\documentclass[12pt]{amsart}

\usepackage[latin1]{inputenc}
\usepackage{amsmath}
\usepackage{amssymb}
\usepackage{amsthm}
\usepackage{ulem,xcolor}
\usepackage{hyperref}
\usepackage{color,graphicx}
\usepackage{geometry}
\theoremstyle{plain}\newtheorem{theo}{Theorem}[section]
\theoremstyle{plain}\newtheorem{cor}[theo]{Corollary}
\theoremstyle{defn}\newtheorem{rem}[theo]{Remark}
\theoremstyle{plain}
\theoremstyle{plain}\newtheorem{lem}[theo]{Lemma}
\theoremstyle{plain}
\theoremstyle{plain}\newtheorem{prop}[theo]{Proposition}
\theoremstyle{defn}
\theoremstyle{defn}
\theoremstyle{defn}

\newenvironment{pfof}[1]{\vspace{1ex}\noindent{\bf Proof of
#1}\hspace{0.5em}}{\hfill\qed\vspace{1ex}}

\def\R{\mathbb{R}}
\def\N{\mathbb{N}}
\def\Z{\mathbb{Z}}

\def\C{\mathbb{C}}
\def\cB{\mathcal{B}}

\def\cJ{\mathcal{J}}\def\eps{\varepsilon}

\def\cL{\mathcal{L}}

\def\cO{\mathcal{O}}

\DeclareMathOperator{\esup}{ess\ sup}

\begin{document}

\title[Sharp error term in LLT and mixing for Lorentz gases with infinite horizon]{Sharp error term in local limit theorems 
and mixing for Lorentz gases with infinite horizon}

\author{Fran\c{c}oise P\`ene}
\address{Univ Brest, Universit\'e de Brest, LMBA,
UMR CNRS 6205, 
6 avenue Le Gorgeu, 29238 Brest cedex, France}
\email{francoise.pene@univ-brest.fr}

\author{Dalia Terhesiu}
\address{Mathematisch Instituut
University of Leiden
Niels Bohrweg 1, 2333 CA Leiden}
\email{d.e.terhesiu@math.leidenuniv.nl}


\begin{abstract}
We obtain sharp error rates in the local limit theorem for the Sinai billiard map (one and two dimensional) with infinite horizon. 
This result allows us to further obtain higher order terms and thus, sharp mixing rates in the speed of mixing of 
dynamically H{\"o}lder observables for the planar and tubular infinite horizon Lorentz gases in the map (discrete time) case. We also obtain an asymptotic estimate for the tail probability of the first return time to the initial cell.
In the process, we study  families of transfer operators for
infinite horizon Sinai billiards perturbed with the free flight function and 
obtain higher order expansions for the associated families of eigenvalues and eigenprojectors.
\end{abstract}
\maketitle

\section{Introduction}

\subsection*{Lorentz gas}
The Lorentz gas has been introduced in \cite{Lorentz} to model the displacement 
of electrons in metals. This model describes the evolution of a point particle moving freely with unit velocity and elastic reflections off pairwise disjoint strictly convex obstacles (with smooth boundary) located $\mathbb Z^d$-periodically (with $d\in\{1,2\}$) in the plane if $d=2$ or on the tube $\mathbb T\times\mathbb R=\mathbb R^2/(\mathbb Z\times\{0\})$.
These obstacles are written $\cO_j+\ell$ with $j=1,...,\mathcal J$ and $\ell\in\mathbb Z^d$ (where $\mathcal J$ a non empty finite set). In this model,
the phase space  consists of positions-unit velocity vectors, which we call configurations.
%a configuration (for a point particle) at some time is given by its position and its unit velocity vector.
%a configuration is a couple position-speed.

\begin{figure}[ht]
\centering
\includegraphics[trim = 10mm 
45mm 20mm 10mm, clip, width=8cm]{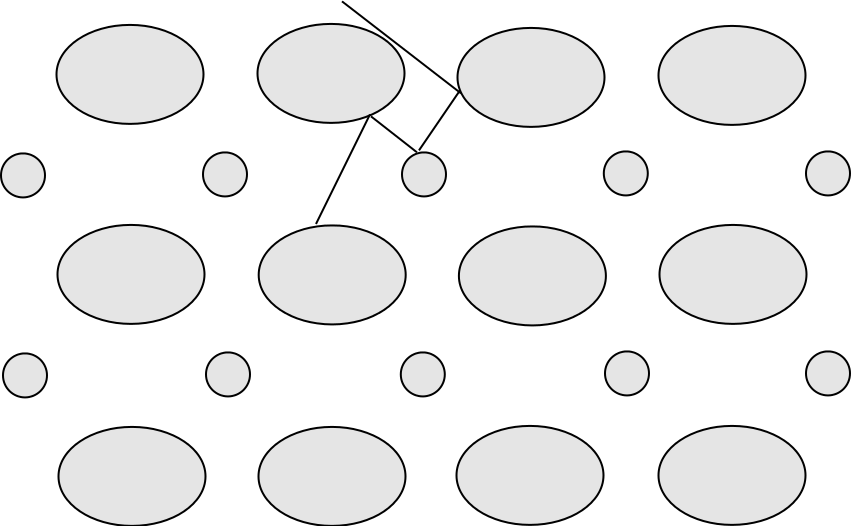}
\caption{A trajectory in a $\mathbb Z^2$-periodic planar domain ($d=2$)}
\label{fig1}
\end{figure}
\begin{figure}[ht]
\centering
\includegraphics[trim = 48mm 89mm 35mm 31mm, clip, width=9cm]{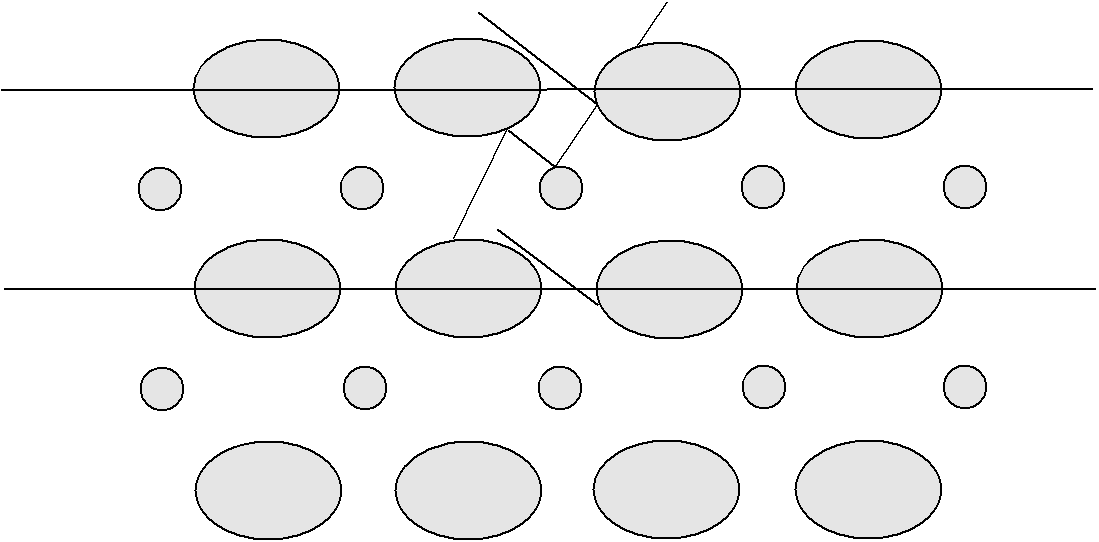}
\caption{A trajectory in a $\mathbb Z$-periodic tubular domain  ($d=1$)}
\label{fig1}
\end{figure}
In this paper, we are interested in discrete time Lorentz gases (the dynamical system corresponding to the collision times) with infinite horizon.
The horizon is said to be infinite if there exists an infinite trajectory intersecting no obstacle and it is said to be finite otherwise. Understanding the stochastic behaviour of the  Lorentz gas in the infinite horizon case is much more challenging than the finite horizon case
and below we recall the main differences along with previous results.
%In this article, we obtain new results on infinite horizon Lorentz gases. (i.e. if there exists a line in the plane touching no obstacle) 
%The Lorentz gas with a single obstacle repeated $\mathbb Z^d$-periodically has infinite horizon and, despite its apparent simplicity  as we will explain the infinite horizon case is much more challenging than the finite horizon case. We will enlighten this by comparing in this introduction some different behaviour. 
The exposition below focuses on  the $\mathbb Z^2$-periodic case, but we mention that similar statements hold for $\mathbb Z$-periodic tubular model.
%{\bf{DT: Add 'in the infinite horizon' in the caption for Fig 1?} Maybe omit Fig 2?}

The space $M$ of configurations of this dynamical system is the set
of postcollisional  unit vectors based on
$\bigcup_{j\in\cJ,\ \ell\in\mathbb Z^2}\partial \cO_j+\ell$. For any $\ell\in\mathbb Z^2$, we define the $\ell$-th cell $M_\ell$ as the set of configurations with position in $\bigcup_{j\in\mathcal J}\cO_j+\ell$. We write $\kappa_n$ for the label in $\mathbb Z^2$ of the
cell in which the particle is at the $n$-th collision time:
\[
\kappa_n=N\quad\mbox{if the position of the particle is based in }\bigcup_{j\in\mathcal J}\cO_j+N\mbox{ at the }n\mbox{th collision time}\, .
\]
The finiteness of the horizon is equivalent to the uniform boundedness of $\kappa_1$.
Whereas the model we consider is purely deterministic (position and velocity at collisions can be computed explicitly in terms of the initial one), $(\kappa_n)_{n\ge 0}$ behaves asymptotically as a random walk on $\mathbb Z^2$. In the finite horizon case, 
 $(\kappa_n)_{n\ge 0}$ behaves asymptotically as a simple symmetric random walk, while in the infinite horizon we have a  symmetric random walk with displacement of infinite variance.
%with important qualitative differences depending whether the horizon is finite or not.
It is worth noticing that the dynamics of the discrete time Lorentz gas is given by the sequence of couples $(X_n,\kappa_n)_{n\ge 1}$ where $X_n$ with values in $M_0$
is the configuration modulo $\mathbb Z^2$ of the position at the $n$-th collision time. More precisely,
\[
X_n=(q,\vec v)\in M_0\quad\mbox{if the configuration at the $n$-th collision time is }(q+\kappa_n,\vec v)\, .
\]
In this representation $(\kappa_n)_n$ corresponds to the evolution of the Lorentz gas at a macroscopic scale while $(X_n)_n$ corresponds to its evolution at a microscopic scale.
The dynamics of $(X_n)_n$ is refereed to as  Sinai billiard and recall that the ergodicity of this dynamical system has been established in the seminal work by Sinai \cite{Sinai70}.
Several limit laws for  Lorentz process are obtained for the invariant probability measure $\mathbb P=\bar\mu$ on $M_0$ absolutely continuous with respect to Lebesgue and for which $(X_n)_{n\ge 0}$ and $(\kappa_{n+1}-\kappa_n=\kappa(X_n))_{n\ge 0}$ is a stationary process.
Under this invariant probability measure, $\kappa_1$ is not square integrable when the horizon is infinite, whereas it
is (as already mentioned) bounded when the horizon is finite.

Limit properties of discrete time finite horizon Lorentz gases
have been obtained in several recent works, among which we mention~\cite{SV04, pene09IHP, FPBS10, Pene18, PTho18,PTho19a}.
Very recent notable progress for continuous time 
%(not considered in the present work) 
Lorentz gases with finite
horizon has been made for mixing local limit theorem by Dolgopyat and Nandori~\cite{DN1, DN2}, for mixing rates by Dolgopyat, Nandori, P{\`e}ne~\cite{DNP} and for suitable versions of CLT  by P{\`e}ne and Thomine~ \cite{PTho19a}. 
%{\color{red} Citer DSV.!!!}
Obtaining the analogue of any the said results in the infinite horizon case is very challenging because $\kappa_1$ has infinite variance. The main difficulty  in carrying out similar arguments  when $\kappa_1$ has infinite variance
comes down to the weak regularity properties of a family $(P_t)_t$ of perturbed operators. 
Whereas in the finite horizon case eigenelements of these operators are $C^\infty$ in $t$, in the infinite horizon case the family of operator is 
%not continuous in $s$ as operators of the reference Banach space $\cB$ but 
just continuous in $t$ as an operator from the Young  Banach space $\cB$ to $L^p$
(see additional explanations in Section \ref{technicalkeyresults}. In this paper we obtain refined expansions going beyond mere continuity estimates and use this to answer unsolved problems in the discrete time model with infinite horizon.
In particular, we obtain: i) higher order (optimal) local limit theorem and mixing; ii) first order expansion of the tail probability of the first return time to the initial cell.
In what follows we provide a simplified statement of our main results, recalling and comparing with previous results. 
%We start with the second since it does not require further terminology.
%Our goal here is to study  the mixing properties in the infinite horizon case as precisely as possible. We also obtain an equivalent for ....

\subsection*{Previous results: CLT and Local Limit Theorem (LLT) for $\kappa_n$}
When the horizon is finite, it has been proved in \cite{BunimovichSinai:1981,BCS91,Young98} that $(\kappa_n)_n$ behaves asymptotically as a Gaussian random variable with the standard normalization in $\sqrt{n}$, that is
\[
\forall A_i<B_i,\quad \lim_{n\rightarrow +\infty}
\mathbb P
%\bar\mu
\left(\frac{\kappa_n}{\sqrt{n}}\in [A_1,B_1]\times[A_2,B_2]\right)=\mathbb P(W\in  [A_1,B_1]\times[A_2,B_2])\, ,
\]
where $W$ is a Gaussian random variable.
When the horizon is infinite, it has been conjectured in \cite{Bleher} and proved rigorously  \cite{SV07} that the Lorentz gas is 
superdiffusive and more precisely that $(\kappa_n)_n$ satisfies the CLT
with nonstandard normalization in $\sqrt{n\log n}$, i.e.
\[
\forall A_i<B_i,\quad \lim_{n\rightarrow+\infty}
\mathbb P
%\bar\mu
\left(\frac{\kappa_n}{\sqrt{n\log n}}\in [A_1,B_1]\times[A_2,B_2]\right)=\mathbb P(W\in  [A_1,B_1]\times[A_2,B_2])\, ,
\]
where $W$  is some Gaussian random variable. 
These two different behaviours can be heuristically explained by the fact that $\kappa_1$ is uniformly bounded when the horizon is finite and is not square integrable when the horizon is infinite, more precisely that
\begin{equation}\label{tailprobakappa}
\mathbb P(|\kappa_1|>N)\approx N^{-2}\, .
\end{equation}
While~\cite{SV07} focuses on the case where only obstacle (modulo $\mathbb Z^2$) is tangent to a same line, we consider the more general case and establish the following formula for the asymptotic variance matrix of $W$
generalizing \cite{SV07}:
\begin{equation}\label{ExprSigma2}
(a_{i,j})_{i,j=1,2}
:=\frac 1{|\partial\bar Q|}\sum_{C\in\mathcal C}\frac {
\mathfrak d_{C}^2}{|w_C|} w_C\otimes w_C
\, ,
\end{equation}
where $w\otimes w$ represents the matrix $\left(\begin{array}{cc}w_1^2&w_1w_2\\w_1w_2&w_2^2\end{array}\right)$ if $w=(w_1,w_2)$ and where $\mathcal C$ is the set of different "corridors" distinct modulo $\mathbb Z^2$ (see picture and begining of Section \ref{freeflight} for details) that can be drawn in $Q$; for each corridor $C\in\mathcal C$, $\mathfrak d_C$ is its width and $w_C$ a vector in $\mathbb Z^2$  in the direction of the corridor with coprime coordinates.

\begin{figure}[ht]
\centering
\includegraphics[trim = 50mm 50mm 10mm 60mm, clip,
 width=6cm]{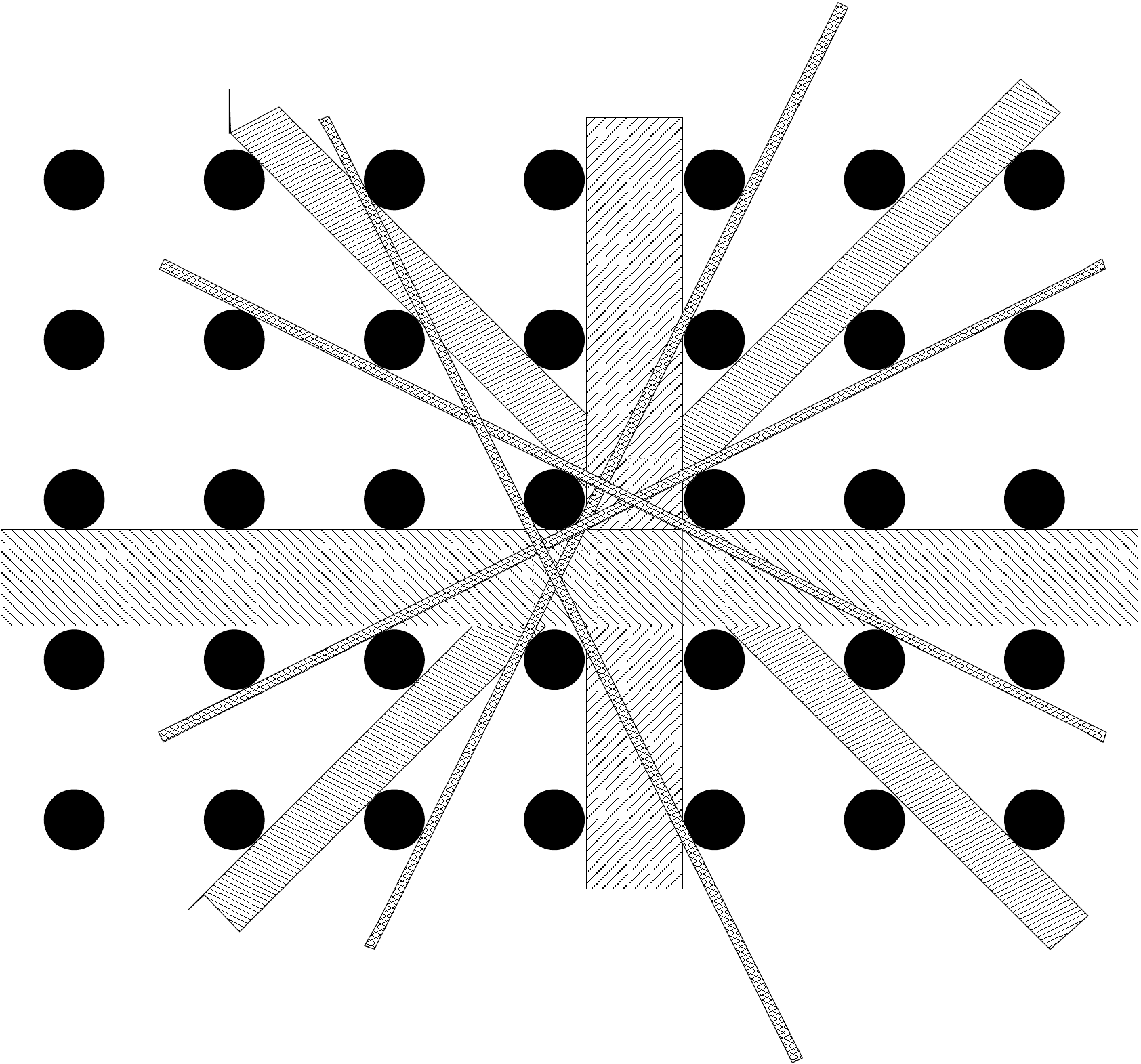}\ \ \ \ 
\includegraphics[trim = 10mm 45mm 30mm 60mm, clip, width=8cm]{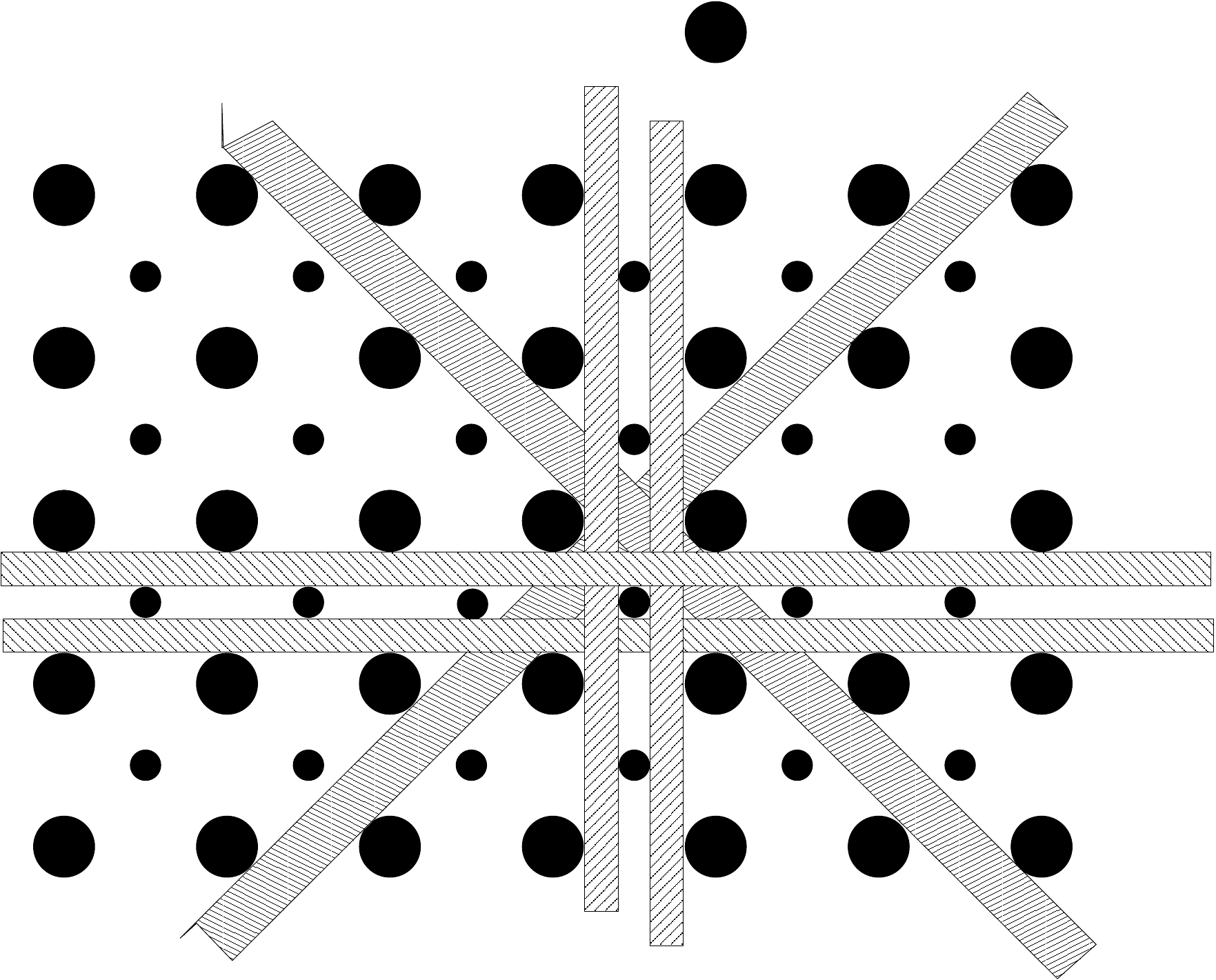}
\caption{Corridors for two different periodic billiard domains}
\label{fig2}
\end{figure}

%Another important issue is  the LLT for $ which is related to the notion of mixing of dynamical systems preserving an infinite measure. 

We recall the  local version of the said CLT, namely LLT gives the asymptotic behaviour of $\mathbb P(\kappa_n=0)$, that is of the probability that the particle returns to the initial cell in $\mathbb Z^2$ at the $n$-th reflection time.
%The standard local limit theorem deals
%with the study of the asymptotic behaviour of $\mathbb P(\kappa_n=0)$ the probability that the particle returns to the initial cell in $\mathbb Z^2$ at the $n$-th reflection time.
Such LLTs have been obtained  by Sz\'asz and Varj\'u in \cite{SV04}
in the finite horizon case and in \cite{SV07} in the infinite horizon case, stating that
\begin{equation}\label{LLT0}
\mathbb P(\kappa_n=0)\sim \frac{\Phi(0)}{a_n^2}\quad\mbox{as}\ n\rightarrow+\infty\, ,
\end{equation}
with $a_n=\sqrt{n}$ in the finite horizon case and $a_n=\sqrt{n\log n}$ in the
infinite horizon case. Here $\Phi$ is the density function of the corresponding Gaussian random variable $W$. Further, \cite{SV07} uses a version of the local theorem to deduce the recurrence of $(X_n,\kappa_n)_{n\ge 1}$.
We recall that when the horizon is finite, the recurrence comes directly from the CLT thanks to a general argument due to Conze \cite{Conze99} and Schmidt \cite{Schmidt98}.
When the horizon is finite, a sharp error rate in the LLT 
has  been obtained by P{\`e}ne~\cite{pene09IHP}
and further extensions including expansions of any order in the LLT have been shown in \cite{Pene18}. 
%The infinite horizon case is much more delicate and challenging because of the lack of moment of order 2. 
In this paper, we establish several extensions of the LLT in the infinite horizon case,  which is much more delicate due to the lack of finite variance. 
We emphasize that previous results on LLT~\cite{SV07} and mixing~\cite{pene09IHP}
for the infinite measure preserving, infinite horizon Lorentz maps reduce to first order terms. 
%Let $X_0$ be the initial configuration (i.e. post-collisional vector) of the particle modulo $\mathbb Z^2$ and $X_n$ is the configuration modulo $\mathbb Z^2$ of the particle at the $n$-th reflection time.
%Let us present quickly our main results.

\subsection*{Tail probability of the first return time to the initial cell (see Theorem~\ref{THMreturntime})}
%A natural question is the study of the length of long excursions out of a cell.
Let $\tau_0$ be the first return time to the initial cell, that is
\[
\tau_0:=\min\{n\ge 1\, :\, \kappa_n=(0,0)\}\, .
\]
When the horizon is finite, it was proved in \cite{DSV} that $\mathbb P(\tau_0>N)\sim\frac 1{\Phi(0)\log N}$. When the horizon is infinite, we show  that
\begin{equation}\label{returntime0}
\mathbb P(\tau_0>N)\sim\frac 1{\Phi(0)\log\log N}\quad\mbox{as}\quad N\rightarrow +\infty\, .
\end{equation}

\subsection*{Study of long free flights (see Lemma~\ref{lemm-tail})}
An important ingredient of our proofs for higher order LLT and mixing exploits higher order expansion of $\kappa$. We recall that in the case where only one obstacle (modulo $\mathbb Z^2$ is tangent to an infinite line contained in the billiard domain),  \cite[Proposition 6]{SV07} shows that
%As seen previously in \eqref{tailprobakappa}, the study of long free flights in Lorentz process with infinite horizon is crucial. The proof of  \eqref{tailprobakappa} follows from a computation (see \cite[Prop 6]{SV07} in the case where only one obstacle (modulo $\mathbb Z^2$ is tangent to an infinite line contained in the billiard domain) %based on geometric arguments ensuring that, for $L,w\in\mathbb Z^2$
\[
\mathbb P(\kappa_1=L+Nw)\sim\frac{\mathfrak a}{|Nw|^3}+o(N^{-3})\, .
\]
This form was enough in \cite{SV07} for obtaining the LLT.  In our proofs, we need an error of $\mathcal O(N^{-4})$. 
%This was enough in \cite{SV07} for proving the nonstandard limit theorem. But, unfortunately, this is not enough for our purpose.
Our Lemma~\ref{lemm-tail} provides (under an additional regularity assumption) a precise estimate  of the following form (with explicit constants $\mathfrak a$ and $\mathfrak a'$):
%for the probability of a long free flight
% along a direction $w\in\mathbb Z^2$
\begin{equation}\label{Plongtraj}
\mathbb P(\kappa_1=L+Nw)=\frac{\mathfrak a}{|Nw|^3}+\frac{\mathfrak a'}{|Nw|^4}+o(N^{-4})\, .
\end{equation}

\subsection*{Class of functions} All our results below hold for dynamically H{\"o}lder observables and refer to Section \ref{mainresults} for precise definition
and further assumptions, where required.

\subsection*{Higher order in  Mixing LLT (see Theorem~\ref{THM0} for details)}
We study a stronger version of the LLT, the mixing LLT (MLLT)
which consists in establishing
\begin{equation}\label{MLLT0}
\mathbb E[f(X_0).
\mathbf 1_{\{\kappa_n=N\}}.g(X_n)]\sim \frac{\Phi(0)\mathbb E[f(X_0)]
\mathbb E[g(X_0)]}{n\log n} \quad\mbox{as}\ n\rightarrow +\infty\, ,
\end{equation}
when $\mathbb E[f(X_0)]\mathbb E[g(X_0)]\ne 0$. MLLT is about asymptotic independence of $(X_0,\kappa_n,X_n)$ as $n\rightarrow +\infty$.
%Whereas an expansion of any order has been proved in \cite{Pene18} in the finite horizon case, obtaining higher order terms is challenging in the infinite horizon case.
When $N=0$, we prove in particular that
\begin{align*}
\mathbb E[f(X_0).
\mathbf 1_{\{\kappa_n=0\}}.g(X_n)]&=\frac{\Phi(0)\mathbb E[f(X_0)]
\mathbb E[g(X_0)]}{n\log(n\log n))}  +O\left( \frac 1{n(\log n)^2}\right)\, ,
%\\ &\ \ \ \ \ \  + \cO\left(n^{-\frac 32}(\log n)^{-\frac 52}\right)\right)\, ,
\end{align*}
with additional error terms detailed in Theorem~\ref{THM0}.
This result ensures in particular that 
\[
\mathbb E[f(X_0).
\mathbf 1_{\{\kappa_n=0\}}.g(X_n)]=\mathbb E[f(X_0)]
\mathbb E[g(X_0)] \Phi(0)\left(\frac 1{n\log n}-\frac{\log\log n}{n(\log n)^2}\right)  +O\left( \frac 1{n(\log n)^2}\right)
\]
providing a second order term in \eqref{MLLT0}.
When $N\ne 0$ is fixed, we obtain an intermediate term ensuring in particular that
\[
\mathbb E[f(X_0).
\mathbf 1_{\{\kappa_n=\lfloor  N\sqrt{n
%\log(n
\log n
%)
}\rfloor\}}.g(X_n)]=- \frac{ \nabla\Phi(N) \cdot \mathfrak K(f,g)}{(n\log(n\log n))^{\frac {3}2}} +O\left( \frac 1{(n\log n)^{\frac {3}2}\log n}\right)
\]
if $\mathbb E[f(X_0)] \mathbb E[g(X_0)]=0$ with $\mathfrak K(f,g)$ a $\mathbb R^2$-valued bilinear form  linearly independent of $\mathbb E[f(X_0)]
\mathbb E[g(X_0)]$.

\subsection*{Mixing of general observables in the infinite measure case (Theorem~\ref{THM1})} 
The local limit theorem for $\kappa_n$ is strongly related to the notion of mixing of dynamical systems preserving an infinite measure, that is the study of the behaviour of quantities of the form 
\[
\int_M f(X_0,\kappa_0).g(X_n,\kappa_n)\, d\mu\, ,
\]
where $\mu$ is the measure absolutely continuous with respect to the Lebesgue measure which is invariant under $(X_n,\kappa_n)\mapsto (X_{n+1},\kappa_{n+1})$.
A mixing result without error term has been established in~\cite[Theorems 1.1]{Pene18}. In the present  Theorem \ref{THM1} we improve this result to
\begin{align}\label{EQTHM1}
&\int_M f(X_0,\kappa_0).g(X_n,\kappa_n)\, d\mu=O\left(n^{-\frac 32}(\log n)^{-\frac 52}\right)\\
\nonumber&+ \left(\frac {\Phi(0)}{n\log(n\log n)}+O((\log n)^{-1})\right)\int_Mf(X_0,\kappa_0)\, d\mu\, \int_Mg(X_0,\kappa_0)\, d\mu\, .
\end{align}
In  the \emph{finite} horizon case, an expansion of any order of the form $\int_M f(X_0,\kappa_0).g(X_n,\kappa_n)\, d\mu=\sum_{m=1}^{K}\frac{c_m(f,g)}{n^{m}}+o(n^{-K})$ with $c_0(f,g)=\Phi(0)\int_Mf(X_0)\,d\mu\int_Mg(X_0)\,d\mu$ with $(c_m(f,g))_m$ linearly independent
has been established in \cite{Pene18}.  Such a result implies, in particular, that for any positive integer $m$, there exist couples of observables $(f,g)$ such that $\int_M f(X_0,\kappa_0).g(X_n,\kappa_n)\, d\mu\approx n^{-m}$.
%This result provides exact mixing rate for zero mean observables.
In the  infinite horizon case, \eqref{EQTHM1} does not imply directly the  optimal result for zero mean observables and we address this in a result of independent interest.
%Nevertheless, we were able to prove such estimates thanks to a specific result, as explained below.

\subsection*{Mixing of zero integral observables in the infinite measure case (Theorem~\ref{THEOCob})}
In the infinite horizon case,  we  obtain different 
%forms of 
rates of mixing
for null integral functions $f,g$.
%such that $\int_Mf(X_0)\,d\mu\int_Mg(X_0)\,d\mu=0$, depending on the regularity of $f$ and $g$. 
In particular, we show that
\begin{equation}\label{decor1int0}
\int_M f(X_0,\kappa_0).g(X_n,\kappa_n)\, d\mu\approx (n\log (n\log n))^{-1}n^{-m}\, ,
\end{equation}
when $g$ and $f$ are coboundaries of respective orders $k,\ell$ with $k+\ell=m$
%is a coboundary of order $m$ 
(a coboundary of order $m$ is a function $g$ of the form $g(X_n,\kappa_n)=g_0(X_n,\kappa_n)-g_0(X_{n-1},\kappa_{n-1})$, where 
$g_0$ is a coboundary of order $m-1$, considering here that a function with non null integral is a coboundary of order 0)
and that
\begin{equation}\label{decor2int0}
\int_M f(X_0,\kappa_0).g(X_n,\kappa_n)\, d\mu\approx (n\log (n\log n)) ^{-2}\, 
\end{equation}
when $f=\phi.(\mathbf 1_{M_N}+\mathbf 1_{M_{-N}}-2M_0)$ with $\phi$ and $g$
having non null integral.
\subsection*{Method of proof and main challenges}
We heavily exploit that the discrete infinite horizon Lorentz gas is a $\mathbb Z^d$ extension of the dynamical system
$X_n\mapsto X_{n+1}$, that is of the form $(X_n,\kappa_n)\mapsto (X_{n+1},\kappa_{n+1}=\kappa_n+\kappa(X_n))$. We refer to Section \ref{mainresults} for further details.
%Our approach uses crucially the fact that the dynamical system has the form $(X_n,\kappa_n)\mapsto (X_{n+1},\kappa_{n+1}=\kappa_n+\kappa(X_n))$, i.e. is a $\mathbb Z^d$ extension of the dynamical system
%$X_n\mapsto X_{n+1}$. This structure detailed in Section \ref{mainresults} combined with operator perturbation techniques is the key point to obtain our results.
%establish mixing rate for null integral observables.
%A precise description of the terminology is provided in Section \ref{mainresults}.\\
%Let us give some details about the techniques used.
We recall that the LLT (without error term) established by Sz\'asz and Varj\'u in~\cite[Theorem 13]{SV07} 
%for discrete infinite horizon Lorentz gas
uses the abstract results~\cite[Theorem 2]{BalintGouezel06} of B{\'a}lint and Gou{\"e}zel, which establishes the non standard CLT for  observables of  infinite variance (so, not $L^2$) observables acting on exponential Young towers.
The method in~\cite{BalintGouezel06} was developed to establish the non standard CLT for the stadium billiard, which, among other things, was possible  due to the work of Young~\cite{Young98} and Chernov~\cite{Chernov99}.
%In the present work we obtain higher order terms in both,  the LLT~\cite[Theorem 13]{SV07} and mixing~\cite[Theorem 1.1]{Pene18}.
%LLT for Sinai billiards with infinite horizon goes via 
%Our strategy here establishing the LLT for the 
%with the use of Young tower due to the work of Young~\cite{Young98} and Chernov~\cite{Chernov99}.
A classical tool for establishing LLTs for chaotic dynamical systems is the perturbed transfer operator method.
As clarified in Section \ref{technicalkeyresults}, a serious challenge for obtaining error terms in the LLT in~\cite[Theorem 13]{SV07} is that 
we need 'sufficiently high'  expansions (not just continuity) for the  families of eigenvalues and eigenprojectors  associated with the transfer operator perturbed with non $L^2$ functions. 
To give a first insight into this difficulty we point out that given the Young Banach space $\cB$ (see Section~\ref{technicalkeyresults} for definition),
the family of operators is continuous as a family of elements of $\cL(\cB\rightarrow L^p)$ for some $p>1$ (but not as elements of $\cL(\cB)$). 
%as a family of elements of $\cL(\cB_0\rightarrow L^p)$ for some $p>1$ (but not as elements of $\cL(\cB_0)$). 
%Thus, obtaining 'sufficiently high'  expansions for the  families of eigenvalues and eigenprojectors  associated with the transfer operator
%perturbed with non $L^2$ functions is highly non trivial and constitutes a consequent part of this paper.
 Propositions~\ref{PROP1} and~\ref{prop-lambda}  provide the expansions  for the  families of eigenvalues and eigenprojectors we use in the proofs of our main results, already mentioned. 
The proofs of  Propositions~\ref{PROP1} and~\ref{prop-lambda} build on the framework put forward in~\cite{BalintGouezel06}
using several geometrical estimates established in~\cite{SV07}. Under assumptions specific to Young towers for Sinai billiards with infinite horizon,  Propositions~\ref{prop-lambda} and ~\ref{PROP1} can be viewed as refined version of the main technical results in~\cite{BalintGouezel06}.
For a summary of the new ingredients used in the proofs of these propositions we refer to the text after the statement of Proposition~\ref{prop-expproj} in  Section \ref{sec:Pi}.
%Another challenge is to use Propositions~\ref{PROP1} and ~\ref{prop-lambda} in order to obtain sharp mixing rates for dynamically H\"older observables for the infinite measure preserving, infinite horizon Lorentz maps.
% Our main resultsin this sense are Theorems~\ref{THM0} and~\ref{THM1}. Moreover,  in Theorem~\ref{THEOCob}, we obtain improved error terms for particular zero mean functions, including coboundaries, with leading terms (non zero in general) of different orders.

Let us conclude this introduction with a few remarks on the main examples of infinite measure preserving systems of physical interest. 
We have already recalled infinite measure preserving periodic Lorentz gases. As already mentioned, LLTs for Sinai billiards  can be translated into first order mixing for periodic Lorentz gases.
A different type of systems of physical interest are intermittent maps, preserving an infinite measure.
To fix notation, we recall a well known interval map, namely the Liverani Saussol Vaienti map \cite{LSV}, $T:[0,1]\to [0,1]$, $T(x)=x(1+2^\alpha x^\alpha)$ if $x\le 1/2$ and $T(x)=2x-1$ if $x>1/2$. Such maps can be viewed as one sided Markov renewal chains with heavy dependencies.
We only consider $\alpha\ge 1$, as in this case $T$ is preserving an infinite measure.
 First order mixing for such maps was obtained by  Gou\"ezel~\cite{Gouezel11} and by Melbourne and Terhesiu~\cite{MT12}. 
In these works, the $1/\alpha$-stable LLT (for the first return map of intermittent maps, a much more simple dynamical setting than that of Sinai billiards)  is a minor part of the mechanism; in fact, for $\alpha<2$ this type of LLT can be bypassed (see~\cite{MT12}).
In short, the mechanisms for obtaining mixing are  highly non trivial generalizations of : 
\begin{itemize}
\item[i)] the procedure of obtaining the  asymptotic  of renewal sequences for simple symmetric random walks (in the sense that the LLT is the only required ingredient),  in the case of periodic Lorentz gases;
\item[ii)] proofs of strong renewal theorems (for renewal sequences with infinite mean) for one sided Markov renewal chains, in the case of  intermittent maps.
%intermittent maps are generalizations of Markov chains jumping from $k=0$ to a random level $\ell\in\mathbb N$ and going down from $k$ to $k-1$ until it reaches $0$. As clarified in~\cite {Gouezel11,MT12}, 
\end{itemize}
Higher order mixing  for (not necessarily Markov) infinite measure preserving intermittent interval maps  have been obtained first in~\cite{MT12} and refined in~\cite{Terhesiu16}; such results have been 
generalized to suspension flows in~\cite{MelTer17, BMT18}. 
%We mention that w
We do not know whether results similar to the results on mixing rates for mean zero 
functions and coboundaries in the setup of Lorentz maps (in~\cite{Pene18} and also
in Theorem~\ref{THEOCob} therein) hold for infinite measure preserving intermittent interval maps; the previous results~\cite{MT12, Terhesiu16, MelTer17, BMT18} for
 zero integral
 observables are confined to big $O$ error terms. 
\subsection*{Outline of the paper
}
In Section \ref{mainresults}, we introduce
the precise version of the $\mathbb Z^d$-periodic billiard model with infinite horizon we consider and state our main results  Theorems~\ref{THMreturntime}~\ref{THM0} and~\ref{THM1}.
and~\ref{THEOCob}. In section~\ref{technicalkeyresults}, we present our key technical results Propositions~\ref{PROP1},~\ref{prop-lambda}.
In Section \ref{freeflight}, we obtain an expansion  for the probability of long free flights, which is crucial for Proposition~\ref{prop-lambda}.
In Section \ref{sec:Pi}, we prove our first key  result Proposition~\ref{PROP1}, stating
an expansion of the dominating eigenprojector.
In Section \ref{sec:lambda}, we prove our second key result Proposition~\ref{prop-lambda}, which gives an expansion of the eigenvalue using results contained in the two previous sections.
In Section \ref{LLT}, we state an expansion in the LLT
 in a general context and use it to prove our main result
as well as a general decorrelation result for some $\mathbb Z^d$-extensions.
Some further  technical estimates, as well as the proof of Theorem~\ref{THMreturntime}, are included in Appendix.

\section{Model and main results}\label{mainresults}
We start by presenting the two dimensional case, that is when $d=2$.
We consider a planar billiard domain $Q$ given by
$
Q:=\mathbb R^2\setminus\bigcup_{j \in\mathcal J,\, \ell\in\mathbb Z^2} \mathcal O_{j,\ell}$, 
with $\mathcal J$ a non empty finite set and with $\mathcal O_{j,\ell}:=\mathcal O_{j}+\ell$, where the $\mathcal O_j$ are open convex set with boundary $C^3$ and with nonzero curvature, such that $\mathcal O_{j,\ell}$ have pairwise disjoint closures.
We assume that the billiard has infinite horizon, i.e. that $Q$ contains at least one line.
When $d=1$, we replace $Q$ and the $\cO_{j,\ell}$ by their quotients modulo $\mathbb Z\times\{0\}$, that is $Q$ is a subset of the tube $\mathbb T\times\mathbb R$.

%one corridor $A_0+[0,t]\vec w^\perp+\mathbb R \vec w$, where $A_0\in\partial Q$, where $t>0$ and where $\vec w,w^\perp\in\mathbb R^2$ are two orthogonal vectors.\\

\begin{figure}[ht]
\centering
\includegraphics[trim = 10mm 80mm 20mm 60mm, clip, width=12cm]{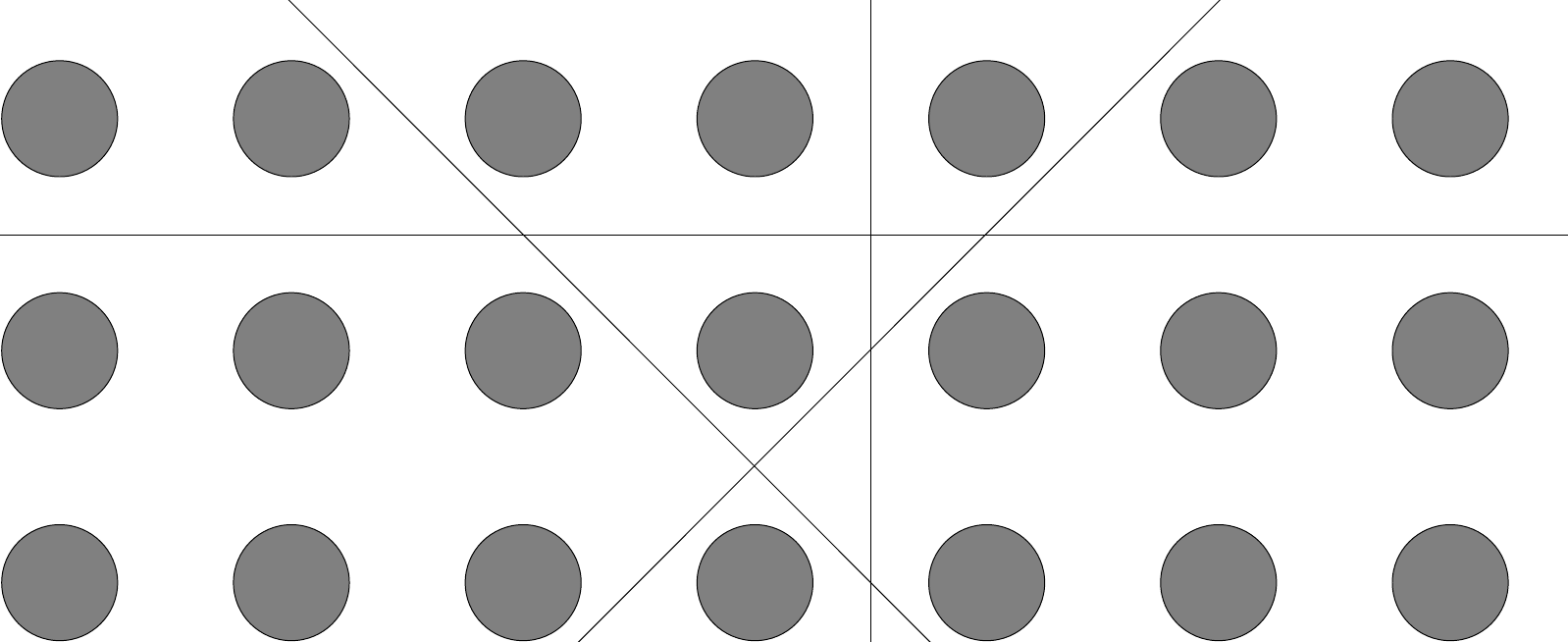}
\caption{A  periodic billiard domain with 4 infinite horizon directions}
\label{fig1}
\end{figure}

We denote by $(M,T,\mu)$ the original billiard dynamical system map corresponding to collision times. The configuration space $M$ is the set of couples of
position and velocity $(q,\vec v)$ with $q\in \partial Q$ and $\vec v$ 
a unit reflected vector, i.e. a unit vector $\vec v$ oriented inside $Q$.
The billiard map $T$ maps a configuration $(q,\vec v)$ corresponding to a collision time to the configuration
corresponding to the next collision time. The measure $\mu$ is the measure on $M$ with density proportional to $\cos\varphi$, where $\varphi$ is the angle of $\vec v$ with the normal vector to $\partial Q$ directed inside $Q$, normalized so that
$\mu(\{(q,\vec v)\in M\, :\, q\in\bigcup_{j\in\mathcal J} \partial \mathcal O_{j}\})=1$.
The infinite measure preserving dynamical system $(M,T,\mu)$ is canonically isomorphic to the $\mathbb Z^d$-extension of $(\bar M,\bar T,\bar\mu)$ 
by $\kappa:\bar M\rightarrow\mathbb Z^d$,
%, with $d\in\{1,2\}$,
where $(\bar M,\bar T,\bar\mu)$ is the probability preserving billiard dynamical system in the billiard domain in 
$\bar Q
=Q/\mathbb Z^2\subset \mathbb T^2$ if $d=2$ (and $\bar Q
=Q/(\{0\}\times\mathbb Z)\subset \mathbb T^2$ if $d=1$)
and $\bar\mu$ the probability measure with density proportional to $\cos\varphi$.
Let us give the formula of assymptotic variance $\Sigma^2$.
We take $\Sigma^2=(a_{i,j})_{i,j=1,...,d}$ with $(a_{i,j})_{i,j=1,2}$ given in \eqref{ExprSigma2}.
Note that \eqref{ExprSigma2} coincide with the formula of \cite[Theorem 20]{SV07} in the case of a single obstacle since in this case a corridor corresponds to four points $x\in R_0$ s.t. $Tx=x$ (two positions, one on each side of the corridor, and two directions $\pm\frac{w_C}{|w_C|}$).
{\bf We assume $\Sigma^2$ invertible}, i.e. the interior of $Q$ contains at least $d$ unbounded lines not parallel to each other (one may observe that when $d=1$ the invertibility of $\Sigma^2=a_{1,1}$ just means that $a_{1,1}\ne 0$). 
%corresponding to its period in $\mathbb Z^2$
%and $\mathfrak n_C$is the number of different  (distinct modulo $\mathbb Z^2$) points of tangencies of $\partial Q$ with the corridor $C$.\\
For any $x\in\mathbb R^d$, we write $\Phi_{\Sigma^2}(x):=\frac{e^{-\frac 12 \Sigma^{-2}x\cdot x}}{\sqrt{(2\pi)^d\det \Sigma^2}}$, $\Phi_{\Sigma^2}$
is the density function of a Gaussian distribution with expectation 0 and variance
matrix $\Sigma^2$. 
%We start by stating a result about the first return time $\tau_0$ to the initial cell.
We set $\tau_0:=\min\{n\ge 1\, :\, \kappa_n=(0,0)\}$ with $\kappa_n:=\sum_{k=0}^{n-1}\kappa\circ\bar T^k$. 
\begin{theo}[Tail probability of the first return time in the initial cell]\label{THMreturntime}
\begin{align}
\label{returntime}
\mbox{If }d=2 :&\quad\bar\mu(\tau_0>N)\sim\frac {2\pi \sqrt{\det \Sigma^2}}{\log\log N}\quad\mbox{as}\quad N\rightarrow +\infty\\
\label{returntimedim1}
\mbox{If }d=1 :&\quad
\bar\mu\left(\tau_0>N\right)\sim \sqrt{\frac {2 a_{1,1}\log N}{\pi N}}\quad\mbox{as}\quad N\rightarrow +\infty\, .
\end{align}
\end{theo}
Our other main theorems will require the following additional assumption (ensuring \eqref{Plongtraj}):
\begin{equation}\label{H0}
\mbox{$\partial Q$ is $C^4$ at points $q\in\partial Q$ such that the tangent line to $\partial Q$ at $q$ is contained in $Q$.}
\end{equation}
Let us introduce the class of smooth functions we consider. Let  $R_0\subset M$ be the set of reflected vectors that are tangent to $\partial Q$. The billiard map $T$
defines a $C^1$-diffeomorphism from $M\setminus(R_0\cup T^{-1}R_0)$ onto  $M\setminus(R_0\cup T R_0)$.
For any integers $k\le k'$, we set $\xi_k^{k'}$ for the partition
of $M\setminus \bigcup_{j=k}^{k'} T^{-j}R_0$ in connected components and $\xi_k^\infty:=\bigvee _{j\ge k}\xi_k^j$.
For any $\phi:M\rightarrow\mathbb R$ and any $\eta\in(0,1)$, we set
\begin{equation}\label{DefiHolder}
L_{\phi,\eta}:=\sup_{k\ge 0}\sup_{A\in \xi_{-k}^k}\sup_{x,y\in A}\frac{|\phi(x)-\phi(y)|}{\eta^k}
\quad\mbox{and}\quad
\Vert\phi\Vert_{(\eta)}:=\Vert\phi\Vert_\infty+L_{\phi,\eta}\, .
\end{equation}
%When considering 
For $\phi: \bar M\rightarrow \mathbb R$, we define $R_0$, $\xi_{k}^j$, $L_{\phi,\eta}$, $\Vert\phi\Vert_{(\eta)}$ in the same way with $(\bar M,\bar T)$ instead of $(M,T)$.
%We define the same notions with same notations with
%$(\bar M,\bar T)$ instead of $(M,T)$.
\begin{theo}[Mixing local limit theorem]\label{THM0}
%Let $(M,T,\mu)$ be the Lorentz map with infinite horizon. 
Assume
 \eqref{H0}.
% holds true.
Let $\phi,\psi: \bar M\rightarrow \mathbb R$ be  two measurable functions such that $\Vert \phi\Vert_{(\eta)}+\Vert \psi\Vert_{(\eta)}<\infty$,
then, uniformly in $N\in\mathbb Z^d$,
\begin{align}\label{LLTfinal}
&\int_{\bar M} \phi 1_{\{\kappa_n=N\}}\psi\circ \bar T^n\, d\bar\mu
=\frac{\mathbb E_{\bar\mu}[\phi]\mathbb E_{\bar\mu}[\psi]}{(n\log( n\log n))^{\frac d2}}\left(\Phi_{\Sigma^2}\left(\frac N{\sqrt{n\log (n\log n)}}\right)
+O((\log n)^{-1})\right)\\
\nonumber& - \nabla\Phi_{\Sigma^2}\left(\frac N{\sqrt{n\log (n\log n)}}\right)
\cdot\frac{\mathfrak K(\phi,\psi)}{(n\log (n\log n))^{
\frac{d+1}2}}
%\left( \mathbb E_{\mu_{\Delta}}[g\Pi'_0h]+ i\mathbb E_{\mu_{\Delta}}[\hat\kappa_k  g]\mathbb E_{\mu_\Delta}[h]+i\mathbb E_{\mu_\Delta}[g]  \mathbb E_{\mu_\Delta}[\hat\kappa_kP^kh]\right)
%\left(1-\frac12\left(d+2-\frac{\Sigma^{-2}N\cdot N}{n\log n}\right)
% \frac{\log\log n}{\log n}\right)
 +O\left(\frac{
1}{(n\log n)^{\frac {d+1}2}\log n}\right)\, .
\end{align}
with $\mathfrak K(\phi,\psi):=\mathbb E_{\bar\mu}[\psi]\sum_{j\ge 0}
\mathbb E_{\bar\mu}[\kappa\circ \bar T^{
j}\phi]+\mathbb E_{\bar\mu}[\phi]\sum_{j\le -1}
\mathbb E_{\bar\mu}[\kappa\circ \bar T^{
j}\psi]$, these sums being absolutely convergent, and
with $\nabla$ the gradient operator.
\end{theo}
Observe that the two first terms of \eqref{LLTfinal} both contain expansions since the first term can be rewritten $
 \frac {\mathbb E_{\bar\mu}[\phi]\mathbb E_{\bar\mu}[\psi]}{(n\log n)^{\frac d2}}\left(
\Phi_{\Sigma^2}\left(\frac N{\sqrt{n\log n}}\right)
 \left(1+
\left(\frac{\Sigma^{-2}N\cdot N}{n\log n}-d\right)\frac{\log\log n}{2\log n}\right)
 + O\left(\frac 1{\log n}\right)
     \right)$ and the second one
$-\nabla\Phi_{\Sigma^2}\left(\frac N{\sqrt{n\log n}}\right)
\cdot\frac {\mathfrak K(\phi,\psi)}{(n\log n)^{\frac{
d+1}2}}
\left(1-\frac12\left(d+2-\frac{\Sigma^{-2}N\cdot N}{n\log n}\right)
 \frac{\log\log n}{\log n}\right)$. 
The assumption that the interior of $Q$ contains at least $d$ non parallel infinite lines ensures that $\det \Sigma^2\ne 0$.
For any $N\in\mathbb Z^d$, we write $M_N$ for the set of $(q,v)\in M$ such that
$q\in \bigcup_{j\in\mathcal J}\partial \mathcal O_{j,N}$.
\begin{rem}
Observe that the bilinear forms $\mathbb E_\mu[\phi]\mathbb E_{\mu}[\psi]$
and $\mathfrak K(\phi,\psi)$ are linearly independent.
Indeed, under the assumptions of Theorem~\ref{THM0},  then $\mathfrak K(\phi-\phi\circ \bar T,\psi)=-\mathbb E_\mu[\psi]\mathbb E_\mu[\kappa.\phi\circ \bar T]$
which is non zero in general. 
\end{rem}
\begin{theo}
[Decorrelation in infinite measure]
\label{THM1}
Assume \eqref{H0}.
Let $\eta,\gamma\in(0,1)$.
%Let $(M,T,\mu)$ be a Lorentz map with infinite horizon satisfying  such that the interior of $Q$ contains at least $d$ infinite lines not parallel to each other.
Let $\phi,\psi:M\rightarrow\mathbb R$ be two measurable observables such that
%uniformly bounded observables constant on every connected component of $M$ (i.e. on each obstacles) such that
$
\sum_{N\in\mathbb Z^2}(1+
|N|^\gamma
)\left(\Vert \phi 1_{M_N}\Vert_{\infty}+\Vert \psi 1_{M_N}\Vert_{\infty}\right)<\infty$
and
%\quad\mbox{and}\quad
$\sum_{N\in\mathbb Z^2}L_{\phi 1_{M_N},\eta}<\infty$.
Then
\[
\int_M \phi. \psi\circ T^n\, d\mu=
\frac {
\Phi_{\Sigma^2}(0)}{(n\log (n\log n))^{\frac d2}}
\int_M\phi\, d\mu\, \int_M\psi\, d\mu+O\left(\frac 1{(n\log n)^{\frac d2}\log n}\right)\, .
\]
If moreover
$
%\exists\gamma\in(0,1),\quad
\sum_{N\in\mathbb Z^2}(1+
|N|^{1+\gamma}
)\left(\Vert \phi 1_{M_N}\Vert_{\infty}+\Vert \psi 1_{M_N}\Vert_{\infty}\right)<\infty$, then
\begin{align*}
\int_M \phi. \psi\circ T^n\, d\mu&=
\frac {
\Phi_{\Sigma^2}(0)}{(n\log (n\log n))^{\frac d2}}
\left(1
%-\frac d2\frac{\log\log n}{\log n}
+O\left(\frac 1{\log n}\right)\right)\int_M\phi\, d\mu\, \int_M\psi\, d\mu\\
&+O\left(\frac 1{(n\log n)^{\frac {d+1}2}
\log n
}\right)\, .
\end{align*}
\end{theo}
Again, in the above result, $(n\log(n\log n))^{-\frac d2}$ can be replaced by $\frac{1-\frac d2\frac{\log\log n}{\log n}}{(n\log n)^{\frac d2}}$ providing a second term in $\frac{\log\log n}{(n\log n)^{\frac d2}\log n}$.
When $\phi$ or $\psi$ has zero mean, Theorem~\ref{THM1} only provides
an estimate in $O(\cdot)$. 
Luckily, our method enables us to establish sharp decorrelation rates
 % in $a_n^{-d-k}$ or $a_n^{-d-k}\log n$ 
for  zero  mean observables under natural regularity assumptions. This includes smooth
coboundaries.
% of the $\mathbb Z^d$-extension, including 
\begin{theo}[Sharper decorrelation rates for particular functions with zero integral]\label{THEOCob}
Assume  \eqref{H0}.
%Let $(M,T,\mu)$ be the Lorentz map with infinite horizon satisfying \eqref{H0} and such that the interior of $Q$ contains at least $d$ infinite lines not parallel to each other.
Let $\gamma\in(0,1)$.
\begin{itemize}
\item[(a)]
Let
$\phi,\psi:M\rightarrow\mathbb C$ be observables such that 
$\sum_{N\in\mathbb Z^d}(1+
|N|^\gamma
)\left(\Vert \phi 1_{M_N}\Vert_{\infty}+\Vert \psi 1_{M_N}\Vert_{\infty}\right)<\infty$
%\quad\mbox{and}\quad
and
$\sum_{N\in\mathbb Z^d}L_{\phi 1_{M_N},\eta}<\infty$.
Then
\begin{align*}
\int_M \phi.\psi\circ(id-T)^m\circ T^n\, d\mu&=-
\frac {\Phi_{\Sigma^2}(0)
\int_M\phi\, d\mu\, \int_M\psi\, d\mu}{(n\log (n\log n))^{\frac d2}n^m}
(-2)^{-m}d(d+2)\cdots(d+2m-2)\\
&\ \ \ +O\left((n\log n)^{-\frac d2}n^{-m}(\log n)^{-1}\right)\, ,
\end{align*}
with $(-2)^{-m}d(d+2)\cdots(d+2m-2)=(-1)^m m!$ when $d=2$.
%with $\Delta$ the Laplacian operator and with $\mathcal G:x\mapsto e^{-\frac {x\cdot x}2}$.
\item[(b)] Let $N\in\mathbb Z^d$.
If $\phi:M\rightarrow\mathbb C$ is invariant by translation of positions by $\mathbb Z^d$ and satisfies  $\Vert\phi\Vert_{(\eta)}<\infty$ and if there exists
$\delta\in(0,1]$ such that $\sum_{N\in\mathbb Z^d}\left(\Vert \psi 1_{M_N}\Vert_{(\eta)}+N^\delta \Vert  \psi 1_{M_N}\Vert_{\infty}\right)<\infty$,
then, setting $f_0=\phi(1_{M_N}+1_{M_{-N}}-2\times 1_{M_0})$,
\begin{align*}
\int_M f_0.\psi\circ T^n\, d\mu
&=-
\frac {\Phi_{\Sigma^2}(0)\int_{M_0}\phi\, d\mu\, \int_M\psi\, d\mu}{(n\log (n\log n))^{\frac d2+1}}\left(\Sigma^{-2}N\cdot N
+O((\log n)^{-1})\right) \\
&
+O\left(\frac{\log n}{(n\log n)^{\frac {d+3}2}}+a_n^{-d-2-\delta}\right)
\, .
\end{align*}
\end{itemize}
\end{theo}
Again $((n\log (n\log n))^{-\frac d2-m}$ in the above formulas can be replaced by
$\frac{1-\frac{(d+2m)\log\log n}{2\log n}}{(n\log n)^{\frac d2+m}}$
an expansion with two terms.
Let us make several observations on this last result.
First, whereas in the finite horizon case, we only have leading terms in $n^{-d/2-m}$ in the decorrelation of smooth functions, in the infinite horizon case we can have leading terms in $n^{-m}(n\log n)^{-d/2}$ 
%(or more precisely in $n^{-m}(n\log (n\log n))^{-d/2}$ )
but also in $(n\log n)^{-d/2-1}$.
%(or more precisely in $(n\log(n\log n))^{-1}(n\log (n\log n))^{-d/2}$).
Other orders are possible. For example, we can easily adapt our proof
to obtain sharp decorrelation rate in $n^{-\frac d2-m-1}(\log n)^{-d/2-1}$ in case (b) with $\psi$ a coboundary of order $m$.
Observe that, when $m=1$, Case (a) of Theorem~\ref{THEOCob} corresponds to the study of $\int_M \phi.\psi\circ T^n\, d\mu-\int_M \phi.\psi\circ T^{n-1}\, d\mu$ and the dominating term given by (a) is equivalent to the difference between the two leading terms of $\int_M \phi.\psi\circ T^n\, d\mu$ 
and of $\int_M \phi.\psi\circ T^{n-1}\, d\mu$ obtained in Theorem~\ref{THM1}.
The leading term is of order $(n\log n)^{-d}-((n+1)\log (n+1))^{-d}\sim d(n\log n)^{-\frac d2-1}\log n $.
Observe that the case when $\phi$ is a coboundary and the case of two coboundaries is included in item (a) of theorem~\ref{THEOCob}. Indeed, by $T$-invariance of $\mu$,
\begin{align*}
\int_{M}\phi\circ(id-T).\psi\circ T^n\, d\mu
&= \int_{M}\phi.\psi\circ T^{n}\, d\mu-\int_{M}\phi.\psi\circ T^{n-1}\, d\mu\\
&=- \int_{\Omega}\phi.\psi\circ(id-T)\circ T^{n-1}\, d\mu\, ,
\end{align*}
and thus
\[
\int_{M}\phi\circ(id-T)^r.\psi\circ(id-T)^s\circ T^n\, d\mu=(-1)^r
      \int_{M}\phi.\psi\circ(id-T)^{r+s}\circ T^{n-r}\, d\mu\, .
\] 
%\begin{rem}
%Our proof is based on a local limit theorem for $\kappa$ with error term. Note that the usual LLT corresponds to applying Theorem~\ref{THM1} with $\phi=\psi=1_{M_0}$ which provides~:
%\[
%\bar\mu\left(\sum_{m=0}^{n-1}\kappa\circ\bar T^m=0\right)=\frac 1{\sqrt{(2\pi)^d\det \Sigma^2}\, (n\log n)^{\frac d2}}\left(1+\frac d2\frac{\log\log n}{\log n}+O\left(\frac 1{\log n}\right)\right)\, .
%\]
%\end{rem}
Theorems~\ref{THM0} and~\ref{THM1} are contained in the more technical Theorems~\ref{LLT20} and~\ref{LLT2} (valid for a class of less regular observables)
which are consequences
of Theorem~\ref{LLT1} that gives higher order terms in LLT and speed of mixing 
under abstract assumptions on families of eigenvalues and eigenprojectors. Most of our work consist in proving the results contained in the next section
enabling the application of Theorem~\ref{LLT1} to the quotiented tower $(\Delta,f,\mu)$ and $\hat\kappa$.

\section{Key technical estimates}\label{technicalkeyresults}
We focus on the case $d=2$ since the similar results in the case $d=1$ follow from them.
We do not assume here that the interior of $Q$ contains at least $d$ non parallel infinite lines.
Our results are based on Fourier analysis and Young towers.
It is known that $(\bar M,\bar T,\bar\mu)$ is a factor, under a projection written $\mathfrak p_1:\bar\Delta\rightarrow \bar M$, of 
a Young tower $(\bar\Delta,\bar f,\bar\mu_\Delta)$ with stable 
and unstable curves. By factorizing/collapsing the stable curves we can reduce it to a one-dimensional Young tower $(\Delta, f,\mu_\Delta)$ by $\mathfrak p_2:\bar\Delta\rightarrow\Delta$.
Throughout we let $\hat\kappa:\Delta\to\Z^d$ be the version of $\kappa$ on $\Delta$, that is $\hat\kappa\circ\mathfrak p_2=\kappa\circ\mathfrak p_1$ (the existence of such a $\hat\kappa$ comes from the fact that $\kappa$ is constant on the stable curves).
Let $P$ be the transfer operator for $(\Delta, f,\mu_\Delta)$. We consider the family $(P_t)_{t\in\mathbb R}$ of perturbations of $P$ given by 
$P_t:=P\left(e ^{i t\cdot \hat\kappa}\cdot \right)$, where $\cdot$ denotes the standard scalar product on $\mathbb R^d$. Note that $P_0\equiv P$. 
As shown in~\cite{SV07} thanks to \cite{Young98,Chernov99} there exist $\beta\in(0,\pi]$ and $\theta\in(0,1)$ such that for every
$t\in[-\beta,\beta]^d$,
\begin{equation}
\label{spgap-Sz}
P_t^n=\lambda_t^n\Pi_t+N_t^n\, ,
\end{equation}
%with
\begin{equation}
\label{spgap-Sz-bis}
\mbox{with}\quad\quad\quad\quad\sup_{t\in[-\beta,\beta]^d}\Vert N_t^n\Vert_{\mathcal B}+\sup_{t\in[-\pi,\pi]^d\setminus[-\beta,\beta]^d}\Vert P_t^n\Vert_{\mathcal B}=O\left(\theta^n\right)\, ,
\end{equation}
%and
\begin{equation}
\label{spgap-Sz-bisbis}
\mbox{and}\quad\quad\quad\quad
\lim_{t\rightarrow 0}\left\Vert \Pi_t-\mathbb E_{\mu_\Delta}[\cdot]\mathbf 1\right\Vert_{\cL(\cB\rightarrow L^1(\mu_\Delta))}=0\, ,
\end{equation}
where $\mathcal B$ is a complex Banach space of $\mathbb C$-valued and $\mu_\Delta$-integrable functions
(considered by Young in \cite{Young98}).
As in the finite horizon case, $(P_t)_t$ defines a family of operators on $\cB$.
But, whereas in the finite horizon case $t\mapsto P_t$ is $C^\infty$ from $\mathbb R$ to $\cL(\cB)$ the set of linear continuous operators on $\cB$, in the infinite horizon case we can just say that $t\mapsto P_t$ is 
%not continuous from $\mathbb R$ to $\cL(\cB)$
continuous from $\mathbb R$ to $\cL(\cB\rightarrow L^1(\mu_\Delta))$.
Additionally the derivative of $P_t$ at $0$ should be $ P(i\hat\kappa\cdot)$ which is not in $\cL(\cB)$ not even in $\cL(L^1)$
(see Lemma~\ref{discont}) but is in $\cL(L^p\rightarrow L^q)$
as soon as $\frac 1p+\frac 12> \frac 1q$.
As shown in~\cite[Proposition 6]{SV07}, $\bar\mu(\kappa=L+N\vec w)\sim C_{L,\vec w}N^{-3}(1+o(1))$ for some $C_{L,\vec w}\ge 0$, strictly positive for some $L$ if there exists a line of direction $\vec w$ contained inside $Q$.
%except for a finite number of lattices $L+\mathbb Z\vec w$.
%$(L,\vec w)\in(\mathbb Z^2)^2$ with $\vec w$ prime .
%; so, $\bar\mu(|\hat\kappa|=N)\sim CN^{-3}$.
Combining this with~\eqref{spgap-Sz}, ~\eqref{spgap-Sz-bis}, several lemmas obtained inside~\cite[Proof of Theorem 13]{SV07}
and ~\cite[Proposition 4.17]{BalintGouezel06}, Sz\'asz and Varj\'u established in \cite{SV07} the following estimate
\begin{equation}\label{devlambda}
\lambda_t=1-\Sigma^2 t\cdot t\, \log (1/|t|)(1+o(1)),\quad\mbox{as}\ t\rightarrow 0\, .
\end{equation}

Inside the proof of our Proposition~\ref{prop-lambda} below, which gives a higher order expansion of $\lambda_t$ at $t=0$, we provide a precise summary of the results in~\cite{SV07} needed to obtain~\eqref{devlambda}.
Let us recall that \eqref{spgap-Sz}, \eqref{spgap-Sz-bis}, \eqref{spgap-Sz-bisbis}
and \eqref{devlambda}
imply the LLT for $\bar T$ and 
thus first order mixing (speed of mixing) for $T$ as in~\cite[Theorem 1.1]{Pene18}.
Here we are interested in higher order terms in both the LLT and mixing for suitable classes of functions.
%We state our main result for the two dimensional maps $T$ and $\bar T$ and mention the difference in the one dimensional case (more detailed results in dimension 1 and 2 are given in Section \ref{resultsbill}).
%To do so we need to introduce some notation.

\begin{prop}\label{PROP1}
%Let $\eta\in(0,1)$ and 
%Fix $p>2$.
There exists a functional Banach space $\mathcal B_0$ of bounded functions 
and $\Pi'_0 \in\mathcal L(\mathcal B_0\rightarrow L^s(\mu_\Delta))$ 
for every $s\in[1,2)$ such that, for any $p>2$, any $p'\in[1,\frac43)$, 
%\hookrightarrow\mathcal B\hookrightarrow L^p(\mu_\Delta)$ such that, for every $p'\in[1,\frac 4 3)$, 
any  $\gamma\in (1,\min(2\frac{p-1}p,\frac 4{p'}-2))$, there exists $C>0$ 
such that
% for every $w\in\mathcal B_0$,
%there exists
$\|(\Pi_t-\Pi_0-t\cdot \Pi_0')(w)\|_{\mathcal L(\cB_0\rightarrow L^{
p'}(\mu_\Delta))}\le  C\,|t|^\gamma$ as $t\rightarrow 0$.
%Moreover $\Pi'_0$ satisfies
%\[
%\Pi_0' (w)(x)=\Pi_0' (w)\circ\pi_0(x)+i\mu_\Delta(w)\sum_{m=0}^{\omega(x)-1}\hat\kappa\circ f^m(\pi_0(x))\, ,
%\]
%and $1_Y\Pi'_0(w)$ is in $\mathcal B_0$.
%$1_Y\Pi'_0(vw)=\mu_\Delta(vw)Q'_0(1_Y)+\left(-2\pi \left( \mu_\Delta(Y) \sum_{j\ge 0}\int_\Delta \hat\kappa\circ f^j \, .vw\, d\mu_\Delta+ \int_\Delta vw \, d\mu_\Delta\sum_{j\ge 0}\int_{\Delta} \hat\kappa\, 1_{Y}\circ f^{j+1}\, d\bar\mu_\Delta -\mu_Y(Q'_0)\right)1_Y-\mu(Q'_0(1_Y))\right) 1_Y\, .$
\end{prop}
Proposition~\ref{PROP1} follows directly from Proposition~\ref{prop-expproj} by taking $v=1_\Delta$.
\begin{rem}
A precise formula  for $\Pi'_0$ is given by~\eqref{eq-derivPi0} of Proposition~\ref{prop-expproj} together with~\eqref{eq1yderivpi} of Proposition~\ref{prop-expproj-base}.
\end{rem}
\begin{prop}
\label{prop-lambda}
Assume \eqref{H0}.
As $t\to 0$, 
$\lambda_t=1- \Sigma^2 t\cdot t\, \log(|t|^{-1})
%\frac 12\sum_{A\, :\, (A,w)\in\mathcal A}\frac{ d_{A}^2}{2|\partial\bar Q|\, |w| }
%(t\cdot w)^2\log||t|^{-1}|+O(t^2)\\
%&=1-\frac 12\sum_{C\in\mathcal C}\mathfrak n_C\frac{ \mathfrak d_{C}^2}{|\partial\bar Q|\, |w_C| }
%(t\cdot w_C)^2\log||t|^{-1}|
+O(t^2)$
%$\lambda_t=1- \frac 14\sum_{(A,w)\in\mathcal A}\sum_{(A,B)\in E_{L, w}}\frac{ d_{(A,B)}^2}{|\partial\bar Q|\, |w| }(t\cdot w)^2|\log|t|^{-1}|+O(t^2)$,
with the notations introduced at the begining of Section~\ref{freeflight}.
% $\mathcal A$, $E_{L,w}$ and $d_{(A,B)}$ defined in Section \ref{sec:lambda}.
\end{prop}
Since the norms of $\mathbb R^d$ are all metrically equivalent, Proposition~\ref{prop-lambda} is true for any choice of norm $|\cdot|$ on $\mathbb R^d$.

\section{Estimate of the probability of long free flights}\label{freeflight}
Our proof of Proposition~\ref{prop-lambda} provided in Section~\ref{sec:lambda}
is based on estimates established in Section~\ref{sec:Pi} and on an higher order expansion of the tail of the free flight $\kappa$.
Lemma~\ref{lemm-tail} below gives the required expansion (to be precise we just use $O(N^{-4})$). We mention that Lemma~\ref{lemm-tail} can be  regarded as a refinement of  \cite[Proposition 6]{SV07}. First, it  provides higher order of the tail of the free flight
 and second, 
we compute the dominating term without any restriction on
%in the model described below we allow more than one 
the number of types
 of scatterer at the boundary of a given corridor.
%Throughout we use $|m|$  prove this is the following higher order expansion of the tail of the displacement functionas the norm of vectors $m \in \Z^2$.
%Let $L,\vec w\in\mathbb Z^2$.
As in \cite{SV07}, we introduce the notion of corridor. 
%A corridor is a bi-infinite strip contained in $Q$ and delimited by two lines that are tangent to $\partial Q$.
We call corridor $C$ %of lattice $L+\mathbb Z w$ 
a strip contained in $Q$ delimited by  $(A+\mathbb R w)\cup(B+L+\mathbb R  w)$ for some 
$A,B\in \bigcup_{j\in\mathcal J}\partial\mathcal O_j$ and $L,w\in\mathbb Z^2$.
We write $\mathcal C$ for the set of corridors.
For any corridor $C\in\mathcal C$, we write $\mathfrak d_{C}$ for its width and
$w_{ C}$ for its direction that is the prime $w\in\mathbb Z^2\setminus\{0\}$ with non negative first coordinate integer coordinates (and with positive second coordinate if the first one is null) such that the $ C$ contains a line of direction
$w_{C}$. 
%We also write $\mathfrak n_C$ for the number of points $E\in  \bigcup_{j\in\mathcal J}\partial\mathcal O_j$ such that $E+\mathbb Z^2$ intersects the boundary of the corridor $C$.

\begin{figure}[ht]
\centering
\includegraphics[trim = 0mm 10mm 10mm 100mm, clip,
 width=7cm]{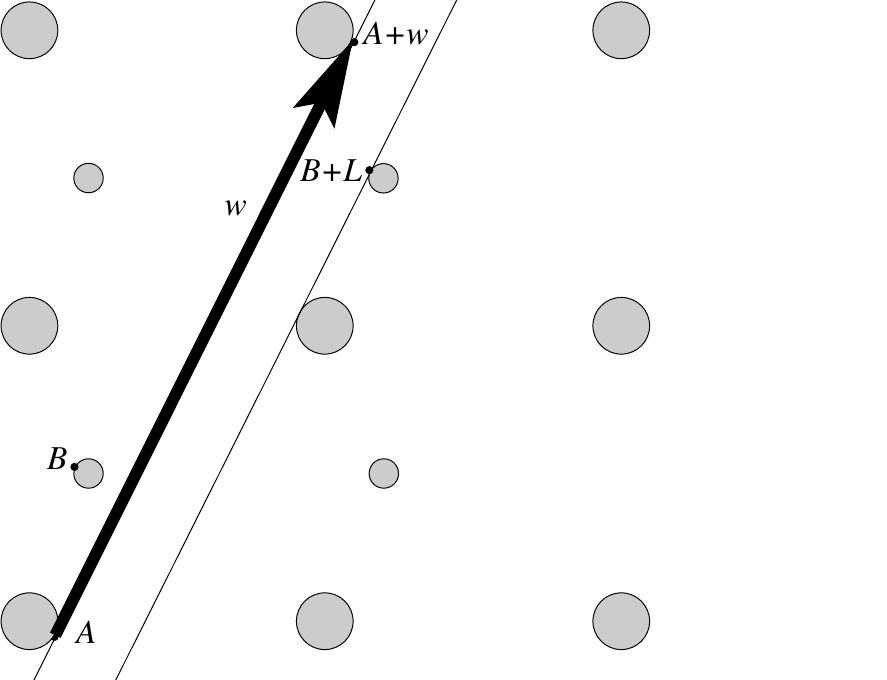}
\caption{A corridor}
\label{fig3}
\end{figure}

Observe that a corridor is fully determined by $A$ and that, given $A$ defining a corridor, there are two possible choices of prime $w\in\mathbb Z^2$, one being the opposite of the other.
We write $\mathcal A$ for the finite set of $(A,w)\in (\bigcup_{j\in\mathcal J}\partial\mathcal O_j)\times\mathbb Z^2$, with $w$ prime, such that there exists $(B,L)$ as above and $d_{A}$ for the width of the corresponding corridor.

\begin{rem}
Let $(A,w)\in \mathcal A$ then $(A, v=w/| w|)\in\bigcap_{k\in\mathbb Z} T^{k}(R_0)$.
% which is locally finite. 
%Indeed, w
With the above notations, let $\mathcal O$ be the obstacle containing $B+L$.
If
$(q',v')\in\bigcap_{k\in\mathbb Z} T^{k}(R_0)\setminus\{(A,v)\}$ is close to $(A,v)$, %but with $(q',v')\ne (A,v)$. T
the line $q'+\mathbb R v'$ passes between $\cO+Mw$ and  $\cO+(M+1)w$ for some $M\in\mathbb Z$ without passing through them. But, this is impossible if  $ |\angle( v, v')|<\arctan\frac{2r}{|w|}$ with $r$ the minimal radius of curvature of $\mathcal O$). Thus $(A,v)$ is isolated in $\bigcap_{k\in\mathbb Z} T^{k}(R_0)$. Therefore $\mathcal A$ is finite.
\end{rem}
% as well as the sets $E_{L, w}$.
%We call corridor any maximal strip of direction $ w$ contained in $Q$ (where we call "strip of direction $ w$" any convex set union of lines of direction $ w$). We denote $\mathcal C_{ w}$ for the set of corridors of direction $ w$. Morever, for any $\mathcal C\in\mathcal C_{\vec w}$, we write $d_{\mathcal C}$ for the width of $\mathcal C$ and $E_{\mathcal C,L}$ for the set of couples $(A,B)\in (\bigcup_{j\in\mathcal J}\partial\mathcal O_j)^2$such that the segment $[A,B+L]$ crosses $\mathcal C$ (i.e. $\mathcal C$ is the convex hull of $(A+\mathbb R\vec w)\cup(B+L+\mathbb R \vec w)$).
Given $L,w\in\mathbb Z^2$, 
we write $E_{L, w}$ for the set of $(A,B)$ such that $(A,w)\in \mathcal A$
and $B\in\partial Q$ is as described above (on the other line delimiting the corridor).

%$E_{L,w}$ is not empty, with $w$ prime and $L$ with minimal nonnegative coordinates (minimal with respect to the alphabetical order: $(a_1,a_1)<(b_1,b_2)$ if $a_1<b_1$ or if $a_1=b_1$ and $a_2<b_2$) in the set $L+\mathbb Z w$.

%The set $E_{L}$ is not restricted to a single point. Indeed, whereas $A$ and $w$ determine a corridor, several couples $(A,B)$ can correspond to the same corridor without having the same trajectories $(A,B+L+\mathbb N w)$. First if several obstacles are tangent to the same line boundary of a corridor, for a same $A$, we can choose different $B$ (for example, on Figure 3, $B$ has been chosen on the smallest obstacle, but for the same $A$ we could have chosen $B$ on the biggest obstacle). Second the couples $(A,B)$ and $(B,A)$ do not correspond to the same trajectories. Finally if we keep the same couple $(A,B)$ but replace $w$ by $-w$, the trajectories are not the same. 

\begin{lem}
\label{lemm-tail}
Assume \eqref{H0}.
Let $ w\in\mathbb Z^2$ prime and $L\in\mathbb Z^2$,
% such that there exists corridor a strip of direction $ w$ contained inside $Q$,
then
\[
\bar\mu(\kappa=L+N w)
=\sum_{(A,B)\in E_{L, w}}\left(\frac{ d_{A}^2\, \mathfrak a_{(A,B)}}{2|\partial\bar Q|\, |w|N^3 }+\frac{\mathfrak a'_{(A,B)}}{2|\partial \bar Q|\, (|w|\, N)^4}\right)
%=\sum_{\mathcal C\subset\mathcal C_{ w}}\frac{\# E_{\mathcal C,L} d_{\mathcal C}^2}{2|\partial\bar Q|N^3 }
+o(N^{-4})\, ,
\]
with 
$\mathfrak a_{(A,B)}:= \frac{AA'\, B''B}{|w| ^2}$ and
\[
\mathfrak a'_{(A,B)}:=
\frac{3d_A^2AA'\, B''B(AA'+B''B-2\frac{w.\overrightarrow{A(B+L)}}{|w|})}2
%\\&\quad\quad\quad
+
\frac{d_A^3}2\left( B''B\left(\frac 1{\mathfrak c_{A}}-\frac 1{\mathfrak c_{A'}}
\right)+AA'\left(\frac 1{ \mathfrak c_{B}}-\frac 1{ \mathfrak c_{B''}}\right)\right)\, .
\]
%where $\mathfrak c_E$ is the curvature of $\partial Q$ at $E$ for any $E\in\bigcup_{i\in\mathcal I}\partial Q$ and 
where $B''$ is the first point of $\partial Q$ met by the half-line $B-\mathbb R_+ w$
and where $A'$ is the first point of $\partial Q$ met by the half-line $A+\mathbb R_+ w$ and where $\mathfrak c_{E}$ is the curvature of $\partial Q$ at any $E\in\partial Q$.
\end{lem}
For example, in the case of a billiard with a single original obstacle (if $\mathcal J=1$), then
$\bar\mu(\kappa=(N,1))=\frac{d^2}{2|\partial \bar Q|}(N^{-3}+ 3(1-h)N^{-4}+o(N^{-4}))$ (then $A=A'$, $B=B''$ and so $\frac 1{\mathfrak c_A}-\frac 1{\mathfrak c_{A'}}=\frac 1{\mathfrak c_B}-\frac 1{\mathfrak c_{B''}}=0$), where $d$ is the vertical distance between two different obstacles and $h$ is the horizontal gap between the highest point and the lowest point of the obstacle ($h$ can be negative or positive).
\begin{cor}\label{lemm-tailrem}
Assume \eqref{H0}. Let us fix a corridor $C$.
The sum  $\mathfrak a_C$ of the $\mathfrak a_{(A,B)}$ over the couples $(A,B)$ associated to $C$ is 2.
\end{cor}
\begin{proof}
%Let us write $w=w_C$.
% and for the sum of the $\mathfrak a_{(a,b)}$ the couples $(A,B)$ associated to $C$.
The corridor $C$ is delimited by two lines $\cL$ and $\cL'$.
Let $(A_i)_{i\in \mathbb Z/N\mathbb Z}$ (resp. $(B_j)_{j\in \mathbb Z/N'\mathbb Z}$) be the points of $\bigcup_{j=1}^{\mathcal J}\partial \cO_j\cap (\cL+\mathbb Z^2)$ (resp. $\bigcup_{j=1}^{\mathcal J}\partial \cO_j\cap(\cL'+\mathbb Z^2)$) ordered so that $A'_i=A_{i+1}+\ell_i$ and $B''_j=B_{j-1}+\ell'_j$ for some $\ell_i,\ell'_j\in\mathbb Z^2$.
Then $\sum_i\overrightarrow{A_iA'_i}=\sum_j\overrightarrow{B''_jB_j}= w_C$,
and so \[
\mathfrak a_C=\sum_{i,j}\left(\mathfrak a_{(A_i,B_j)}+\mathfrak a_{(B_j,A_i)}\right)
=\sum_{i,j}\frac{A_iA'_i\, B''_jB_j+B_jB'_j\, A''_iA_i}{|w_C|^2}
=2\, .
\]
\end{proof}
%\section
%Assumption for the moment: one single obstacle repeated $\mathbb Z^d$-periodically with $C^4$-smooth boundary.

%We consider a maximal corridor contained in $Q$ of the form
%\[
%A_0+[0,d_0] w^\perp+\mathbb R  w\, ,
%\]
%with $A_0\in \partial O_{i_0}$, $\vec w\in\mathbb R^2$ and
%$L>0$ and with $\vec w^\perp$ a unit vector perpendicular to $\vec w$.
%Without loss of generality, we assume that $\vec w\in\mathbb Z^2$ and is minimal.
\begin{proof}[Proof of Lemma~\ref{lemm-tail}]
Observe that, for every $\varepsilon$, outside the $\varepsilon$-neighbourhood of $\bigcap_{k\in\mathbb Z} T^{k} (R_0)$, $\kappa$ is uniformly bounded.  Moreover, outside the $\varepsilon$-neighbourhood of points $(A_0, w/| w|)$
with $(A_0,B)\in E_{L, w}$, for some $N_0$ large enough, $\kappa\not\in L+(N_0+\mathbb N)  w$.
Therefore, it is enough to prove that, for any
 $(A_0,B)\in E_{L, w}$ and any $i_0,i_1\in\mathcal J$ so that $A_0\in\partial \mathcal O_{i_0}$ and $B\in\partial \mathcal O_{i_1}$,
%$A_0$ is on the boundary of the corridor.
%Let $B_0\in\partial O_{i_1,\ell_1}\cap (A_0+d_0 w^\perp+\mathbb R \vec w)$
%be on the other side of the boundary of the corridor, such that $\langle \vec w, \overrightarrow{A_0B_0}\rangle\ge 0$, with $\ell_1$ of minimal length.
%We wish to prove that
\begin{equation}\label{mu0L+Nw}
\mu(M_{(i_0 ,0)}\cap T^{-1}(M_{(i_1,L+N w)}))=\frac{d_0^2 \, \mathfrak a_{(A_0,B)}}{2|\partial\bar Q|\, |w|\, N^3}
+\frac{\mathfrak a'_{(A_0,B)}}{2|\partial \bar Q|\, N^4} + o(N^{-4})\, ,
\end{equation}
where
$d_0:=d_{A_0}$ is the width of the corresponding corridor and $M_{(i,\ell)}$ the set of elements of $ M$ based on a point in $\partial \mathcal O_{i,\ell}$.
%Once this will be proved the conclusion will follow by taking the sum over the couples $( A_0,B_0)$ satisfying the above description since for $N$ large enough$|\kappa|\ge N$ implies that the starting point is close to some $(A_0, w/| w|)$ with $A_0$.\\
%Let $A_0$ be the highest point of $\partial\mathcal O_{(0,0)}$.
Set $B_0:=B+L$.
Let $\mathcal A_N$ be the image set of
$M_{(i_0 ,0)}\cap T^{-1}(M_{(i_1,\ell_1+N w)})$ by the projection map $(q,\vec v)\mapsto (A,\vec v)$
where $A$ is the intersection point of $q+\mathbb R\, \vec v$
with the line $A_0+\mathbb R w^\perp$, where $ w^\perp$ is the unit vector perpendicular to $\vec w$ 
directed inside $Q$ at $A_0$.
Because of the invariance of $\bar\mu$ under the billiard map
and under the inversion $(q,\vec v)= (q,-\vec v)$,
thus, using also coordinates $(x,\alpha)$ on $(A_0+\mathbb R w^\perp)\times S^1$
with $x$ the absciss on the line $A_0+\mathbb R w^\perp$ and $\alpha$ for angle between $w$ and $\vec v$, we know that 
\begin{equation}\label{mu0L+Nw2}
\mu(M_{(i_0 ,0)}\cap T^{-1}(M_{(i_1,\ell_1+N w)})=\frac 1{2|\partial \bar Q|}\int_{\mathcal A_N}\cos\alpha\, dx\, d\alpha \, .
\end{equation}
\begin{figure}[ht]
\centering
\includegraphics[trim = 40mm 65mm 40mm 20mm, clip, width=8cm]{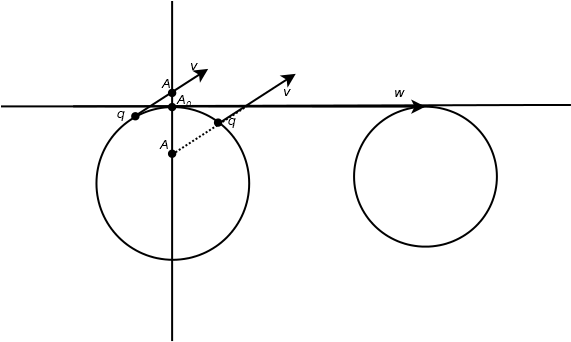}
\caption{Projected point $A$.}
\label{fig0}
\end{figure}

We write $\mathcal O'$ for the first obstacle touched by $A_0+\mathbb R_+ w$
(at $A'_0\in\partial\mathcal O'$) and $\mathcal O"$ for the first obstacle touched by $B_0-\mathbb R_+ w$ (at $B"\in\partial\mathcal O"$).
Let us write $\vec v_\alpha$ for the unit vector making angle $\alpha$ with $ w$, for $\alpha\in[0,\pi/2]$ and we consider $x$ close to 0.
First, for a given position $A=A_0+x w^\perp$ close to $A_0$ with $x<0$,
%the point 
$(A,\vec v_\alpha)$ is in $\mathcal A_N$ if and only if
the first obstacle (other than $\mathcal O_{(i_0,0)}$) met by
the half-line $A+\vec v_\alpha\, \mathbb R_+$ is $\mathcal O_{(i_1,\ell_1+N w)}$, that is
%(for $N$ large enough\footnote{$N$ large enough so that $A+v_{\alpha_{N-1}}\mathbb R$ intersects $\mathcal O_{(N,1)}$.})
if and only if
%\begin{equation}\label{majobeta}
$ \max(\beta,\alpha_{N})\le \alpha<\alpha'_{N}$.
%\end{equation}
where we write $\beta$ for the angle such that the line $A+\vec v_\beta\mathbb R$ is tangent to $\mathcal O'$ from above, $\alpha_{N}$ for the angle such that the line $q+\vec v_{\alpha_N} \mathbb R$
is tangent to $\mathcal O_{(i_1,\ell_1+N w)}$ from underneath (at some point $B'_N$) and $\alpha'_{N}$ for the angle such that the line $q+\vec v_{
\alpha'_N} \mathbb R$
is tangent to $\mathcal O"+N w$ from underneath (at some point $B"_N$).
\begin{figure}[ht]
\centering
\includegraphics[trim = 32mm 0mm 40mm 30mm, clip,
 width=15cm]{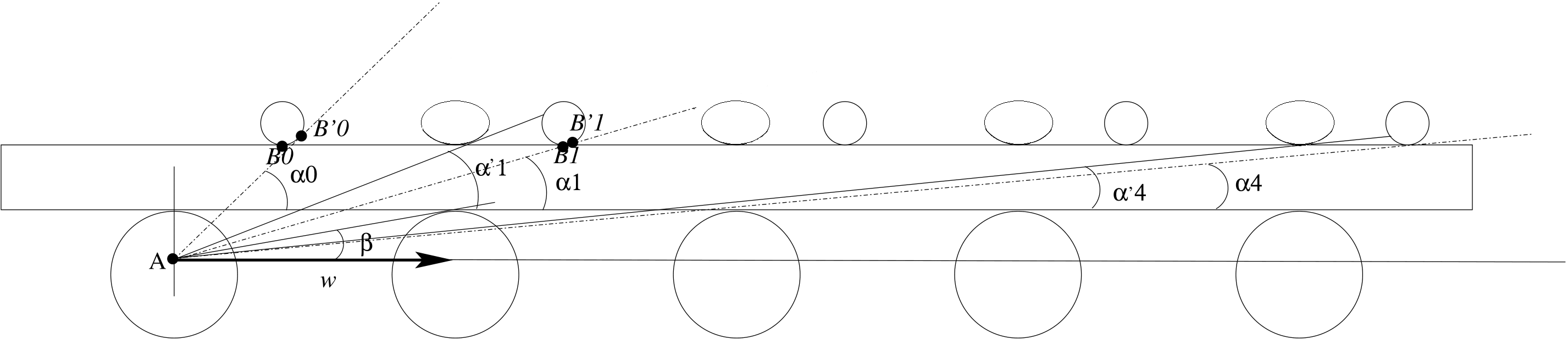}
\caption{The $\alpha_i$ and $\alpha'_i$ in a corridor}
\label{fig1b}
\end{figure}
Second, for a given position $A=A_0+x w^\perp$ close to $A_0$ with $x>0$, 
%the point 
$(A,\vec v_\alpha)$ is in $\mathcal A_N$ if and only if
the first obstacle met by
the half-line $A+\vec v_\alpha\, \mathbb R_+$ is $\mathcal O_{(i_1,\ell_1+Nw)}$ and if
$A+\vec v_\alpha\, \mathbb R_-$ intersects $\mathcal O_{(i_0,0)}$, that is if
%\begin{equation}\label{majobeta'}
$\max(\beta',\alpha_{N})\le \alpha<\alpha'_{N}$,
%\end{equation}
with $\beta'$ the smallest positive angle such that $A+\vec v_{\beta'}\mathbb R$
is tangent to $\partial \mathcal O_{i_0}$.
Thus, the integral \eqref{mu0L+Nw2} becomes
\begin{eqnarray}\label{mu0L+Nw3}
%\mu(M_{(i_0 ,0)}\cap T^{-1}(M_{(i_1,\ell_1+N w)})=
\frac {I_N^{-}+I_N^+}{2|\partial \bar Q|}:=\frac {1}{2|\partial \bar Q|}
\left( \int_{x^-}^0 \int_{\max( \beta,\alpha_N) }^{\alpha'_N} \cos\alpha\, d\alpha \, dx \, 
+  \int_0^{x^+} \int_{\max( \beta',\alpha_N)}^{\alpha'_N} \cos\alpha\, d\alpha \, dx\right)\, .
\end{eqnarray}
The proof continues now by giving precise estimates of
$\beta,\beta'$ and $\alpha_N, \alpha'_N$ and also of the numbers $x^- < 0 < x^+$ beyond which there are no solutions $\alpha$ to respectively
$ \max(\beta,\alpha_{N})\le \alpha<\alpha'_{N}$ and $\max(\beta',\alpha_{N})\le \alpha<\alpha'_{N}$ anymore.\\
\noindent\underline{\bf Estimate of $\beta$.}
Let $q(u)=A'_0-q_1(u)\frac{w}{|w|}-q_2(u)w^\perp$ be the  parametrization of $\partial\cO'$ by arc-length such that $q(0)=A'_0$ (i.e. $(q_1(0),q_2(0))=(0,0)$) and $A\in q(s)+\mathbb R q'(s)$. 
Observe that $(q_1'(0),q'_2(0))=(1,0)$ and
$(q_1''(0),q''_2(0))=(0,\mathfrak c_{A'_0})$ where  $\mathfrak c_{A'_0}$ is the curvature of $\partial \mathcal O'$ at $A'_0$.
Then
%\begin{equation}\label{formule1tanbeta}
$\tan\beta=\frac{q'_2(s)}{q'_1(s)}= \frac{\mathfrak c_{A'_0}s+\frac {q_2'''(0)}2 s^2}{1+\frac {q_1'''(0)}2 s^2}+ O(s^3)$.
%= \mathfrak c_{A'_0}s+\frac {q_2'''(0)}2 s^2+ O(s^3)\, ,
%\end{equation}
So
%\begin{equation*}%\label{formule1beta}
$\beta= \mathfrak c_{A'_0}s+\frac {q_2'''(0)}2 s^2+ O(s^3)$,
%\end{equation*}
with $s$ such that
\begin{align*}
AA_0&=q_2(s)+(A_0A'_0-q_1(s))\frac{q'_2(s)}{q'_1(s)}=\frac{\mathfrak c_{A'_0}}2s^2
    + (A_0A'_0-s)(\mathfrak c_{A'_0} s+\frac {q_2'''(0)}2s^2)+O(s^3)
%&=A_0A'_0\mathfrak c_{A'_0}s+\left(\frac {A_0A'_0q_2'''(0)}2   -\frac{\mathfrak c_{A'_0}}2\right)s^2  +\cO(s^3)\\
%&=A_0A'_0\mathfrak c_{A'_0}s\left(1+\frac 1{2}\left(\frac{q_2'''(0)}{\mathfrak c_{A'_0}}    -\frac 1{A_0A'_0}\right)s \right) +\cO(s^3)\, .
\end{align*}
Thus $
s=\frac{AA_0}{A_0A'_0\mathfrak c_{A'_0}}\left(1-\frac 1{2}\left(\frac{q_2'''(0)}{\mathfrak c_{A'_0}}
    -\frac 1{A_0A'_0}\right)\frac{AA_0}{A_0A'_0\mathfrak c_{A'_0}}\right)+O((AA_0)^3)$.
It follows that 
%\eqref{formule1beta} that
\begin{equation}\label{beta}
\beta=\frac{AA_0}{A_0A'_0}+ \frac{(AA_0)^2}{2(A_0A'_0)^3\mathfrak c_{A'_0}}+ O((AA_0)^3)\, .
\end{equation}

\noindent\underline{\bf Estimate of $\beta'$.}
We proceed analogously for $\beta'$. This time
we consider that $q(u)=A'_0-q_1(u)\frac{w}{|w|}-q_2(u)w^\perp$ is the  parametrization of $\partial\cO_{(i_0,0)}$ by arclength such that $q(0)=A_0$ (i.e. $(q_1(0),q_2(0))=(0,0)$) and $A\in q(s)+\mathbb R q'(s)$. 
Note that $(q_1'(0),q'_2(0))=(1,0)$ and
$(q_1''(0),q''_2(0))=(0,\mathfrak c_{A'_0})$ where where  $\mathfrak c_{A'_0}$ denotes the curvature of $\partial \mathcal O'$ at $A'_0$.
Then
%\begin{equation}\label{formule1tanbeta'}
$\tan\beta'=\frac{q'_2(s)}{q'_1(s)}=\mathfrak c_{A_0}s+O(s^2)$,
%\end{equation}
with $s$ such that
%\begin{equation}
$-AA_0=q_2(s)-q_1(s)q'_2(s)=-\frac{\mathfrak c_{A_0}}2s^2
    +\cO(s^3)$
%\end{equation}
and so 
\begin{equation}\label{beta'}
\beta'
%= \mathfrak c_{A_0}s+O(s^2)
=  \sqrt{2c_{A_0}AA_0}+O(AA_0)\, .
\end{equation} 

\noindent\underline{\bf Estimates of $\alpha_N$ and of $\alpha'_N$.}
Note that $\alpha'_N$ is $\alpha_{N+m_1}$
for another choice of point $B_0$ and for some $m_1\in \mathbb Z^2$ (depending on $B_1$). Thus it is enough to estimate $\alpha_N$.
The notation $O(...)$ below will be uniform in $A$.
Let $B_N:=B_0+N w$.
% be the lowest point of $\mathcal O_{(N,1)}$.
%The tangent line to $\partial\mathcal O_{(i_1,\ell_1+N w)}$ at $B_N$ is parallel to $ w$, so that
%$$ \left(T_{B_N} \partial \mathcal O_{(i_1,\ell_1+N w)}, \overrightarrow{AB_N}\right)=\arcsin     \frac{d_0+ \overrightarrow{A_0A}\cdot w^\perp}{| \overrightarrow{AB'_0}+N w|}=     \frac{d_0+ \overrightarrow{A_0A}\cdot w^\perp}{| \overrightarrow{AB'_0}|+N| w|}+\left(\frac {d_0}{N|w|}\right)^3+o(N^{-3})\, .$$ 
We parametrize $ \partial\mathcal O_{(i_1,\ell_1+N w)}$ by arc-length as $q(u)=B_N+q_1(u)\frac{w}{|w|}+q_2(u)w^\perp$ such that $q(0)=B_N$,  $q(s)=B'_N$, $q'(0)=\frac{w}{|w|}$ and $q''(0)=\mathfrak c_{B_0}w^\perp$ where  $\mathfrak c_{B_0}$ is the curvature of $\partial \mathcal O_{(i_1,\ell_1)}$ at $B_0$. 
Since $\langle q'(s),q'(s)\rangle=1$, 
%it follows that 
$\langle q''(s),q'(s)\rangle=0$,
% so 
$0=\langle q'''(0),q'(0)\rangle+\langle q''(0),q''(0)\rangle=q'''_1(0)+\mathfrak c_{B_0}^2$ 
and 
so $q'''_1(0)=-\mathfrak c_{B_0}^2$.
Then
%\begin{align*}
%\label{BNB'N1} 
$B_NB'_N
%&
=|q(s)-q(0)|
%=|sq'(0)+\frac 12 q''(0)s^2+\frac 16 q'''(0)s^3+o(s^3)|
%\\
%\nonumber &=\sqrt{\left(s+\frac {q'''_1(0)}6 s^3\right)^2+\left(\frac {\mathfrak c_{B_0}}2 s^2+ \frac {q'''_2(0)}6 s^3\right)^2}+o(s^3)\\
%\nonumber &= \sqrt{s^2+\left(-\frac {\mathfrak c_{B_0}^2}3+\frac {\mathfrak c_{B_0}^2}4 \right) s^4 +O(s^5)}+o(s^3)\\
%\nonumber &=  s\sqrt{1-\frac {\mathfrak c_{B_0}^2}{12} s^2 +O(s^3)}+o(s^3)\\
%&
=  s
-\frac {\mathfrak c_{B_0}^2}{24} s^3 +o(s^3)$,
%\end{align*}
and so
\begin{align}
\nonumber \tan\alpha_N=&\frac {q'_2(s)}{q'_1(s)}=\frac {\mathfrak c_{B_0}s+\frac 12q_2'''(0)s^2+\frac 16q_2''''(0)s^3+o(s^3))}{1+\frac 12q_1'''(0)s^2
+o(s^2)}\\
%\nonumber&=\mathfrak c_{B_0}s+\frac{q_2'''(0)}{2}s^2+\left(\frac{q_2''''(0)}{6}-\frac {q_1'''(0)\mathfrak c_{B_0}}2\right)s^3+o(s^3)\\
\label{tanalphaN1}&=\mathfrak c_{B_0}B_NB'_N+\frac{q_2'''(0)}{2}(B_NB'_N)^2+\mathfrak b_{B_0}(B_NB'_N)^3+o((B_NB'_N)^3)\, ,
\end{align}
with $\mathfrak b_{B_0}:=\frac{q_2''''(0)}{6}+\frac{13\mathfrak c_{B_0}^3}{24}$.
Also
\begin{align*}%\label{a2}
\tan \alpha_N&=\tan\left(T_{B_N} \partial \mathcal O_{(i_1,\ell_1+N w)}, \overrightarrow{AB'_N}\right)=
       \frac{  w^\perp\cdot  \overrightarrow{AB'_N}}{  w\cdot  \overrightarrow{AB'_N}/| w|}
= 
\frac{d'_0(A)+w^\perp\cdot \overrightarrow{B_NB'_N}
}{|w|\, N+\frac{w.(\overrightarrow{A_0B_0}+\overrightarrow{B_NB'_N})}{|w|}}\, ,
\end{align*}
with 
%\begin{equation}\label{d'0A}
$d'_0(A):=w^\perp\cdot \overrightarrow{AB_0}=d_0-w^\perp\cdot\overrightarrow{A_0A}$.
%\end{equation}
Thus $\alpha_N=O(N^{-1})$, $B_NB'_N=O(N^{-1})$, $A_0A=O(N^{-1})$.
Since $B_N\in \partial \mathcal O_{(i_1,\ell_1+N w)}$ and since
$T_{B_N}\partial  \mathcal O_{(i_1,\ell_1+N w)}=B_N+\mathbb R w$, it follows that
$ w^\perp\cdot \overrightarrow{B_NB'_N}=q_2(s)=\frac 12
\mathfrak c_{B_0}(B_NB'_N)^2+o(N^{-2})$ and 
$w\cdot \overrightarrow{B_NB'_N}
=|w|B_NB'_N
+o\left(N^{-2}\right)$. 
So
\begin{align}
\tan \alpha_N
%&=  \frac{d'_0(A)+  \frac 12\mathfrak c_{B_0}(B_NB'_N)^2}{\frac{w.\overrightarrow{A_0B_N}}{|w|}+B_NB'_N}+o(N^{-3})\\
\label{a2bis}&=  
\frac{d'_0(A)}{\frac{w.\overrightarrow{A_0B_N}}{|w|}}+
\frac 12
\frac{\mathfrak c_{B_0}(B_NB'_N)^2
}{\frac{w.\overrightarrow{A_0B_N}}{|w|}}
-\frac{d'_0(A)  B_NB'_N}{(\frac{w.\overrightarrow{A_0B_N}}{|w|})^2}
+o(N^{-3})
\, .
\end{align}
Identifying \eqref{tanalphaN1} with \eqref{a2bis}, we obtain 
%first 
that
%\[B_NB'_N=\frac{d_0(A)}{\mathfrak c_{B_0}|w|\, N}+O\left(N^{-2}\right)\]
%Injecting this in \eqref{tanalphaN1}=\eqref{a2bis}, we obtain that
%\[B_NB'_N= \frac{d_1}{ N}+\frac{d_2}{N^2}+O\left(N^{-3}\right)\]
%and finally
\[
B_NB'_N= \frac{d_1}{ \frac{w.\overrightarrow{A_0B_N}}{|w|}}+\frac{d_2}{(\frac{w.\overrightarrow{A_0B_N}}{|w|})^2}+\frac{d_3}{(\frac{w.\overrightarrow{A_0B_N}}{|w|})^3}+o\left(N^{-3}\right)\, ,\]
with
$
d_1:=\frac{d'_0(A)}{\mathfrak c_{B_0}}$,
$d_2:=-\frac{q_2'''(0)d_1^2}{2\mathfrak c_{B_0}}$, 
$d_3:=
-\frac{d_1^2}{2}
-q_2'''(0)\frac{d_1d_2}{\mathfrak c_{B_0}}-\mathfrak b_{B_0}\frac{d_1^3}{\mathfrak c_{B_0}}$.
Using \eqref{tanalphaN1} and $\arctan u=u-\frac {u^3}3$, it follows that
\begin{align}
%\alpha_N
%&= \frac{\mathfrak c_{B_0}d_1}{ \frac{w.\overrightarrow{A_0B_N}}{|w|}}+\frac{\mathfrak c_{B_0}d_2+\frac{q_2'''(0)d_1^2}2}{(\frac{w.\overrightarrow{A_0B_N}}{|w|})^2}+\frac{\mathfrak c_{B_0}d_3+\mathfrak b_{B_0}d_1^3+q'''_2(0)d_1d_2-\frac{\mathfrak c_{B_0}^3d_1^3}3}{(\frac{w.\overrightarrow{A_0B_N}}{|w|})^3}+o\left(N^{-3}\right)\\
%&= \frac{d'_0(A)}{ \frac{w.\overrightarrow{A_0B_N}}{|w|}}+\frac{-\frac{\mathfrak c_{B_0}d_1^2}{2}-\frac{\mathfrak c_{B_0}^3d_1^3}3}{(\frac{w.\overrightarrow{A_0B_N}}{|w|})^3}+o\left(N^{-3}\right)
\label{alphaN}
\alpha_N= \frac{d'_0(A)}{\frac{w.\overrightarrow{A_0B_N}}{|w|}} -\frac{  \frac{(d'_0(A))^2}{2 \mathfrak c_{B_0}}+\frac{(d'_0(A))^3}{3}    }{(\frac{w.\overrightarrow{A_0B_N}}{|w|})^3}+o\left(N^{-3}\right)\, .
\end{align}
Let us recall that
$d'_0(A):=w^\perp\cdot \overrightarrow{AB_0}$.
%We have proved that
%\begin{equation} \label{alphaN}
%\alpha_N= \frac{d'_0(A)}{\frac{w.\overrightarrow{A_0B_N}}{|w|}} -\frac{  \frac{(d'_0(A))^2}{2 \mathfrak c_{B_0}}+\frac{(d'_0(A))^3}{3}    }{(\frac{w.\overrightarrow{A_0B_N}}{|w|})^3}+o\left(N^{-3}\right)\, .
%\end{equation}
Therefore, we also have
\begin{equation}\label{alpha'N}
\alpha'_N=\frac{d'_0(A)}{\frac{w.\overrightarrow{A_0B''_N}}{|w|}} -\frac{  \frac{(d'_0(A))^2}{2 \mathfrak c_{B''}}+\frac{(d'_0(A))^3}{3}    }{(\frac{w.\overrightarrow{A_0B''_N}}{|w|})^3}+o\left(N^{-3}\right)   \, .
\end{equation}
Hence
\begin{equation}\label{alpha'N-alphaN}
\alpha'_N-\alpha_N=d'_0(A)\left(\frac 1{\frac{w.\overrightarrow{A_0B''_N}}{|w|}}
-\frac{1}{\frac{w.\overrightarrow{A_0B_N}}{|w|}}\right)
+\frac{d_0^2}{2(\frac{w.\overrightarrow{A_0B''_N}}{|w|})^3}\left(\frac 1{ \mathfrak c_{B_0}}-\frac 1{ \mathfrak c_{B''}}\right)
% -\left(\frac{(d'_0(A))^2}{2 \mathfrak c_{B_0}}+\frac{(d'_0(A))^3}{3} \right)\left(\frac{1}{(\frac{w.\overrightarrow{A_0B''_N}}{|w|})^3}-\frac{1}{(\frac{w.\overrightarrow{A_0B_N}}{|w|})^3}\right)
+o\left(N^{-3}\right) \, .
\end{equation}

\noindent\underline{\bf Conclusion}
Observe that since $\alpha_N=O(N^{-1})$ and $\alpha'_N-\alpha_N=O(N^{-2})$,
\[
\int_{\max(\beta'',\alpha_N)}^{\alpha'_N}\cos(\varphi)\, d\varphi=\sin \alpha'_N- \sin\max(\beta'',\alpha_N)=\alpha'_N-\max(\beta'',\alpha_N)+O(N^{-4})\, ,
\]
for $\beta''\in\{\beta,\beta'\}$.
Recalling \eqref{mu0L+Nw3}, since $AA_0=O(N^{-1})$, it follows that
\[
%\int_{\mathcal A_N}\cos\alpha\, d\alpha\, dx
I_N^++I_N^-=O(N^{-5})+\int_{0}^{x^+}\max(0,\alpha'_N- \max(\beta',\alpha_N))\, dx+\int_{x^-}^{0}\max(0,\alpha'_N-\max(\beta,\alpha_N))\, dx\, .
\]
%where $x$ is such that $A=A_0+ xw^\perp$.
%We will consider separately the sets $\cA_N^\pm$ of points $(A=A_0\pm xw^\perp,\vec v)\in\cA_N$ with $x>0$.\\
%Let $A=A_0+xw^\perp$. Recall that $(A,\vec v_\alpha)\in \cA_N$ if  $\max(\beta',\alpha_N)\le\alpha<\alpha'_N$. 
Due \eqref{alpha'N} and  \eqref{beta'} and since $A_0A=O(N^{-1)}$,
for $N$ large enough, $\beta'<\alpha'_N$ if and only if
\[
  \sqrt{2\mathfrak c_{A_0}AA_0}+O(AA_0)
    <   \frac{d_0}{|w|N}+O(N^{-2})\, , \mbox{i.e.}\quad
  AA_0
    <   \frac{d_0^2}{2\mathfrak c_{A_0}|w|^2N^2}+O(N^{-3})\, .
\]
%and if it is not the case, then there is no angle $\alpha$ satisfying \eqref{majobeta'}.
Analogously the condition $\alpha_N<\beta'<\alpha'_N$ is satisfied if and only if
%\[ \frac{d_0}{|w|N}+O(N^{-2})< \sqrt{2\mathfrak c_{A_0}AA_0}+\cO(AA_0)  <   \frac{d_0}{|w|N}+\cO(N^{-2})\, , \mbox{ i.e.}\quad
$  AA_0
    = \frac{d_0^2}{2c_{A_0}|w|^2N^2}+O(N^{-3})$.
%\]
Therefore, using the fact that $\alpha'_N-\alpha_N=O(N^{-2})$, we obtain
\begin{align}
%\int_{0}^{x^+}\!\!\!\max(0,\alpha'_N- \max(\beta',\alpha_N))\, dx
%\nonumber\int_{\mathcal A_N^+}\cos\varphi\, d\varphi\, dr
I_N^+&=\int_{0}^{ \frac{d_0^2}{2\mathfrak c_{A_0}|w|^2N^2}+O(N^{-3})}\!\!\!\!\left(\alpha'_N-\alpha_N\right)\, dx
% \\
%&=\int_{0}^{ \frac{d_0^2}{2\mathfrak c_{A_0}|w|^2N^2}+O(N^{-3})}  \frac{d_0 B''B_0}{|w|^2N^2}\, dr +O(N^{-5})
= \frac{d_0^3 B''B_0}{2\mathfrak c_{A_0}|w|^4N^4}+O(N^{-5})\, ,\label{IntAN+}
\end{align}
where we used \eqref{alpha'N-alphaN}.
%Now we concentrate on the set $\cA_N^-$ of points $(A=A_0-xw^\perp,\vec v)\in\cA_N$ with $x>0$. Let $A=A_0-xw^\perp$. Recall that  $(A,\vec v_\alpha)\in \cA_N$ if\eqref{majobeta'} is satisfied, i.e. $\max(\beta,\alpha_N)\le\alpha<\alpha'_N$. 
Now \eqref{alpha'N} combined with \eqref{beta} implies that  $\beta<\alpha'_{N}$ if and only if
$$
\frac{AA_0}{A_0A'_0}+ \frac{(AA_0)^2}{2(A_0A'_0)^3\mathfrak c_{A'_0}}+ O((AA_0)^3)<\frac{d_0+A_0A}{\frac{w.\overrightarrow{A_0B''_N}}{|w|}}+O(N^{-3})
$$
i.e.
$
%\begin{equation}\label{beta<alpha'N}
A_0A<\mathfrak T'_N:=\frac{d_0A_0A'_0}{\frac{w.\overrightarrow{A_0B''_N}}{|w|}}+\left(d_0(A_0A'_0)^2-\frac{d_0^2}{2\mathfrak c_{A'_0}}\right)
\frac {1}{(\frac{w.\overrightarrow{A_0B''_N}}{|w|})^2}
+O(N^{-3})$.
%\end{equation}
Analogously
$\beta<\alpha_{N}$ if and only if
\begin{equation}\label{beta<alphaN}
A_0A<\mathfrak T_N:=\frac{d_0A_0A'_0}{\frac{w.\overrightarrow{A_0B_N}}{|w|}}+\left(d_0(A_0A'_0)^2-\frac{d_0^2}{2\mathfrak c_{A'_0}}\right)
\frac {1}{(\frac{w.\overrightarrow{A_0B_N}}{|w|})^2}
+O(N^{-3})\, .
\end{equation}
So
%In particular 
$\mathfrak T'_N-\mathfrak T_N\sim \frac{d_0A_0A'_0\, B''B_0}{(\frac{w.\overrightarrow{A_0B_N}}{|w|})^2}$.
%If \eqref{beta<alpha'N} is not satisfied there exists no $\alpha$ such that $\max(\beta,\alpha_N)<\alpha<\alpha'_N$. 
%Thus, 
%using the fact that 
Since $\alpha_N=O(N^{-1})$ and 
%that 
$\alpha'_N-\alpha_N=O(N^{-2})$, it follows that
\begin{equation}
I_N^-
%\int_{x^-}^{0}\max(0,\alpha'_N-\max(\beta,\alpha_N))\, dx
%\int_{\cA_N^-}\cos\varphi\, drd\varphi
=\int_{-\mathfrak T'_N}^{0}(\alpha'_N-\max(\beta,\alpha_N))\, dx\, +O(N^{-5})\, .
\end{equation}
Now observe that 
$\alpha_N<\beta<\alpha'_N$ if and only if $\mathfrak T_N<A_0A<\mathfrak T'_N$ %(with $\mathfrak T'_N-\mathfrak T_N=O(N^{-2})$) 
and in this case:
\[
\alpha'_N-\max(\beta,\alpha_N)=\alpha'_N-\beta
=(\mathfrak T'_N-AA_0)\left(\frac 1{A_0A'_0}-\frac 1{\frac{w.\overrightarrow{A_0B''_N}}{|w|}}\right)+O(N^{-3})\, .
\]
Otherwise $\alpha'_N-\max(\beta,\alpha_N)=\alpha'_N-\alpha_N$ and so
\begin{align*}
\alpha'_N-\max(\beta,\alpha_N)&=(d_0+A_0A)\left(\frac 1{\frac{w.\overrightarrow{A_0B''_N}}{|w|}}-\frac{1}{\frac{w.\overrightarrow{A_0B_N}}{|w|}}\right)
% -\left(\frac{(d_0+A_0A)^2}{2 \mathfrak c_{B_0}}+\frac{(d_0+A_0A)^3}{3} \right)\left(\frac{1}{(\frac{w.\overrightarrow{A_0B''_N}}{|w|})^3}-\frac{1}{(\frac{w.\overrightarrow{A_0B_N}}{|w|})^3}\right)
+\frac{d_0^2}{2(\frac{w.\overrightarrow{A_0B_N}}{|w|})^3}\left(\frac 1{ \mathfrak c_{B_0}}-\frac 1{ \mathfrak c_{B''}}\right)
+o\left(N^{-3}\right) \, .
\end{align*}
Therefore
$
%\begin{align*}
I_N^-
=\frac{\Gamma_1}{(\frac{w.\overrightarrow{A_0B_N}}{|w|})^3}+\frac{\Gamma_2}{(\frac{w.\overrightarrow{A_0B_N}}{|w|})^4}+o(N^{-4})$,
%\end{align*}
with $\Gamma_1:=d_0^2A_0A'_0\, B''B_0$ and 
\[
\Gamma_2:= \frac{3d_0^2\, A_0A'_0B''B_0(A_0A'_0+B''B_0)}{2}
-\frac{d_0^3B''B_0}{2\mathfrak c_{A'_0}}+\frac{d_0^3A_0A'_0}{2}\left(\frac 1{ \mathfrak c_{B_0}}-\frac 1{ \mathfrak c_{B''}}\right) \, .
\]
Thus, due to \eqref{mu0L+Nw2}, \eqref{mu0L+Nw3} and \eqref{IntAN+}, we conclude that
\begin{align*}
\bar\mu(M_{(i_0 ,0)}\cap T^{-1}(M_{(i_1,L+N w)}))
%&=\frac 1{2|\partial \bar Q|}\int_{\cA_N}\cos\varphi\, dxd\varphi\\
%&=\frac 1{2|\partial \bar Q|}\int_{\cA_N^+}\cos\varphi\, drd\varphi+\frac 1{2|\partial \bar Q|}\int_{\cA_N^-}\cos\varphi\, drd\varphi\\
&=\frac{\Gamma_1}{2|\partial \bar Q|(\frac{w.\overrightarrow{A_0B_N}}{|w|})^3}+\frac{\Gamma_3}{2|\partial \bar Q|(\frac{w.\overrightarrow{A_0B_N}}{|w|})^4}+o(N^{-4})\, ,
\end{align*}
with 
$
\Gamma_3:=\Gamma_2+\frac{d_0^3 B''B_0}{2\mathfrak c_{A_0}}
$.
%=\frac{3d_0^2A_0A'_0\, B''B_0(A_0A'_0+B''B_0)}2+\frac{d_0^3 B''B_0}2\left(\frac 1{\mathfrak c_{A_0}}-\frac 1{\mathfrak c_{A'_0}}\right)+\frac{d_0^3A_0A'_0}{2}\left(\frac 1{ \mathfrak c_{B_0}}-\frac 1{ \mathfrak c_{B''}}\right)\, .\]
We obtain \eqref{mu0L+Nw} by noticing that, since $\overrightarrow{BB_N}=N w$,
\[
\frac 1{(\frac{w.\overrightarrow{A_0B_N}}{|w|})^3}=
\frac 1{(N|w|+\frac{w.\overrightarrow{A_0B_0}}{|w|})^3}=
\frac 1{(N|w|)^3}\left(1-3\frac{w.\overrightarrow{A_0B_0}}{N|w|^2}\right)+o(N^{-4})\, ,
\]
and analogously that
$
\frac 1{(\frac{w.\overrightarrow{A_0B_N}}{|w|})^4}=
\frac 1{(N|w|)^4}+o(N^{-4})$.

\end{proof}

\section{Regularity of the projector $t\mapsto\Pi_t$ at $t=0$}
\label{sec:Pi}
In this section we state and prove Proposition~\ref{prop-expproj}, which is a generalization of Proposition~\ref{PROP1}. We do not assume \eqref{H0}.
% we obtain some technical estimates to be used in the proof of Proposition~\ref{prop-lambda}. The main result of this section is 
We write $\pi_0:\Delta\rightarrow Y$ for the vertical projection from $\Delta$ to
its base $Y=\Delta_0$ given by $\pi_0(x,\ell)=(x,0)$ and $\omega:\Delta\rightarrow\mathbb N$ for the level map given by $\omega((x,\ell))=\ell$. 
To simplify notations we write $\mu_\Delta(h)$ for $\int_\Delta h\, d\mu_\Delta$.

\subsection{Banach spaces and regularity of the dominating eigenprojector}\label{sec:Banach}
We start by recalling  results from \cite{BCS91,Young98,Chernov99,SV07}.
First recall that the operator $P$ can be written as follows
\begin{equation}\label{formulaP}
\forall h\in L^1(\mu_\Delta),\quad Ph(x)=\sum_{Y\in f^{-1}(x)}e^{-\alpha(y)}h(y)\, ,\quad x\in \Delta\, ,
\end{equation}
with $\alpha\equiv 0$ outside the base of $\Delta$.
We write $s_0(\cdot,\cdot)$ for the separation time for $f$ on $\Delta$ corresponding to the separation time $s(\cdot,\cdot)$ in \cite{Young98}.
In particular, if $s_0(x,y)\ge n$ then the corresponding elements in $\bar M$ (or in $M_N$) are in the
same atom of $\xi_0^n$.
Recall the class of $\eta\in (0,1)$-H\"older functions defined in \eqref{DefiHolder}.
Choose $\beta\in(\eta,1)$ large enough so that $|\alpha(y_1)-\alpha(y_2)|\le C_\alpha \beta^{s_0(y_1,y_2)+1}$
if $s_0(y_1,y_2)\ge 1$.
The condition $\beta>\eta$ will ensure that the Banach spaces $\mathcal B_0$ and $\mathcal B$ described below will be tailored to the approximation of observables
considered in Theorem~\ref{THM1}.
We define the space $\mathcal B_0$ of Lipschitz functions $h:\Delta\rightarrow\mathbb C$ with respect to the metric $\beta^{s_0(\cdot,\cdot)}$, with norm
\[
\Vert h\Vert_{\mathcal B_0}:=\Vert h\Vert_\infty+\esup_{x,y\in\Delta\, :\, s_0(x,y)\ge 1}\frac{|h(x)-h(y)|}{\beta^{s_0(x,y)}}\, .
\]
Let $p>2$ and choose $\varepsilon>0$ small enough so that in particular $\sum_{\ell\ge 0} e^{p\ell\varepsilon}\mu_\Delta(\Delta_\ell)<\infty$,
writing $\Delta_\ell%=\{f^\ell(y),\ y\in Y,\ \ell<\sigma(y)\}
$ for the $\ell$-th floor of the tower $\Delta$.
We let $\mathcal B$ be the space
of functions $h:\Delta\rightarrow\mathbb C$ such that the following quantity is finite
\[
\Vert h\Vert_{\mathcal B}:=\sup_{\ell\ge 1}e^{-\ell\varepsilon}\left(\Vert h_{|\Delta_\ell}\Vert_\infty+\esup_{x,y\in\Delta_\ell\, :\, s_0(x,y)\ge 1}\frac{|h(x)-h(y)|}{\beta^{s_0(x,y)}}\right)\, .
\]
 The choice of $\eps$ 
%so that so that $\sum_{\ell\ge 0} e^{\ell\varepsilon}\theta_1^\ell<\infty$
ensures that the Young Banach space $\mathcal B$ can be continuously injected in $L^p(\mu_\Delta)$. While $\cB_0$ is independent of $p$, $\cB$ depends on $p$ via $\eps$.
%(and in the first place, the Doeblin-Fortet-Lasota-Yorke inequality $\|P_t^n v\|_{\cB}\le C(\|v\|_{L^p}+\theta^n\|v\|_{\cB})$ from which~\eqref{spgap-Sz} follows).
\begin{lem}\label{discont}
Let $H:g\mapsto P(\hat\kappa g)$. Then
$H^3(\mathbf 1)\not\in L^1(\mu_\Delta)$.
As a consequence, since $\mathbf 1\in\cB\subset L^1(\mu_\Delta)$, $P'_0=iH$ does acts neither on $\cB$ nor on $L^1(\mu_\Delta)$
\end{lem}
\begin{proof}
Note that
%\begin{align*}
$H^3(\mathbf 1)=
P(\hat\kappa P(\hat\kappa P(\hat\kappa)))
=P^3(\hat\kappa\circ f^2. \hat\kappa\circ f. \hat\kappa)$,
%\end{align*}
using the fact that $\hat\kappa P(g)=P(g.\hat\kappa\circ f)$.
In particular $H^3(\mathbf 1)$ is integrable if and only if
$
%\kappa\circ T^4. \kappa\circ T^3. 
\kappa\circ \bar T^2. \kappa\circ\bar  T.\kappa$
is integrable, i.e. if and only if $
%\kappa\circ T^2. 
\kappa\circ \bar T.\kappa. \kappa\circ \bar T^{-1}
%\hat\kappa\circ T^{-2}
$.
Let us prove that this random variable is not integrable.
Due to \cite[Proposition 9]{SV07}, there exist $K,K_0>0$ such that $|\kappa\circ \bar  T^{-1}|$ and $|\kappa\circ\bar  T|$
are both larger that $K\sqrt{|\kappa|}-K_0$.
%Let $w,L,A,B$ as in Lemma~\ref{lemm-tail}. Let $x=(q,\vec v)$ and $T(x)=(q',\vec v')$. Assume $ \kappa( x)= L+ Nw $ with $N$ large, $q$ close to $A$ and $q'$ close to $B+L+Nw$.  Then $\widehat{(w,\vec v)}=O(N^{-1})$ and so $w^\perp \cdot\overrightarrow{Bq'}=O(N^{-1})$ and $w \cdot\overrightarrow{Bq'}=O(N^{-\frac 12})$ and so $\widehat{(w,\vec v')}=O(N^{-\frac 12})$, which implies that $|\kappa(Tx)|\ge K\sqrt{N}$. Analogously, we conclude that $|\kappa(T^{-1}x)|\ge K\sqrt{N}$.
Since $\kappa^2$ is not integrable, nor is
%\begin{align*}
%&\mathbb E_{\bar\mu}\left[
%\left|\kappa\circ T^2.
%
$\left| \kappa\circ\bar  T. \kappa. \kappa\circ \bar T^{-1}
\right|$.
%\right]\ge K''
%\sum_{M\in\mathbb Z}\bar\mu(|\kappa|_1=M) 
%M^{\frac 12}M M^{\frac 12}
%\\
%&\ \ \ \ \ \ge K''
%\sum_{M\in\mathbb Z}\bar\mu(|\kappa|_1=M)M^{2}=\mathbb E_{\bar\mu}[|\kappa|_1^2]=\infty\, .
%\end{align*}
\end{proof}
In what follows we exploit that $\cB$ is continuously embedded in $L^p(\mu_\Delta)$ and satisfies~\eqref{spgap-Sz} and~\eqref{spgap-Sz-bis}.
Using just  the information on big tail $\bar\mu(|\kappa|>N)=O( N^{-2})$,
\begin{lem}
\label{lem-expproj} Let $q\in[1,\frac{2p}{p+2})$ (so that $2(\frac 1q-\frac 1p)>1$) and $\gamma\in(1, 2\left(\frac 1q-\frac 1p\right))$, there exists  $C=C(q,\gamma)>0$  such that for all  $t\in[-\pi,\pi]^2$, 
$\|P_t-P_0-t\cdot P_0'\|_{\cB\to L^q}\le C t^\gamma$, 
with $P'_0:=P(i\kappa \cdot)$.
\end{lem}
\begin{proof}
Note that $q<p$.
Set $r:=\frac{qp}{p-q}$ so that $\frac 1q=\frac 1p+\frac 1r$. With this notation,
the upper bound on $\gamma$ can be rewritten $\gamma r<2$ and 
\begin{align*}
\|P_t-&P_0-t\cdot P_0'\|_{\cB\to L^q}\ll\|P_t-P_0-t\cdot P_0'\|_{L^p\to L^q}\\
&\le\left\|v\mapsto P\left((e^{it\cdot\hat\kappa}-1-it\cdot \hat\kappa )v\right)\right\|_{L^p\to L^q}\\
&\le\left\|v\mapsto (e^{it\cdot\hat\kappa}-1-it\cdot \hat\kappa )v\right\|_{L^p\to L^q}\le\left\|e^{it\cdot\hat\kappa}-1-it\cdot \hat\kappa \right\|_{L^r}\\
&\le\left\|t^\gamma\hat\kappa^\gamma \right\|_{L^r}
=t^\gamma\left(\mathbb E\left[\hat\kappa^{\gamma r}\right]\right)^{1/r}=O(t^\gamma)\, ,
\end{align*}
where we used the H\"older inequality at the penultimate line and the fact that $\hat\kappa$ admits moments of every order smaller than 2
at the last line. We have also used the fact that $|e^{i\hat\kappa}-1-i\hat\kappa|\le \min(|\hat\kappa|, | \hat\kappa|^2)\le |\hat\kappa|^\gamma$  for any $\gamma\in[1,2]$ and for all $x$ and that $\gamma r<2$.
\end{proof}
Given $q\in[1,2)$ up to taking $p>2q/(2-q)$ large enough (and so up to our choice of $\mathcal B$), we can adjust $\gamma$ so that $\gamma<2/q$ is as close to $2/q$ as we wish
(in particular as close to 2 as we wish if 
%we  take 
$q=1$).
Let $Y:=\Delta_0$ be the base of the tower $\Delta$. Throughout, we let $\mu_Y=\mu_\Delta(\cdot|Y)$.
As shown in~\cite{Chernov99, SV07},
the height of the tower $(\Delta, f,\mu_\Delta)$, which we denote by $\sigma:Y\to \N$, has exponential tail $\mu_Y(\sigma>n)\ll \theta_1^n$ for some $\theta_1\in(0,1)$. 
Using Lemma~\ref{lem-expproj} and building on the arguments used in~\cite{BalintGouezel06}, we obtain the following expansion of eigenprojector $\Pi_t$. 

\begin{prop}
\label{prop-expproj}
Let $b\in(p,+\infty]$.
For every $w\in\mathcal B_0$, every $v\in L^b(\mu_\Delta)$ constant on each $(a,\ell)$ (with $a\in\alpha$ and $0\le \ell<\sigma(a)$) such that $vw\in\mathcal B$, 
there exists $\Pi_0'(vw)$ belonging to $L^{\eta}(\mu_\Delta)$ for every $\eta\in [1,2)$ such that $1_Y\Pi'_0(vw)\in\mathcal B_0$.\\
Moreover, for every $p'\in(1, 4/3)$ and every $\gamma\in (1,\min(2\frac{p-1}p,\frac 4{p'}-2))$, there exists $C,\delta>0$ such that, for every $(v,w)$ as above and all  $t\in B_\delta(0)$,
\[
\|(\Pi_t-\Pi_0-t\cdot \Pi_0')(vw)\|_{L^{
p'}(\mu_\Delta)}\le  C\,|t|^\gamma (\|vw\|_{\mathcal B}+
\|w\|_{\mathcal B_0}\|v\|_{L^{b}(\mu_\Delta)}).
\]
and  $\Vert \Pi'_0(vw)\Vert_{L^\eta}\le C (\|vw\|_{\mathcal B}+
\|w\|_{\mathcal B_0}\|v\|_{L^{b}(\mu_\Delta)})$.
Moreover $\Pi'_0$ satisfies
\begin{equation}
\label{eq-derivPi0}
\Pi_0' (vw)(x)=\Pi_0' (vw)\circ\pi_0(x)+
i\sum_{m=0}^{\omega(x)-1}\hat\kappa\circ f^m(\pi_0(x))\int_\Delta vw\, d\mu_\Delta\, .
\end{equation}
\end{prop}
We postpone the proof of Proposition~\ref{prop-expproj} to the end of this section. We remark that Proposition~\ref{prop-expproj} cannot be proved
merely via the continuity arguments in~\cite{KellerLiverani99}, which is why we resort to building on the arguments in~\cite{BalintGouezel06}.
The key elements used in the proof of Proposition~\ref{prop-expproj} are : i)  obtain Lemma~\ref{lem-eigf2}, which gives 
the expansion of $t\mapsto\int_Y\Pi_t v\, d\mu_Y$  for suitable functions $v$, as  in Proposition~\ref{prop-expproj-base};  ii) the key observation in i) is that it is sensible to study the smoothness in $t$ of $\int_Y\Pi_t\, d\mu_Y$, first and use this to define $\mathbf 1_Y\Pi'_0$
and further, to control $\|1_Y\Pi_0'v\|_{\cB_0}$  for suitable functions $v$.
iii) use  i) and ii) together with formula \eqref{eq-derivPi0} to control $\Pi_0'v$ in $L^{p'}$
for suitable $v$ and $p'$. For further details on the use of Lemma~\ref{lem-eigf2} in defining $\mathbf 1_Y\Pi'_0$ we refer 
to the last paragraph in Subsection~\ref{subsec-def}.

\subsection{Regularity of $t\mapsto \mathbf 1_Y\Pi_t$.}
\label{subsec-def}

Recall that $\sigma$ corresponds to the first return time of $f$ to the base  $Y$.
In what follows, we let $F=f^\sigma:Y\rightarrow Y$ with
$F(y)=f^{\sigma(y)}(y)$ for all $y\in Y$ be the first return map.
Recall that $(Y, F,\mu_Y)$ is a Gibbs Markov map
 with respect to a suitable countable partition $\mathcal{Y}$ and that $\sigma$ is constant on each atom
of $\mathcal{Y}$ (the required definitions are recalled below).
Let $R: L^1(\mu_Y)\to L^1(\mu_Y)$ be the transfer operator of the Gibbs Markov base map $(Y, F=f^\sigma,\mu_Y)$ given by 
\begin{equation}\label{defR}
Rv(x):=\sum_{y\, :\, Fy=x}\chi(y)v(y)=\sum_{n\ge 1}P^n(1_{Y\cap\{\sigma=n\}}v)\, ,
\end{equation}
with $\chi=\frac{d\mu_Y}{d\mu_Y\circ F}:Y\to\R$. 
In particular, on $f^{-1}(Y)$, $e^{-\alpha}=\chi\circ\pi_0$.
Define $d_\beta(y,y')=\beta^{s(y,y')}$ where the  separation time for $F$, $s(y,y')$, is the least integer $n\ge0$ 
such that $F^ny$ and $F^ny'$ lie in distinct partition elements in $\alpha$.
The partition $\mathcal{Y}$ separates trajectories with $s(y,y')=\infty$ if and only if $y=y'$;  so $d_\beta$ is a metric.
The map $F$ is a {\em (full-branch) Gibbs-Markov map}, which means that
\begin{itemize}
\item $F|_a:a\to Y$ is a measurable bijection for each $a\in\mathcal{Y}$, and
\item 
$|\log\chi(y)-\log\chi(y')|\le C_\alpha d_\beta(y,y')$ for all $y,y'\in a$, $a\in\mathcal{Y}$
(since $s(\cdot,\cdot)\le s_0(\cdot,\cdot)$).
\end{itemize}
A consequence of this definition is that there is a constant $C>0$ such that 
\begin{align} \label{eq:GM}
\chi(y)\le C\mu_Y(a) \qquad\text{and} \qquad
|\chi(y)-\chi(y')|\le C\mu_Y(a)d_\beta(y,y'),
\end{align}
for all $a\in\mathcal{Y}$ and $y,y'\in a$.
%The tower map $(\Delta,f,\mu_\Delta)$ can be represented as follows:
%\[\Delta=\{f^\ell(y),\ y\in Y,\ \ell\in\{0,...,\sigma(y)-1\}\},\quad f(y,\ell)=\begin{cases} (y,\ell+1) & \ell\le\sigma(y)-2 \\ (Fy,0) & \ell=\sigma(y)-1 \end{cases}\, ,\]
%and  $\mu_\Delta=(\mu_Y\times{\rm counting})/\bar\sigma$ with $\bar\sigma=\int_Y\sigma\,d\mu_Y$. In our situation, $(\Delta,f,\mu_\Delta)$ is mixing
%and the tower has {\em exponential tails}, i.e. $\mu_Y(\sigma>n)=O(e^{-cn})$ for some $c>0$.
Since $F$ is Gibbs Markov, it follows that
 (see, for instance,~\cite[Section 5]{Sarig02}):
% and \cite[Appendix B]{pene09IHP}):
\begin{itemize}
\item
The space $(\cB_1,\Vert\cdot\Vert_{\mathcal B_0})$ of $\beta$-H\"older continuous functions on $Y$ contains constant functions and  $\mathcal{B}_1\subset L^\infty( \mu_Y)$.
Note that $\cB_1$ corresponds to functions $h\in\mathcal B_0$ supported on $Y$.
For this reason and for the reader convenience, we have chosen to write also $\Vert\cdot\Vert_{\mathcal B_0}$
for the norm of $\cB_1$ (in order to avoid the introduction of unnecessary notation).
\item $R$ is quasi-compact on $\cB_1$ and $1$ is a simple eigenvalue for $R$, isolated in
the spectrum of $R$.
\end{itemize}

Building on the argument of~\cite[Lemma 3.15]{BalintGouezel06}, in this section we obtain
\begin{prop}
\label{prop-expproj-base}
Let $b\in(p,+\infty]$ and $\gamma\in (1,2\frac{p-1}p)$.
There exists $C>0$ such that,
for every $w\in\mathcal B_0$ and $v\in L^b(\mu_\Delta)$ constant on each $(a,\ell)$ (with $a\in\alpha$ and $0\le \ell<\sigma(a)$)
%so that $\int_Y v\,d\mu_Y\neq 0$ and 
so that $vw\in\mathcal B$, the following holds true
%and $1_Yv\in L^b(\mu_Y)$,
% Let $v:\Delta\to\R$ be a function so that $v$ satisfies  the assumptions in the statement of Proposition~\ref{prop-expproj},
%in particular so that $1_Yv\in L^{p}(\mu_Y)$.
%there exists $\mathbf 1_Y\Pi'_0 (vw)\in\mathcal B_0$ such that
\[\|1_Y(\Pi_t-\Pi_0-t\cdot \Pi_0')(vw)\|_{\cB_0}\le C\,|t|^\gamma (\|vw\|_{\mathcal B}+\|w\|_{\mathcal B_0}\|v\|_{L^{b}(\mu_\Delta)})\, ,
\]
\begin{equation}
\label{eq1yderivpi}
\mbox{with}\quad\quad\quad\quad\quad 1_Y\Pi'_0(vw):=\mu_\Delta(vw)Q'_0(1_Y)+c_0(vw)1_Y\in\cB_0\, ,
\end{equation}
\[
\mbox{and}\quad\quad\quad
c_0(vw):=i   \left(\sum_{j\ge 0}\int_\Delta \hat\kappa\circ f^j \, .vw\, d\mu_\Delta+ \mu_\Delta( vw )\sum_{j\ge 0}\int_{\Delta} \hat\kappa\,\frac{ 1_{Y}\circ f^{j+1}}{\mu_\Delta(Y)}\, d\mu_\Delta\right)\, ,
\]
and with $Q'_0(1_Y)\in\mathcal B_1$ given by \eqref{eq-Qtderiv} and $\Vert 1_Y\Pi'_0(vw)\Vert_\infty\le C'  
\|w\|_{\mathcal B_0}\|v\|_{L^{b}(\mu_\Delta)}$.
\end{prop}
\begin{rem}
The expression of  $c_0(vw)$ comes from Lemma~\ref{lem-eigf2}.
% below.  
 As shown in Lemma~\ref{lemma-lemexpk1y} below, the sums in the formula for $c_0(vw)$ are absolutely convergent and in $O( \|w\|_{\mathcal B_0}\|v\|_{L^{b}(\mu_\Delta)})$.
\end{rem}

\begin{rem}\label{int00}
Under the assumptions of Proposition~\ref{prop-expproj},
if $\mu_\Delta(vw)=0$, then
\[
\Pi'_0(vw)=\Pi'_0(vw)\circ\pi_0=c_0(vw)=i   \sum_{j\ge 0}\int_\Delta \hat\kappa\circ f^j \, .vw\, d\mu_\Delta\, .
\]
If moreover $vw=v'w'-(v'w')\circ f$ is a coboundary of functions of the same kind, then
\[
\Pi'_0(vw)=\Pi'_0(vw)\circ\pi_0=c_0(vw)=
-i  \int_\Delta \hat\kappa \, .(v'w')\circ f\, d\mu_\Delta\, .
\]
\end{rem}
Recall that  $\sigma$ is the first return time of $f$ to the base $Y$. 
Write $\hat \kappa_n:=\sum_{j=0}^{n-1}\hat\kappa\circ f^j$
 and note that $\hat\kappa_\sigma:=\sum_{j=0}^{\sigma(\cdot)-1}\hat\kappa\circ f^j$ is the induced (to the base $Y$) version of $\hat\kappa$.
In order to define $1_Y\Pi'_0v$ we will justify that the derivative at $t=0$ of LHS of the following identity
$$1_Y\Pi_t (vw)=\int_Y\Pi_t (vw)\, d\mu_Y\frac{Q_t(1_Y)}{\int_Y Q_t (1_Yvw)\, d\mu_Y} \, $$
is well defined. Here, $Q_t$ is an eigenprojector of a perturbation $\tilde R_t$ of $R$.
More precisely, while $\Pi_t$ is the eigenprojector of $ P(\lambda_t^{-1}e^{it\hat\kappa}\cdot)$ associated to the eigenvalue 1, $Q_t$ is the main eigenprojector of $\tilde R_t=R(\lambda_t^{-\sigma(\cdot)}e^{it\hat\kappa_\sigma}\cdot)$  associated to the eigenvalue 1 (see below for the formal definition of $\tilde R_t$).
The above displayed formula allows us to exploit that in the RHS we only have 
$\int_Y\Pi_t (vw)\, d\mu_Y$ and $Q_t$. As explained  below $Q_t$ is much easier to understand than $\Pi_t$. Among other technical lemmas,  in the next subsection, we obtain the required expansion (in norm) for $Q_t$ (see Lemma~\ref{lem-eigf}).
%$1_Y\Pi_t (vw)$.
Lemma~\ref{lem-eigf2} below allows us to control the derivative  of $\int_Y\Pi_t (vw)\, d\mu_Y$ at $0$.

%Our proof of Proposition \ref{prop-expproj-base} will use several preliminary facts, that are contained in the following lemmas.
%
%
%
%
%
\subsection{Technical lemmas}
\label{subsec-techlem}
We start with the following lemma on the integrability of $\hat\kappa_\sigma$.\begin{lem}
\label{lem-indk} For any $r\in (1,2)$, $\int_Y|\hat\kappa_\sigma|^r\, d\mu_Y<\infty$.
\end{lem}
\begin{proof} First, due to the H\"older inequality,
\[
%begin{align*}
 \int_Y|\hat\kappa_\sigma|^r\, d\mu_Y 
 =\sum_{n\ge 1}\int_{Y\cap\{\sigma=n\}}|\hat\kappa_n|^r\, d\mu_Y
 \le\sum_{n\ge 1}\mu(\sigma=n)^{1/q}\int_Y|\hat\kappa_n|^{rp}1_{\{\sigma=n\}}\, d\mu_Y\, ,
\]
with $p\in(1,2/r)$ and $q=p/(p-1)$ so that $1/p+1/q=1$.
%\end{align*}
Let $s=rp/(rp-1)$ so that $1/(rp)+1/s=1$.
Using the H{\"o}lder inequality for inner products, we have that for any $n\ge 1$
% and $a, b>1, 1/a+1/b=1$,
\[
%begin{align*}
 |\hat\kappa_n|^{rp} =|\sum_{j=0}^{n-1}\hat\kappa \circ f^j\cdot 1|^{rp}
 \le 
 \left(\left|\sum_{j=0}^{n-1}1\right|^{1/s} \left(\sum_{j=0}^{n-1}|\hat\kappa \circ f^j|^{rp}\right)^{1/rp}\right)^{rp}\\
  \le n^{rp/s}\sum_{j=0}^{n-1}|\hat\kappa \circ f^j|^{rp}\, ,
\]
%\le n^{a/b+a}\sup_{j\le n}|\hat\kappa\circ f^j|.
%\end{align*}
and thus
%Taking $a=r+\eps$ for $r\in (1,2)$ and $\eps$ as small as we wish, we have
%\begin{align*}
$ \int_Y|\hat\kappa_\sigma|^r\, d\mu_Y 
% =\sum_{n\ge 1}\mu(\sigma=n)^{1/q}\int_Y|\hat\kappa_n|^{rp}\, d\mu_Y\\
%\sum_{n\ge 1}\int_{\sigma=n}|\hat\kappa_n|^r\, d\mu_Y
 \le\sum_{n\ge 1}\mu(\sigma=n)^{1/q}n^{rp/s}\int_Y\sum_{j=0}^{n-1}|\hat\kappa \circ f^j|^{rp} 1_{\{\sigma=n\}}\, d\mu_Y$
which leads to
%\\&
\[
\int_Y|\hat\kappa_\sigma|^r\, d\mu_Y \le\sum_{n\ge 1}C_1\theta_1^{n/q}n^{rp/s}(\mu(Y))^{-1}\int_\Delta|\hat\kappa |^{rp}\, d\mu_\Delta<\infty\, .
\]
%\end{align*}
%The conclusion follows since $\mu_Y(\sigma=n)$ and since $r\in (1,2)$ and we can choose $\eps$ as small as we wish.~
\end{proof}

We define $\tilde R_t:=\sum_{n\ge 1}\lambda_t^{-n}P_t^n(1_{Y\cap\{\sigma=n\}}v)=\sum_{n\ge 1}\lambda_t^{-n}R_n(e^{it\hat\kappa_n}\cdot)$,
with $R_nv:=R(1_{\{\sigma=n\}}v)=P^n(1_{Y\cap\{\sigma=n\}}v)$.
The next lemma 
provides some useful estimates on $R_n$.
% which we will use in the proofs below.
%\begin{lem}\label{lem-Rn}
%Fix $q\in (\frac p{p-1},2)$ and $\varepsilon_0\in(0,1)$. Let $v:\Delta\to\R$ so that $v$ satisfies the assumptions in the statement of Proposition~\ref{prop-expproj}.Then for every   $\gamma\in (1,2/q)$, there exist $C_0$ and$\rho<1$ so that  for all $t$ small enough, all $w \in \cB_1$ and all $n\ge 1$,

\begin{lem}\label{lem-Rn}
Let $b\in(2,+\infty]$.
Fix $q\in (\frac b{b-1},2)$ (with convention $\frac b{b-1}=1$ if $b=\infty$) and $\varepsilon_0\in(0,1)$. Then, for every   $\gamma\in [1,2/q)$, there exist $C_0$ and $\rho\in(0,1)$ so that  for all $t$ small enough, all $w \in \cB_1$, all $n\ge 1$ and all $v_Y\in  L^b(\mu_Y)$ constant on each atom of the partition $\alpha$,
\begin{align}
\label{eq-rnv0}
\|R_n((e^{it\hat\kappa_\sigma}-1-it\hat\kappa_\sigma) v_Y w)\|_{\cB_0}
&\le C_0 \rho^n |t|^\gamma\, \|\hat\kappa_\sigma\|^\gamma_{L^{\gamma q}(\mu_Y)}\, \|v_Y\|_{ L^{b}(\mu_Y)} \| w \|_{\cB_0}\\
\label{eq-rnv1}\Vert R_n(\hat \kappa_\sigma v_Y w)\Vert_{\cB_0} &\le C_0\rho^n \|v_Y\|_{L^b(\mu_Y)}
 \| \hat \kappa_\sigma\|^\gamma_{L^{\gamma q}(\mu_Y)} \| w\|_{\cB_0}\, , \\
 \label{eq-rnv2} \Vert R_n(v_Y w)\Vert_{\cB_0} &\le C_0\rho^n \|v_Y\|_{L^2(\mu_Y)}  \| w\|_{\cB_0}\, .
%\\   \label{eq-rnv3}\Vert R_n(\hat \kappa_\sigma w)\Vert_{\cB_0} &\le C_0\rho^n  \| \hat \kappa_\sigma\|_{L^{1+\varepsilon_0}(\mu_Y)} \| w\|_{\cB_0}\, .
\end{align}
\end{lem}
Note that, in this lemma,  $1<\gamma<2(1-\frac 1b)$ and that up to adapting the value of $q$, we can take $\gamma$ 
as close  to $2(1-\frac 1b)$ as we wish.
\begin{proof}
By the arguments used in ~\cite[Proof of Proposition 12.1]{MelTer17} and exploiting that $v_Y$ and $\hat\kappa_\sigma$ are constant on every $a\in\mathcal{A}$,
%\begin{equation}\label{RnDA}
%\|R_n(()\cdot)\|_{\cB_0}\ll \int_Y1_{\{\sigma=n\}} |(e^{it\hat\kappa_\sigma}-1-it\hat\kappa_\sigma) v_Y|\, d\mu_Y
%\end{equation}
%and
\begin{equation}\label{RnDA}
\forall w\in\cB_1,\ \forall h\in L^1(\mu_Y),\quad \|R_n( h w)\|_{\cB_0}\ll \|1_{\{\sigma=n\}} h\|_{L^1(\mu_Y)}\|w\|_{\cB_0}\, ,
\end{equation}
for $h\in\{(e^{it\hat\kappa_\sigma}-1-it\hat\kappa_\sigma) v_Y, v_Y,\hat\kappa_\sigma, \hat\kappa_\sigma v_Y\}$.
A justification of \eqref{RnDA} based on ~\cite[Proof of Proposition 12.1]{MelTer17} is provided in Appendix~\ref{sec-RnDA}.
We note that  since $\sigma$ has exponential tail, equation~\eqref{RnDA}  and H{\"o}lder inequality imply~\eqref{eq-rnv1} and \eqref{eq-rnv2}.
%and \eqref{eq-rnv3}.
Next, by Lemma~\ref{lem-indk},  
for any $r \in (1,2)$,
$\hat\kappa_\sigma\in L^r(\mu_Y)$.
%By assumption, $1_Yv\in L^{p}(\mu_Y)$ (since $v\in\mathcal B\subset L^b(\mu_\Delta)$).  
%Thus, for any $q\in (1, 4/3)$,
%\[\int_Y|v\, \hat\kappa_\sigma|^q\, d\mu_Y<\int_Y|v|^{3q}\, d\mu_Y\cdot \int_Y|\hat\kappa_\sigma|^{3q/2}\, d\mu_Y<\infty.\]
%The restriction on $q$ comes from needing to require  $|\hat\kappa_\sigma|^{3q/2} \in  L^{1}(\mu_Y)$.
Since $\gamma q\in[1,2)$, using the same argument as in the proof of Lemma~\ref{lem-expproj} combined with H\"older inequality, we have
\[
\int_Y 1_{\{\sigma=n\}}|(e^{it\cdot \hat\kappa_\sigma}-1-it\cdot \hat\kappa_\sigma)\, v_Y|\, d\mu_Y\le  \Vert v_Y\Vert_{L^b(\mu_Y)}\Vert (t\cdot \hat\kappa_\sigma)^\gamma\Vert_{L^q(\mu_Y)}  (\mu_Y(\sigma=n))^{1-\frac 1b-\frac 1q}\, ,
% \Big(\int_Y |(e^{it\hat\kappa_\sigma}-1-it\hat\kappa_\sigma)\, v|^q\, d\mu_Y\Big)^{1/q}\mu_Y(\sigma=n)^{1-1/q}\\
%&\ll t^\gamma\mu_Y(\sigma=n)^{1-1/q}.
\]
which leads to \eqref{eq-rnv0}.
Note that $\gamma q<2$ ensures that $\hat\kappa_\sigma^\gamma\in L^{q}(\mu_Y)$.
%Note that for any $\gamma\in (1,2)$, $|(e^{it\hat\kappa_\sigma}-1-it\hat\kappa_\sigma)|^q\ll |t^\gamma| |\hat\kappa_\sigma|^{q\gamma}$
%and we require that  $|\hat\kappa_\sigma|^{q\gamma}\in  L^{1}(\mu_Y)$. This gives a restriction on $\gamma$.
%We further restrict the range of $\gamma>1$  so that  
%\[\int_Y |(e^{it\hat\kappa_\sigma}-1-it\hat\kappa_\sigma)\, v|^q\, d\mu_Y\ll \int_Y|v|^{3q}\, d\mu_Y\cdot \int_Y|\hat\kappa_\sigma|^{(3\gamma q)/2}\, d\mu_Y.\]
%Recalling that  $1_Yv\in L^{4}(\mu_Y)$, we just need to note that every   $\gamma\in (1,3/(2q))$ and fixed $q\in (1,4/3)$,
% $\hat\kappa_\sigma\in L^{\gamma q}(\mu_Y)$. This is the assumption on $\gamma$ we have used in the statement of the lemma.
The result follows  from the previous two displayed inequalities since $\sigma$ has exponential tail.~\end{proof}

%Recall that $\lambda_t$ is the eigenvalue associated with $\Pi_t$. Recall that $\lambda_t=1-t^2\log (1/|t|)(1+o(1))$ (see~\cite[Theorem 13]{SV07}); in particular, $\lambda_0=1$.
%Define $R(\lambda_t, t):=\sum_{n=0}^\infty \lambda_t^{-n} R_{n,t}$, where $R_{n,t}v:=R_n(e^{it\hat\kappa_\sigma}v)$.
Note that $\tilde R_0=R$ and that \eqref{devlambda}
% $\lambda_t=1-t^2\log (1/|t|)(1+o(1))$ 
implies that  $\lambda_0'= \frac{d}{dt}\lambda_t|_{t=0}=0$.
The next lemma shows that $t\mapsto \tilde R_t \in\mathcal B_1$ is differentiable at $t=0$
with derivative $\tilde R'_0:=R(i\hat\kappa_\sigma \cdot)=\sum_{n=0}^\infty R_n(i\hat\kappa_n \cdot)$, and that this is also true if we replace $w\in\mathcal B_1$
by $v_Yw$ as in Lemma~\ref{lem-Rn}.
%By Lemma~\ref{lem-Rn}, for $t$ small enough, $\|R_{n,t}(\hat\kappa_\sigma\cdot)\|_{\cB_0}\ll \rho_0^n$ for some $\rho_0<1$.
%Thus, for $t$ small enough, the derivative
%\begin{align}
%\label{eq-der}
%\frac{d}{dt}R(\lambda_t, t)=-\lambda_t'\sum_{n=0}^\infty  n\,\lambda_t^{-(n-1)}R_{n,t}+\sum_{n=0}^\infty  
%\lambda_t^{-n}R_n(i\kappa_\sigma e^{it\kappa_\sigma}\cdot)
%\end{align}
%converges in $\|\cdot\|_{\cB_0}$. Let $R'(1,0):=\sum_{n=0}^\infty R_n(i\hat\kappa_\sigma)$.
%The next  result  shows that   $R'(1,0)1_Y=\frac{d}{dt}R(\lambda_t, t)|_{t=0}$ is well defined in $\cB_1$ and it gives a bound for $\|R'(1,0)v_Y\|_{\cB_0}=\|\frac{d}{dt}R(\lambda_t, t)|_{t=0} v_Y\|_{\cB_0}$ in terms of $\|v_Y\|_{ L^{4}(\mu_Y)}$ for suitable $v_Y$.
\begin{lem}
\label{lem-deriv} 
%Let $v_Y: Y\to \R$ so that $v$ is constant on  each $a\in\alpha$ and so that $v_Y\in L^{4}(\mu_Y)$.
Let $b,q,\gamma$ as in Lemma~\ref{lem-Rn}.
Then for every   $\gamma\in [1,2/q)$, there exists 
$C>0$ such that for $t$ small enough and all $w \in \cB_1$
and for any $v_Y$ as in Lemma~\ref{lem-Rn},
\begin{align}
\|R(v_Yw)\|_{\cB_0} &\le C  \|v_Y\|_{ L^{b}(\mu_Y)}\Vert w\Vert_{\cB_0},\label{Rp-0}\\
\|\tilde R'_0(v_Yw)\|_{\cB_0} &\le C \|\hat\kappa_\sigma\|^\gamma_{L^{\gamma q}(\mu_Y)}\, \|v_Y\|_{ L^{b}(\mu_Y)}\Vert w\Vert_{\cB_0},\label{Rp-a}\\
\|(\tilde R_t-\tilde R_0-t\cdot \tilde R'_0)(v_Y w)\|_{\cB_0} &\le C |t|^\gamma(\|\hat\kappa_\sigma\|_{L^{\gamma q}(\mu_Y)}^\gamma\, \|v_Y\|_{ L^{b}(\mu_Y)})\Vert w\Vert_{\cB_0}\, .\label{Rp-c}
%\\\|(\tilde R_t-\tilde R_0-t\cdot \tilde R'_0)w\|_{\cB_0} &\le C |t|^\gamma \Vert w\Vert_{\cB_0}\, . \label{Rp-d}
\end{align}

\end{lem}

\begin{proof}
Summing \eqref{eq-rnv1} (resp. \eqref{eq-rnv2}) gives \eqref{Rp-a} (resp. \eqref{Rp-0}).
% 
% First,   the second displayed equation in the statement of  Lemma~\ref{lem-Rn} says that for some $\rho<1$,
% $\|R_n(\hat\kappa_\sigma v_Y)\|_{\cB_0}\ll \rho^n (\|\hat\kappa_\sigma\|_{L^{\gamma q}(\mu_Y)}\, \|v_Y\|_{ L^{4}(\mu_Y)})$.
% The first part of the statement follows since $R'(1,0)=\sum_{n=0}^\infty R_n(i\hat\kappa_\sigma)$
% 
Next, note that
\begin{align*}
&\ \|(\tilde R_t-\tilde R_0-t\cdot \tilde R'_0)(v_Y w)\|_{\cB_0} \le \sum_{n\ge 1} |\lambda_t^{-n}|\, \|R_n((e^{it\cdot \hat\kappa_\sigma}-1-it\cdot \hat\kappa_\sigma)(v_Yw))\|_{\cB_0}\\
&\ \ \ \ \ \ \ \ \ \ \ \ \ \ \ \ \ \ \ \ \ \ \ \ \ + \sum_{n\ge 1} |\lambda_t^{-n}-1|\left(\|R_n (v_Yw)\|_{\cB_0}+\|R_n(\hat\kappa_\sigma (v_Yw))\|_{\cB_0}\right)\\
&\le C_0|t|^\gamma(\|\hat\kappa_\sigma\|^\gamma_{L^{\gamma q}(\mu_Y)}\, \|v_Y\|_{ L^{b}(\mu_Y)}) \sum_{n\ge 1} 
 \rho^{\frac n2}\Vert w\Vert_{\cB_0}	\\
&\ \ \  + |1-\lambda_t|
\sum_{n\ge 1} n
\, |\lambda_t^{-n}|
\left(\|R_n (v_Yw)\|_{\cB_0}+
t\|R_n(\hat\kappa_\sigma (v_Yw))\|_{\cB_0}\right)\, ,
\end{align*}
where in the last inequality we have used Lemma~\ref{lem-Rn}
for a suitable $\rho$
together with the fact that $ |\lambda_t^{-n}-1|= |1-\lambda_t|\left|\sum_{k=1}^{n}\lambda_t^{-k}\right|$
and $|\lambda^{-1}_t|<\rho^{-\frac 12}$ if $t$ is small enough (due to \eqref{devlambda}).
We conclude by applying \eqref{eq-rnv1} and \eqref{eq-rnv2}
combined with $|\lambda^{-1}_t|<\rho^{-\frac 12}$.
%We already know that $\|R_n(\hat\kappa_\sigma v_Y)\|_{\cB_0}\ll \rho^n (\|\hat\kappa_\sigma\|_{L^{\gamma q}(\mu_Y)}\, \|v_Y\|_{ L^{4}(\mu_Y)})$.Also, by  equation~\eqref{RnDA} in Lemma~\ref{lem-Rn}, $\|R_n v_Y\|_{\cB_0}\ll \rho^n  \|v_Y\|_{ L^{2}(\mu_Y)}$.
%the second term in the previous displayed equation is
%\[
%O(|\lambda_t-1|  (\|\hat\kappa_\sigma\|_{L^{\gamma q}(\mu_Y)}\, \|1_Yv\|_{ L^{4}(\mu_Y)}))=
%O(t^2\log (1/|t|) (\|\hat\kappa_\sigma\|_{L^{\gamma q}(\mu_Y)}\, \|1_Yv\|_{ L^{4}(\mu_Y)})),
%\]
%as required. 
%The claims on $\|\tilde R'_0\cdot\|_{\cB_0}$  and  $\|\tilde R_t-\tilde R_0-t\tilde R'_0\|_{\cB_0}$ follows by taking $v_Y=1_Y$.
\end{proof}
Recall that $\Pi_t$ acts on $\cB$. The next lemma is a restatement of~\cite[Lemma 3.14]{BalintGouezel06} in terms of the 
eigenprojection $\Pi_t$, as opposed to the (normalized) eigenvector $\frac{\Pi_t 1}{\int_\Delta \Pi_t 1\, d\mu_\Delta}$, as there.
\begin{lem}
\label{lem-projbase}For any small $t$ small enough
 and for any $v\in\cB$
$\tilde R_t(1_Y\Pi_tv)=(1_Y\Pi_t v)$.
\end{lem}
\begin{proof}
Since $\Pi_tv\in\cB$, observe that $1_Y\Pi_tv\in\mathcal B_1$.
For all $x\in f^{-1}(Y)$,
\begin{align}
\label{eq-Pit}
\Pi_t v(x)=\lambda_t^{-{\omega(x)}}P_t^{\omega(x)}(\Pi_t v)(x)=\lambda_t^{-{\omega(x)}} e^{it\cdot \hat\kappa_{\omega(x)}(
\pi_0(x))}\Pi_t (v)\circ\pi_0(x)\, .
\end{align}
Therefore, for all $y\in Y$,
\begin{align*}
\label{eq-Pit}
\Pi_t v(y)&=\lambda_t^{-1}P_t\Pi_tv(y)=\sum_{x\in f^{-1}(y)}\chi\circ\pi_0(x)\lambda_t^{-1}e^{it\cdot\hat\kappa(x)}\Pi_t(v)(x)\\
&=\sum_{z\in F^{-1}(y)}\chi(z)\lambda_t^{-\sigma(z)} e^{it\cdot \hat\kappa_{\sigma(z)}(z)}\Pi_t (v)(z)=\tilde R_t(1_Y\Pi_tv)(x)\, ,
\end{align*}
due to \eqref{eq-Pit}, since $\omega(x)+1=\sigma(\pi_0(x))$ and since, for every $x\in f^{-1}(Y)$, 
$x=f^{\omega(x)}(\pi_0(x))$.~\end{proof}
By Lemma~\ref{lem-projbase}, for $t$ small enough, the eigenvalue $\tilde\lambda_t$  of $\tilde R_t$ associated with   the projection $(1_Y\Pi_t) v$
is so that $\tilde\lambda_t=1$; this lemma tells us how the projection acts on $\cB$.
Let $Q_t$ be the eigenprojection for $\tilde R_t$ associated with  $\tilde\lambda_t=1$. 
Since $1$ is an isolated eigenvalue in the spectrum of $\tilde R_0=R$ and  $\tilde R_t$  is a continuous family of operators (by Lemma~\ref{lem-deriv}),
we have that $\tilde \lambda_t=1$ is  isolated in the spectrum of $\tilde R_t$ for every $t$ small enough.
Hence,  there exists $\delta_0>0$ so that for any $\delta\in(0,\delta_0)$,
\begin{equation}
\label{eq-Qt}
Q_t=\frac 1{2i\pi}\int_{|\xi-1|=\delta}(\xi I-\tilde R_t)^{-1}\, d\xi.
\end{equation}
Since $(\xi I-\tilde R_0)^{-1}$ and  $\tilde R'_0$ are well defined  operators in $\cB_1\subset L^\infty(\mu_Y)$, so 
is the derivative at $t=0$ of $t\mapsto Q_t\in\cB_1$ and write
\begin{align}
\label{eq-Qtderiv}
Q'_0=\frac 1{2i\pi}\int_{|\xi-1|=\delta}(\xi I-\tilde R_0)^{-1}\tilde R'_0(\xi I-\tilde R_0)^{-1}\, d\xi\, ,
\end{align}
for $\delta>0$ small enough.
Recall that $Q_0, Q_0'$ are well defined in $\cB_1$.
The next result shows that $\|Q_0'h\|_{\cB_0}$ is also well defined for $h=v_Yw$ with $v_Y,w$ as in Lemma~\ref{lem-Rn}.
\begin{lem}
\label{lem-eigf} 
Let $b,q,\gamma$ as in Lemma~\ref{lem-Rn}.
Then  there exist $C_1, C_2>0$ such that for $t$ small enough,
for every $v_Y\in L^{b}(\mu_Y)$ constant on  each $a\in\alpha$ and every $w\in \cB_1$,
\begin{align*}
\Vert Q_t(v_Yw)\Vert_{\cB_0}&\le C_1%\Vert v_Y\Vert_{L^b(\mu_Y)}
\Vert w\Vert_{\cB_0}\|v_Y\|_{L^b(\mu_Y)}\, ,\\
%{\color{red}Control of $Q_0(vw), Q_t(vw)$ by $\|vw\|_{\mathcal B}+\Vert w\Vert_{\beta}\|v_Y\|_{L^b(\mu_Y)}$?????}
%Fix $q\in (1,4/3)$. Then for every $\gamma\in (1,3/(2q))$,
\|Q_0'(v_Yw)
\|_{\cB_0}&\le C_1\|\hat\kappa_\sigma\|_{L^{\gamma q}(\mu_Y)}^\gamma\, \|v_Y\|_{ L^{b}(\mu_Y)}\Vert w\Vert_{\cB_0}\, ,\\
%and
 \|(Q_t -Q_0-tQ_0')(v_Yw)\|_{\cB_0}&\le C_2 |t|^\gamma(\|\hat\kappa_\sigma\|_{L^{\gamma q}(\mu_Y)}^\gamma\, \|v_Y\|_{ L^{b}(\mu_Y)})\Vert w\Vert_{\cB_0}\, .
\end{align*}
\end{lem}
\begin{proof}
% Using~\eqref{eq-Qt}, compute that
%\begin{align*}
%\Vert Q_t(w)\Vert_{\cB_0} 
%%&=\Vert Q_0(v_Yw)\Vert_{\cB_0}+\Vert (Q_t-Q_0)(v_Yw)\Vert_{\cB_0} \\
%&\le \int_{|\xi-1|=\delta}\Vert(\xi I-\tilde R_t)^{-1}(w)\Vert_{\cB_0} d\xi\ll 
%%+\Vert (Q_t-Q_0)(v_Yw)
%\|w\|_{\cB_0}\, ,
%\end{align*}
%for any $t$ small enough
The first estimate will come from our estimates of $\|(Q_t -Q_0-tQ_0')(v_Yw)\|_{\cB_0}$ and $\|Q_0'(v_Yw) \|_{\cB_0}$
%. The first estimate of
%the lemma will then follow from the following estimate.  
and from \eqref{eq-Qt} ensuring that
\begin{align*}\label{eq-xiR0Y}
\Vert Q_0(vw)\Vert_{\cB_0}&=\frac 1{2\pi}
\left\|\int_{|\xi-1|=\delta}(\xi I-\tilde R_0)^{-1}(v_Yw) d\xi\right\Vert_{\cB_0}\\
&=\frac 1{2\pi}\left\|\int_{|\xi-1|=\delta}\!\!\!\!\!\! (\xi I- R)^{-1}(v_Yw-\mu_Y(v_Yw)) d\xi+\mu_Y(v_Yw)\int_{|\xi-1|=\delta}\!\!\!\!\!\! (\xi I-1)^{-1}1_Y\, d\xi\right\Vert_{\cB_0}
\end{align*}
Thus
$\Vert Q_0(vw)\Vert_{\cB_0}\le \frac 1{2\pi}\left( \sum_{j\ge 1}\int_{|\xi-1|=\delta}|\xi|^{-j-1}\|R^j(v_Yw-\mu_Y(v_Yw))\|_{\cB_0} d\xi\right)+   |\mu_Y(v_Yw)|$. So
\[
\Vert Q_0(vw)\Vert_{\cB_0}\ll \left(\sum_{j\ge 1}(1-\delta)^{-j-1}\rho^{j-1}\|v_Y\|_{L^b(\mu_Y)}\|w\|_{\cB_0}\right)+   |\mu_Y(v_Yw)|\ll \Vert w\Vert_{\cB_0}\Vert v_Y\Vert_{L^b(\mu_Y)}\, ,
\]
for some $\rho\in(0,1)$, where we used \eqref{Rp-0} combined with the spectral properties of $R$, up to take $\delta$ small enough so that $(1-\delta)^{-1}\rho<1$.
Second
\begin{align*}
&(Q_t -Q_0)= \frac  1{2i\pi}\int_{|\xi-1|=\delta}(\xi I-\tilde R_t)^{-1}(\tilde R_0- \tilde R_t)(\xi I-\tilde R_0)^{-1}  d\xi\\\
&=E_1(t)+E_2(t)= \frac  1{2i\pi}\int_{|\xi-1|=\delta}(\xi I-\tilde R_0)^{-1}(\tilde R_0- \tilde R_t)(\xi I-\tilde R_0)^{-1} \, d\xi\\
&+\frac  1{2i\pi}
 \int_{|\xi-1|=\delta}\Big((\xi I-\tilde R_t)^{-1}-(\xi I-\tilde R_0)^{-1}\Big)(\tilde R_0- \tilde R_t)(\xi I-\tilde R_0)^{-1} d\xi\, .
\end{align*}
Note that
\[
%begin{align*}
\|E_2(t)( v_Yw)\|_{\cB_0}
\ll \int_{|\xi-1|=\delta} \Big\|(\xi I-\tilde R_t)^{-1}-(\xi I-\tilde R_0)^{-1} \Big\|_{\cB_0} 
\|(\tilde R_0- \tilde R_t)(\xi I-\tilde R_0)^{-1}(v_Yw)\|_{\cB_0}\, d\xi.
\]
%\end{align*}
Recall that for all $\xi$ so that $|\xi-1|=\delta$, $\|(\xi I-\tilde R_t)^{-1}\|_{\cB_0}\ll 1$, for all $t$ small enough.
This together with Lemma~\ref{lem-deriv} with $v_Y=1_Y$ implies that for all  $t$ small enough,
\[
\left\|(\xi I-\tilde R_t)^{-1}-(\xi I-\tilde R_0)^{-1}\right\|_{\cB_0}
=\left\|(\xi I-\tilde R_0)^{-1}(\tilde R_t-\tilde R_0)(\xi I-\tilde R_t)^{-1}\right\|_{\cB_0}
\ll  |t|.
\]
Hence 
$
\|E_2(t)(v_Yw )\|_{\cB_0}\ll |t| \int_{|\xi-1|=\delta}\|(\tilde R_0- \tilde R_t)(\xi I-\tilde R_0)^{-1}(v_Yw)\|_{\cB_0}\, d\xi$. 
To simplify notations, we write $\mu_Y(\cdot)$ for $\int_Y\cdot\, d\mu_Y$.
We claim that
\begin{equation}
\label{eq-xiRvY}
\int_{|\xi-1|=\delta}\!\!\!\!\!\!\!\!\! \|(\tilde R_0-\tilde R_t)(\xi I-\tilde R_0)^{-1}(v_Yw)\|_{\cB_0}\, d\xi\ll   |t|  \|\hat\kappa_\sigma\|_{L^{\gamma q}(\mu_Y)}^\gamma   \|v_Y\|_{L^b(\mu_Y)}\|w\|_{\cB_0} 
%=\int_{|\xi-1|=\delta}\left\|(\tilde R_0-\tilde R_t)(v_Yw-\mu_Y( v_Yw))\, d\mu_Y\right\|_{\cB_0}\!\!\! d\xi+D(v_Yw)
\, .
\end{equation}
%where $\|D(v_Yw)\|_{\cB_0}\ll t\|v_Yw-\mu_Y( v_Yw)\|_{L^1(\mu_Y)}$.
%By the second part of the conclusion in Lemma~\ref{lem-deriv} (with $v_Yw-\mu_Y( v_Yw)$ instead of $v_Yw$ there),
%$\|(\tilde R_0- \tilde R_t)(v_Yw-\mu_Y(v_Yw))\|_{\cB_0}\ll  |t|\, \|\hat\kappa_\sigma\|_{L^{\gamma q}(\mu_Y)}^\gamma\|v_Y\|_{L^b(\mu_Y)}\|w\|_{\cB_0}$.
This 
%together with~\eqref{eq-xiRvY} 
implies that
\begin{align*}
\|E_2(t)(v_Y w)\|_{\cB_0}\ll |t|^2\|\hat\kappa_\sigma
\|_{L^{\gamma q}(\mu_Y)}^\gamma\|v_Y\|_{L^b(\mu_Y)}\|w\|_{\cB_0}\, .
\end{align*}
Next, compute that
\begin{align*}
E_1(t)=tQ_0'  +E(t)&=\frac  t{2i\pi}\int_{|\xi-1|=\delta}(\xi I-\tilde R_0)^{-1}\tilde R'_0(\xi I-\tilde R_0)^{-1} \, d\xi\\
&+\frac  1{2i\pi} \int_{|\xi-1|=\delta}(\xi I-\tilde R_0)^{-1}(\tilde R_0- \tilde R_t-t\tilde R'_0)(\xi I-\tilde R_0)^{-1}  \, d\xi\, .
\end{align*}
Proceeding as in estimating $E_2(t)(v_Yw)$ above and using the first part of the conclusion in Lemma~\ref{lem-deriv}
 and a formula analgous to \eqref{eq-xiRvY},
\begin{align}
\|Q_0'(v_Yw )\|_{\cB_0}&\ll \int_{|\xi-1|=\delta}\|\tilde R'_0(\xi I-\tilde R_0)^{-1}(v_Yw)\|_{\cB_0}\, d\xi\nonumber\\
%&\ll \int_{|\xi-1|=\delta}\|\tilde R'_0 (v_Yw-\mu_Y(v_Yw))\|_{\cB_0}\, d\xi+ \|v_Yw-\mu_Y(v_Yw)\|_{L^1(\mu_Y)}\\
&\ll\|\hat\kappa_\sigma\|_{L^{\gamma q}(\mu_Y)}^\gamma
%\int_{|\xi-1|=\delta} 
\| v_Y\|_{L^b(\mu_Y)}
%\, d\xi 
\|w\|_{\cB_0} \, ,\label{claim}
\end{align}
%Therefore
%%This together with~\eqref{eq-xiRvY}) gives that
%\[
%\|Q_0'(v_Yw )\|_{\cB_0}\ll \|\hat\kappa_\sigma\|_{L^{\gamma q}(\mu_Y)}^\gamma\|v_Y\|_{L^b(\mu_Y)}\|w\|_{\cB_0} 
%\]
and thus, the second part of the conclusion follows.
Finally, similarly to~\eqref{eq-xiRvY}, we claim that
\begin{align}
\label{eq-claim2}
\nonumber \|E(t)(v_Y w)\|_{\cB_0}&\ll \int_{|\xi-1|=\delta}\|(\tilde R_0- \tilde R_t-t\cdot \tilde R'_0)(\xi I-\tilde R_0)^{-1}(v_Yw)\|_{\cB_0}\, d\xi\\
%&= \int_{|\xi-1|=\delta}\|(\tilde R_0- \tilde R_t-t\cdot \tilde R'_0)(v_Yw-\mu_Y(v_Yw)
%%\int v_Y\, d\mu_Y
%)\|_{\cB_0}\, d\xi + D_0(v_Yw),
%\end{align}
%where $\|D_0(v_Yw)\|_{\cB_0}\le \|v_Yw-\int v_Yw\, d\mu_Y\|_{L^1(\mu_Y)}\|w\|_{\cB_0}$. This together with Lemma~\ref{lem-deriv} gives that 
%\[
%\|E(t)(v_Y w)\|_{\cB_0} 
&\ll |t|^\gamma|\,  \|\hat\kappa_\sigma\|_{L^{\gamma q}(\mu_Y)}^\gamma\|v_Y\|_{L^b(\mu_Y)}\|w\|_{\cB_0}\, .
\end{align}
The third part of the conclusion follows by putting all the above together.
It remains to prove the claims~\eqref{eq-xiRvY}, \eqref{claim} and~\eqref{eq-claim2}. 
 Note that for any operator $\tilde P$  bounded in $\cB_1$, we have
\begin{align*}
&\int_{|\xi-1|=\delta}\|\tilde P(\xi I-\tilde R_0)^{-1}(v_Yw)\|_{\cB_0}\, d\xi\\
&=\int_{|\xi-1|=\delta} \Big\|\tilde P(\xi I-\tilde R_0)^{-1} (v_Yw-\mu_Y(v_Yw)1_Y)+ (\xi-1)^{-1}\mu_Y( v_Yw)\tilde P 1_Y\Big\|_{\cB_0}\, d\xi\, .
\end{align*}
In the sequel, we take $\tilde P\in\{Id, \tilde R_0- \tilde R_t,\tilde R'_0,\tilde R_0-\tilde R_t-t\tilde R'_0\}$, respectively. 
Since $\tilde R_0=R$ has a spectral gap in $\cB_1$ with decomposition $R^j=\mu_Y(\cdot)1_Y+N'_j$ with $\|N'_j \|_{\cB_0}\le C\rho^j$ for some $C>0$ and some $\rho<1$, for every $j\ge 1$, we can write
\[
R^j(v_Yw-\mu_Y(v_Yw)1_Y)=R^{j-1}R(v_Y-\mu_Y(v_Yw)1_Y)=
%Q_0(R(v_Y-\int_Y v_Y\, d\mu_Y))+ 
N'_{j-1}(R(v_Yw-\mu_Y(v_Yw)1_Y)),
\]
since $\mu_Y(R(v_Yw-\mu_Y(v_Yw)))=0$.
We note that although $(v_Y-\mu_Y(v_Y)1_Y)\notin\cB_1$. But, by the first two displayed inequalities in Appendix~\ref{sec-RnDA}
(with $vw$ replaced by $v_Yw$ and $\mu_Y(v_Yw)1_Y$),
\[
\Big\|R(v_Yw-\mu_Y(v_Yw))\|_{\cB_0}\le C' \|w\|_{\mathcal B_0}\|v_Y\|_{L^1(\mu_Y)}\, ,
\]
for some $C'>0$. Thus, there exist $C">0$ and $\rho<1$ so that for $j\ge 1$, 
\begin{equation}\label{majoRj}
\Big\|R^j(v_Yw-\mu_Y(v_Yw))\Big\|_{\cB_0}\le C"\rho^j
 \|w\|_{\mathcal B_0}\|v_Y\|_{L^1(\mu_Y)}
%\|v_Yw-\mu_Y(v_Yw)\|_{L^1(\mu_Y)}.
\end{equation}
Below we write  $\gamma(\tilde P)$ to indicate that this is a positive number that depend on the operator $\tilde P\in\{\tilde R_0-\tilde R_t,\tilde R'_0 ,\tilde R_0-\tilde R_t-t\tilde R_0\}$.
Using the fact that $\tilde R_0=R$ and that $(\xi I-R)^{-1}=\sum_{j\ge 0}\xi^{-j-1}R^j$ and putting the above together, we obtain
\begin{align*}
&\int_{|\xi-1|=\delta} \|\tilde P(\xi I-R)^{-1} v_Yw\|_{\cB_0}\, d\xi\le\int_{|\xi-1|=\delta} \left\| \xi^{-1}\tilde P(v_Yw-\mu_Y( v_Yw))\right\|_{\cB_0}\, d\xi\\
&\ \ \ \ \ \ \ \ +\int_{|\xi-1|=\delta}\left\Vert \sum_{j\ge 1} \xi^{-j-1}\tilde P R^{j}(v_Yw-\mu_Y( v_Yw))+ (\xi-1)^{-1}\mu_Y( v_Yw)\tilde P 1\right\Vert_{\cB_0}\, d\xi\\
&\ll t^{\gamma(\tilde P)} \left(\|\hat\kappa_\sigma\|_{L^{\gamma q}(\mu_Y)}^\gamma\, \|v_Y\|_{L^b(\mu_Y)}\|w\|_{\cB_0}+\sum_{j\ge 1}\int_{|\xi-1|=\delta}  |\xi|^{-j-1}\Big\|  R^{j}(v_Yw-\mu_Y( v_Yw))\Big\|_{\cB_0}\, d\xi\right)\\
&\ll t^{\gamma(\tilde P)} \left(\|\hat\kappa_\sigma\|_{L^{\gamma q}(\mu_Y)}^\gamma\, \|v_Y\|_{L^b(\mu_Y)}\|w\|_{\cB_0}+\sum_{j\ge 1}\int_{|\xi-1|=\delta}  |\xi|^{-j-1}\rho^j \Vert v_Y\Vert_{L^1(\mu_Y)}\|w\|_{\cB_0}\, d\xi\right).
\end{align*}
due to Lemma \ref{lem-deriv} and to \eqref{majoRj}, with $\gamma(\tilde R_0-\tilde R_t)=1$, $\gamma(\tilde R'_0)=0$, $\gamma(\tilde R_0-\tilde R_t-t\tilde R'_0)=\gamma$.
%Taking $\tilde P=(R(1,0)- R(\lambda_t, t)$ and  $\tilde P=(R(1,0)- R(\lambda_t, t)-tR'(1,0)$, respectively and using Lemma~\ref{lem-deriv}, we have
%Therefore
So
\[
\int_{|\xi-1|=\delta} \|\tilde P(\xi I-R)^{-1} (v_Yw)\|_{\cB_0}\, d\xi\ll 
 t^{\gamma(\tilde P)}   \|\hat\kappa_\sigma\|_{L^{\gamma q}(\mu_Y)}^\gamma  \|v_Y\|_{L^b(\mu_Y)}\|w\|_{\cB_0}
\left( 1+\delta \sum_{j\ge 1} \rho^j(1-\delta)^{-j}\right)
\]
The claims follow by choosing $\delta$ 
%small enough 
so that 
$\sum_{j\ge 2}\rho^j|\xi|^{-j}\le \sum_{j\ge 2}\rho^j(1-\delta)^{-j}<\infty$.
~\end{proof}

\subsection{Expansion of $t\mapsto\int_Y\Pi_t\, d\mu_Y$}

The next estimate, of independent interest, requires a more careful analysis and strongly exploits that the modulus is outside of the  integral.
Its proof uses arguments somewhat similar to the ones in~\cite{KellerLiverani99} together with arguments exploiting symmetries on the tower $\Delta$.
Recall that $p>2$.
\begin{lem}\label{lem-eigf2}
\label{cor-betcont} 
Let $b\in(p,+\infty]$. Let $c_0 (v, w)$ be as  defined in 
%the statement of 
Proposition~\ref{prop-expproj-base}, namely
\[
c_0(vw)=i   \left(\sum_{j\ge 0}\int_\Delta \hat\kappa\circ f^j \, .vw\, d\mu_\Delta+ \int_\Delta vw \, d\mu_\Delta\sum_{j\ge 0}\int_{\Delta} \hat\kappa\, \frac{1_{Y}\circ f^{j+1}}{\mu_\Delta(Y)}\, d\mu_\Delta\right)\in\mathbb C^2\, .
\]
Then for every $(v,w)$ as in
%satisfying the assumptions of 
Proposition~\ref{prop-expproj}
%there exists
%$c_0(v)\in\mathbb R$ satisfying $|c_0(v)|\le C\|v\|_{L^1(\mu_\Delta)}$ for some $C>0$
%so that 
and for every $\gamma\in (1,2\frac{p-1}p)$,
\begin{equation*}
\left|\int_Y (\Pi_t-\Pi_0)(vw)\, d\mu_Y-c_0(vw)\cdot t \right|\ll |t|^\gamma (\Vert vw\Vert_{\mathcal B}+\Vert w\Vert_{\mathcal B_0}\|v\|_{L^b(\mu_\Delta)}) \mbox{ as }t \to 0\,.
\end{equation*}

%with $c_0(1_\Delta):=-2\pi\, \sum_{j\ge 0}\int_{\Delta} \hat\kappa\, 1_{Y}\circ f_\Delta^{j+1}\, d\bar\mu_\Delta$,
\end{lem}

%\begin{rem} 
%\label{rmk=derint}It follows from  Lemma~\ref{lem-eigf2} and Lemma~\ref{cor-betcont} that $\int_Y\Pi_0' 1\, d\mu=c_0$ 
%and $\int_Y\Pi_0' v\, d\mu=c_0(v)$, respectively
%\end{rem}

\begin{proof}
Let $\gamma\in \left(1, 2\frac{p-1}p\right)$, i.e. $1<\gamma<2\left( 1-\frac 1p\right)<2\left( 1-\frac 1b\right)$. Fix $\varepsilon>0$ so that $\gamma+\varepsilon<2\left( 1-\frac 1p\right)<2$.
Fix $1<q<\frac {2b}{b+2}$ close enough to 1 so that $1<\gamma+\varepsilon<2\left(\frac 1q-\frac 1b\right)$.
%Let $\eta\in\left(1, 4\left(\frac 1q-\frac 1p\right)-2\right)$, and let $\varepsilon>0$.
%Let us prove that $|\int_Y (\Pi_t-\Pi_0-t\Pi_0') 1|\ll |t|^{\gamma}$.
Let $p_1\in(1,2)$.
Consider $\theta\in(0,1)$ satisfying \eqref{spgap-Sz-bis} and such that  $\theta_1^{\frac {(p_1-1)}{p_1}}<\theta$. Set $\theta_0:=\sqrt{\theta}$ 
%Up to increase the value of $\theta\in(0,1)$, we also assume, thanks to Proposition~\ref{prop-convsum}, that
%\begin{equation}\label{expdeccov}
%\sum_{j\ge N'}\left|\int_{\Delta}\hat\kappa\, 1_{Y}\circ f_\Delta^{j+1}\, d\bar\mu_\Delta\right|\ll \theta^{N'}\, .
%\end{equation}
and $r:=\theta_0^{1/\vartheta}$ with $\vartheta
%>\frac{2\eta}{2-\eta}
>2$ large enough (i.e. $r$ close enough to 1) so that
\begin{equation}\label{condr}
\frac{2(\gamma+\varepsilon)}{\vartheta-1}<\varepsilon\quad\mbox{and}\quad
\frac{\left(\gamma+\varepsilon\right)(\vartheta-2)}{\vartheta-1}>\gamma-1\, .
\end{equation}
We choose $\delta$ small enough so that $(1-\delta)^{-\frac p{p-1}}\theta_1<1$,
$(1-\delta)^{-\frac b{b-1}}\theta_1^{
\frac b{b-1}(\frac 1p-\frac 1{b})}<1$,
\begin{equation}
\label{eq-delta}
\delta+r<1, \quad
(1-\delta)^{-1}\rho<1,\quad  
(1-\delta)^{-1}\theta_0<1
\end{equation}
and so that 
$\Pi_t=\int_{|\xi-1|=\delta}(\xi I-P_t)^{-1}\, d\xi$ for every $t$ in a small neighbourhood of 0 (this is possible thanks to~\cite{KellerLiverani99} since
1 is isolated in the spectrum of $P_0$).
We will make use of this choice from equation~\eqref{eq-enm} onwards.
Recall that $(\xi I-P_t)^{-1}-(\xi I-P_0)^{-1}=(\xi I-P_0)^{-1}(P_t-P_0)(\xi I-P_t)^{-1}$, and so
\begin{align*}
&\int_Y(\Pi_t -\Pi_0) (vw) \, d\mu_Y=\frac  1{2i\pi}\int_Y \int_{|\xi-1|=\delta}(\xi I-P_0)^{-1}(P_t-P_0)(\xi I-P_t)^{-1} (vw)\, d\xi\, d\mu_Y\\
&=I_1(t)+I_2(t)=\frac  1{2i\pi}\int_Y \int_{|\xi-1|=\delta}(\xi I-P_0)^{-1}(P_t- P_0)(\xi I-P_0)^{-1} (vw)\, d\xi\, d\mu_Y\\
&+\frac  1{2i\pi}\int_Y \int_{|\xi-1|=\delta}(\xi I-P_0)^{-1}(P_t- P_ 0)\Big((\xi I-P_t)^{-1}-(\xi I-P_0)^{-1}\Big) (vw)\, d\xi\, d\mu_Y\, .
\end{align*}
\begin{itemize}
\item Let us start with the computation of $I_1(t)$. Setting 
$N':=\lfloor(\gamma+\varepsilon)\log t/\log \theta_0\rfloor$,
\begin{align}
I_1(t)
&=\frac  1{2i\pi}\int_Y \int_{|\xi-1|=\delta}(\xi I-P_0)^{-1}(P_t- P_0)(\xi I-P_0)^{-1} (vw)\, d\xi\, d\mu_Y\nonumber\\
&= I_{1,N'}(t)+I'_{1,N'}(t)=\frac  1{2i\pi}\sum_{j=0}^{N'-1}\int_Y \int_{|\xi-1|=\delta}\xi^{-j-1}P_0^j(P_t- P_0)(\xi I-P_0)^{-1} (vw)\, d\xi\, d\mu_Y\nonumber\\
&\ \ \ +\frac  1{2i\pi}\int_Y \int_{|\xi-1|=\delta}\xi^{-N'}(\xi I-P_0)^{-1}P_0^{N'}(P_t- P_0)(\xi I-P_0)^{-1} (vw)\, d\xi\, d\mu_Y\, .\label{eqI1}
\end{align}
Observe that $|\xi-1|=\delta$ implies $|\xi^{-1}|\le(1-\delta)^{-1}<r^{-1}$.
Due to Lemma \ref{lem-expproj},
\begin{align*}
& \left|I_{1,N'}(t)-\frac  1{2i\pi}\sum_{j=0}^{N'-1}\int_Y \int_{|\xi-1|=\delta}\xi^{-j-1}P_0^j(t\cdot P'_0)(\xi I-P_0)^{-1} (vw)\, d\xi\, d\mu_Y \right|\\
&\ll r^{-N'}\left\Vert (P_t- P_0-t\cdot P'_0)(\xi I-P_0)^{-1} (vw)\right\Vert_{L^1(\Delta)}\ll r^{-N'}t^{\gamma+\varepsilon}\Vert vw\Vert_{\mathcal B}\, ,
\end{align*}
since our assumption on $\varepsilon,q$ ensures that $\gamma+\varepsilon<2\left(\frac 1q-\frac 1p\right)$.
Therefore
\begin{equation}
 \left|I_{1,N'}(t)-
\frac  1{2\pi}
\sum_{j=0}^{N'-1}\int_{f^{-j-1}Y} \int_{|\xi-1|=\delta}\xi^{-j-1}(t\cdot \kappa)(\xi I-P_0)^{-1} (vw)\, d\xi\, d\mu_Y \right|\ll r^{-N'}t^{\gamma+\varepsilon}\Vert vw\Vert_{\mathcal B}\, .
\end{equation}
Moreover, 
since $\sup_{\vert \xi-1\vert=\delta}\left\Vert(\xi I-P_0)^{-1}\right\Vert_{\mathcal B}<\infty$, using the Fubini theorem, we obtain
\begin{align*}
&I'_{1,N'}(t)
=\frac  1{2i\pi}\int_Y \int_{|\xi-1|=\delta}\xi^{-N'}(\xi I-P_0)^{-1}(\Pi_0+N_0^{N'})(P_t- P_0)(\xi I-P_0)^{-1} (vw)\, d\xi\, d\mu_Y\nonumber\\
&= \frac  1{2i\pi \mu_\Delta(Y)}\int_{|\xi-1|=\delta}\!\!\!\!\!\!\!\! \xi^{-N'}\Pi_0\left(1_Y(\xi I-P_0)^{-1}1\right)\Pi_0((P_t- P_0)(\xi I-P_0)^{-1} (vw))   \, d\xi +O(
r^{-N'}\theta^{N'}
%  t
)\nonumber\\
&= 
\frac  1{2\pi\mu_\Delta(Y)}
\int_{|\xi-1|=\delta}\xi^{-N'}\Pi_0\left(1_Y(\xi I-P_0)^{-1}1\right)\Pi_0(t\cdot \hat\kappa(\xi I-P_0)^{-1} (vw))   \, d\xi\,\\
&+O((r^{-N'}|t|^{\gamma+\varepsilon}+r^{-N'}\theta^{N'} % |t|
)\Vert vw\Vert_{\mathcal B})\, ,
\end{align*}
where we used again Lemma~\ref{lem-expproj}.
Since $r=\theta_0^{1/\vartheta}$, $\vartheta>2$ and $\theta=\theta_0^{1/2}$, we obtain
\begin{equation}\label{I1Formule}
I_1(t)= 
\frac  1{2\pi}
t\cdot A_{1,N'}(t) +O((r^{-N'}t^{\gamma+\varepsilon}+\theta_0^{\frac{(\vartheta-2)N'}{2}}
% t
)\Vert vw\Vert_{\mathcal B})\, ,
\end{equation}
%with
\begin{align}
\label{eq-1prime}
 A_{1,N'}(t)&:=U_{N'}(t)+V_{N'}(t)=
\sum_{j=0}^{N'-1}\int_{Y} \int_{|\xi-1|=\delta}\xi^{-j-1}P^{j+1}\left( \hat \kappa(\xi I-P_0)^{-1} (vw)\right)\, d\xi\, d\mu_Y\\
\nonumber &+\frac 1{2\pi\mu_\Delta(Y)}\int_{|\xi-1|=\delta}\xi^{-N'}\Pi_0\left(1_Y(\xi I-P_0)^{-1}1\right)\Pi_0( \hat\kappa(\xi I-P_0)^{-1} (vw))   \, d\xi\, .
\end{align}
Due to our choice of $N'$ and since
$\theta_0=r^\vartheta
%<r^{\frac \eta{2-\eta}}=r^{1+\frac {\eta-1}{1-\frac \eta 2}}
$, we have
\begin{equation}\label{estimN'}
r^{N'\vartheta}=\theta_0^{N'} \ll t^{\gamma+\varepsilon}
\end{equation}
so $r^{-N'} \ll t^{-\frac{\gamma+\varepsilon
%-1
 }{\vartheta}}\ll t^{-\varepsilon}
% t^{-(1-\frac \eta 2)}
$, due to our assumptions on $\vartheta$. 
Note that
\[
(\xi I-P_0)^{-1}(vw)=\sum_{j\ge 0} \xi^{-j-1} P_0^j(vw)\quad\mbox{and}\quad
(\xi I-P_0)^{-1}1_\Delta= (\xi-1)^{-1}1_\Delta
\]
and so $V_{N'}(t)=\frac 1{\mu_\Delta(Y)}\int_{|\xi-1|=\delta}\sum_{j\ge 0}\xi^{-N'-j-1}(\xi-1)^{-1}\mu_\Delta(Y)\Pi_0( \hat\kappa P_0^j (vw))   \, d\xi$. Thus
%, due to the Cauchy integral formula, 
\begin{align*}
V_{N'}(t)&=\frac 1{\mu_\Delta(Y)}\sum_{j\ge 0}\int_{|\xi-1|=\delta}\xi^{-N'-j-1}(\xi-1)^{-1}\mu_\Delta(Y)\Pi_0( \hat\kappa N_0^j (vw))   \, d\xi
\end{align*}
since $P_0^j=\mu_\Delta(\cdot)1_\Delta+N_0^j$, since $\Pi_0(\hat\kappa)=0$ and since $\sum_{j\ge 1}|\xi^{-j}\Pi_0(\hat \kappa N_0^j(vw))|\ll \sum_{j\ge 1}|(1-\delta)^{-1}\theta|^j\Vert \hat\kappa\Vert_{L^{p/(p-1)}(\mu_\Delta)}\Vert vw\Vert_{\mathcal B}<\infty$, due to our choice of $\delta$.
Therefore,  due to the Cauchy integral formula, 
\begin{align}
\label{formule TN'}
V_{N'}(t)&=2i\pi  \sum_{j\ge 0}\Pi_0( \hat\kappa N_0^j (vw))
=2i\pi  \sum_{j\ge 0}\int_\Delta \hat\kappa\circ f^j \, vw\, d\mu_\Delta\, .
\end{align}
Combining this with \eqref{I1Formule} and \eqref{eq-1prime}, we obtain
\begin{equation}\label{I1Formulebis}
I_1(t)= 
t\cdot \left(\frac  1{2\pi}U_{N'}(t)+i\sum_{j\ge 0}\int_\Delta \hat\kappa\circ f^j \, .vw\, d\mu_\Delta\right) +O(|t|^\gamma\Vert vw\Vert_\mathcal B)\, .
\end{equation}

%Note that, for $v=1_\Delta$, 

We compute that
\begin{align}
\label{eq-s12v}
&U_{N'}(t) =\frac 1{\mu_\Delta(Y)}\sum_{j=0}^{N'-1} \int_{|\xi-1|=\delta}\xi^{-j-1}\int_{\Delta} 1_Y P^{j+1}\left( \hat \kappa(\xi I-P_0)^{-1} (vw)\right)\,  d\mu_\Delta\, d\xi\\
 \nonumber &=\frac 1{\mu_\Delta(Y)}\sum_{j=0}^{N'-1}\int_{|\xi-1|=\delta}\xi^{-j-1}\int_{\Delta}  \hat \kappa(\xi I-P_0)^{-1} (vw). 1_Y\circ f^{j+1}\, d\mu_\Delta\, d\xi=\frac {U_{1,N'}(t)+U_{2,N'}(t)}{\mu_\Delta(Y)}\, ,
\end{align}
\begin{align*}
U_{1,N'}(t) &=\sum_{j=0}^{N'-1}\int_{|\xi-1|=\delta}\xi^{-j-1}\int_{\Delta} \hat \kappa\, (\xi I-P_0)^{-1} \left(vw-\mu_\Delta(vw)\right)\, 1_Y\circ f^{j+1}\, d\mu_\Delta\, d\xi\, ,\\
U_{2,N'}(t))&=\mu_\Delta( vw) \sum_{j=0}^{N'-1}\int_{|\xi-1|=\delta}\xi^{-j-1}\,(\xi-1)^{-1}\,\int_{\Delta} \hat \kappa\, 1_Y\circ f^{j+1}\, d\mu_\Delta\, d\xi\, ,
\end{align*}
since $(\xi I-P_0)^{-1}1_\Delta=(\xi-1)^{-1}1_\Delta$.
Due to  the Cauchy
integral
formula, 
\begin{align}
U_{2,N'}(t):=& 2i\pi\mu_\Delta( vw)  \sum_{j=0}^{N'-1}\int_{\Delta} \hat\kappa\, 1_{Y}\circ f^{j+1}\, d\mu_\Delta\nonumber \\
&=2i\pi \mu_\Delta( vw) \sum_{j\ge 0}\int_{\Delta} \hat\kappa\, 1_{Y}\circ f^{j+1}\, d\mu_\Delta +O(\theta_0^{N'}\Vert vw\Vert_{L^1(\mu_\Delta)})\nonumber\\
&=
2i\pi \mu_\Delta( vw) \sum_{j\ge 0}\int_{\Delta} \hat\kappa\, 1_{Y}\circ f^{j+1}\, d\mu_\Delta +O(|t|^\gamma\Vert vw\Vert_{\mathcal B})\, ,\label{U2N't}
%\label{A1N}
%&=-2i\pi \sum_{j=0}^{N'}\int_{\bar M} \kappa\, 1_{Y}\circ \bar T^{j+1}\, d\bar\mu\, 
\end{align}
where in  the penultimate line we have used that $\int_{\Delta} \hat\kappa\, 1_{Y}\circ f^{j}\, d\bar\mu_\Delta=O(\theta_0^j)$ (see Lemma~\ref{lemma-lemexpk1y} below) and  in the last line we have used~\eqref{estimN'}.
Next, we estimate $U_{1,N'}$ defined in~\eqref{eq-s12v}.
Since $v|_{(a,\ell)}=C_{(a,\ell)}$ where $C_{(a,\ell)}$ are constants, $U_{1,N'}(t)$ is equal to the following quantities
\begin{align}%\label{U1N't}
\nonumber&=\sum_{j=0}^{N'}\int_{|\xi-1|=\delta}\xi^{-j-1}\int_{\Delta}\sum_{r\ge0} \xi^{-r-1}  P_0^{r}\sum_{(a,\ell)}\left(vw1_{(a,\ell)}-\int_{(a,l)} vw\, d\mu_\Delta\right)\, (\hat\kappa1_Y\circ f^j)\, d\mu_\Delta\, d\xi\\
\nonumber&=\sum_{j=0}^{N'}\int_{|\xi-1|=\delta}\xi^{-j-1}\sum_{r\ge0} \xi^{-r-1}\sum_{(a,\ell)}\int_{\Delta}  P_0^{r}\left(vw1_{(a,\ell)}-\int_{(a,l)} vw\, d\mu_\Delta\right)\, (\hat\kappa1_Y\circ f^j)\, d\mu_\Delta\, d\xi\\
\nonumber&=\sum_{j=0}^{N'}\int_{|\xi-1|=\delta}\xi^{-j-1}\sum_{r\ge0} \xi^{-r-1}\sum_{(a,\ell)}C_{(a,\ell)}\int_{\Delta} \left(w1_{(a,\ell)}- \mu_\Delta(w1_{(a,l)})\right)\,(\hat\kappa\, 1_Y\circ f^j)\circ f^r\, d\mu_\Delta\, d\xi\\
\nonumber&= \sum_{r\ge0} \sum_{(a,\ell)}\sum_{j=0}^{N'}C_{(a,\ell)}\int_{\Delta} \left(w1_{(a,\ell)}- \mu_\Delta(w1_{(a,l)}\right)\,(\hat\kappa\, 1_Y\circ f^j)\circ f^r\, d\mu_\Delta\, \int_{|\xi-1|=\delta}\xi^{-j-r-2}\, d\xi
\end{align}
and thus $U_{1,N'}(t)=0$. The interchange of sums and integrals in the above 
%chain of 
equations is justified by Lemma~\ref{lem-justif...}. This is the only part of the proof where this assumption that is crucially used.
Combining \eqref{I1Formulebis}, \eqref{eq-s12v}, \eqref{U2N't} and $U_{1,N'}(t)=0$, we obtain that
\begin{align*}
I_1(t)&=i  t\cdot \left( \sum_{j\ge 0}\int_\Delta \hat\kappa\circ f^j \, .vw\, d\mu_\Delta+ \int_\Delta vw \, d\mu_\Delta\sum_{j\ge 0}\int_{\Delta} \hat\kappa\, \frac{1_{Y}\circ f^{j+1}}{\mu_\Delta(Y)}\, d\bar\mu_\Delta\right)\\
&+O(t^\gamma(\Vert vw\Vert_{\mathcal B}+\|w\|_{\mathcal B_0}\|v\|_{L^b(\mu_\Delta)}))\, .
\end{align*}
\item Let us prove that
$| I_2(t)|\ll |t|^{\gamma}\Vert vw\Vert_{\mathcal B}$.
This will give the conclusion. 
First, compute that for any $N\ge 1$,
\begin{align}
\label{eq-i2}
\nonumber I_2(t)&=\int_Y \int_{|\xi-1|=\delta} (\xi I-P_0)^{-1}(P_t-P_0)(\xi I-P_0)^{-1}(P_t- P_ 0)(\xi I-P_t)^{-1} (vw)\, d\xi\, d\mu_Y\\
\nonumber &=\sum_{j=0}^{N-1}\int_{|\xi-1|=\delta}  \xi^{-j-1} \int_Y P_0^j (P_t-P_0)(\xi I-P_0)^{-1}(P_t- P_ 0)(\xi I-P_t)^{-1} (vw)\,d\mu_Y\, d\xi\\
\nonumber &+\int_Y \int_{|\xi-1|=\delta} \xi^{-N} (\xi I-P_0)^{-1} P_0^N (P_t-P_0)(\xi I-P_0)^{-1}(P_t- P_ 0)(\xi I-P_t)^{-1} (vw)\, d\xi\, d\mu_Y\\
&=S_N(t)+E_N(t).
\end{align}
We further decompose $S_N$ and $E_N$, choosing 
$N:=\lfloor
(\gamma+\varepsilon)\frac{\log t}{\log(\theta_0/r)}\rfloor$ to prove the claim.
We first deal with $S_N$. 
\begin{align}
\label{eq-sn}
\nonumber S_N(t)&=\sum_{j=0}^{N-1} \sum_{\ell=0}^{N-1}\int_{|\xi-1|=\delta}  \xi^{-j
-1}\xi^{-\ell
-1}  \int_Y P_0^j (P_t-P_0)P_0^\ell (P_t- P_ 0)(\xi I-P_t)^{-1} (vw)\,d\mu_Y\, d\xi\\
\nonumber &+\sum_{j=0}^{N-1} \int_Y\int_{|\xi-1|=\delta}  \xi^{-j
-1} P_0^j (P_t-P_0) (\xi I-P_0)^{-1}\xi^{-N}  P_0^N(P_t- P_ 0)(\xi I-P_t)^{-1} (vw)\, d\xi\, d\mu_Y\\
&=S_{N, N}(t)+E_{N, N}(t).
\end{align}
Set $\gamma_0=\frac {\gamma+\varepsilon}2\in(0,1)$.
%and having adjusting $\Delta$ and $\mathcal B$ such that
Note that $\gamma+\varepsilon<2\left(1-\frac 1p\right)$ implies that $\frac p{p-1}\gamma_0<1$.
\begin{align}
\Big |\int_Y P_0^j &(P_t-P_0)P_0^\ell (P_t- P_ 0)(\xi I-P_t)^{-1} (vw)\,d\mu_Y\Big|\nonumber\\
&=\Big |\int_Y (e^{it\cdot\hat\kappa}-1)P_0^\ell (P_t- P_ 0)(\xi I-P_t)^{-1} (vw)\,d\mu_Y\Big|
\nonumber\\
%&\ll \|(\xi I-P_t)^{-1}\|_{\cB}\,\left|g\mapsto \int_Y (e^{it\hat\kappa}-1)\circ f^{\ell+1} (e^{it\hat\kappa}-1)g \, d\mu_Y\right|_{\mathcal B\rightarrow\mathbb C}\\
%&\ll \|(\xi I-P_t)^{-1} (vw)\|_{\cB}\,\left|g\mapsto \int_\Delta ((e^{it\cdot\hat\kappa}-1)1_Y)\circ f^{\ell+1}\, (e^{it\cdot\hat\kappa}-1)g \, d\mu_\Delta\right|_{\mathcal B\rightarrow\mathbb C}\nonumber\\
&\ll \|(\xi I-P_t)^{-1}(vw)\|_{\cB}\,\left\Vert ((e^{it\cdot\hat\kappa}-1)1_Y)\circ f^{\ell+1}\, (e^{it\cdot\hat\kappa}-1)\right\Vert_{L^{
\frac p{p-1}}
(\mu_\Delta)}\nonumber\\
&\ll 
% t^{2\gamma_0}\int_\Delta |\hat\kappa\circ f^\ell|^{\gamma_0}\, |\hat\kappa|^{\gamma_0}\, d\mu_\Delta\ll 
|t|^{2\gamma_0}\|\hat\kappa\|_{L^{2\gamma_0\frac p{p-1}}}^{2\gamma_0}\Vert vw\Vert_{\mathcal B}\, ,\label{AAA1}
\end{align}
where 
%in the penultimate inequality 
we have used the H\"older inequality combined with
$\mathcal B\hookrightarrow L^p(\mu_
\Delta
)$ and finally,
in the last line, we used
the inequality $|e^{it\cdot\hat\kappa}-1|\ll |t\hat\kappa|^{\gamma_0}$ combined with
the Cauchy-Schwarz inequality.  Hence,
%for any $\gamma_0\in (0,1)$
using the fact that $|\xi|^{-1}\le (1-\delta)^{-1}\le r^{-1}$,
\begin{equation}
\label{eq-snm}
|S_{N, N}(t)|\ll r^{-2N} |t|^{2\gamma_0}=r^{-2N} |t|^{\gamma+\varepsilon}.
\end{equation}

By~\eqref{spgap-Sz}, we have $P_0^N=\Pi_0+N_0^N$ with $\|N_0^N\|_{\cB}\ll \theta_0^N$. Hence,
\begin{align*}
E_{N, N}(t)&=\sum_{j=0}^{N-1} \int_Y \int_{|\xi-1|=\delta}  \xi^{-j
-1}   P_0^j (P_t-P_0)\xi^{-N}(\xi I-P_0)^{-1}\Pi_0\Big((P_t- P_ 0)(\xi I-P_t)^{-1} (vw)\, \Big)\, d\xi\, d\mu_Y\\
&+\sum_{j=0}^{N-1}\int_Y \int_{|\xi-1|=\delta}  \xi^{-j
-1
}  P_0^j (P_t-P_0)\xi^{-N} (\xi I-P_0)^{-1} N_0^N\Big((P_t- P_ 0)(\xi I-P_t)^{-1} (vw)\, \Big)\, d\xi\, d\mu_Y\\
&=E_{N, N}^1(t)+E_{N, N}^2(t).
\end{align*}
But  for any $\xi$ so that $|\xi-1|=\delta$, 
\begin{align*}
\Big|\Pi_0\Big((P_t- P_ 0)(\xi I-P_t)^{-1} (vw)\, \Big)\Big|&\ll \int_\Delta |(P_t- P_ 0)
(\xi I-P_t)^{-1} (vw)
| \, d\mu_\Delta\\
&\ll  \|P_t- P_ 0\|_{\cB\to L^1}\Vert vw\Vert_{\mathcal B}\ll |t| \Vert vw\Vert_{\mathcal B}\, .
\end{align*}
Hence,
\begin{align*}
|E_{N, N}^1(t)|\ll |t|\, \Vert vw\Vert_{\mathcal B} \sum_{j=0}^{N-1} \int_\Delta  \int_{|\xi-1|=\delta} | \xi^{-j
-1
}|\, | \xi^{-N}|\, |(P_t-P_0)(\xi I-P_0)^{-1} (1_Y)|\, d\xi\, d\mu_\Delta.
\end{align*}
The above together with~\eqref{eq-delta} implies that
\begin{equation}
\label{eq-enm}
|E_{N, N}^1(t)|\ll |t|\, r^{-2N}\,  \|P_t- P_ 0\|_{\cB\to L^1}  \Vert vw\Vert_{\mathcal B}\ll |t|^2\, r^{-2N} \Vert vw\Vert_{\mathcal B}\, .
\end{equation}
Recall that $\|N_0^N\|_{\cB}\ll \theta_0^N$. This together with ~\eqref{eq-delta} implies that
\begin{equation}
\label{eq-enm2}
|E_{N, N}^2(t)|\ll (\theta_0/r^2)^{N}\, \|P_t- P_ 0\|_{\cB\to L^1} \Vert vw\Vert_{\mathcal B}\ll |t|\, (\theta_0/r^2)^{N} \Vert vw\Vert_{\mathcal B}\, .
\end{equation}
%Due to our choice for $N$, since $\theta<r^{\frac{{\color{blue}2}\eta}{2-\eta}}$, we have
%$$(\theta/r^2)^{N}\ll t^{\eta-1}\ll (\theta/r^2)^{N} \le r^{2N\frac {\eta-1}{1-\frac \eta 2}}$$
%and so $r^{-2N}\ll t^{-(1-\frac\eta 2)}$. 
This together with~\eqref{eq-sn}, \eqref{eq-snm}, \eqref{eq-enm} and~\eqref{eq-enm2}
implies that
\begin{equation}
\label{eq-snest}
|S_N(t)|\ll    (r^{-2N}|t|^{\gamma+\varepsilon}
+|t|^2 r^{-2N}
 + |t|(\theta_0/r^2)^N) \Vert vw\Vert_{\mathcal B}
%\ll t^{\eta}
\, .
\end{equation}
%It remains to estimate $E_N(t)$ defined in~\eqref{eq-i2} and choose $N$ as a function of $t$.
Proceeding as above in estimating $E_{N,N}(t)$, we compute that
\begin{align*}
%\label{eq-en12}
& E_N(t)=E_{N}^1(t)+E_{N}^2(t)\\
\nonumber&=\int_Y \int_{|\xi-1|=\delta} \xi^{-N} (\xi I-P_0)^{-1} \Pi_0\Big( (P_t-P_0)(\xi I-P_0)^{-1}(P_t- P_ 0)(\xi I-P_t)^{-1} (vw)\Big)\, d\xi\, d\mu_Y\\
\nonumber &+\int_Y \int_{|\xi-1|=\delta} \xi^{-N} (\xi I-P_0)^{-1} N_0^N\Big( (P_t-P_0)(\xi I-P_0)^{-1}(P_t- P_ 0)(\xi I-P_t)^{-1} (vw)\Big)\, d\xi\, d\mu_Y\, .
\end{align*}
Similar to~\eqref{eq-enm2}, $|E_{N}^2(t)|\ll (\theta_0/r)^{N} \Vert vw\Vert_{\mathcal B}$. We need to decompose $E_N^1(t)$ further.
To start we reduce the analysis to $\Pi_0\Big( (P_t-P_0)(\xi I-P_0)^{-1}(P_t- P_ 0)(\xi I-P_t)^{-1} (vw)\Big)$,
for $\xi$ so that  $|\xi-1|=\delta$. With this specified we compute that for any $L\ge 1$,
\begin{align*}
\Pi_0\Big( (P_t-P_0)&(\xi I -P_0)^{-1}(P_t - P_ 0)(\xi I-P_t)^{-1} (vw)\Big)\\
=&\int_Y  (P_t-P_0)(\xi I-P_0)^{-1}(P_t- P_ 0)(\xi I-P_t)^{-1} (vw)\, d\mu\\
=&S'_L(t) +E'_L(t)=\sum_{j=0}^{L-1}\xi^{-j-1} \int_Y  (P_t-P_0)P_0^j (P_t- P_ 0)(\xi I-P_t)^{-1} (vw)\, d\mu \\
&+\int_Y  (P_t-P_0)(\xi I-P_0)^{-1}\xi^{-L}P^L(P_t- P_ 0)(\xi I-P_t)^{-1} (vw)\, d\mu\, .
\end{align*}
%Proceeding as in~\eqref{eq-snm}, 
Due to \eqref{AAA1},
% for $\widetilde\gamma_0\in (1/2+\eta/4,1/q)$,
%for any $\gamma_0\in (0,1)$,
%\begin{equation*}
$|S'_{L}(t)|\ll 
%r^{-L}\, t^{2\widetilde\gamma_0}
 r^{-L}\, |t|^{\gamma+\varepsilon} \Vert vw\Vert_{\mathcal B}$.
%\end{equation*}
 Next, to estimate $E'_L(t)$, we repeat the argument used in estimating $E_{N,N}(t)$ above. Namely,
\begin{align*}
|E'_L(t)|&\ll \Big|\int_Y  (P_0-P_t)(\xi I-P_0)^{-1}\xi^{-L}\Pi_0(P_t- P_ 0)(\xi I-P_t)^{-1} (vw)\, d\mu \Big|\\
&+\Big|\int_Y  (P_0-P_t)(\xi I-P_0)^{-1}\xi^{-L}N_0^L(P_t- P_ 0)(\xi I-P_t)^{-1}(vw)\, d\mu\Big|\\
&\ll (r^{-L}\, |t|^2+ |t|\, (\theta_0/r)^L) \Vert vw\Vert_{\mathcal B}\, .
\end{align*}
Therefore
$\left|E_N(t)\right|\ll(\theta_0/r)^N+  r^{-N}\left(r^{-L}\, |t|^{\gamma+\varepsilon}+ |t|\, (\theta_0/r)^L\right) \Vert vw\Vert_{\mathcal B}$.
Gathering this and \eqref{eq-snest}, we obtain
\begin{equation}
\label{eq-i2tuselat}
|I_2(t)|\ll    \left(r^{-2N}|t|^{\gamma+\varepsilon}
 + |t|(\theta_0/r^2)^N+(\theta_0/r)^N+  r^{-N}\left(r^{-L}\, |t|^{\gamma+\varepsilon}+ |t|\, (\theta_0/r)^L\right)\right) \Vert vw\Vert_{\mathcal B}\, .
\end{equation}
Recall that 
$N:=\lfloor
(\gamma+\varepsilon)\frac{\log |t|}{\log(\theta_0/r)}\rfloor$, so that
 $(\theta_0/r)^N\sim |t|^{\gamma+\varepsilon}$.
Recall also that $\theta_0= r^\vartheta$ with $\vartheta>
%\frac{2\eta}{2-\eta}>2$, so
2$ so that
$
\frac{2(\gamma+\varepsilon)}{\vartheta-1}<\varepsilon$
%\quad\mbox{and}\quad
and
$\frac{\left(\gamma+\varepsilon\right)(\vartheta-2)}{\vartheta-1}>\gamma-1$.
Hence 
$\theta_0/r=r^{\vartheta-1}$ and $r^{-N}\sim |t|^{-\frac{\gamma+\varepsilon}{\vartheta-1}}$
and $(\theta_0/r^2)^N=r^{N(\vartheta-2)}\approx |t|^{\frac{(\gamma+\varepsilon)(\vartheta-2)}{\vartheta-1}}$ thus
\begin{align*}
\label{eq-i2tuselat}
|I_2(t)|&\ll  \left(  |t|^{\gamma+\varepsilon-2\frac{\gamma+\varepsilon}{\vartheta-1}}
 +| t|^{1+\frac{(\gamma+\varepsilon)(\vartheta-2)}{\vartheta-1}}+|t|^{\gamma+\varepsilon}+  |t|^{-\frac{\gamma+\varepsilon}{\vartheta-1}}
\left(r^{-L}\, |t|^{\gamma+\varepsilon}+ |t|\, (\theta_0/r)^L\right)\right) \Vert vw\Vert_{\mathcal B}\\
&\ll  \left(|t|^{\gamma} +  |t|^{-\frac{\gamma+\varepsilon}{\vartheta-1}}
\left((\theta_0/r)^{-\frac{L}{\vartheta-1}}\, |t|^{\gamma+\varepsilon}+ |t|\, (\theta_0/r)^L\right)\right) \Vert vw\Vert_{\mathcal B}\, ,
\end{align*}
due to our choice of $\vartheta$.
%with $\eta_0=\min\left(\frac\eta 2,\frac{(1+\frac\eta 2)(\vartheta-2)}{\vartheta-1},\frac\eta 2-2\frac{1+\frac\eta 2}{\vartheta-1} \right)$.
%Let $\varepsilon>0$.
%Assume from now on that $\eta=\eta/2$ and that $\vartheta$ is large enough so that $\eta_0\ge \eta/2-\varepsilon$.\\
Now, choose $L:=\left\lfloor\frac{(\gamma+\varepsilon-1)\log |t|}{\log(\theta_0/r)}\right\rfloor$, so, due to our choice of $\vartheta$,
$$
|I_2(t)|
\ll  \left(|t|^{\gamma} +  |t|^{-\frac{\gamma+\varepsilon}{\vartheta-1}}
\left( |t|^{\gamma+\varepsilon-\frac{\gamma+\varepsilon-1}{\vartheta-1}}+ |t|^{\gamma+\varepsilon}\right)\right) \Vert vw\Vert_{\mathcal B}\ll |t|^{\gamma} \Vert vw\Vert_{\mathcal B} \, .$$
%Now we we need to have $ t^{-\frac{1+\eta}{\vartheta-1}}r^{-L}\, t^{1+\frac \eta 2}\le t^{1+\eta'}$ with $\eta'>0$.Recall that $1+\frac \eta 2<2/q$ and $q>1$ can be taken as close to 1 as we wish

%Now, choose $L=\left\lfloor\eta'\frac{\log t}{\log(\theta/r)}\right\rfloor$; for $\widetilde\eta_1\in (\frac{2-\eta}4,1)$;
%, which we choose very close to $1$; 
%so $(\theta/r)^{L}\approx  t^{\eta'}$.
%\ll  (\theta/r)^L< r^{\frac{3\eta-2}{2-\eta}L}$ (since $\theta<r^{\frac{2\eta}{2-\eta}}$).
%Moreover $ t^{\eta-1}\sim (\theta/r^2)^N<r^{N\frac{4(\eta-1)}{2-\eta}}$, so$r^{-N}\ll t^{-\frac{2-\eta}4}$. Hence
%$$\left|E_N(t)\right|\ll(\theta/r)^N+   t^{-\frac{2-\eta}4}\left(t^{2\widetilde\eta_0-\widetilde\eta_1\frac{2-\eta}{3\eta-2}}+ t^{1+\widetilde\eta_1}\right)\, .$$

%and $r^{-N}\ll t^{-\frac{\eta_0}2}$. Similarly to equation~\eqref{eq-snm}, for any  $\eta_0\in (0,1)$, we have
%\begin{align*}
%\Big|\Pi_0\Big( (P_0&-P_t)(\xi I -P_0)^{-1}(P_t - P_ 0))(\xi I-P_t)^{-1} 1\Big)\Big|\ll |S'_{L}(t)|+ |E'_L(t)|\ll t^{\eta}\, .
%\end{align*}
%The above displayed equation together with~\eqref{eq-en12} implies that  for any  $\gamma_0\in (0,1)$, 
%\begin{align*}
%|E_N(t)|\ll (\theta/r)^{N} +r^{-N}\, t^{\gamma_0}\, \ll t^{\gamma-\frac{\gamma_0}2}.
%\end{align*}
\end{itemize}
\end{proof}

\subsection{Proof of the expansion of $\Pi_t$~: proofs of Propositions~\ref{prop-expproj-base} and \ref{prop-expproj}}

We first provide the proof of Propositions~\ref{prop-expproj-base} relying on our technical Lemmas~\ref{lem-eigf} and 
~\ref{lem-eigf2}, and then we prove Proposition~\ref{prop-expproj} .\\

\begin{pfof}{Proposition~\ref{prop-expproj-base}}
We proceed as in~\cite[Proof of Lemma 3.14]{BalintGouezel06} using our estimates.
Since 1 is a simple eigenvalue of $\tilde R_t$, and since $Q_t(1_Y)$ and $1_Y\Pi_t(vw)$ are both eigenfunctions belonging to $\mathcal B_1$ of $\tilde R_t$ associated to 1, these two vectors are proportional and so, 
\begin{equation}\label{1YPit}
1_Y\Pi_t(vw)=\mu_Y(1_Y\Pi_t(vw))\frac{Q_t(1_Y)}{\mu_Y(Q_t(1_Y))}=\left(\mu_\Delta(vw)+t\cdot c_0(vw)+\boldsymbol{\varepsilon}^{(1)}_t
\right)\frac{Q_t(1_Y)}{\mu_Y(Q_t(1_Y))}\, ,
\end{equation}
%\[
%\mu_Y\left(Q_0(1_Yvw)+t\cdot Q'_0(1_Yvw)+  \boldsymbol{\varepsilon}^{(1)}_t\right)(1_Y\Pi_t(vw))=\left[\mu_Y(1_Y\Pi_0)+t\cdot c_0(vw)+\boldsymbol{\varepsilon}^{(2)}_t\right]\left[Q_0(1_Yvw)+t\cdot Q'_0(1_Yvw)+\boldsymbol{\varepsilon}^{(1)}_t\right]
%\]
with
$
\boldsymbol{\varepsilon}^{(1)}_t:=\mu_Y(1_Y(\Pi_t-\Pi_0-t\cdot c_0)(vw))=O(|t|^\gamma(\|vw\|_{\mathcal B}+\|w\|_{\mathcal B_0}\|v\|_{L^b(\mu_\Delta)}))$
in $\mathbb C$, 
by Lemma~\ref{cor-betcont}.
Moreover, due to Lemma~\ref{lem-eigf},
$
\boldsymbol{\varepsilon}^{(2)}_t:=(Q_t-Q_0-t\cdot Q'_0)(1_Y)=O(|t|^\gamma)$
%\quad \mbox{in } 
in $\mathcal B_1$.
Hence, still in $\mathcal B_1$,
\begin{align*}
\frac{Q_t(1_Y)}{\mu_Y(Q_t(1_Y))}&=
\frac{1_Y+t\cdot Q'_0(1_Y)+\boldsymbol{\varepsilon}^{(1)}_t}
{\mu_Y(1_Y+t\cdot Q'_0(1_Y)+\boldsymbol{\varepsilon}^{(1)}_t)}
=1_Y+t\cdot(Q'_0(1_Y)-\mu_Y(Q'_0(1_Y))1_Y)+O(|t|^\gamma)\\
&=1_Y+t\cdot Q'_0(1_Y)+O(|t|^\gamma)\, ,
\end{align*} 
since the derivation at $t=0$ of $Q_t^2(1_Y)=Q_t(1_Y)$ combined with $Q_01_Y=1_Y$ leads to $\mu_Y(Q'_01_Y)=Q_0Q'_01_Y=0$.
Combining the above estimate with \eqref{1YPit}, we conclude that
\[
1_Y\Pi_t(vw)=\mu_\Delta(vw)+t\cdot \left(c_0(vw)+\mu_\Delta(vw)Q'_0(1_Y)\right)+O\left(|t|^\gamma(\|vw\|_{\mathcal B}+\|w\|_{\mathcal B_0}\|v\|_{L^b(\mu_\Delta)})\right)\, ,
\]
which leads to \eqref{eq1yderivpi}.
~\end{pfof}
\begin{rem}
\label{rem:derivopi}
By the same reasoning used in obtaining  $\mu_Y(Q'_01_Y)=Q_0Q'_01_Y=0$ with $\Pi_0'$ instead of $Q_0'$ and $1$ instead of $1_Y$ we have $\mu_\Delta(\Pi'_01)=0$.
\end{rem}

%Using~\eqref{eq-Qt} and Lemma~\ref{lem-deriv},we obtain

%\subsection{ Asymptotic of $\Pi_t$: proof of Proposition~\ref{prop-expproj}}
%Due to Proposition ~\ref{prop-expproj-base}, $t\mapsto \mathbf 1_Y\Pi_t (vw)\in \mathcal B$ is differentiable at $0$, with derivative given by
%\begin{equation}\label{eq1yderivpi}
%(1_Y\Pi'_0)(vw):=(c_0(vw)-\mu(Q'_0(1_Y))) 1_Y+\mu_\Delta(vw)Q'_0(1_Y)\, ,
%\end{equation}
%with $c_0$ of Lemma~\ref{lem-eigf2}.\\

Using Proposition~\ref{prop-expproj-base} and equation~\eqref{eq-Pit} we can complete

\begin{pfof}{Proposition~\ref{prop-expproj}} 
To simplify notations, we write $\|(v,w)\|:=\|vw\|_{\mathcal B}+\|w\|_{\mathcal B_0}\|v\|_{L^b(\mu_\Delta)}$.
Starting from~\eqref{eq-Pit}, 
$\Pi_t(vw)(x)=\lambda_t^{-\omega(x)} e^{it\cdot\hat\kappa_{\omega(x)}\circ \pi_0(x)}\Pi_t (vw)\circ\pi_0(x)$.
Using the fact that $\lambda_0=1$ and $\lambda'_0=0$, we conclude that
$t\mapsto \Pi_t(vw)(x)$ is differentiable at $0$, with derivative
\begin{align}
\label{eq-derivPi0ult}
\Pi_0' (vw)(x)&=
i[\hat\kappa_{\omega(x)}\Pi_0 (vw)]\circ\pi_0(x)
+\Pi_0' (vw)\circ\pi_0(x),
\end{align}
where we use the definition of $1_Y\Pi'_0$ in~\eqref{eq1yderivpi}.
This provides the last formula of Proposition~\ref{prop-expproj}.
With this definition, since $\Pi'_0$ is uniformly bounded on $Y$, for every $\eta\in(1,2)$
\begin{align*}
\Vert \Pi'_0 (vw)\Vert^\eta_{L^\eta(\mu_\Delta)}
&\ll \Vert 1_Y\Pi'_0 (vw)\Vert_\infty^\eta+
       \sum_{n\ge 1}\int_{\{\sigma>n\}}|\hat\kappa_n|^\eta\, d\mu_Y\|vw\|_{L^1(\mu_Y)}\\
&\ll \left(1+ \sum_{n\ge 1}(\mu_Y(\sigma>n))^{\frac{2-\eta}{2+\eta}}\Vert|\hat\kappa_n|^\eta\Vert_{L^\frac{\eta+2}{2\eta}(\mu_Y)}\right)\|(v,w)\|\\
%&\ll 1+\sum_{n\ge 1}(\mu_Y(\sigma>n))^{\frac{2-\eta}{2+\eta}}\Vert \hat\kappa_n\Vert_{L^\frac{\eta+2}{2}(\mu_\Delta)}^\eta\\
&\ll\left( 1+\sum_{n\ge 1}\theta_1^{\frac{2-\eta}{2+\eta}n} n^\eta \Vert\hat\kappa\Vert_{L^\frac{\eta+2}{2}(\mu_\Delta)}^\eta\right)\|(v,w)\|=O(\|(v,w)\|)\, .
\end{align*}
%compute that
%\begin{align*}
%\Pi_t' v(x)&=-\omega(x)\lambda_t'\cdot \lambda_t^{-(\omega(x)-1)} e^{it\kappa_\omega(x)}\Pi_t (v)\circ\pi_0(x)\\
%&+i\kappa_\omega(x)\lambda_t^{-\omega(x)} e^{it\kappa_\omega(x)}\Pi_t (v)\circ\pi_0(x)+\lambda_t^{-\omega(x)} e^{it\kappa_\omega(x)}\cdot\Pi_t' (v)\circ\pi_0(x),
%\end{align*}
%{\color{red} At this step we have just defined $1_Y\Pi'_0$ with another formula. Do we have to extend the proof to prove the existence of $1_Y\Pi'_tv$ to $v\ne 0$?}\\
%As in~\cite[Proof of Lemma 3.16]{BalintGouezel06} (which, in particular, shows that $\lambda_t=1-t^2\log(1/|t|)(1+o(1))$), we divide the computation into: i) $\omega(x)\le b\log(1/|t|)$,
%for some large $b$; ii) $\omega(x)> b\log(1/|t|)$.
%{\bf Case i): $\omega(x)\le b\log(1/|t|)$.} 
%Compute that
%{\color{blue}
%\begin{align}
%\label{eq-expbigP}
%(\Pi_t-&\Pi_0-t\cdot \Pi_0')v(x)=
%\left[(e^{it\kappa_{\omega(x)}}-1-it\cdot \kappa_{\omega(x)})\Pi_0 (v)\right]\circ\pi_0(x)+\left[(e^{it\kappa_{\omega(x)}}-1)t\cdot\Pi'_0(v)\right]\circ\pi_0\\
%&+\left[e^{it\kappa_{\omega(x)}}(\Pi_t-\Pi_0-t\Pi_0') (v)\right]\circ\pi_0(x)+(\lambda_t^{-\omega(x)}-1)[e^{it\kappa_{\omega(x)}}\Pi_t (v)]\circ\pi_0(x)\, .
%\end{align}
Formula \eqref{eq-derivPi0ult} together with~\eqref{eq-Pit} implies that
\begin{align}
\label{eq-expbigP}
(\Pi_t-&\Pi_0-t\cdot \Pi_0')(vw)(x)=
\left[(e^{it\cdot\hat\kappa_{\omega(x)}}-1-it\cdot \hat\kappa_{\omega(x)})\Pi_0 (vw)\right]
\circ\pi_0(x)\\
&\nonumber
+\left[(e^{it\hat\kappa_{\omega(x)}}-1)t\cdot\Pi'_0(vw)\right]\circ\pi_0
\left[e^{it\cdot\hat\kappa_{\omega(x)}}
(\Pi_t-\Pi_0-t\cdot\Pi_0') (vw)\right]\circ\pi_0(x)
\\
&
+(\lambda_t^{-\omega(x)}-1)
[e^{it\hat\kappa_{\omega(x)}}\Pi_t (vw)]\circ\pi_0(x)\, .\nonumber
\end{align}
%\begin{align}\label{eq-expbigP_v2}
%(\Pi_t-&\Pi_0-t\cdot \Pi_0')v(x)=\left[(e^{it\hat\kappa_{\omega(x)}}-1-it\cdot \kappa_{\omega(x)})\Pi_0 (v)\right]\circ\pi_0(x)+\left[(e^{it\kappa_{\omega(x)}}-1)t\cdot\Pi'_0(v)\right]\circ\pi_0\\\nonumber &+\lambda_t^{-\omega(x)}\left[e^{it\kappa_{\omega(x)}}(\Pi_t-\Pi_0-t\Pi_0') (v)\right]\circ\pi_0(x)+(\lambda_t^{-\omega(x)}-1)[e^{it\kappa_{\omega(x)}}(\Pi_0+t\cdot \Pi'_0) (v)]\circ\pi_0(x)\, .
%\end{align}
This 
%decomposition 
leads to
%\begin{equation}\label{Decomp}
$\left\Vert(\Pi_t-\Pi_0-t\cdot \Pi_0')(vw)\right\Vert_{L^{p'}(\mu_{\Delta})}\le I_1(t)+I_2(t)+I_3(t)+I_4(t)$.
%\end{equation}
First
\begin{align*}
(I_1(t))^{p'}&:=\int_{\Delta}\left| e^{it\hat\kappa_{\omega(x)}}-1-it\cdot \hat\kappa_{\omega(x)}\right|^{p'}\circ\pi_0(x)d\mu_{\Delta}(x)\left(\int_{Y}|vw|\, d\mu_Y\right)^{p'}\\
&\ll\int_{\Delta}\left|t\cdot  \hat\kappa_{\omega(x)}\circ\pi_0(x)\right|^{p'\gamma} d\mu_{\Delta}(x)\Vert vw\Vert_{L^1(\mu_\Delta)}^{p'}\\
&\ll \sum_{n\ge 1}\int_{\sigma> n}  \left|t \cdot \hat\kappa_{n}\right|^{p'\gamma} d\mu_{Y}\Vert vw\Vert_{L^1(\mu_\Delta)}^{p'}
\end{align*}
Thus $(I_1(t))^{p'}\ll\sum_{n\ge 1}|t|^{p'\gamma}\mu_Y(\sigma\ge n)^{\frac {2-\gamma}{2+\gamma}} \Vert \hat\kappa_n^{p'\gamma}\Vert_{L^{\frac{\gamma+2}{2\gamma}}(\mu_Y)}\Vert vw\Vert_{L^1(\mu_\Delta)} ^{p'}$. It follows that
\[
\ll\sum_{n\ge  1}|t|^{p'\gamma}\mu_Y(\sigma> n)^{\frac {2-\gamma}{2+\gamma}} \Vert \hat\kappa_n\Vert^{p'\gamma}_{L^{p'\frac{\gamma+2}{2}}(\mu_Y)}\Vert vw\Vert_{L^1(\mu_\Delta)}^{p'}\ll |t|^{p'\gamma}\|vw\|^{p'}_{\mathcal B}\, ,
\]
%end{align*}
for any $\gamma\in(1,\frac 4{p'}-2)$, since $p'\frac{\gamma+2}2\in(1,2)$ and $\Vert\hat\kappa_n\Vert_{L^{p'\frac{\gamma+2}{2}}(\mu_Y)}\le n\Vert \hat\kappa\Vert_{L^{p'\frac{\gamma+2}{2}}(\mu_\Delta)}$ and $\mu_Y(\sigma>n)\ll \theta_1^n$.
Second, using $|e^{iy}-1|\le|y|$, it comes
\begin{align*}
(I_2(t))^{p'}&:=|t|^{p'}\,\int_{\Delta}\left|(e^{it\cdot \hat\kappa_{\omega(x)}}-1)\Pi'_0(vw)\right|^{p'}\circ\pi_0\, d\mu_{\Delta}(x)
\ll |t|^{2p'}\sum_{n\ge 1}\int_{\sigma> n} |\hat\kappa_n\Pi'_0(vw)|^{p'}\, d\mu_Y \\
&\ll |t|^{2p'} \Vert 1_Y\Pi'_0(vw)\Vert_\infty \sum_{n\ge 1}\mu(\sigma> n)^{\frac {2-p'}{p'+2}}\Vert\hat\kappa_n^{p'}\Vert_{L^{\frac{p'+2}{2p'}}(\mu_\Delta)}\\
&\ll |t|^{2p'} \Vert 1_Y\Pi'_0(vw)\Vert_\infty \ll |t|^{2p'}\|(v,w)\|,
\end{align*}
proceeding as for $I_1$ and using $\Vert 1_Y\Pi'_0(vw)\Vert_\infty\ll\|(v,w)\|$.
Third
\begin{align}
\label{I3}
(I_3(t))^{p'}&:=
\int_{\Delta}\left|
e^{it\cdot \hat\kappa_{\omega(x)}}
(\Pi_t-\Pi_0-t\Pi_0') (vw)\right|^{p'}\circ\pi_0(x)\,d\mu_{\Delta}(x)\\
&=O\left(|t|^{\gamma p'}(\Vert vw\Vert_{
\mathcal B}+\Vert w\Vert_{\mathcal B_0}\|v\|_{L^b(\mu_\Delta)})^{p'}\right)\, ,
\end{align}
with $\gamma$ as in Proposition~\ref{prop-expproj-base}.
For the last term, we observe that the expansion of $\lambda$ at $0$ implies in particular that $1>|\lambda_t|>e^{-a_0\frac {|t|^2\log(1/|t|)} 2}$ if $t$ is small enough and so
\begin{align*}
|I_4(t)|^{p'} &=\int_{\Delta} \left|(\lambda_t^{-\omega(x)}-1)
e^{it\cdot \hat\kappa_{\omega(x)}}\Pi_t
%(\Pi_0+t\cdot\Pi'_0) 
(vw)
\right|^{p'}\circ\pi_0(x)\, d\mu_\Delta(x)\\
&\le \int_{\Delta}\left (\omega(x)\, |\lambda_t-1|\, e^{a_0\frac {|t|^2\log(1/|t|)} 2(\omega(x)-1)}\left\Vert 1_Y\Pi_t (vw)\right\Vert_\infty\right)^{p'}\, d\mu_\Delta(x)\\
&\ll (|t|^2\log(1/|t|)\left\Vert 1_Y\Pi_t (vw)\right\Vert_\infty)^{p'} \sum_{n\ge 1}\int_{\sigma> n}n^{p'} e^{a_0p'\frac {n|t|^2\log(1/|t|)} 2}\, d\mu_Y
\end{align*}
and so $|I_4(t)|^{p'}\ll (|t|^2\log(1/|t|)\left\Vert 1_Y\Pi_t (vw)\right\Vert_\infty)^{p'} \sum_{n\ge 1}n^{p'} e^{a_0p'\frac {n|t|^2\log(1/|t|)} 2}\mu_Y(\sigma> n)$. Thus
\[
|I_4(t)|^{p'}\ll (|t|^2\log(1/|t|) \|(v,w)\|)^{p'}\, ,
\]
%\end{align*}
for small $|t|$
% small enough 
since $\mu_Y(\sigma>n)\ll \theta_1^n$ and
$\left\Vert 1_Y\Pi_t (vw)\right\Vert_\infty\ll\|(v,w)\|$ (see \eqref{1YPit}).
So, taking $\gamma$ as in Proposition~\ref{prop-expproj},
%Hence
%\begin{align}
%\label{eq-pito1}
$\left\Vert(\Pi_t-\Pi_0-t\cdot \Pi_0')(vw)\right \|_{L^{p'}(\mu_{\Delta})}
\ll |t|^\gamma\, \left(\|vw\|_{\mathcal B}+\Vert w\Vert_{\mathcal B_0}\|v\|_{L^b(\mu_\Delta)}\right)$.
%\end{align}
\end{pfof}

\section{Expansion of $\lambda_t$: Proof of Proposition~\ref{prop-lambda} }
\label{sec:lambda}

Let $v_t$ be the eigenfunction of $P_t$ associated with $\lambda_t$ so that $\mu_\Delta(v_t)=1$, that is $v_t=\frac{\Pi_t 1_\Delta}{\mu_\Delta(\Pi_t 1_\Delta)}$.
Using a classical argument (see, for instance,~\cite{BalintGouezel06}), we write
\begin{align}
1-\lambda_t&=\int_{\Delta} (P_0-P_t)v_t\, d\mu_\Delta=\int_{\Delta}(1-e^{it\cdot\hat\kappa})v_t\, d\mu_\Delta\nonumber\\
&=\int_{\Delta} (1-e^{it\cdot\hat\kappa})\, d\mu_\Delta+\int_{\Delta}(1-e^{it\cdot\hat\kappa})(v_t-v_0)\, d\mu_\Delta:=\Psi(t) +V(t)\, .\label{eq-lambdato}
\end{align}
Estimates of $V(t)$ and $\Psi(t)$ are given respectively in Lemmas~\ref{lem:V} and \ref{lem:Psi}. 
%In the above formula, 
$\Psi(t)$ is the pure scalar part which will be estimated as if dealing with the characteristic function of an i.i.d. process
and for this we only need to exploit Lemma~\ref{lemm-tail}. 
The function $V(t)$ will be estimated via the estimates used in the proof of Proposition~\ref{prop-expproj}.
%and the argument inside the proofs of ~\cite[Lemmas 3.16 and 3.17]{BalintGouezel06}.
We start with the latter.

\begin{lem}\label{lem:V} There exists $C>0$ such that for all $t\in B_\delta(0)$, $|V(t)|\le Ct^2$.
\end{lem}
\begin{proof} 
Note that
$v_t-v_0=\frac{\Pi_t(1_\Delta)-\mu_\Delta(\Pi_t(1_\Delta))}{\mu_\Delta(\Pi_t(1_\Delta))}$.
Therefore it is enough to prove that
\begin{equation}\label{Goal}
\left|
\int_{\Delta}(1-e^{it\cdot\hat\kappa})(\Pi_t(1_\Delta)-\mu_\Delta(\Pi_t(1_\Delta)))\, d\mu_\Delta\right|\ll t^2\, .
\end{equation}
We know from Proposition~\eqref{prop-expproj} that
$\mu_\Delta((\Pi_t-\Pi_0)(1_\Delta))=O(t)$, which implies that
\begin{equation}\label{Esti1}
\left|\int_{\Delta}(1-e^{it\cdot\hat\kappa})\mu_\Delta((\Pi_t-\Pi_0)(1_\Delta))\, d\mu_\Delta\right| \ll t^2\Vert \hat\kappa\Vert_{L^1(\mu_\Delta)}\ll t^2\, .
\end{equation}
It remains to estimate
$V_0(t):=\int_{\Delta}(1-e^{it\cdot\hat\kappa})(\Pi_t-\Pi_0)(1_\Delta)\, d\mu_\Delta$.
To this end, thanks to \eqref{eq-Pit}, we use the following decomposition similar to \eqref{eq-expbigP}
\begin{align}
\label{eq-expbigPbis}
(\Pi_t-\Pi_0)1_\Delta(x)=&\lambda_t^{-\omega(x)} e^{it\hat\kappa_{{\omega(x)}}
(\pi_0(x))}\Pi_t (1_\Delta)\circ\pi_0(x)-1\nonumber\\
=&
(e^{it\cdot\hat\kappa_{\omega(x)}(\pi_0(x))}-1)
+
\left[e^{it\cdot\hat\kappa_{\omega(x)}}
(\Pi_t-\Pi_0) (1_\Delta)\right]\circ\pi_0(x)\\
&+(\lambda_t^{-\omega(x)}-1)
[e^{it\cdot\hat\kappa_{\omega(x)}}\Pi_t (1_\Delta)]\circ\pi_0(x)\, ,
\nonumber
\end{align}
which leads to
\begin{equation}\label{DecompJ}
\left|\int_{\Delta}(1-e^{it\cdot\hat\kappa})(\Pi_t-\Pi_0)(1_\Delta)\, d\mu_\Delta\right|\le J_1(t)+J_2(t)+J_3(t)\, ,
\end{equation}
%with
\begin{align*}
J_1(t)&:=\left|\int_{\Delta}(1-e^{it\cdot\hat\kappa(x)}) (e^{it\cdot \hat\kappa_{\omega(x)}(\pi_0(x))}-1)\, d\mu_{\Delta}(x)\right|\, ,\nonumber\\
J_2(t)&:=
\int_{\Delta}\left|
(1-e^{it\cdot\hat\kappa(x)})
e^{it\hat\kappa_{\omega(x)}}
(\Pi_t-\Pi_0) (1_\Delta)\right|\circ\pi_0(x)\,d\mu_{\Delta}(x)\, ,\\
J_3(t)&:=
\int_{\Delta} \left|(1-e^{it\cdot\hat\kappa(x)})(\lambda_t^{-\omega(x)}-1)
e^{it\kappa_{\omega(x)}}\Pi_t
(1_\Delta)
\right|\circ\pi_0(x)\, d\mu_\Delta(x)\, .
\end{align*}
First let us prove that
\begin{align}
J_1(t)=O(t^2)\quad \mbox{as }t\rightarrow 0\, .
\label{MajoJ1}
\end{align}
%Observe that
\begin{align}
J_1(t)&=\left|\int_Y \sum_{k=0}^{\sigma-1}(1-e^{it\cdot \hat\kappa\circ f^k})
(1-e^{it\cdot \hat\kappa_k})\, d\mu_\Delta\right|
%\nonumber\\
%&=\int_{Y\cap\{\sigma<-b\log t\}} \sum_{k=1}^{\sigma-1}(1-e^{it\cdot \hat\kappa\circ f^k})(1-e^{it\cdot \hat\kappa_k})\, d\mu_\Delta+O(t^2)\nonumber\\
%&
\le t^2\int_{Y
%\cap\{\sigma<-b\log t\}
} \sum_{k=1}^{\sigma-1}|\hat\kappa\circ f^k|\, |\hat\kappa_{k}|\, d\mu_\Delta
%+O(t^2)
\label{TOTO}\, .
\end{align}
To prove that $J_1(t)=O(t^2)$, we will use the fact that, due to ~\cite[Propositions 11--12]{SV07} and~\cite[Lemma 16]{SV07}, there exists $\varepsilon>0$ such that, for every $V$,
%large enough, 
\[
\mu_\Delta\left(A_{n,V}\right)=O(n^{-3-\varepsilon})\quad \mbox{with }A_{n,V}:=\left\{\hat\kappa=n,\ \exists |j|\le V\log (n+2),\ |\hat\kappa\circ f^j|>n^{4/5}\right\}.
\]
For any $y\in Y$, we write $\ell(y)$ for the largest integer in $\{1,...,\sigma(y)-1\}$ such that 
\[
N(y):=\sup_{k=1,...,\sigma(y)-1}|\hat\kappa\circ f^k(y)|=|\hat\kappa\circ f^{\ell(y)}(y)|\, .
\]
Set $Y_n:=Y\cap\{\hat \kappa\circ f^{\ell(y)}=n$, 
$Y'_n:=\left\{y\in Y_n\, :\ \sigma(y)< b\log (n+2)\right\}$
and also
$
Y^{(0)}_n:=\left\{y\in Y'_n\, :\ \forall j<\sigma(y),\ |\hat\kappa\circ f^j|\le n^{4/5}\right\}\, .$
Notice that
\begin{align*}
&\int_{Y} \sum_{k=1}^{\sigma-1}|\hat\kappa\circ f^k|\, |\hat\kappa_{k}|\, d\mu_\Delta\le \sum_{n\ge 0}\int_{Y'_n} \sum_{k=1}^{b\log(n+2)-1}|\hat\kappa\circ f^k|\, |\hat\kappa_{k}|\, d\mu_\Delta+\sum_{n\ge 0}  \int_{Y_n\setminus Y'_n} n^2\sigma\, d\mu_\Delta\\
&\le  \sum_{n\ge 0}\int_{Y^{(0)}_n} b\, n^{9/5}\log (n+2)+\sum_{n\ge 0}\int_{Y'_n\setminus Y_n^{(0)}}n^2b\log(n+2)\, d\mu_\Delta
+\sum_{n\ge 0}   n^2\mathbb E_{\mu_\Delta}[\sigma 1_{Y\setminus Y'_n}]\\
&\le  \sum_{n\ge 0}b\, n^{9/5}\log (n+2)
\mu_\Delta(Y'_n)
+\sum_{n\ge 0}n^2b\log(n+2)\mu_\Delta(Y'_n\setminus Y_n^{(0)})\\
&\ \ \ \ \ \ \ \ \ \ \ \ \ \ \ 
+\sum_{n\ge 0}   n^2\!\!\!\sum_{m\ge b\log(n+2)}\mu_\Delta(\sigma>m)\, .
\end{align*}
The last term of the above right hand side is less than
$\sum_{n\ge 0}   n^2\sum_{m\ge b\log(n+2)}C_1\theta_1^m\le \sum_{n\ge 0}  O( n^2 \theta_1^{ b\log(n+2)})<\infty$ by taking $b$ large enough. The other terms are dominated by
\begin{align*}
&\ \sum_{n\ge 0}O\left(n^{9/5}\log (n+2)
 \sum_{m=0}^{b\log(n+2)}\mu_\Delta(\hat\kappa=n)
\right)
+\sum_{n\ge 0}n^2\log(n+2)\sum_{m=0}^{b\log(n+2)}\mu_\Delta(A_{n,b})
\\
&\le  \sum_{n\ge 0}O\left(n^{9/5-3}(\log (n+2))^2 \right)+\sum_{n\ge 0}n^{2-3-\varepsilon}(\log(n+2))^2<\infty\, ,
\end{align*}
Combined with~\eqref{TOTO}, this ends the proof of \eqref{MajoJ1}.
Second, due to Proposition~\ref{prop-expproj-base}
\begin{align}
J_2(t)&=
\int_{\Delta}\left|
(1-e^{it\cdot\hat\kappa(x)})
e^{it\hat\kappa_{\omega(x)}}
(\Pi_t-\Pi_0) (1_\Delta)\right|\circ\pi_0(x)\,d\mu_{\Delta}(x)\nonumber\\
&\ll \int_{\Delta}\left|t\cdot 
\kappa(x)
(\Pi_t-\Pi_0) (1_\Delta)\right|\circ\pi_0(x)\,d\mu_{\Delta}(x)\ll t^2 \Vert\hat\kappa\Vert_{L^1(\mu_\Delta)}\, ,\label{MajoJ2}
\end{align}
and finally, using the fact that  $1>|\lambda_t|>e^{-\frac {t^2\log(1/|t|)} 2}$ if $t$ is small enough,
\begin{align}
|J_3(t)| &=\int_{\Delta} \left|(1-e^{it\cdot\hat\kappa(x)})(\lambda_t^{-\omega(x)}-1)
e^{it\kappa_{\omega(x)}}\Pi_t
(1_\Delta)
\right|\circ\pi_0(x)\, d\mu_\Delta(x)\nonumber\\
&\le \int_{\Delta} |t\cdot\hat\kappa(x)|\omega(x)\, |\lambda_t-1|\, e^{\frac {t^2\log(1/|t|)} 2(\omega(x)-1)}\left\Vert 1_Y\Pi_t (1_\Delta)\right\Vert_\infty\, d\mu_\Delta(x)\nonumber\\
&\ll t^3\log(1/|t|) \sum_{n\ge 1}\int_{Y\cap\{\sigma> n\}}|\hat\kappa\circ f^n|n e^{\frac {nt^2\log(1/|t|)} 2}\, d\mu_Y\nonumber
\end{align}
Thus $|J_3(t)|\ll t^3\log(1/|t|) \sum_{n\ge 1}n e^{\frac {nt^2\log(1/|t|)} 2}\Vert\hat\kappa\Vert_{L^p(\mu_\Delta)}(\mu_Y(\sigma> n))^{1/q}$ and it follows that
\begin{align}
|J_3(t)|&\ll t^3\log(1/|t|) \sum_{n\ge 1}n e^{\frac {nt^2\log(1/|t|)} 2}\Vert\hat\kappa\Vert_{L^p(\mu_\Delta)}(C_1\theta_1^n)^{1/q}=O(t^2)\, ,
\label{MajoJ3}
\end{align}
by taking $p<2$ close to 2 and using the fact that for $|t|$ small enough
$e^{\frac {t^2\log(1/|t|)} 2}\theta_1<1$. 
Combining \eqref{Esti1}, \eqref{DecompJ}, \eqref{MajoJ1}, \eqref{MajoJ2} and \eqref{MajoJ3}, 
we obtain \eqref{Goal} and conclude.~\end{proof}

To complete the proof of Proposition~\ref{prop-lambda}, we still need to estimate $\Psi(t)$ in~\eqref{eq-lambdato}.
The next lemma can be viewed as an extension of the asymptotic of the pure scalar quantity in~\cite[Proof of Theorem 3.1]{ADb}
under the conclusion of Lemma~\ref{lemm-tail} via the arguments used in~\cite{MT12, Terhesiu16}.
\begin{lem}\label{lem:Psi}
Let $\hat\kappa_0:\Delta\rightarrow \mathbb Z^d$ be a $\mu_\Delta$-centered function. Set $\Psi(t):=\int_{\Delta} (1-e^{it\cdot\hat\kappa_0})\, d\mu_\Delta$.
Assume that there exists $M_0>0$ such that $\{|\hat\kappa_0|\ge M_0\}=\bigcup_{(L,w)\in
\mathcal S}\{\hat\kappa_0\in L+\mathbb N w\}$
with $\mathcal S$ a finite subset of $(\mathbb Z^2)^2$ such that the lattices $L+\mathbb N w\subset\mathbb Z^d$ are disjoint for distinct $(L,w)\in \mathcal S$ outside the open ball $B(0,M_0)$.
Assume moreover that
\[
\forall (L,w)\in \mathcal S,\quad\mu_\Delta\left(\hat\kappa_0= L+ n w\right)=\frac{c_{L,w}}{n^3}+O(n^{-4}),\quad\mbox{as}\quad n\rightarrow+\infty\, .
\]
Then, as $t\to 0$, $\Psi(t)=\sum_{(L,w)\in\mathcal S}\frac{c_{L,w}}2(t\cdot w)^2\log||t|^{-1}|+O(t^2)$.
\end{lem}
\begin{proof}
We start by writing $\Psi(t)=\int_\Delta (1-e^{it\cdot\hat\kappa_0})\, d\mu_\Delta
=\int_\Delta (1+it\cdot\hat\kappa_0-e^{it\cdot\hat\kappa_0})\, d\mu_\Delta
$. So
\begin{align*}
\Psi(t)
%&=\sum_{n\in \mathbb Z^d}(1+it\cdot\hat\kappa_0-e^{it\cdot n})\mu_{\Delta}(\hat\kappa_0=n)\\
&=\sum_{(L,w)\in\mathcal S}\sum_{n= 1}^{\lfloor |t|^{-1}\rfloor}(1+it\cdot (L+nw)-e^{it\cdot (L+nw)})\mu_{\Delta}(\hat\kappa_0=L+nw)+O(t^2)\\
&=\sum_{(L,w)\in\mathcal S}\sum_{n= 1}^{\lfloor |t|^{-1}\rfloor}(1+it\cdot (L+nw)-e^{it\cdot (L+nw)})\left(\frac{c_{L,w}}{n^3}+O(n^{-4})\right)+O(t^2)\\
&=\sum_{(L,w)\in\mathcal S}\sum_{n= 1}^{\lfloor |t|^{-1}\rfloor}(1+it\cdot (L+nw)-e^{it\cdot (L+nw)})\frac{c_{L,w}}{n^3}+O(t^2)\, ,
\end{align*}
using the properties of $M_0,\mathcal S$.
% that together with the fact that $\mu_{\Delta}(\hat\kappa_0=L+nw)=O(|n|^{-3})$ for every $(L,w)\in\mathcal B$.
Setting $a_{L,w,n}:=\frac{c_{L,w}}{n^3}$, $A_{L,w,n}:=\sum_{k\ge n}a_{L,w,k}$ and $b_{L,w,n}:=1+it\cdot (L+nw)-e^{it\cdot (L+nw)}$, 
using the Abel transform, we obtain
\begin{align*}
&\Psi(t)=O(t^2)+\sum_{(L,w)\in\mathcal S}\sum_{n=1}^{\lfloor |t|^{-1}\rfloor}a_{L,w,n}b_{L,w,n}\\
&=O(t^2)+\sum_{(L,w)\in\mathcal S}\left\{A_{L,w,1}b_{L,w,1}-A_{L,w,\lfloor|t|^{-1}\rfloor+1}b_{L,w,\lfloor|t|^{-1}\rfloor}+\sum_{n=2}^{\lfloor |t|^{-1}\rfloor}A_{L,w,n}(b_{L,w,n}-b_{L,w,n-1})\right\}\\
&=O(t^2)+\sum_{n=2}^{\lfloor |t|^{-1}\rfloor}\sum_{(L,w)\in\mathcal S}\left\{\left(\frac{c_{L,w}}{2n^2}+O(n^{-3})\right)(it\cdot w+e^{it\cdot(L+(n-1)w)}-e^{it\cdot(L+nw)})\right\}
\end{align*}
and thus
\begin{align*}
\Psi(t)&=O(t^2)+\sum_{n=2}^{\lfloor |t|^{-1}\rfloor}\sum_{(L,w)\in\mathcal S}\left\{\left(\frac{c_{L,w}}{2n^2}+O(n^{-3})\right)(it\cdot w(1-e^{it\cdot(L+nw)}) +O(t^2))\right\}\\
&=O(t^2)+\sum_{n=2}^{\lfloor |t|^{-1}\rfloor}\sum_{(L,w)\in\mathcal S}\left\{\left(\frac{c_{L,w}}{2n^2}+O(n^{-3})\right)(it\cdot w)(-it\cdot(L+nw)+O(t^2n^2))\right \}\\
&=O(t^2)+(t\cdot w)^2\sum_{(L,w)\in\mathcal S}\sum_{n=2}^{\lfloor |t|^{-1}\rfloor}\frac{c_{L,w}}{2n}=O(t^2)+\sum_{(L,w)\in\mathcal S}\frac{c_{L,w}}2(t\cdot w)^2\log||t|^{-1}|.
\end{align*}
\end{proof}
\begin{proof}[Proof of Proposition~\ref{prop-lambda}]
We first observe that  Lemma~\ref{lemm-tail} ensures that the general conditions of
Lemma~\ref{lem:Psi} are satisfied with
$c_{L,w}:=
\sum_{(A,B)\in E_{L,w}}\frac{ d_{A}^2\, \mathfrak a_{(A,B)}}{2|\partial\bar Q|\, |w| }$, 
and with
$\mathcal S$
 is the set of $(w,L)\in(\mathbb Z^2)^2$, with $w$ with co-prime coordinates and $L\in\mathcal E_w$ (where $\mathcal E_w$ is the set of $L\in\mathbb Z^2$ such that $L\cdot w\ge 0>(L-w)\cdot w$)
%\in\mathcal E_w$ (where $\mathcal E_w$is the set of representants of $\mathbb Z^2/(\mathbb Zw)$ with minimal nonnegative coordinates with respect to the alphabetical order: $(a_1,a_1)<(b_1,b_2)$ if $a_1<b_1$ or if $a_1=b_1$ and $a_2<b_2$) 
for which
there exist $(A,B)\in E_{L,w}$.
The finiteness of $\mathcal S$ comes then from the finiteness of $\mathcal A$ and the finiteness of the possibilities for $B$ once $(A,w)$ is fixed, the finiteness of $L$ once $A,w,B$ are fixed comes from our constraint on $L$.
The disjointness assumption of Lemma~\ref{lem:Psi} comes from our first conditions on $w$ and $L$. Since  $1-\lambda_t=\Psi(t)+V(t)$, due to Lemmas~\ref{lem:V} and~\ref{lem:Psi},
we know that
\[
1-\lambda_t=\sum_{(L,w)\in\mathcal S}\frac{c_{L,w}}2(t\cdot w)^2\log||t|^{-1}|+O(t^2)\, .
\]
For any prime $w\in\mathbb Z^2$, due to Lemma~\ref{lemm-tail} 
and Corollary~\ref{lemm-tailrem}, 
%we obtain
\[
\sum_{L\, :\, (L,w)\in\mathcal S}c_{L,w}=\sum_{C\in\mathcal C\, :\, w_C=\pm w}
\frac{ \mathfrak d_{C}^2}{|\partial\bar Q|\, |w_C| }
%\sum_{L\in \mathcal E_w}\sum_{(A,B)\in E_{L,w}}\frac{ d_{A}^2\, \mathfrak a_{(A,B)}}{2|\partial\bar Q|\, |w| }=\sum_{A\, :\, (A,w)\in\mathcal A}\frac{ d_{A}^2}{2|\partial\bar Q|\, |w| }
\, .\] 
%Given a corridor $C$, the corresponding $(A,w)\in\mathcal A$ are given by $(A,\pm w_{\mathcal C})$ with $A$ taken among the points $E\in  \bigcup_{j\in\mathcal J}\partial\mathcal O_j$ such that $E+\mathbb Z^2$ intersects  the boundary of the corridor $C$. 
Therefore 
$1-\lambda_t
%&=\frac 12\sum_{A\, :\, (A,w)\in\mathcal A}\frac{ d_{A}^2}{2|\partial\bar Q|\, |w| }(t\cdot w)^2\log||t|^{-1}|+O(t^2)\\
=2 \sum_{C\in\mathcal C}
%\mathfrak n_C
\frac{ \mathfrak d_{C}^2}{2|\partial\bar Q|\, |w_C| }
(t\cdot w_C)^2\log||t|^{-1}|+O(t^2)$.
\end{proof}

\section{Expansions in the local limit theorem and of decorrelation rate}\label{LLT}
%Let $(M,T,\mu)$ be the $\mathbb Z^d$-extension of $(\bar M,\bar T,\bar\mu)$by $\kappa:\bar M\rightarrow\mathbb Z^d$, with $d\in\{1,2\}$.\\
%Let $P$ be the transfer operator for $(\bar M,\bar T,\bar\mu)$.$P_t:=P\left(e ^{i\langle t,\kappa\rangle}\right)$
\subsection{Expansion in the local limit theorem in a general context}
Let $(\Delta, f,\mu_\Delta)$ be a probability preserving dynamical system with transfer
operator $P$. 
Let $p,p_0,p_1\in[ 1,+\infty]$ with $p_0\le p_1$.
Set $\hat \kappa_n:=\sum_{k=0}^{n-1}\hat\kappa\circ f^k$, with $\hat\kappa:\Delta\rightarrow\mathbb Z^d$ integrable with zero mean. Assume that, for every $t\in[-\pi,\pi]^d$, the operator $P_t:f\mapsto P(e^{i t\cdot\hat\kappa}f)$ acts on a complex Banach space $\mathcal B$ of functions $f:\Delta\rightarrow\mathbb C$ and satisfies the following properties.
Assume that $\hat\kappa\in L^p(\mu_\Delta)$, that $\mathcal B\hookrightarrow L^{p_1}(\mu_\Delta)$ and 
that there exists $\beta\in(0,\pi)$ such that for every
$t\in[-\beta,\beta]^d$,
\begin{equation}\label{quasicompact}
P_t^n=\lambda_t^n\Pi_t+N_t^n\, ,
\end{equation}
%with
\begin{equation}\label{borneexpo}
\sup_{t\in[-\beta,\beta]^d}\Vert N_t^n\Vert_{\mathcal L(\mathcal B,L^{p_0}(\mu_\Delta))}+\sup_{t\in[-\pi,\pi]^d\setminus[-\beta,\beta]^d}\Vert P_t^n\Vert_{\mathcal L(\mathcal B,L^{p_0}(\mu_\Delta))}=O\left(\theta^n\right)\, ,\quad\mbox{with}\ 0<\theta<1\, .
\end{equation}
Assume moreover that there exists an invertible positive symmetric matrix $A$ such that
\begin{equation}\label{lambda}
\lambda_t=1- A t\cdot t\, \log (1/|t|)+O(t^2)\, ,\ i.e. \ \lambda_t=e^{- A t \cdot t\, \log (1/|t|)+O(t^2)}\, .
\end{equation}
Moreover assume either that there exists
% $p_2\in[ 1,p_1]$ and 
$\gamma\in(0,1]$ such that
\begin{equation}\label{Pi}
\Pi_t:=
\mathbb E_{\mu_{\Delta}}\left[\cdot\right]\mathbf 1_{\Delta}+O\left( |t|^\gamma\right)\quad\mbox{in}\quad\mathcal L\left(\mathcal B, L^{p_0}(\mu_{\Delta})\right)
%\,,\quad \gamma\in(0,1]
\, , 
\end{equation}
or that 
there exists
%$\mathcal B$ and $p_0\in[ 1,p_1]$ and 
$\gamma'\in(0,1]$ such that
\begin{equation}\label{Pidiff}
%\hat\kappa\in L^{p_0}(\mu_\Delta),\ 
\exists \Pi'_0\in\mathcal L(\mathcal B,L^{p_0}(\mu_\Delta)),\ 
\Pi_t:=\mathbb E_{\mu_{\Delta}}\left[\cdot\right]\mathbf 1_{\Delta}+t\cdot \Pi'_0+ O\left( t^{1+\gamma'}\right)\  \mbox{in}\ \mathcal L(\mathcal B,L^{p_0}(\mu_\Delta))\, .
\end{equation}
Note that,  due to Proposition~\ref{PROP1},
under \eqref{H0}, both~\eqref{Pi} and~\eqref{Pidiff} are satisfied for the tower $(\Delta,f,\mu_\Delta)$ associated to the discrete infinite horizon Lorentz gas.
%Notice that \eqref{lambda} can be rewritten
%\begin{equation}\label{lambda2}
%\lambda_t=e^{- A t \cdot t\, \log (1/|t|)+O(t^2)}\, .
%\end{equation}
Let $q_0,q_1\in[1,+\infty]$ be such that $\frac 1{q_i}+\frac 1 {p_i}= 1$.
% and $q_1\in[1,+\infty]$ such that $\frac 1{q_1}+\frac 1 {p_1}\le 1$.
Let us define
\begin{equation}\label{I0I2}
I_{0}(X):=\frac{e^{-\frac 1{2}  A^{-1}X\cdot X}}{\sqrt{(2\pi)^{d}\det A}},\quad
I_{2}(X):=
-
\frac{ A^{-1} X\cdot X-d}{\sqrt{(2\pi)^{d}\det A}}\, e^{-\frac 1{2}  A^{-1}X\cdot X}\, ,
\end{equation}
\begin{equation}\label{I1I3}
I_{1}(X):=
-i
\frac{A^{-1}X}{\sqrt{(2\pi)^{d}\det(A)}}e^{-\frac 1{2}  A^{-1}X\cdot X},\quad
I_{3}(X):=
-i
\frac{A^{-1}X(d+2- A^{-1} X\cdot X)}{\sqrt{(2\pi)^{d}\det(A)}}\, e^{-\frac 1{2}  A^{-1}X\cdot X}
.
\end{equation}
Under these assumptions, we state the following general local limit theorem with expansion, that can be read first considering $k=0$. The generalization to $k=O(n/\log n)$ will be useful in the proof of our main results (Theorems~\ref{THM0} and \ref{THM1}) due to approximations of observables $\phi,\psi$ using functions that are constant on every stable curve. We set $a_n:=\sqrt{n\log n}$. 
\begin{prop}\label{LLT1} Assume~\eqref{quasicompact}--\eqref{lambda}.

If \eqref{Pi} holds, then, 
for every $h\in\mathcal B%\cap L^{p_3}(\mu_\Delta)
$, for $g\in L^{q_0}(\mu_\Delta)$, uniformly in $N\in\mathbb Z^d$ 
\begin{align}\label{eq:LLT1}
&\mathbb E_{\mu_{\Delta}}[h\mathbf 1_{\{\hat \kappa_n\circ f^k=N\}}g\circ f^n]=
\frac {\mathbb E_{\mu_{\Delta}}[h]\mathbb E_{\mu_{\Delta}}[g]}{a_n^d}
\left(I_0\left(\frac N{a_n}\right)-\frac {I_2\left(\frac N{a_n}\right)\log\log n +O(1)}{2\log n}
%+O((\log n)^{-1})
\right)\\
&
+O\left(
a_n^{-d}(a_n^{-\gamma}+(k/a_n)^{\min(1,\frac p{q_1},\frac p{p_0})})
\Vert g\Vert_{L^{q_0}(\mu_\Delta)}
(\sup_t\Vert P_t^kP^kh\Vert_{\mathcal B}
%+\Vert h\Vert_{L^{p_3}})
\right)\, ,\nonumber
%&+O\left(\frac{\Vert g\Vert_{L^{q_1}(\mu_\Delta)}\sup_u\Vert P_u^{k}P^kh\Vert_{\mathcal B}}{a_n^d\log n}\right)\, ,
\end{align}
for $k=0$ (without additional assumption on $p_1$) and uniformly in 
$k\le C n/\log n
%C a_n/(\log n)^{\max(p_1,q_1)}$ if $\max(p_1,q_1)<\infty
$.\\
%, with $\frac 1{p'_3}+\frac 1{p_3}=\frac 1{q_1}+\frac 1{q'_1}=1$.\\

If \eqref{Pidiff} holds, then, for every $h\in\mathcal B$, for $g\in L^{q_0}(\mu_\Delta)$, uniformly in $N\in\mathbb Z^d$, 
\begin{align*}%\label{eq:LLT2}
&\mathbb E_{\mu_{\Delta}}[h\mathbf 1_{\{\hat \kappa_n\circ f^k=N\}}g\circ f^n]\\
&=\frac 1{a_n^d}\mathbb E_{\mu_{\Delta}}[h]\mathbb E_{\mu_{\Delta}}[g]
\left(I_0\left(\frac N{a_n}\right)-\frac {I_2\left(\frac N{a_n}\right)\log\log n+O(1)}{2\log n}
%+O\left(\frac 1{\log n}\right)
\right)\\
&+\frac {
C_k(g,h)}{a_n^{d+1}}
\left(I_{1}\left(\frac N{a_n}\right)-\frac {I_3\left(\frac N{a_n}\right)}2 \frac{\log\log n}{\log n}\right)+O\left(\frac{\Vert g\Vert_{L^{q_0}(\mu_\Delta)}\sup_u\Vert P_u^{k}P^kh\Vert_{\mathcal B}}{a_n^{d+1}\log n}\right)\\
&
+O\left(\frac {1 }{(2\pi )^d(a_n)^{d+1}}\int_{[-\beta a_n,\beta a_n]^d}\!\!\!\!\!\!\!\!\! |u|e^{-\frac{c_0\min(|u|^{2-\epsilon},|u|^{2+\epsilon})}2} |\mathbb E_{\mu_{\Delta}}[g\Pi'_0(P_{u/a_n}^kP^kh-P^{2k}h)]|\, du\right)
\, ,
\end{align*}
with 
$C_k(g,h):= \mathbb E_{\mu_{\Delta}}[g\Pi'_0P^{2k}h]+ i\mathbb E_{\mu_{\Delta}}[\hat\kappa_k  g]\mathbb E_{\mu_\Delta}[h]+i\mathbb E_{\mu_\Delta}[g]  \mathbb E_{\mu_\Delta}[\hat\kappa_kP^kh]$,
for $k=0$ (without additional assumption on $p_1,p_0,q_0$) and uniformly in 
$k\le C a_n^{\frac{\tilde\gamma}{1+\tilde\gamma}}/(\log n)^{1/(1+\tilde\gamma)}$ 
with $\tilde\gamma=\min(1,\frac{p(q_0-1)-q_0}{q_0},\frac{p-q_1}{q_1})$
if
%$p_1>2$, 
$q_1<p$, $q_0<+\infty$,  $\frac 1{q_0}+\frac 1 {p}<1$.

%\sum_{k\ge 0}\frac {1}{(\log n)^k}\sum_{\ell=0}^k\frac {(-\log\log n)^{k-\ell}}{\ell!(k-\ell)! 2^{k-\ell}}I_{k,\ell}\, +E_n(h,g)\, ,$$
%with
%$$I_{k,\ell}:=\int_{\mathbb R^d}\left(\langle A u,u\rangle\right)^ke^{-\frac 1{2}\langle A u,u\rangle} \left(\log|u|\right)^{\ell}\, du\, ,$$
%and with $E_n(h,g)=O\left(a_n^{-d}\Vert g\Vert_\infty\Vert f\Vert_{\mathcal B}/\log n\right)$.\\
%Moreover, under \eqref{Pidiff}, we have
%$$E_n(h,g)= O\left(a_n^{-d-1-\gamma'}\Vert g\Vert_\infty\Vert f\Vert_{\mathcal B}\right)\, .$$
\end{prop}
%Let us insist on the fact that w
When $k=0$, one can take $p_0=p_1=1$, $q_0=q_1=\infty$ and that the last error term in the second estimate of Proposition~\ref{LLT1} vanishes.
\begin{rem}
If $k=0$ and $\mathbb E_{\mu_\Delta}[g]\mathbb E_{\mu_\Delta}[h]=0$,
 the second estimate of Proposition~\ref{LLT1} provides an estimate in $O(a_n^{-d-1}/\log n)$ if $N$ is fixed or uniformly bounded, but provides 
an estimate with a leading term (non null in general) of order $a_n^{-d-1}$ if $N$ has order $a_n$.
\end{rem}

\begin{proof}[Proof of Proposition~\ref{LLT1}]
We assume that $\beta<1$, that
\begin{equation}\label{lambda0}
\forall t\in[-\beta,\beta]^d,\quad e^{-2 At\cdot t\log(1/|t|)}\le \left|\lambda_t\right|\le e^{-\frac 12 At\cdot t\log(1/|t|)}\, 
\end{equation}
for $|\cdot|$ the supremum norm on $\mathbb R^d$
and that
$\forall y>x>\beta^{-1}$, $\frac 12 (x/y)^\epsilon \le \log(x)/\log(y)\le 2 (y/x)^{\varepsilon}$
(using for example Karamata's representation of slowly varying functions).
This last condition will imply that, for every $n$ large enough (so that $a_n>\beta^{-1}$) and for every $|u|\in[-\beta a_n,\beta a_n]$, the following 
inequalities hold true
\begin{equation}\label{logx/logy}
\frac 12\min(|u|^\varepsilon,|u|^{-\varepsilon})\le  \log(a_n/|u|)/\log a_n  \le    2\max(|u|^\varepsilon,|u|^{-\varepsilon})\, .
\end{equation}
Let $h,g$ be as in the statement of Proposition~\ref{LLT1}.
Note that $P_t^n(H.G\circ f^n) =GP_t^n(H)$ (since $\int_\Delta   h  P(H \cdot G\circ f)\,d\mu_{\Delta}=  \int_\Delta  ( h\circ f)\, H\,  (G\circ f) \,d\mu_{\Delta} =  \int_\Delta     h\,G\, P(H) \,d\mu_{\Delta}$). Thus,
\begin{align*}
\mathbb E_{\mu_{\Delta}}&[h\mathbf 1_{\{\hat \kappa_n\circ f^k=N\}}g\circ f^n]
=\frac 1{(2\pi)^d}\int_{[-\pi,\pi]^d}e^{-it\cdot N}\mathbb E_{\mu_{\Delta}}\left[he^{it\cdot\hat \kappa_n\circ f^k}g\circ f^n\right]\, dt\\
&=\frac 1{(2\pi)^d}\int_{[-\pi,\pi]^d}e^{-it\cdot N}\mathbb E_{\mu_{\Delta}}\left[P^{n+k}\left(he^{it\cdot\hat \kappa_n\circ f^k}g\circ f^n\right)\right]\, dt\\
&
=\frac 1{(2\pi)^d}\int_{[-\pi,\pi]^d}e^{-it\cdot N}\mathbb E_{\mu_{\Delta}}\left[P^{k}\left(P^n\left(he^{it\cdot\hat \kappa_{n-k}\circ f^k}(ge^{it\cdot\hat \kappa_k})\circ f^n\right)\right)\right]\, dt\, .
\end{align*}
Thus $\mathbb E_{\mu_{\Delta}}(h\mathbf 1_{\{\hat \kappa_n\circ f^k=N\}}g\circ f^n)
=\frac 1{(2\pi)^d}\int_{[-\pi,\pi]^d}e^{-it\cdot N}\mathbb E_{\mu_{\Delta}}\left[P_t^k\left( gP^n\left(he^{it\cdot\hat \kappa_{n-k}\circ f^k}\right)\right)\right]\, dt$ and so
\[
\mathbb E_{\mu_{\Delta}}[h\mathbf 1_{\{\hat \kappa_n\circ f^k=N\}}g\circ f^n]=\frac 1{(2\pi)^d}\int_{[-\pi,\pi]^d}e^{-it\cdot N}\mathbb E_{\mu_{\Delta}}\left[ P_t^k \left(gP_t^{n-2k}P_t^kP^kh\right)\right]\, dt\, .
\]
%\end{align*}
Now, due to \eqref{quasicompact} and \eqref{borneexpo}, $\mathbb E_{\mu_{\Delta}}(h\mathbf 1_{\{\hat \kappa_n\circ f^k=N\}}g\circ f^n)$ is equal to
\begin{align}
%\nonumber&\\
\nonumber&\frac 1{(2\pi)^d}\int_{[-\beta,\beta]^d}e^{-it\cdot N}\lambda_t^{n-2k}\mathbb E_{\mu_{\Delta}}\left[e^{it\cdot\hat\kappa_k}g\Pi_tP_t^kP^kh\right]\, dt+O(\theta^{n-2k}\Vert g\Vert_{L^{q_0}(\mu_\Delta)}\sup_t\Vert P_t^kP^kh\Vert_{\mathcal B})\\
\label{FormulaInt}&=\frac 1{(2\pi a_n)^d}\int_{[-\beta a_n,\beta a_n]^d}e^{-i\frac u{a_n}\cdot N}\lambda_{u/a_n}^{n-2k}\mathbb E_{\mu_{\Delta}}\left[e^{i\frac u{a_n}\cdot\hat\kappa_k}g\Pi_{u/a_n}P_{u/a_n}^kP^kh\right]\, du\\
&\ \ \ +O\left(\theta^{n-2k}\Vert g\Vert_{L^{q_0}(\mu_\Delta)}\sup_t\Vert P_t^kP^kh\Vert_{\mathcal B}\right)\, .\nonumber
\end{align}
Observe that $a_n^2\sim 2 n\log(a_n)$. 
This, combined 
with~\eqref{lambda0} and~\eqref{logx/logy},
ensures that there exists $c_0,c'_0>0$ such that, 
for every $n$ large enough
% and $k=O(n/\log n)$ with $c\in(0,1)$ 
and for every $u\in[-\beta a_n,\beta a_n]^d$,
% the following holds true
\begin{equation}\label{bornelambda}
e^{-c'_0\max(|u|^{2-\epsilon},|u|^{2+\epsilon})}\le e^{-\frac {2n}{a_n^2} Au\cdot  u\log(a_n/|u|)}\le \left|\lambda_{u/a_n}^{n} \right|\le e^{-\frac {n}{2a_n^2} Au\cdot  u\log(a_n/|u|)}\le 
% e^{-{c'_0}{|u|^2}\frac{\log(a_n/|u|)}{\log a_n}}\le 
e^{-c_0\min(|u|^{2-\epsilon},|u|^{2+\epsilon})}  \, .
\end{equation}
Therefore, using \eqref{Pi}, we obtain that
$$\left\Vert\lambda_{u/a_n}^{n-2k}\left(\Pi_{u/a_n}-\mathbb E_{\mu_{\Delta}}[\cdot]\mathbf 1_{\Delta}\right)\right\Vert_{\mathcal L(\mathcal B,L^{p_0}(\mu_{\Delta}))}\le  K_1 e^{-c_0\min(|u|^{2+\epsilon},|u|^{2-\epsilon})}(u/a_n)^\gamma \, ,$$
and so
\begin{align}
&\mathbb E_{\mu_{\Delta}}(h\mathbf 1_{\{\hat \kappa_n\circ f^k=N\}}g\circ f^n)=O\left(a_n^{-d-\gamma}\Vert g\Vert_{L^{q_0}(\mu_\Delta)}\sup_t\Vert P_t^kP^kh\Vert_{\mathcal B}\right)\nonumber\\
&+\frac 1{(2\pi a_n)^d}\int_{[-\beta a_n,\beta a_n]^d}\mathbb E_{\mu_{\Delta}}[e^{i\frac u{a_n}\cdot\hat\kappa_k}g]\, \mathbb E_{\mu_{\Delta}}[e^{i\frac u{a_n}\cdot\hat\kappa_k\circ f^k}h]e^{-i\frac u{a_n}\cdot N}\lambda_{u/a_n}^{n-2k}\, du\nonumber\\
&=\frac 1{(2\pi a_n)^d} \mathbb E_{\mu_{\Delta}}[g] \, \mathbb E_{\mu_{\Delta}}[h]\int_{[-\beta a_n,\beta a_n]^d}e^{-i\frac u{a_n}\cdot N}\lambda_{u/a_n}^{n-2k}\, du\nonumber\\
&\ \ \  +O(
a_n^{-d}(a_n^{-\gamma}+(k/a_n)^{\min(1,\frac p{q_1})}+(k/a_n)^{\min(1,\frac p{p_0})})
\Vert g\Vert_{\mathbb L^{q_0}(\mu_\Delta)}
\sup_t\Vert P_t^kP^kh\Vert_{\mathcal B}
%\Vert h\Vert_{\mathbb L^{p_1}(\mu_\Delta)}
)\, ,\label{CCC1}
\end{align}
due to \eqref{bornelambda}
 combined with the two following estimates
\begin{align*}
\mathbb E_{\mu_\Delta}[(e^{i\frac u{a_n}\cdot\hat\kappa_k}-1)g]&=O\left(|u|^{\gamma_0}\frac{k^{\gamma_0}\Vert\hat\kappa\Vert^{\gamma_0}_{L^{p}(\mu_\Delta)}\Vert g\Vert_{\mathbb L^{q_0}(\mu_\Delta)}}
      {(a_n)^{\gamma_0}}\right)\\
\mathbb E_{\mu_\Delta}[(e^{i\frac u{a_n}\cdot\hat\kappa_k\circ f^k}-1)P^kh]&=O\left(|u|^{\gamma_0}\frac{k^{\gamma_0}\Vert\hat\kappa\Vert^{\gamma_0}_{L^{p}(\mu_\Delta)}\Vert h\Vert_{\mathbb L^{p_1}(\mu_\Delta)}}
      {(a_n)^{\gamma_0}}\right)\, ,
\end{align*}
with $\gamma_0:=\min(1,\frac p{p_0}
%)$ and $\delta:=\min(1
,\frac p{q_1})$.
Here we have used  H{\"o}lder inequality,
the fact that  
%$k\ll a_n$,  
$|e^{ix}-1|\le 2|x|^{\gamma_0}$
% for $p\ge 1$ 
and 
the following inequality applied with $s=\gamma_0$ and $r\in\{p_0,q_1\}$
% and $s=\min(1,p/p_j)$
%for any $r\ge 1$ and $s>1/r$,
\begin{equation}\label{normrs}
\forall r\ge 1,\quad\forall s\ge 1/r,\quad 
\||\hat\kappa_k|^{s}\|_{L^r(\mu_\Delta)}=
\|\hat\kappa_k\|^{s}_{L^{rs}(\mu_\Delta)}\le 
%\left(k \|\hat\kappa\|_{L^{rs}(\mu_\Delta)}\right)^{s}=
 k^{s} \|\hat\kappa\|^s_{L^{rs}(\mu_\Delta)}
\, ,
\end{equation}
where we used the triangular inequality for $\Vert\cdot\Vert_{L^{rs}(\mu_\Delta)}$
and the fact that $\Vert\hat\kappa\circ f^m\Vert_{L^{rs}(\mu_\Delta)}=\Vert\hat\kappa\Vert_{L^{rs}(\mu_\Delta)}$ by $f$-invariance of $\mu_\Delta$.
%The last displayed inequality is a consequence of  the fact that $\|v^{1/r}\|_{L^r}=\|v\|_{L^1}^{1/r}$ for any $r>1$ and all $v\in L^r$.
%since $1/\gamma_0=\min(1/p_1,1/q_1)$ and 
Moreover, due to \eqref{lambda%2
} 
and  \eqref{bornelambda}
and since $k=O(n/\log n)$,
\begin{align}
\label{lambdan-k}\lambda_{u/a_n}^{n-2k} 
&=\lambda_{u/a_n}^{n} \lambda_{u/a_n}^{-2k} 
%\nonumber&
=  e^{-\frac n{ a_n^2} A u\cdot u\, \left(\log (a_n/|u|)\right)+O(n|u|^2/a_n^2)}  e^{O(\frac kn\max(|u|^{2-\epsilon},|u|^{2+\epsilon}))}\\
\nonumber&=  e^{-\frac 1{2 \log n} A u\cdot u\, \left(\log n+\log(\log n/|u|^2)\right)+O(\max(|u|^{2-\epsilon},|u|^{2+\epsilon})/\log n)}\\
%&=e^{-\frac 1{2} A u\cdot u\, \left(1+\frac{\log(\log n/|u|^2)}{\log n}\right)+O\left({-k a_n^{-2} A u\cdot u\, \log (a_n/|u|)\right)+O\left(\frac{\max(|u|^{2-\varepsilon},|u|^{2+\varepsilon})}{\log n}\right)}\\
\nonumber&=e^{-\frac 1{2} A u\cdot u\, \left(1+\frac{\log(\log n/|u|^2)}{\log n}\right)}
+O\left(e^{-\frac{c_0\min(|u|^{2-\epsilon},|u|^{2+\epsilon})}2}\frac{\max(|u|^{2-\varepsilon},|u|^{2+\varepsilon})}{\log n}\right)\, .
%&=e^{-\frac 1{2}\langle A u,u\rangle}e^{-\frac 1{2}\langle A u,u\rangle\, \frac{\log(\log n/|u|^2)}{\log n}+O(|u|^2/\log n)}\, .
\end{align}
Therefore
\begin{align*}
&\int_{[-\beta a_n,\beta a_n]^d}u^\ell e^{-i\frac u{a_n}\cdot N}\lambda_{u/a_n}^{n-2k}\, du=\int_{\mathbb R^d}u^\ell e^{-i\frac u{a_n}\cdot N}e^{-\frac 1{2} A u\cdot u\, \left(1+\frac{\log(\log n/|u|^2)}{\log n}\right)}\, du+O\left(\frac 1{\log n}\right)\nonumber\\
=&\int_{\mathbb R^d}u^\ell e^{-i\frac u{a_n}\cdot N}e^{-\frac 1{2} A u\cdot u}
\left(1 -\frac 12  A u\cdot u\frac{\log(\log n/|u|^2)}{\log n}\right. \nonumber\\
&\left. \ \ \ +O\left(e^{\frac 12 Au\cdot u\frac{|\log(\log n/|u|^2)|}{\log n}}
\frac{
|u|^4(\log(\log n/|u|^2))^2}{(\log n)^2}\right)\right)\, du+O\left(\frac 1{\log n}\right)\, ,\nonumber
\end{align*}
where we used $e^{x}=1+x+O(e^{|x|}x^2)$. Thus
\begin{align}
&\int_{[-\beta a_n,\beta a_n]^d}u^\ell e^{-i\frac u{a_n}\cdot N}\lambda_{u/a_n}^{n-2k}\, du
=\int_{\mathbb R^d}u^\ell e^{-i\frac u{a_n}\cdot N}e^{-\frac 1{2} A u\cdot u}
\left(1 -\frac 12  A u\cdot u\frac{\log(\log n/|u|^2)}{\log n}\right)\nonumber\\
&\ \ \ +O\left(|u|^{4+\ell} 
e^{-\frac 1{2} A u\cdot u\left(1-\frac{\log\log n}{\log n}-2\frac{|\log |u||}{\log n}\right)}\frac{(\log\log n)^2+|\log |u||^2}{(\log n)^2}\right)\, du+O\left(\frac 1{\log n}\right)\, , \nonumber
\end{align}
%and so
\begin{align}
\int_{[-\beta a_n,\beta a_n]^d}u^\ell &e^{-i\frac u{a_n}\cdot N}\lambda_{u/a_n}^{n-2k}\, du=\int_{\mathbb R^d}u^\ell e^{-i\frac u{a_n}\cdot N}e^{-\frac 1{2} A u\cdot u}
\left(1 -\frac 12  A u\cdot u\frac{\log(\log n/|u|^2)}{\log n}\right)\nonumber\\
&\ \ +O\left(|u|^{4+\ell} 
e^{-\frac 1{2} A u\cdot u\left(\frac 12 - 2\max(|u|,|u|^{-1})\right)}\frac{(\log\log n)^2+|\log |u||^2}{(\log n)^2}\right)\, du+O\left(\frac 1{\log n}\right)\nonumber\\
=&\int_{\mathbb R^d}u^\ell e^{-i\frac u{a_n}\cdot N}e^{-\frac 1{2} A u\cdot u}
\left(1 -\frac 12  A u\cdot u\frac{\log\log n}{\log n}\right)\, du+O\left(\frac 1{\log n}\right)\, .\label{intlambda}
%&=&\int_{\mathbb R^d}e^{-\frac 1{2}\langle A u,u\rangle}\sum_{k\ge 0} \frac 1{k!}\left(-\frac 12 \langle A u,u\rangle\frac{\log(\log n/|u|^2)}{\log n}\right)^k\, du+O\left(\frac 1{\log n}\right)\, .
%\label{intlambdabis}
\end{align}
Applying this formula with $\ell=0$ and using
 $(k/a_n)^{\min(1/p_1,1/q_1)}=O((\log n)^{-1})$,
\eqref{CCC1} becomes
\eqref{eq:LLT1}
%\[
%\mathbb E_{\mu_{\Delta}}(h\mathbf 1_{\{\hat S_n\circ f^k=N\}}g\circ f^n)=\frac 1{ a_n^d}\mathbb E_{\mu_{\Delta}}[h]\mathbb E_{\mu_{\Delta}}[g]\left(I_0\left(\frac N{a_n}\right)-\frac 12 I_2\left(\frac N{a_n}\right)\frac{\log\log n}{\log n}\right)+O\left(\frac{\Vert g\Vert_{\mathbb L^{q_1}(\mu_\Delta)}\Vert f\Vert_{\mathcal B}}{a_n^d\log n}\right)\, ,\]
with $I_{2k}(x):=\frac 1{(2\pi)^{d}}\int_{\mathbb R^d}e^{-i u\cdot x}\left( A u\cdot u\right)^{k}e^{-\frac 1{2} A u\cdot u}\, du$, which coincide with \eqref{I0I2}.
Indeed, with the change of variable $v=A^{1/2}u$,
$I_0(x)=\frac{\Phi(A^{-1/2}x)}{\sqrt{\det A}}$ and $I_2(X)=-\frac{(\Delta \Phi)(A^{-1/2}x)}{\sqrt{\det A}}$ where $\Delta$ is the Laplacian operator and with $\Phi(x):=\frac 1{(2\pi)^{d}}\int_{\mathbb R^d}e^{-i u\cdot x}
%\left(  u\cdot u\right)^{k}
e^{-\frac 1{2}  u\cdot u}\, du=\frac{e^{-\frac 12 x\cdot x}}{\sqrt{(2\pi)^d}}$ the standard $d$-dimensional Gaussian density.
This ends the proof of the first assertion of Proposition~\ref{LLT1}.
\\
Assume from now on \eqref{Pidiff}.
% and that $\mathbb E_{\mu_\Delta}[g]\mathbb E_{\mu}[h]=0$.
Then,
\begin{align*}
\mathbb E_{\mu_\Delta}[P_t^k(g\Pi_tP_t^k P^k h)]&=\mathbb E_{\mu_\Delta}[e^{it\cdot\hat\kappa_k}g]\mathbb E_{\mu_\Delta}[P_t^kP^k h]+ t\cdot\mathbb E_{\mu_\Delta}[e^{i t\cdot\hat\kappa_k}g\Pi'_0P_t^kP^k h]\\
&\ \ \  \ \ \ +O(t^{1+\gamma'}\Vert g\Vert_{L^{q_0}(\mu_\Delta)}\Vert P_t^kP^k h\Vert_{\mathcal B})\, ,
\end{align*}
and so \eqref{FormulaInt} leads to
\begin{align*}%\label{AAA2}
&\mathbb E_{\mu_{\Delta}}[h\mathbf 1_{\{\hat \kappa_n\circ f^k=N\}}g\circ f^n]
=\frac 1{(2\pi a_n)^d}\int_{[-\beta a_n,\beta a_n]^d}e^{-i\frac u{a_n}\cdot N}\lambda_{u/a_n}^{n-2k}\\
&
\ \ \ \ \ \  \ \ \ \ \ \ \ \left(\mathbb E_{\mu_{\Delta}}[e^{i\frac u{a_n}\cdot\hat\kappa_k}g]\, \mathbb E_{\mu_{\Delta}}[e^{i\frac u{a_n}\cdot\hat\kappa_k\circ f^k}h]+\frac u{a_n}\cdot\mathbb E_{\mu_\Delta}[e^{i \frac u{a_n}\cdot\hat\kappa_k}g\Pi'_0  P_{\frac u{a_n}}^kP^kh]
\right)
\, du\\
& \ \ \ \ \ \  \ \ \ \ \ \ \ +O\left(a_n^{-d-1-\gamma'}\Vert g\Vert_{L^{q_0}(\mu_\Delta)}\sup_t\Vert P_t^kP^kh\Vert_{\mathcal B}\right)\, .
\end{align*}
If $k=0$, we go directly to \eqref{AAA2b}. Otherwise we proceed as follows.
Since  $|e^{ix}-1-ix|\le 2|x|^{1+\widetilde\gamma}$ for 
%every $\gamma\in(0,1]$, observe that
$\widetilde\gamma:=\min(1,\frac{p(q_0-1)}{q_0}-1,\frac p{q_1}-1)\in (0,1]$
(since $\frac 1{q_0}+\frac 1 {p}<1$ and $q_1<p$), 
\begin{align*}
\mathbb E_{\mu_\Delta}\left[\left(e^{i\frac u{a_n}\cdot\hat\kappa_k}-1-i\frac u{a_n}\cdot\hat\kappa_k\right)g\right]
& =O\left(\frac{1}{a_n^{1+\widetilde\gamma}}\mathbb E_{\mu_{\Delta}}\left[\left|u\cdot\hat\kappa_k\right|^{1+\widetilde\gamma} |g|\right]\right)\\
&=O\left(\frac{|u|^{1+\widetilde\gamma}}{a_n^{1+\widetilde\gamma}}\left\Vert |\hat\kappa_k|^{1+\widetilde\gamma}\right\Vert_{L^{\frac{q_0}{q_0-1}}(\mu_\Delta)} \left\Vert g\right\Vert_{L^{q_0}(\mu_\Delta)}\right)\, ,
\end{align*}
due to the H\"older inequality,
%with $\gamma_1=\min(1,\frac{p(q_0-1)}{q_0}-1)$ so that 
noting that $\frac{q_0}{q_0-1}(1+\widetilde\gamma)\le p$.
Thus, using \eqref{normrs} with $s=1+\widetilde\gamma$ and $r=\frac{q_0}{q_0-1}$, 
we obtain
\begin{align*}
\mathbb E_{\mu_\Delta}\left[\left(e^{i\frac u{a_n}\cdot\hat\kappa_k}-1-i\frac u{a_n}\cdot\hat\kappa_k\right)g\right]=O\left(|u|^{1+\widetilde\gamma}\frac{k^ {1+\widetilde\gamma}\Vert\hat\kappa\Vert^{1+\widetilde\gamma}_{L^{\frac{q_0}{q_0-1}(1+\widetilde\gamma)}(\mu_\Delta)}\Vert g\Vert_{\mathbb L^{q_0}(\mu_\Delta)}}
      {a_n^{1+\widetilde\gamma}}\right)\, ,
\end{align*}
%Similarly,
\[
\mathbb E_{\mu_\Delta}[(e^{i\frac u{a_n}\cdot\hat\kappa_k}-1-i\frac u{a_n}\cdot\hat\kappa_k) P^kh]=O\left(|u|^{1+\widetilde\gamma}\frac{k^{1+\widetilde\gamma}\Vert\hat\kappa^{1+\widetilde\gamma}\Vert_{L^{q_1}(\mu_\Delta)}\Vert  P^kh\Vert_{\mathbb L^{p_1}(\mu_\Delta)}}
      {a_n^{1+\widetilde\gamma}}\right)\, ,
\]
%with $ \gamma_0=\min(1,\frac p{q_1}-1)$ so
noticing that $(1+\widetilde\gamma)\frac{q_0}{q_0-1}\le p$ and $(1+\widetilde\gamma)q_1\le p$.
% and  $\gamma_0\in(0,1]$.
Therefore
\begin{align*}
&\mathbb E_{\mu_{\Delta}}[h\mathbf 1_{\{\hat \kappa_n\circ f^k=N\}}g\circ f^n]
=\frac 1{(2\pi a_n)^d} \mathbb E_{\mu_{\Delta}}[g] \, \mathbb E_{\mu_{\Delta}}[h]\int_{[-\beta a_n,\beta a_n]^d}e^{-i\frac u{a_n}\cdot N}\lambda_{u/a_n}^{n-2k}\, du\\
&+\frac 1{(2\pi )^d(a_n)^{d+1}}\int_{[-\beta a_n,\beta a_n]^d}
 u\cdot   D_{n,k}(g,h) e^{-i\frac u{a_n}\cdot N}\lambda_{u/a_n}^{n-2k} \, du\\
&
 +O\left(a_n^{-d-1}/\log n
\Vert g\Vert_{L^{q_0}(\mu_\Delta)}\sup_t\Vert P_t^kP^kh\Vert_{\mathcal B}\right)\, ,
\end{align*}
with
$
D_{n,k}(g,h):=\left(\mathbb E_{\mu_{\Delta}}[ i\hat\kappa_k (g\Pi_0(h)+\Pi_0(g) P^kh) +e^{i \frac u{a_n}\cdot\hat\kappa_k\circ f^k}g\Pi'_0 P_{\frac u{a_n}}^kP^kh]\right)$, 
since $k^{1+\widetilde\gamma}\ll a_n^{\widetilde\gamma}/\log n$
% and $k^{1+\gamma_0}\ll a_n^{\gamma_0}/\log n$ 
due to our assumption on $k$.
Recall that $\Pi'_0\in\mathcal L(\mathcal B, L^{p_0}(\mu_\Delta))$. Choosing
$\gamma_2=\min(1,\frac{p_0}{q_0})\in (0,1]$  (since $q_0<\infty$)
% $\gamma_2\in(0,1]$ small enough 
so that $1\le\gamma_2q_0\le p_0$,
\begin{align*}
\left|\mathbb E_{\mu_{\Delta}}[(e^{i \frac u{a_n}\cdot\hat\kappa_k\circ f^k}-1)g\Pi'_0 P^k_{u/a_n}P^kh]
    \right|&\le 2|u/a_n|^{\gamma_2}\mathbb E_{\mu_\Delta}[|\hat\kappa_k|^{\gamma_2}\Pi'_0P^k_{u/a_n}P^kh]\\
&\le  2|u/a_n|^{\gamma_2}  k^{\gamma_2}\Vert\hat\kappa\Vert_{L^{\gamma_2 q_0}(\mu_\Delta)}^{\gamma_2}\Vert \Pi'_0P^k_{u/a_n}P^kh\Vert_{L^{p_0}(\mu_\Delta)}\, ,
\end{align*}
since  $\Vert \hat\kappa_k\Vert_{L^{q_0\gamma_2}(\mu_\Delta)}\le  k\Vert\hat\kappa\Vert_{L^{q_0\gamma_2}(\mu_\Delta)}$ (by triangular inequality and $f$- invariance of $\mu_\Delta$). So
\begin{align}\label{AAA2b}
&\mathbb E_{\mu_{\Delta}}(h\mathbf 1_{\{\hat S_n=N\}}g\circ f^n)=
\frac 1{(2\pi a_n)^d}\int_{[-\beta a_n,\beta a_n]^d}e^{-i\frac u{a_n}\cdot N}\lambda_{u/a_n}^{n-2k}\, du\, \mathbb E_{\mu_\Delta}[g]\mathbb E_{\mu_\Delta}[h]\\
\nonumber&+
\frac {1 }{(2\pi )^d(a_n)^{d+1}}\int_{[-\beta a_n,\beta a_n]^d} ue^{-i\frac u{a_n}\cdot N}\lambda_{u/a_n}^{n-2k}\mathcal E_k(g,h,u/a_n)\, du\,
 +O(a_n^{-d-1}/\log n)\, ,
\end{align}
since $(k/a_n)^{\gamma_2}\ll (\log n)^{-1}$ and with
\[
\mathcal E_k(g,h,t):=  \mathbb E_{\mu_{\Delta}}[g\Pi'_0P_t^kP^kh]+i\mathbb E_{\mu_{\Delta}}[\hat\kappa_k  g]\mathbb E_{\mu_\Delta}[h]+i\mathbb E_{\mu_\Delta}[g]  \mathbb E_{\mu_\Delta}[\hat\kappa_kP^kh]\, .
\]
Therefore
\begin{align*}
&\mathbb E_{\mu_{\Delta}}(h\mathbf 1_{\{\hat S_n=N\}}g\circ f^n)=
\frac 1{(2\pi a_n)^d}\int_{[-\beta a_n,\beta a_n]^d}e^{-i\frac u{a_n}\cdot N}\lambda_{u/a_n}^{n-2k}\, du\, \mathbb E_{\mu_\Delta}[g]\mathbb E_{\mu_\Delta}[h]\\
&+\frac {\mathcal E_k(g,h,0) }{(2\pi )^d(a_n)^{d+1}}\cdot\int_{[-\beta a_n,\beta a_n]^d} ue^{-i\frac u{a_n}\cdot N}\lambda_{u/a_n}^{n-2k}\, du\, +O(a_n^{-d-1}/\log n)\\
&+O\left(\frac {1 }{(2\pi )^d(a_n)^{d+1}}\int_{[-\beta a_n,\beta a_n]^d} |u|e^{-\frac{c_0\min(|u|^{2-\epsilon},|u|^{2+\epsilon})}2} |\mathbb E_{\mu_{\Delta}}[g\Pi'_0(P_{u/a_n}^kP^kh-P^{2k}h)]|\, du\right)\, ,
\end{align*}
due to \eqref{bornelambda}.
Due to \eqref{intlambda} for $\ell=0$ and $\ell=1$, 
%we conclude that 
$\mathbb E_{\mu_{\Delta}}(h\mathbf 1_{\{\hat S_n=N\}}g\circ f^n)$ is equal to
\begin{align*}
%&\mathbb E_{\mu_{\Delta}}(h\mathbf 1_{\{\hat S_n=N\}}g\circ f^n)\\
&
\frac 1{(2\pi a_n)^d}
\left(I_{0}\left(\frac N{a_n}\right)-\frac 12
I_{2}\left(\frac N{a_n}\right)\frac{\log\log n}{\log n}+O\left(\frac 1{\log n}\right)\right)
\mathbb E_{\mu_\Delta}[g]\mathbb E_{\mu_\Delta}[h]\\
&+\frac {\mathcal E_k(g,h,0)}{a_n^{d+1}} \mathbb E_{\mu_{\Delta}}\left(I_{1}\left(\frac N{a_n}\right)-\frac 12
I_{3}\left(\frac N{a_n}\right)\frac{\log\log n}{\log n}\right)+
O\left(a_n^{-d-1}/\log n\right)\\
&+O\left(\frac {1 }{(2\pi )^d(a_n)^{d+1}}\int_{[-\beta a_n,\beta a_n]^d} |u|e^{-\frac{c_0\min(|u|^{2-\epsilon},|u|^{2+\epsilon})}2} |\mathbb E_{\mu_{\Delta}}[g\Pi'_0(P_{u/a_n}^kP^kh-P^{2k}h)]|\, du\right)\, ,
\end{align*}
with
$I_{2k+1}(x):=\frac 1{(2\pi)^d}\int_{\mathbb R^d}u( Au\cdot u)^k e^{-iu\cdot x}e^{-\frac 12 Au\cdot u}\, du=
i(-1)^{k}
\frac{(\Delta_{1,k}\Phi)(A^{-\frac 12}x)}{\sqrt{\det A}}$, with $\Delta_{1,k}\Phi(x):=A^{-\frac 12}\left(\sum_{j=1}^d\frac{\partial^{2k+1}}{\partial x_i\partial^{2k} x_j}\Phi(x)\right)_{i=1,...,d}$, which leads to \eqref{I1I3},
%\[
%I_{3}(X):=\frac{2A^{-1}X-A^{-1}X( A^{-1} X\cdot X-d)}{\sqrt{(2\pi)^{d}\det(A)}}\, e^{-\frac 1{2}  A^{-1}X\cdot X}\, ,
%\]
ending the proof of the Proposition.
%Under \eqref{Pidiff}, the above error term $O(a_n^{-d-\gamma})$ becomes
%\begin{eqnarray*}
%&\ &\frac 1{(2\pi a_n)^d}\int_{[-\beta a_n,\beta a_n]^d}\lambda_{u/a_n}^n   \left(\frac{u}{a_n}\cdot \mathbb E_{\mu_{\Delta}}\left[g\Pi'_0(f)\right]+O\left(|u/a_n|^{1+\gamma'}\right)\right)\, du\\
%&=&\frac 1{(2\pi)^d a_n^{d+1}}\int_{[-\beta a_n,\beta a_n]^d}\lambda_{u/a_n}^n   u\cdot \mathbb E_{\mu_{\Delta}}\left[g\Pi'_0(f)\right]\, du+O\left(a_n^{-d-1-\gamma'}\right)\, ,
%\end{eqnarray*}
%and as for \eqref{intlambda} above, we also have
%\begin{eqnarray*}
%\int_{[-\beta a_n,\beta a_n]^d}\lambda_{u/a_n}^n   u\, du
%&=&\int_{\mathbb R^d}e^{-\frac 1{2}\langle A u,u\rangle\, \left(1+\frac{\log(\log n/|u|^2)}{\log n}\right)} u\, du
%+O\left(\frac 1{n(\log n)^2}\right)\\
%&=&O\left(\frac 1{n(\log n)^2}\right).
%\end{eqnarray*}
\end{proof}
\subsection{Decorrelation expansion for $\mathbb Z^d$-extensions}

We are now interested in decorrelation expansion for  $\mathbb Z^d$-extensions satisfying the set up of Proposition~\ref{LLT1}.
In this subsection, we consider the $\mathbb Z^d$-extension  $(\Omega,\nu,S)$ of $(\Delta,f,\mu_\Delta)$
with step function $\hat\kappa:\Delta\rightarrow\mathbb Z^d$, where
 $\Omega=\Delta\times\mathbb Z^d$, the transformation $S$ is defined by $S(x,L)=(fx,L+\hat\kappa(x))$ and preserves
the infinite measure $\nu=\mu_\Delta\otimes \mathfrak m_d$ with  $\mathfrak m_d$  being the counting measure on $\mathbb Z^d$.
\begin{cor}\label{CORO}
Assume~\eqref{quasicompact}--\eqref{lambda}.
Suppose that \eqref{Pi} holds with $\gamma>0$.
Let $\phi,\psi:\Delta\times\mathbb Z^d\rightarrow\mathbb R$ be two observables
such that
$
\sum_{N\in\mathbb Z^d}\left[\Vert\phi(\cdot,N)\Vert_{\mathcal B}+\Vert\psi(\cdot,N)\Vert_{L^{q_2}(\mu_\Delta)}\right]<\infty
$
and
$
\sum_{N\in\mathbb Z^d}|N|^\gamma\left[\left\vert\int_\Delta \phi(\cdot,N)\, d\mu\right|+\left\vert\int_\Delta \psi(\cdot,N)\, d\mu\right|\right]<\infty$. 
Then 
\[
\int_\Omega \phi. \psi\circ S^n\, d\nu=\frac {I_0(0)}{ (n\log n)^{\frac d2}}\left(1
-
\frac d2\frac{\log\log n}{\log n}\right)\int_\Omega\phi\, d\mu\, \int_\Omega\psi\, d\mu+O\left(\frac 1{(n\log n)^{\frac d2}\log n}\right)\, .
\] 

Suppose that \eqref{Pidiff}(instead of \eqref{Pi}) holds with
$\gamma'>0$
 and assume that there exists $\gamma_0>0$ such that
$
\sum_{N\in\mathbb Z^d}(1+|N|^{\gamma_0})\left(\Vert\psi(\cdot,N)\Vert_{L^{q_0}(\mu_\Delta)}+\Vert \phi(\cdot,N)\Vert_{\mathcal B}\right)<\infty\, .
$
Then
\[
\int_\Omega \phi. \psi\circ S^n\, d\nu=\frac {I(0)}{(n\log n)^{\frac d2}}\left(1-\frac{d\log\log n+O(1)}{2\log n}\right)\int_\Omega\phi\, d\nu\, \int_\Omega\psi\, d\nu+O\left(\frac 1{(n\log n)^{\frac {d+1}2}\log n}\right)\, .
\]

\end{cor}
\begin{proof}
Applying Proposition~\ref{LLT1} to the couples $(\phi(\cdot,N_1),\psi(\cdot,N_2))$ with $k=0$ leads to
\begin{align*}
&\int_\Omega \phi.\psi\circ S^n\, d\nu=\sum_{N_1,N_2\in\mathbb Z^d}\mathbb E_{\mu_\Delta}\left[\phi(\cdot,N_1)
\mathbf 1_{\{\hat\kappa_n=N_2-N_1\}}\, \psi(f^n(\cdot), {N_2})\right]\\
&=\sum_{N_1,N_2\in\mathbb Z^d}\frac 1{a_n^d}\mathbb E_{\mu_{\Delta}}[\phi(\cdot,N_1)]\mathbb E_{\mu_{\Delta}}[\phi(\cdot,N_2)]
\left(I_0\left(\frac {N_2-N_1}{a_n}\right)-\frac {I_2\left(\frac {N_2-N_1}{a_n}\right)}2 \frac{\log\log n}{\log n}\right)\\
&\ \ \ \ \ \ \ \ \ \ +O\left(\frac{1}{a_n^d\log n}\right)\, .
\end{align*}
Thus $\int_\Omega \phi.\psi\circ S^n\, d\nu
=\frac 1{a_n^d}\sum_{N_1,N_2\in\mathbb Z^d}\int_{M_{N_1}}\psi\, d\mu\, \int_{M_{N_2}}\phi\, d\mu I_{0}(0)\left(1-\frac d2\frac{\log\log n}{\log n}\right)+O\left(\frac{1}{a_n^d\log n}\right)$.
We used the fact that for every $\gamma\in(0,2]$, there exists $C_\gamma$
such that, for every $X\in\mathbb R^2$, $|I_0(X)-I_0(0)|+|I_2(X)-I_2(0)|\le C_\gamma  |X|^\gamma$, that $|N_2-N_1|\le|N_2|+|N_1|$ combined with our assumption on $\phi$ and $\psi$.
Setting $\tilde I_k(n,x):=I_k(x)-\frac {I_{k+2}(x)}2\frac{\log\log n}{\log n}$ and applying 
the second point of Proposition~\ref{LLT1}, we obtain that $\int_\Omega \phi.\psi\circ S^n\, d\nu$ is equal to
\begin{align*}%\label{eq:LLT2}
%&\\
&\frac 1{a_n^d}\sum_{N_1,N_2\in\mathbb Z^d}\mathbb E_{\mu_{\Delta}}[\phi(\cdot,N_1)]\mathbb E_{\mu_{\Delta}}[\psi(\cdot,N_2)]
\left(\tilde I_0\left(n,\frac {N_2-N_1}{a_n}\right)+O\left(\frac 1{\log n}\right)\right)\\
&+\frac 1{(a_n)^{d+1}}\sum_{N_1,N_2\in\mathbb Z^d}\mathbb E_{\mu_{\Delta}}[\psi(\cdot,N_2)\Pi'_0\phi(\cdot,N_1)] \tilde I_1\left(n,\frac {N_2-N_1}{a_n}\right)
%\left(I_{1}\left(\frac {N_2-N_1}{a_n}\right)-\frac {I_3\left(\frac {N_2-N_1}{a_n}\right)}2 \frac{\log\log n}{\log n}\right)\\
+O\left(\frac{1}{a_n^{d+1}\log n}\right)\\
&=\frac {I_0(0)}{a_n^d}\sum_{N_1,N_2\in\mathbb Z^d}\!\!\!\! \mathbb E_{\mu_{\Delta}}[\phi(\cdot,N_1)]\mathbb E_{\mu_{\Delta}}[\psi(\cdot,N_2)]
\left(1
-\frac d2\frac{\log\log n}{\log n}+O\left(\frac 1{\log n}\right)\right)+O\left(\frac{1}{a_n^{d+1}\log n}\right)\, ,
\end{align*}

where we used the same argument as before with $\gamma\in(1,2]$ combined 
with the fact that $I_1$ and $I_3$ are uniformly $\gamma_0$-H\"older with $\gamma_0\in(0,1)$ and that $I_1(0)=I_3(0)=0$.
\end{proof}
Observe that when $\phi$ or $\psi$ has null integral, then Corollary \ref{CORO}
only provides an upper bound (given by the term in $O(\cdot)$). 
Nevertheless, the method we used to establish Proposition~\ref{LLT1} and Corollary \ref{CORO} enables the establishment of explicit decorrelation rates
 % in $a_n^{-d-k}$ or $a_n^{-d-k}\log n$ 
for some specific but natural
null integral observables of the $\mathbb Z^d$-extension, including coboundaries.
Before stating these decorrelation results, let us introduce the following notations.
Set
\begin{align}
I_{\ell_1,\ell_2,N}(x)&:= \frac1{(2\pi)^d}\int_{\mathbb R^d} (-iu\cdot N)^{\ell_1}(-A u\cdot u)^{\ell_2}e^{-iu\cdot x}e^{-\frac 12Au\cdot u}\, du\nonumber\\
&= \frac 1{(2\pi)^d\sqrt{\det A}}\int_{\mathbb R^d} (-iu\cdot A^{-1/2} N)^{\ell_1}(- u\cdot u)^{\ell_2}e^{-iu\cdot   A^{-1/2}x}e^{-\frac 12u\cdot u}\, du\nonumber\\
&=%i^{\ell_1+2\ell_2}
\frac{\Delta_{\ell_1,\ell_2}\Phi(A^{-\frac 12}x,N)}{\sqrt{\det A}}\, ,
\label{formulaI}
\end{align}
with $\Delta_{\ell_1,\ell_2}\Phi(x,N):=(A^{-\frac 12}N\cdot\nabla)^{\ell_1}\Delta^{\ell_2}\Phi (A^{-1/2}x)$ 
%the standard $d$-dimensional Gaussian density function 
with $\Phi(x)=\frac{e^{-\frac 12 x\cdot x}}{(2\pi)^{d/2}}$ and where $\nabla$ and $\Delta$ are the usual differential
gradient and Laplacian operators. Recall that, for $\psi:\mathbb R^d\rightarrow \mathbb C$; $\nabla\psi=\left(\frac{\partial}{\partial x_i}\psi\right)_{i=1,...,d}$ and that $\Delta\psi=(\nabla\cdot\nabla)\psi=\sum_{i=1}^d\frac{\partial^2}{\partial^2 x_i}\psi$.

%\left(\sum_{j_1,...,j_{\ell_2}=1}^d\left(\frac{\partial^{\ell_1+2\ell_2}}{\partial x_i\partial^2 x_{j_1}...\partial^2 x_{j_{\ell_2}}}\Phi\right)(A^{-1/2}x)\right)_{i=1,...,d}$ with $\Phi(x):=e^{-\frac 12 A^{-1}x\cdot x}$.

\begin{prop}\label{LLT00}
Assume assumptions of Proposition~\ref{LLT1} with 
%either \eqref{Pidiff} (or 
\eqref{Pi}.
%Then there exist $J_{\ell_1,\ell_2,n,N}$ such that for every $\delta\in(0,1]$,
%\begin{equation}
%\label{relationIJ}
%J_{\ell_1,\ell_2,n,N}=I_{\ell_1,\ell_2,N}(0)-I_{\ell_1,\ell_2+1,N}(0) \frac{\log\log n}{2\log n}+O\left(\frac{N^{\ell_1}}{\log n}
%%+(N/a_n)^\delta\right)
%\right)
%\end{equation}
%and such that the following assertions hold true.\\
\begin{itemize}
\item[(A)] (Coboundary) If $\phi(x,\ell)=h_\ell(x)
%\sum_{N\in\mathbb Z^d}h_N(x)1_{\ell=N}
$
and $\psi(x,\ell)=g_\ell(x)
%\sum_{N\in\mathbb Z^d}g_N(x)1_{\ell=N}
$, with
%\begin{equation}\label{summass}
$\sum_{N\in\mathbb Z^d}\left(\Vert g_N\Vert_{L^{q_2}(\mu_\Delta)}+\Vert h_N\Vert_{\mathcal B}\right)<\infty$,
%\end{equation}
and
% there exists 
%\begin{equation}\label{summass-bis}
%\exists\delta\in(0,1)$ such that 
$\sum_{N\in\mathbb Z^d}|N|^\delta\left(\Vert g_N\Vert_{L^{1}(\mu_\Delta)}+\Vert h_N\Vert_{L^{1}(\mu_\Delta)}\right)<\infty$ for some $\delta\in(0,1)$,
%\end{equation}
then, for all integer $m\ge 1$, $\int_\Omega \phi\circ(id-S)^m.\psi\circ S^n\, d\nu$ is equal to
\begin{align*}
%&\\
%&=
&\frac {(\log n)^m+m\log\log n(\log n)^{m-1}}{2^ma_n^{d+2m}}
\left(
I_{0,m,0}(0)+
I_{0,m+1,0}(0) \frac{\log\log n}{2\log n}\right)\\
&\times \int_\Omega g\, d\nu\, \int_\Omega h\, d\nu+O(a_n^{-d-2m}(\log n)^{m-1})\, .
\end{align*}

\item [(B)] Let $N\in\mathbb Z^d$ be fixed.
If $\phi(x,\ell)=h(x)(1_{\ell=N}+1_{\ell=-N}-2\times 1_{\ell=0})$
and  $\psi(x,\ell)=g(x)$
with $g\in L^{q_2}(\mu_\Delta)$ and $h\in\cB$,
then
\begin{align*}
\int_\Omega \phi.\psi\circ S^n\, d\nu&=-
\frac {A^{-1}N\cdot N}{a_n^{d+2}\sqrt{(2\pi)^d\det(A)}}\mathbb E_{\mu_\Delta}[h]\, \mathbb E_{\mu_\Delta}[g]\left(1-
 \frac{(d+2)\log\log n}{2\log n}+O\left((\log n)^{-1}\right)\right)\\
&+O(a_n^{-d-2-\min(\delta,\gamma)})
\, .\end{align*}
\end{itemize}
\end{prop}
Let us observe that item (B) of Proposition~\ref{LLT00} (with $h=g=1_\Delta$) implies in particular that
$\mu_\Delta(\hat\kappa_n=1)+\mu_\Delta(\hat\kappa_n=-1)-2\mu_\Delta(\hat\kappa_n=0)\approx a_n^{-d-2}$, 
as $n\rightarrow +\infty$.

%\begin{rem}
%Observe that the condition $\lambda_t=\lambda_{-t}$ holds true if $\hat\kappa_n$ and $-\hat\kappa_n$ have the same distribution since, for every $t\in[-\beta,\beta]^d$, \[\lambda_t^n\mathbb E_{\mu_\Delta}[P_t^n1]\sim\mathbb E_{\mu_\Delta}[P_t^n1]=\mathbb E_{\mu_\Delta}[P_t^n1]\sim \mathbb E_{\mu_\Delta}[P_{-t}^n1]\lambda_{-t}^{n}\, .\]
%\end{rem}
\begin{proof}[Proof of Proposition~\ref{LLT00}]
Set 
$
J_{\ell_1,\ell_2,n,N,N_0}:=\frac 1{(2\pi)^d}\int_{-[\beta a_n,\beta a_n]^d}e^{-iu\cdot N/a_n}(-iu\cdot N_0)^{\ell_1}(-Au\cdot u)^{\ell_2}\lambda_{u/a_n}^{n-m}\, du\, ,
$
with $m=0$ in item (B). 
Observe first that, due to \eqref{lambdan-k}
\begin{equation}
\label{diffJ1}
\forall\delta_0\in(0,1],\quad 
J_{\ell_1,\ell_2,n,N,N_0}:=J_{\ell_1,\ell_2,n,0,N_0} + O\left(\frac{N_0^{\ell_1}N^{\delta_0}}{a_n^{\delta_0}}\right)
%\frac 1{(2\pi)^d}\int_{-[\beta a_n,\beta a_n]^d}e^{-iu\cdot N/a_n}(-iu\cdot N)^{\ell_1}(-Au\cdot u)^{\ell_2}\lambda_{u/a_n}^{n-m}\, du\, ,
\end{equation}
using  $|e^{ix}-1|\le 2 |x|^{\delta_0}$,
and, moreover, due to \eqref{intlambda}
\begin{align}\label{relationIJ}
J_{\ell_1,\ell_2,n,0,N_0}&=I_{\ell_1,\ell_2,N_0}(0)
+I_{\ell_1,\ell_2+1,N_0}(0) \frac{\log\log n}{2\log n}+O\left(\frac {N_0^{\ell_1}}{\log n}\right)\, ,
%\label{relationIJ}&=I'_{\ell_1,\ell_2}(n)+O\left(\frac 1{\log n}+(N/a_n)^\delta\right)\, ,
\end{align}
%and thus $J_{\ell_1,\ell_2,n,N}:=J_{\ell_1,\ell_2,n,0,N}$ satisfies \eqref{relationIJ}
%Moreover, 
%\begin{equation}\label{DLI}
%I_{\ell_1,\ell_2,N}(N/a_n)=I_{\ell_1,\ell_2,N}(0)+O(|N|^{\delta}/a_n^\delta)=I_{\ell_1,\ell_2,N}(0)-i\frac{I_{\ell_1+1,\ell_2,N}}{a_n}+O(|N|^{1+\delta}/a_n^{1+\delta})\, ,
%\end{equation}
for every $\delta\in [0,1]$.
% since $I_{\ell_1,\ell_2,N}$ and $I_{\ell_1,\ell_2+1,N}$ are Lipschitz continuous.
Note that $I_{\ell_1,\ell_2,N}$ has same parity as $\ell_1$ and thus that
$I_{\ell_1,\ell_2,N}(0)=0$ if 
%and only if 
$\ell_1$ is an odd number.

\begin{itemize}
\item  Assume 
%$\phi(x,\ell)=\sum_{N\in\mathbb Z^d}h_N(x)1_{\ell=N}$ and $\psi(x,\ell)=\sum_{N\in\mathbb Z^d}g_N(x)1_{\ell=N}$ satisfying
the assumptions of (A).
% \eqref{summass} and \eqref{summass-bis}. 
Then
\begin{align*}
&\int_\Omega \phi\circ(id-S)^m.\psi\circ S^n\, d\nu
=\sum_{r=0}^m\frac{m!(-1)^r}{r!\, (m-r)!}
\int_\Omega \phi\circ S^r.\psi\circ S^n\, d\nu\\
&=\sum_{N_1,N_2\in\mathbb Z^d}\sum_{r=0}^m\frac{m!(-1)^r}{r!\, (m-r)!}
\mathbb E_{\mu_\Delta}\left[h_{N_1}1_{\{\hat\kappa_{n-r}=N_2-N_1\}}g_{N_2}\circ f^{n-r}\right]
%\int_\Omega \phi.\psi\circ S^{n-r}\, d\nu\\
%&=\sum_{r=0}^m\frac{m!(-1)^r}{r!\, (m-r)!}\sum_{N_1,N_2\in\mathbb Z^d}\frac 1{(2\pi)^d}\int_{[-\pi,\pi]^d}e^{-it\cdot(N_2-N_1)}\mathbb E_{\mu_{\Delta}}[g_{N_2}.P_t^{n-r}h_{N_1}]\, dt
\, .
\end{align*}
Thus, due to \eqref{FormulaInt} with $k=0$ and since $\sum_{r=0}^m\frac{m!(-1)^r\lambda_t^{n-r}}{r!(m-r)!}=(-1)^m\lambda_t^{n-m}(1-\lambda_t)^m$, $\int_\Omega \phi\circ(id-S)^m.\psi\circ S^n\, d\nu$ is equal to
%\begin{align*}
%&\sum_{r=0}^m\frac{m!(-1)^m(-1)^{m-r}}{k!\, (m-k)!}\sum_{N_1,N_2\in\mathbb Z^d}\frac 1{(2\pi)^d}\int_{[-\pi,\pi]^d}e^{-it\cdot(N_2-N_1)}\mathbb E_{\mu_{\Delta}}[g_{N_2}.P_t^{n-m+(m-r)}h_{N_1}]\, dt\\
%&=(-1)^m\sum_{N_1,N_2\in\mathbb Z^d}\frac 1{(2\pi)^d}\int_{[-\pi,\pi]^d}e^{-it\cdot(N_2-N_1)}\mathbb E_{\mu_{\Delta}}[g_{N_2}.P_t^{n-m}(I-P_t)^mh_{N_1}]\, dt\, ,
%\end{align*}
%and so, due to \eqref{borneexpo}, $\int_\Omega \phi\circ(id-S)^m.\psi\circ S^n\, d\nu$ is equal to
\begin{align*}
%&\int_\Omega \phi\circ(id-S)^m.\psi\circ S^n\, d\nu\\
&\frac {(-1)^m}{(2\pi)^da_n^d}\sum_{N_1,N_2\in\mathbb Z^d}
\int_{[-\beta a_n,\beta a_n]^d}
e^{-iu\cdot(N_2-N_1)/a_n}\lambda_{u/a_n}^{n-m}(1-\lambda_{u/a_n})^m\mathbb E_{\mu_{\Delta}}[g_{N_2}\Pi_{u/a_n}h_{N_1}]\, du\\
&\ \ \ \ \ \ \ +O\left(\sum_{N_1,N_2\in\mathbb Z^d}\theta^n
\Vert g_{N_2}\Vert_{L^{q_1}(\mu_\Delta)}\Vert h_{N_1}\Vert_{\mathcal B}\right)\\
\end{align*}
that can be rewritten
\begin{align*}
&\frac {(-1)^m}{(2\pi)^da_n^{d+2m}}\sum_{N_1,N_2\in\mathbb Z^d}
\int_{[-\beta a_n,\beta a_n]^d}e^{-iu\cdot(N_2-N_1)/a_n}\lambda_{u/a_n}^{n-m}(Au.u\log(|a_n/u|))^m\mathbb E_{\mu_{\Delta}}[g\Pi_{u/a_n}h]\, du\\
&+O\left(\sum_{N_1,N_2\in\mathbb Z^d}\!\!\!\!\!\! a_n^{-d-2m}
\Vert g_{N_2}\Vert_{L^{q_1}(\mu_\Delta)}\Vert h_{N_1}\Vert_{\mathcal B}
\left(1+\int_{[-\beta a_n,\beta a_n]^d}\!\!\!\!\!\!\! \!\!\!\!\!\!\!\! \!\!\!\!\!\!\!\! e^{-c_0\min(|u|^{2-\epsilon},|u|^{2+\epsilon})}|u|^{2m}(\log(|a_n/u|))^{m-1} du\right)\right)
\end{align*}
where we used \eqref{lambda} and \eqref{bornelambda}. Now writing
$\log(|a_n/u|)=\log a_n+O(\log|u|)$ and using \eqref{Pi}, we obtain that
$\int_\Omega \phi\circ(id-S)^m.\psi\circ S^n\, d\nu$ is equal to
\begin{align}
%&\!\!\!\! \sum_{N_1,N_2\in\mathbb Z^d}\!\!\!\!\frac {(\log a_n)^m}{a_n^{d+2m}} J_{0,m,n,N_2-N_1,0}\mathbb E_{\mu_\Delta}[g_{N_2}\Pi_0h_{N_1}]+O(a_n^{-d-2m}(\log a_n)^{m-1} \sum_{N_1,N_2}\Vert g_{N_2}\Vert_{L^{q_2}(\mu_\Delta)}\Vert h_{N_1}\Vert_{\mathcal B})\\
&%=\!\!\!\! 
\label{cobfinal}
\sum_{N_1,N_2\in\mathbb Z^d}\!\!\!\!
\frac {(\log a_n)^m}{a_n^{d+2m}} J_{0,m,n,0,0}\mathbb E_{\mu_\Delta}[g_{N_2}]\mathbb E_{\mu_\Delta}[h_{N_1}]+O(a_n^{-d-2m}(\log a_n)^{m-1}
\sum_{N_1,N_2}\Vert g_{N_2}\Vert_{L^{q_2}(\mu_\Delta)}\Vert h_{N_1}\Vert_{\mathcal B}),
\end{align}
due to \eqref{diffJ1} with $\delta_0=\delta$, to $\Pi_0=\mathbb E_{\mu_\Delta}[\cdot]1_\Delta$ and to our summability assumptions. 
%\eqref{summass-bis}.
%Thus
%\begin{equation}\label{cobfinal}\int_\Omega \phi\circ(id-S)^m.\psi\circ S^n\, d\nu=\sum_{N_1,N_2\in\mathbb Z^d}\frac {(\log a_n)^m}{a_n^{d+2m}} J_{0,m,n,0,0}\mathbb E_{\mu_\Delta}[g_{N_2}]\mathbb E_{\mu_\Delta}[h_{N_1}]+O\left(a_n^{-d-2m}(\log a_n)^{m-1}\right)\, .
%\end{equation}
This ends the proof of (A) since $\int_\Omega\phi\, d\nu=\sum_{N\in\mathbb Z^d}\mathbb E_{\mu_\Delta}[h_{N}]$ and $\int_\Omega\psi\, d\nu=\sum_{N\in\mathbb Z^d}\mathbb E_{\mu_\Delta}[g_{N}]$.
\item Assume the assumptions of (B).
% $\phi(x,\ell)=h(x)(1_{\ell=N}+1_{\ell=-N}-2\times 1_{\ell=0})$  and  $\psi(x,\ell)=g(x)$. 
Then, due to \eqref{FormulaInt} with $k=0$, we obtain
\begin{align*}
&\int_\Omega \phi.\psi\circ S^n\, d\nu
=\mathbb E_{\mu_{\Delta}}[h( 1_{\{\hat \kappa_n\circ f^k=N\}}+ 1_{\{\hat \kappa_n\circ f^k=N\}}-2\times 1_{\{\hat \kappa_n\circ f^k=0\}})g\circ f^n]\\
%&=\frac 1{(2\pi)^da_n^d}\int_{[-\pi,\pi]^d}(e^{-it\cdot N}+e^{it\cdot N}-2)\mathbb E_{\mu_{\Delta}}[gP_t^nh]\, dt\\
&=
\frac 1{(2\pi)^da_n^d}\int_{[-\beta a_n,\beta a_n]^d}\!\!\!\!\! (e^{-iu\cdot N/a_n}+e^{iu\cdot N/a_n}-2)\lambda_{u/a_n}^n\mathbb E_{\mu_{\Delta}}[g\Pi_{u/a_n}h]\, dt+O(\theta^n
\Vert g\Vert_{L^{q_1}(\mu_\Delta)}\Vert h\Vert_{\mathcal B})\, .
\end{align*}
Therefore $\int_\Omega \phi.\psi\circ S^n\, d\nu$ is equal to
\begin{align*}
&
\frac 1{(2\pi)^da_n^{d+2}}\int_{[-\beta a_n,\beta a_n]^d}(iu\cdot N)^2\lambda_{u/a_n}^n\mathbb E_{\mu_{\Delta}}[g\Pi_{u/a_n}h]\, dt+O(a_n^{-d-4}
\Vert g\Vert_{L^{q_1}(\mu_\Delta)}\Vert h\Vert_{\mathcal B})\\
&=
\frac {1}{a_n^{d+2}}\mathbb E_{\mu_{\Delta}}[g]\mathbb E_{\mu_\Delta}[h] J_{2,0,n,0,N}
+O(a_n^{-d-2-\gamma}
\Vert g\Vert_{L^{q_1}(\mu_\Delta)}\Vert h\Vert_{\mathcal B})\, ,
\end{align*}
due to \eqref{Pi}.
We end the proof of (B) by using \eqref{relationIJ} combined with 
\begin{equation}\label{I20N0}
I_{2,0,N}(0)=-\frac {A^{-1}N\cdot N}{\sqrt{(2\pi)^d\det(A)}}
 ,\quad I_{2,1,N}(0)=\frac {(d+2)A^{-1}N\cdot N}{\sqrt{(2\pi)^d\det(A)}}
 \, .
\end{equation}

\end{itemize}

\end{proof}

\subsection{Decorrelation for the $\mathbb Z^d$-periodic billiard map}\label{resultsbill}
We consider the setting of Section~\ref{mainresults} with \eqref{H0}. Recall that the billiard map $T$ can be represented as the $\mathbb Z^d$-extension of $(\bar M,\bar T,\bar\mu)$ with step function 
$\kappa:\bar M\rightarrow \mathbb Z^d$ (the step function of the one-dimensional case corresponding to the first coordinate of the step function of the two dimensional case).
% Analogously the billiard map in the domain $Q/(\{0\}\times\mathbb Z)$ can be represented  as the $\mathbb Z$-extension of $(\bar M,\bar T,\bar\mu)$ with step function $\kappa_1:\bar M\rightarrow \mathbb Z$ the first coordinate of $\kappa$.
Therefore these two billiard maps can be represented as the $\mathbb Z^d$-extension of $(\bar M,\bar\mu,\bar T)$ by a step function $\kappa:\bar M\rightarrow\mathbb Z^d$ with $d\in\{1,2\}$ such that
there exists $\hat\kappa:\Delta\rightarrow\mathbb Z^d$ satisfying
$\hat\kappa\circ\mathfrak p_2=\kappa\circ\mathfrak p_1$.
We treat together these two models in the following.
We consider the Banach spaces $\mathcal B_0$ and $\mathcal B$ defined in Section
\ref{sec:Banach}.
As recalled  in Section~\ref{sec:Banach}, with these choices, \eqref{spgap-Sz} and \eqref{spgap-Sz-bis} hold.
%have been proved for $\mathcal B$ by Sz\'asz and Varj\'u in \cite{SV07} using Young's towers construction \cite{Young98} adapted by Chernov in \cite{Chernov99} thanks to estimate established in \cite{BCS91}. 
Since $\mathcal B_0$ is continuously embedded in $\mathcal B$, 
\eqref{quasicompact} and \eqref{borneexpo} holds true with the Banach space $\mathcal B_0$ and for any $p_1\in [1,+\infty)$.
Moreover \eqref{lambda} has been proved in Proposition~\ref{prop-lambda} and,
due to Proposition~\ref{prop-expproj} (the fact that the symmetric matrix is invertible follows from the assumption of total dimension of the horizon),
\eqref{Pidiff} (and so  holds true \eqref{Pi} with any $\gamma\in(0,1]$) for $\mathcal B_0$  for any $p_0\in(1,\frac 43)$.
%Thus  holds also true for any $p_0\in(1,\frac 43)$ and .
\begin{rem}
The second item of Proposition~\ref{LLT1} applies providing a mixing local limit theorem for $\mathbb E_{\mu}(\phi\mathbf 1_{\{ \kappa_n=N\}}\psi\circ \bar T^n)$ with $\phi$ and $\psi$ both constant on stable curves, with $\phi$ dynamically H\"older and $\psi\in L^{q_0}(\bar\mu)$. Corollary~\ref{CORO} applies analogously.
\end{rem}
Our goal is to prove Theorems \ref{THM0} and \ref{THM1} for general dynamically H\"older observables. To this end, we will use approximations by functions on $\Delta$ and Proposition \ref{LLT1} with $k=k(n)\rightarrow+\infty$. 
Recall that we have defined
$\xi_k^{k'}$, $\xi_k^\infty$ before Theorem~\ref{THM0}.
For any $\phi:M \rightarrow \mathbb C$
or $\phi:\bar M\rightarrow \mathbb C$
and $-\infty< k\le k'\le \infty$, we define the local variation
$\omega_k^{k'}(\phi, x):= \sup_{y\in\xi_{k}^{k'}(x)}|\phi( x)-\phi(y)|$, 
where $\xi_{k}^{k'}( x)$ is the element of $\xi_k^{k'}$
containing $x$.
We start
with a local limit theorem (generalizing Theorem~\ref{THM0}). Set $\kappa_n:=\sum_{k=0}^{n-1}\kappa\circ \bar T^k$.
\begin{theo}\label{LLT20}
\begin{itemize}
\item[(a)]
Let $(k(n))_n$ be a sequence of integers diverging to $+\infty$ such that $k(n)=O( a_n/\log n)$.
Let $N\in\mathbb Z^d$ and $\phi,\psi: \bar M\rightarrow \mathbb C$ be  two bounded measurable functions,
% such that $\Vert \phi\Vert_\infty+\Vert \psi\Vert_\infty<\infty$,
then
\begin{align}\nonumber%\label{TLLBILL0}
&\int_{\bar M} \phi 1_{\{\kappa_n=N\}}\psi\circ \bar T^n\, d\bar\mu= \frac 1{a_n^d}
   \left(I_0\left(\frac{N}{a_n}\right)-\frac 12 \, I_2\left(\frac{N}{a_n}\right)\frac{\log\log n}{\log n}%+O((\log n)^{-1})
\right)\int_{\bar M}\phi\, d\bar\mu  \int_{\bar M}\psi\, d\bar\mu\\
\label{error}&+  
O\left(a_n^{-d}I_0\left(\frac{N}{2a_n}\right)
\left(\Vert\phi\Vert_{\infty}\Vert\omega_{-k(n)}^\infty(\psi,\cdot)\Vert_{L^1(\bar\mu)})+\Vert\psi\Vert_{\infty}\Vert\omega_{-k(n)}^{k(n)}(\phi,\cdot)\Vert_{L^1(\bar\mu)})\right)
+\frac {\Vert\phi\Vert_\infty\Vert\psi\Vert_\infty}{a_n^d\log n}\right)\, ,
\end{align}
uniformly in $N\in\mathbb Z^d$. In particular, if $\lim_{k\rightarrow +\infty}  \Vert \omega_{-k}^k(\phi,\cdot)
  + \omega_{-k}^{\infty}(\psi,\cdot))\Vert_{L^1(\bar\mu)}=0$,
then the error term in \eqref{error} is in $o(a_n^{-d})$, and if
$\int_{\bar M}(\omega_{-k(n)}^{k(n)}(\phi,\cdot)
  + \omega_{-k(n)}^{\infty}(\psi,\cdot))\, d\bar\mu=O((\log n)^{-1})$,
then the error in \eqref{error} is in $O\left(\frac 1{a_n^d\log n}\right)$.\\
\item[(b)]
Assume that $k(n)=O(a_n^{\frac 12-u})$ for some $u\in(0,1)$
and that there exists $p_2>2$ such that, for $h\in\{\phi,\psi\}$,
%\begin{equation}\label{AAA002}
$\|\omega_{-k(n)}^{k(n)}(h,\cdot)\|_\infty
%+\|\omega_{-k(n)}^{k(n)}(\psi,\cdot)\|_\infty
=O((\log n)^{-1})$,
%\end{equation}
and $\|\omega_{-k(n)}^{k(n)}(h,\cdot)\|_{L^1(\bar\mu)}=O((a_n\log n)^{-1})$
%\begin{equation}\label{AAA02}
%\Vert\omega_{-k(n)}^{k(n)}(\phi,\cdot)
%  + \omega_{-k(n)}^{k(n)}(\psi,\cdot))\Vert_{L^1(\bar\mu)}=O((a_n\log n)^{-1})
%\end{equation}
and
%\begin{equation}\label{COND2}
$\Vert  \omega_{-k(n)}^{k(n)}(h,\cdot)\Vert_{L^{p_2}(\bar\mu) }=o(k(n)^{-1}/\log n)\quad\mbox{and}\quad\sum_{j\ge k(n)}\Vert\omega_{-j}^{j}(h,\cdot)\Vert
_{L^{p_2}(\bar\mu)}=O((\log n)^{-1})$,
%\end{equation}
%and
%\begin{equation}\label{COND2bis}
%\Vert  \omega_{-k(n)}^{k(n)}(\psi,\cdot)\Vert_{L^{p_2}(\bar\mu) }=o(k(n)^{-1}/\log n)\quad\mbox{and}\quad\sum_{j\ge k(n)}\Vert\omega_{-j}^{j}(\psi,\cdot)\Vert_{L^{p_2}(\bar\mu)}=O((\log n)^{-1})\, ,
%\end{equation}
then
the numeric series
 $A_+(\phi):=\sum_{j\ge 0}
\mathbb E_{\bar\mu}[\kappa\circ \bar T^{j}.\phi]$  and $A_-(\psi):=\sum_{j\le -1}
\mathbb E_{\bar\mu}[\kappa\circ \bar T^{j}.\psi]$ converge absolutely
and $\int_{\bar M} \phi 1_{\{\kappa_n=N\}}\psi\circ \bar T^n\, d\bar\mu$ is equal to
\begin{align}\label{LLTfinal0}
&
%\int_{\bar M} \phi 1_{\{\kappa_n=N\}}\psi\circ \bar T^n\, d\bar\mu =
\frac 1{a_n^d}
   \left(I_0\left(\frac{N}{a_n}\right)-\frac 12 I_2\left(\frac{N}{a_n}\right) \frac{\log\log n}{\log n} + O\left(\frac 1{\log n}\right)
     \right)\mathbb E_{\bar\mu}[\phi]\mathbb E_{\bar\mu}[\psi]\\
\nonumber&+i\frac {\mathbb E_{\bar\mu}[\psi]
A_+(\phi)
%\sum_{j\ge 0}\mathbb E_{\bar\mu}[\kappa\circ \bar T^{j}\phi]
+\mathbb E_{\bar\mu}[\phi]
A_-(\psi)
%\sum_{j\le -1}\mathbb E_{\bar\mu}[\kappa\circ \bar T^{j}\psi]
}{a_n^{d+1}}
%\left( \mathbb E_{\mu_{\Delta}}[g\Pi'_0h]+ i\mathbb E_{\mu_{\Delta}}[\hat\kappa_k  g]\mathbb E_{\mu_\Delta}[h]+i\mathbb E_{\mu_\Delta}[g]  \mathbb E_{\mu_\Delta}[\hat\kappa_kP^kh]\right)
\left(
I_{1}\left(\frac {N}{a_n}\right)
-\frac {I_3\left(\frac {N}{a_n}\right)}2 \frac{\log\log n}{\log n}\right)
%\\   &
+O\left(\frac{a_n^{-d-1}}{\log n}\right)\nonumber\, .
\end{align}
\end{itemize}
\end{theo}
\begin{proof}
For the first assertion, we use the first part of Proposition~\ref{LLT1}
with $p>2$ (close to 2), $p_1=\infty$ ($q_1=1$), $p_0\in(1,4/3)$ ($p_0<p$) and $q_0>4$ so that $\min(1,\frac p{q_1},\frac p{p_0})=1$. 
%(both close enough to 2) 
%$k(n)\ll a_n/(\log n)^{q_1}$
%Observe that
Moreover $k/a_n
%)^{\min(1/p_1,1/q_1)}=(k/a_n)^{1/q_1}
\ll (\log n)^{-1}$.
%Note that the condition$k(n)=O(a_n^{1-u})$
We assume  from now on, without loss of generality, that $\phi,\psi$ take their values in 
$\mathbb R$.
To simplify notations, we write $k=k(n)$.
We define $\phi^{(k)}$ and $\psi^{(k)}$ by
\begin{equation}\label{phikpsik}
\phi^{(k)}(\bar x):=\inf_{\xi_{-k}^k( x)} \phi_+-\sup_{\xi_{-k}^k( x)} \phi_-
\quad\mbox{and}\quad
\psi^{(k)}( \bar x):=\inf_{\xi_{-k}^\infty(x)} \psi_+ - \sup_{\xi_{-k}^\infty(x)} \psi_-
\, ,
\end{equation}
where $\phi_+=\max(\phi,0)$, $\psi_+=\max(\psi,0)$, $\phi_-=\max(-\phi,0)$, $\psi_-=\max(-\psi,0)$.
Observe that, for every $x\in\bar M$, 
$|\phi^{(k)}(x)|\le|\phi(x)|$ and $|\psi^{(k)}(x)|\le|\psi(x)|$,
\begin{equation}\label{0MA1}
\left|\phi^{(k)}( x)- \phi(x)\right|\le \omega_{-k}^k(\phi,x)\le 2\Vert \phi\Vert_\infty
\quad\mbox{and}\quad
\left|\psi^{(k)}( x)- \psi(x)\right|\le  \omega_{-k}^\infty(\psi,x)\le 2\Vert\psi\Vert_\infty.
\end{equation}
Since $\phi^{(k)}\circ \bar T^k$ and $\psi^{(k)}\circ \bar T^k$ are constant
on the stable curves, there exist $\tilde\phi^{(k)},\tilde\psi^{(k)}:\Delta\rightarrow\mathbb C$ such that $\tilde\phi^{(k)}\circ \mathfrak p_2=\phi^{(k)}\circ \bar T^k\circ\mathfrak p_1$,
$\tilde\psi^{(k)}\circ \mathfrak p_2=\psi^{(k)}\circ \bar T^k\circ\mathfrak p_1$.
Observe that $\Vert \tilde\psi^{(k)}\Vert_{\infty}\le
 \Vert \psi^{(k)}\Vert_\infty$.
Note that $\tilde\phi^{(k)}:\Delta\rightarrow\mathbb C$ 
is constant on balls of the form $\{y\in\Delta\, :\, s(x,y)\ge 2k\}$ for every $x\in\Delta$.
%belong to $\mathcal B_0$ with $\mathcal B_0$-norm bounded by $2\Vert\phi_N\Vert_\infty\beta^{-2k}$ and $2\Vert\psi_N\Vert_\infty\beta^{-2k}$ respectively.
%Therefore
%\begin{align*}
%  \int_M \phi^{(k,\pm)}.\psi^{(k,\pm)} \circ f^n \, d\mu&=       \sum_{N_1,N_2\in\mathbb Z^2}      \int_{\Delta} \tilde\phi_{N_1}^{(k,\pm)}\mathbf 1_{\{\hat\kappa_n\circ f^k=N_2-N_1\}}    \tilde\psi_{N_2}^{(k,\pm)}\circ f^n  \, d\mu_\Delta\\
%&=       \sum_{N_1,N_2\in\mathbb Z^2}   \int_{\Delta} \tilde\phi_{N_1}^{(k,\pm)}\frac 1{(2\pi)^d}\int_{[-\pi,\pi]^d}e^{it\cdot\hat\kappa_n}\circ f^ke^{-it\cdot(N_2-N_1)} \, dt   \tilde\psi_{N_2}^{(k,\pm)}\circ f^n  \, d\mu_\Delta\\
%&=      \sum_{N_1,N_2\in\mathbb Z^2} \frac 1{(2\pi)^d}\int_{[-\pi,\pi]^d}e^{-it\cdot(N_2-N_1)}    \int_{\Delta} [\tilde\phi_{N_1}^{(k,\pm)}e^{-it\cdot\hat\kappa_k}]  e^{it\cdot\hat\kappa_n}[e^{it\cdot\hat\kappa_k}\tilde\psi_{N_2}^{(k,\pm)}]\circ f^n  \, d\mu_\Delta\,  dt \\
%&=      \sum_{N_1,N_2\in\mathbb Z^2} \frac 1{(2\pi)^d}\int_{[-\pi,\pi]^d}e^{-it\cdot(N_2-N_1)}     \int_{\Delta} [\tilde\phi_{N_1}^{(k,\pm)}e^{-it\cdot\hat\kappa_k}]    e^{it\cdot\hat\kappa_n}[e^{it\cdot\hat\kappa_k}\tilde\psi_{N_2}^{(k,\pm)}]\circ f^n  \, d\mu_\Delta\,  dt 
%\end{align*}
This will be useful to show that
\begin{equation}\label{PukPuphi}
\sup_u\Vert P_u^kP^k\tilde\phi^{(k)}\Vert_{\mathcal B_0}\ll\Vert\phi^{(k)}\Vert_\infty\, .
\end{equation}
To this end, due to \eqref{formulaP},
for every $x_1,x_2\in\Delta$ such that $s_0(x_1,x_2)\ge 1$
\begin{align*}
&\left |P_u^kP^k\tilde\phi^{(k)}(x_1)-P_u^kP^k\tilde\phi^{(k)}(x_2)\right|\\
&\le \sum_{(y_1,y_2)\, :\, s(y_1,y_2)\ge 2k+1,\ f^{2k}(y_i)=x_i}
     \left| e^{-\alpha_{2k}(y_1)+iu\hat\kappa_k(y_1)}\tilde\phi^{(k)}(y_1)-
       e^{-\alpha_{2k}(y_2)+iu\hat\kappa_k(y_2)}\tilde\phi^{(k)}(y_2) \right|\\
&= \sum_{(y_1,y_2)\, :\, s(y_1,y_2)\ge 2k+1,\ f^{2k}(y_i)=x_i}
     \left| \left(e^{-\alpha_{2k}(y_1)}-
       e^{-\alpha_{2k}(y_2)}\right)e^{iu\hat\kappa_k(y_1)}\tilde\phi^{(k)}(y_1) \right|\, ,
\end{align*}
where we used the notation $\alpha_{m}:=\sum_{k=0}^{m-1}\alpha\circ f^k$.
We end the proof of \eqref{PukPuphi} by noticing that 
\begin{align*}
\left|e^{-\alpha_{2k}(y_1)}-e^{-\alpha_{2k}(y_2)}\right|&\le
(e^{-\alpha_{2k}(y_1)}+e^{-\alpha_{2k}(y_2)})\left|\alpha_{2k}(y_1)-\alpha_{2k}(y_2)\right|\\
&\le (e^{-\alpha_{2k}(y_1)}+e^{-\alpha_{2k}(y_2)})\sum_{m=0}^{2k-1} C_\alpha \beta^{s(y_1,y_2)+1-m}\\
&\le (e^{-\alpha_{2k}(y_1)}+e^{-\alpha_{2k}(y_2)})C_\alpha \beta^{s(x_1,x_2)}(1-\beta)^{-1}\, .
\end{align*}
Applying the first item of Proposition
 \ref{LLT1} to the Banach space $\mathcal B_0$ and to the couples
$(h,g)=(\tilde\phi^{(k)},\tilde\psi^{(k)})$ with $q_1>2$, we obtain that
\begin{align}
\nonumber&\int_{\bar M} \phi^{(k)}1_{\{\kappa_n=N\}}\psi^{(k)} \circ \bar T^n \, d\bar\mu\\
\label{AAAA1}&=\frac 1{a_n^d}
   \left(I_0\left(\frac{N}{a_n}\right)-I_2\left(\frac{N}{a_n}\right) \frac{\log\log n}{2\, \log n } \right)\int_{\bar M} \phi^{(k)}\, d\bar\mu\, \int_{\bar M} \psi^{(k)}\, d\bar\mu
+O\left(\frac{ \Vert \psi\Vert_{
\infty
%L^{q_1}(\bar\mu)
}\Vert \phi\Vert_\infty}{(a_n)^{d}\, \log n }\right)\, ,
%\label{AAAA2}&=\frac 1{a_n^d}  I_0\left(\frac{N}{a_n}\right)\int_{\bar M} \phi^{(k)}\, d\bar\mu\, \int_{\bar M} \psi^{(k)}\, d\bar\mu+O\left(\frac{\log\log n}{(a_n)^{d}\, \log n } \Vert \psi\Vert_{L^{q_1}(\bar\mu)}\Vert \phi\Vert_\infty\right)\, ,
\end{align}
where we used the fact that $|\psi^{(k)}|\le |\psi|$, $|\phi^{(k)}|\le |\phi|$ and 
$k/a_n\ll (\log n)^{-1}$.
% and the fact that $I_2$ is uniformly bounded.
Moreover
%It follows from the dominated convergence theorem that
\begin{equation}\label{diffint}
\left|\int_{\bar M} \phi\, d\bar\mu\, \int_{\bar M} \psi\, d\bar\mu-
\int_{\bar M} \phi^{(k)}\, d\bar\mu\, \int_{\bar M} \psi^{(k)}\, d\bar\mu\right|
\le \Vert \psi\Vert_{\infty}\Vert\omega_{-k}^k(\phi,\cdot)\Vert_{L^1(\bar\mu)}+\Vert \phi\Vert_{\infty}\Vert\omega_{-k}^{\infty}(\psi,\cdot)\Vert_{L^1(\bar\mu)}\,
\end{equation}
and, setting $\phi^{(k,+)}(x):=\sup_{\xi_{-k}^k( x)} |\phi|$ and
$\psi^{(k,+)}( x):=\sup_{\xi_{-k}^\infty(x)} |\psi|$,
\begin{align}
&\int_{\bar M}1_{\{\kappa_n=N\}}|\phi.\psi\circ \bar T^n-\phi^{(k)}.\psi^{(k)}\circ \bar T^n|\, d\bar\mu\nonumber\\
&\le
\int_{\bar M}1_{\{\kappa_n=N\}}\left(|\phi|.|\psi-\psi^{(k)}|\circ \bar T^n+|\phi-\phi^{(k)}|.|\psi^{(k)}|\circ\bar T^n\right)\, d\bar\mu\nonumber\\
&\le
\int_{\bar M}1_{\{\kappa_n=N\}}\left(\phi^{(k,+)}.\omega_{-k}^{+\infty}(\psi,\cdot)\circ \bar T^n+\omega_{-k}^{k}(\phi,\cdot).\psi^{(k,+)}\circ \bar T^n\right)\, d\mu\label{0majoint}
\end{align}
which is in
\begin{align*}
&O\left(\ I_0\left(\frac{N}{2a_n}\right)a_n^{-d}\left(|\bar\mu(\phi^{(k,+)})| \bar\mu(\omega_{-k}^{\infty}(\psi,\cdot))+\bar\mu(\omega_{-k}^{k}(\phi,\cdot))\, |\bar\mu(\psi^{(k,+)})|\right)
%\right)\nonumber\\
%&\ \ \ +O\left(
+\frac{\log\log n}{(a_n)^d\log n}\Vert\phi\Vert_\infty\Vert\psi\Vert_\infty\right) \nonumber
\end{align*}
applying \eqref{AAAA1} with $\phi,\psi$ replaced respectively by $(\phi^{(k,+)},\omega_{-k}^{\infty}(\psi))$ and by $(\omega_{-k}^k(\phi),\psi^{(k,+)})$,
and using $|I_0(x)+I_2(x)|\ll I_0(x/2)$.
%and  $(k/a_n)^{\min(1/p_1,1/q_1)}\ll (\log n)^{-1}$ 
This ends the proof of \eqref{error}.

Assume now the assumptions of (b)
%$k=k(n)=O(a_n^{\frac 12-u})$ for some $u\in(0,1)$, \eqref{AAA002}, and \eqref{COND2}
%and \eqref{COND2bis} 
and let us prove \eqref{LLTfinal0}.
We replace \eqref{phikpsik} by 
\begin{equation}\label{phikpsik2}
\phi^{(k)}(\bar x):=\mathbb E_{\bar\mu}[\phi|\xi_{-k}^k( x)]
\quad\mbox{and}\quad
\psi^{(k)}(\bar x):=\mathbb E_{\bar\mu}[\phi|\xi_{-k}^k( x)]
\, .
\end{equation}
Observe that
$\mathbb E_{\bar\mu}[\phi^{(k)}]=\mathbb E_{\bar\mu}[\phi]$, 
$\mathbb E_{\bar\mu}[\psi^{(k)}]=\mathbb E_{\bar\mu}[\psi]$ and
\[
\mathbb E_{\mu_\Delta}[\hat\kappa_k P^k \tilde\phi^{(k)}]=\mathbb E_{\bar\mu}[\kappa_k\circ \bar T^k \phi^{(k)}\circ \bar T^k]=\mathbb E_{\bar\mu}[\kappa_k \phi^{(k)}]
=\mathbb E_{\bar\mu}[\kappa_k \phi]\, ,
\]
due to \eqref{phikpsik2} since $\kappa_k$ is $\xi_{-k}^k$-measurable
and similarly,
%Observe moreover that
setting $\kappa_{-k}:=\sum_{m=-k}^{-1}\kappa\circ \bar T^{m}$,
\[
\mathbb E_{\mu_\Delta}[\hat\kappa_k \tilde\psi^{(k)}]=\mathbb E_{\bar\mu}[\kappa_k \psi^{(k)}\circ \bar T^k]=\mathbb E_{\bar\mu}[\kappa_k\circ \bar T^{-k} \psi^{(k)}]=\mathbb E_{\bar\mu}[\kappa_{-k} \psi^{(k)}]
=\mathbb E_{\bar\mu}[\kappa_{-k} \psi].
\]
To prove \eqref{LLTfinal0},
% the last point of Theorem~\ref{LLT20}, 
we apply the second item of Proposition \ref{LLT1} to the Banach space $\mathcal B_0$ and to the couples
$(h,g)=(\tilde\phi^{(k)},\tilde\psi^{(k)})$ with $p<2$ (close to 2), $q_0>4$ (large), $p_1=\infty$, $q_1=1$ so that 
%our assumption on $k(n)$ implies that
 the condition $k(n)\le C a_n^{\frac{\tilde\gamma}{1+\tilde\gamma}}/(\log n)^{1/(1+\tilde\gamma)}$  of Proposition~\ref{LLT1}
holds. We obtain that
\begin{equation}
\label{BBB1}
\int_{\bar M} \phi^{(k)}1_{\{\kappa_n=N\}}\psi^{(k)} \circ \bar T^n \, d\bar\mu\\
=J_1+J_2+J_3+O\left(\frac{\Vert  \psi\Vert_{L^{q_0}(\bar\mu)}\Vert \phi\|_\infty}{a_n^{d+1}\log n}\right)
\end{equation}
%with
\[
J_1:=\frac 1{a_n^d}
   \left(I_0\left(\frac{N}{a_n}\right)-\frac 12 I_2\left(\frac{N}{a_n}\right) \frac{\log\log n}{\log n} + O\left(\frac 1{\log n}\right)
     \right)
\mathbb E_{\bar\mu}[\phi]\mathbb E_{\bar\mu}[\psi]\, ,
%\int_{\bar M} \phi\, d\bar\mu\, \int_{\bar M} \psi\, d\bar\mu\, ,
\]
\begin{align*}
J_2&:=\frac {
C_k(\tilde\psi^{(k)},\tilde\phi^{(k)})}{(a_n)^{d+1}}\left(I_{1}\left(\frac {N}{a_n}\right)-\frac {I_3\left(\frac {N}{a_n}\right)}2 \frac{\log\log n}{\log n}\right)
+O(a_n^{-d-1}(\log n)^{-1})\, ,
\end{align*}
\[
J_3:=O\left(\frac {1 }{(2\pi )^d(a_n)^{d+1}}\int_{[-\beta a_n,\beta a_n]^d}\!\!\!\!\!\!\!\!\!\!\!\!\!\!\!\! |u|e^{-\frac{c_0\min(|u|^{2-\epsilon},|u|^{2+\epsilon})}2} |\mathbb E_{\mu_{\Delta}}[\tilde\psi^{(k)}\Pi'_0(P_{u/a_n}^kP^k\tilde\phi^{(k)}-P^{2k}\tilde\phi^{(k)})]|\, du\right)\, .
\] 
%Due to \eqref{diffint}% and \eqref{0majoint} combined with\eqref{AAA02},  replacing $(\phi^{(k)},\psi^{(k)})$ with $(\phi,\psi)$in $J_1$,  gives the error
%\[
%O\left(a_n^{-d}I_0\left(\frac{N}{2a_n}\right)\left(\Vert \psi\Vert_{\infty}\Vert\omega_{-k}^k\Vert_{L^1(\bar\mu)}+\Vert \phi\Vert_{\infty}\Vert\omega_{-k}^{\infty}\Vert_{L^1(\bar\mu)}\right)\right)=O(a_n^{-d-1}/\log n),
%\]
%which is negligible in the view of the final, desired,  error term. 
Note that $\phi=\phi_0+\bar\mu(\phi)$ and $\psi=\psi_0+\bar\mu(\psi)$, with $\phi_0=\phi-\bar\mu(\phi)$ and $\psi_0=\psi-\bar\mu(\psi)$.
We can resume the rest of the proof to estimating $J_2$ and $J_3$.
To study these two terms,
% $J_2$ and $J_3$,  
we shall exploit the expression $\Pi'_0$ obtained in Propositions~\ref{prop-expproj} and~\ref{prop-expproj-base}. 
%To deal with 
For $J_3$ we exploit that
% let us notice that, due to Propositions~\ref{prop-expproj} and~\ref{prop-expproj-base},
\begin{equation}
\label{eq-blabla2}
\Pi'_0
%P^{2k}
w=\mathbb E_{\bar\mu}[w]H+i\sum_{j\ge 0}
\mathbb E_{\mu_\Delta}[\hat\kappa\circ f^{j}w]1_\Delta\, ,
\end{equation}
with
$\displaystyle
H(x):=Q'_0(1_Y)\circ\pi_0(x)+i\sum_{m=0}^{\omega(x)-1}\hat\kappa\circ f^m\circ\pi_0(x)+\frac i{\mu_\Delta(Y)}\sum_{j\ge 0}\mathbb E_ {\mu_\Delta}[\hat\kappa 1_Y\circ f^{j+1}]$. 
%In particular
Writing $t$ for $u/a_n$ in the expression of $J_3$ and using the above formula for $\Pi'_0$ we note that
\begin{align}
\label{eq-blabla1}
%\nonumber O(|(a_n)^{d+1} J_2|)=
&\left|\mathbb E_{\mu_\Delta}\left[\tilde\psi^{(k)}
\Pi'_0(P_t^{k}P^k-P^{2k})\tilde\phi^{(k)}\right]\right|\\
\nonumber &=\left\vert\mathbb E_{\bar\mu}[(e^{it\kappa_k}-1)\phi^{(k)}]\mathbb E_{\mu_\Delta}\left[\tilde\psi^{(k)} H\right]+i\mathbb E_{\bar\mu}[\psi^{(k)}]\sum_{j\ge 0}
\nonumber \mathbb E_{
%\bar
\mu_\Delta}[\hat\kappa P^j(P_t^kP^k-P^{2k})
\tilde\phi^{(k)}
%\circ \bar T^{k+j}(e^{it\kappa_k}-1)\phi^{(k)}
]\right\vert\\
&\le |t| k\Vert\kappa\Vert_{L^1(\bar\mu)} \Vert \phi\Vert_\infty\Vert\psi\Vert_\infty \Vert H\Vert_{L^1(\bar\mu)}+\Vert\psi\Vert_\infty E_k
\, ,
\end{align}
%, for $\delta_0\in(0,1]$
\begin{align*}
E_k
%&:=\sum_{j\ge 0}\left|\mathbb E_{\bar\mu}[\kappa\circ \bar T^{k+j}(e^{it\kappa_k}-1)\phi^{(k)}]\right|\\ &\le
&=\sum_{j= 0}^{
%2k+
\log n-1}\left|\mathbb E_{\bar\mu}[\kappa\circ \bar T^{k+j}(e^{it\kappa_k}-1)\phi^{(k)}]\right|+\sum_{j\ge
% 2k+
\log n}
%\left|\mathbb E_{\bar\mu}[\kappa\circ \bar T^{k+j}(e^{it\kappa_k}-1)\phi^{(k)}]\right|
\Vert\hat\kappa\Vert_{L^{q_1}}\theta^j\Vert (P_t^kP^k-P^{2k})
\tilde\phi^{(k)}\Vert_{\mathcal B_0}\\
&
%\le
\ll \sum_{j= 0}^{
%2k+
\log n-1} \Vert \kappa\circ \bar T^{k+j}(t\kappa_k)^{p-1}
%^{\delta_0}
\Vert_{L^1(\bar\mu)}   \Vert\phi\Vert_\infty+\sum_{j\ge 
%2k+
\log n}
\theta^j\Vert\phi\Vert_\infty
%\left(\left|\hat\kappa P^{j-2k}P^{2k}\tilde\phi^{(k)}\right|+\left|\hat\kappa P^{j-2k}P_t^{k}P^k\tilde\phi^{(k)}\right|\right)
\, ,
\end{align*}
\begin{align*}
E_k&\ll %(2k+
\log n
%)
 \Vert \kappa\Vert_{L^{p}(\mu_\Delta)}\Vert t\kappa_k
%^{\delta_0}
\Vert^{p-1}_{L^{p}(\bar\mu)}   \Vert\phi\Vert_\infty+
%\sum_{j\ge 2k+\log n}\Vert\hat\kappa\Vert_{L^{p_0}(\mu_\Delta)}\theta^{j-2k}
\theta^{\log n}\Vert\phi\Vert_\infty\ll 
(|t k|^{p-1}\log n +\theta^{\log n})
\Vert\phi\Vert_\infty\, .
\end{align*}
%and so $E_k\ll   |t|a_n^{1-2u}+(\log n)^{-1} \Vert\phi\Vert_\infty \ll (\log n)^{-1}  (|t|a_n+1) \Vert\phi\Vert_\infty$, since $k(n)\ll a_n^{\frac 12-u}$. 
This  together with equation~\eqref{eq-blabla1} and $k(n)\ll a_n^{\frac 12-u}$ implies that $J_3=O(a_n^{-d-1}/\log n)$.
%Moreover 
 The rest of the proof is allocated to the study of  $J_2$.
%=O\left(a_n^{-d-1}/\log n\right)$. 
We will 
use
%exploit that
\[
\mathbb E_{\mu_\Delta}\left[\tilde\psi^{(k)}
\Pi'_0P^{2k}\tilde\phi^{(k)}\right]=\mathbb E_{\bar\mu}[\phi^{(k)}]\mathbb E_{\mu_\Delta}\left[\tilde\psi^{(k)} H\right]+i\mathbb E_{\bar\mu}[\psi^{(k)}]\sum_{j\ge 0}
\mathbb E_{\bar\mu}[\kappa\circ \bar T^{k+j}\phi^{(k)}]\, ,
\]
obtained from~\eqref{eq-blabla2} with $w=\tilde\phi^{(k)}$, after multiplication with $\tilde\psi^{(k)}$ and integration. Thus
%Using this in the expression of $J_2$ have 
%From which we deduce

\begin{align}\label{intermediate}
J_2&=\frac {\mathfrak C_{k(n)}(\phi,\psi)}{(a_n)^{d+1}}
%\left( \mathbb E_{\mu_{\Delta}}[g\Pi'_0h]+ i\mathbb E_{\mu_{\Delta}}[\hat\kappa_k  g]\mathbb E_{\mu_\Delta}[h]+i\mathbb E_{\mu_\Delta}[g]  \mathbb E_{\mu_\Delta}[\hat\kappa_kP^kh]\right)
\left(I_{1}\left(\frac N{a_n}\right)-\frac {I_3\left(\frac N{a_n}\right)}2 \frac{\log\log n}{\log n}\right)
%\\   &
+O\left(a_n^{-d-1}/\log n\right)\, ,
\end{align}
with $\mathfrak C_k(\phi,\psi):=C_k(\tilde\psi^{(k)},\tilde\phi^{(k)})$, that is
%\begin{align}\label{intermediate}
%&\int_{\bar M} \phi 1_{\{\kappa_n=N\}}\psi\circ \bar T^n\, d\bar\mu=\frac 1{a_n^d}
%   \left(I_0\left(\frac{N}{a_n}\right)-\frac 12 I_2\left(\frac{N}{a_n}\right) \frac{\log\log n}{\log n} + O\left(\frac 1{\log n}\right)
%     \right)\int_{\bar M} \phi\, d\mu\, \int_{\bar M} \psi\, d\bar\mu\\
%\nonumber&+\frac {\mathfrak C_{k(n)}(\phi,\psi)}{(a_n)^{d+1}}
%\left( \mathbb E_{\mu_{\Delta}}[g\Pi'_0h]+ i\mathbb E_{\mu_{\Delta}}[\hat\kappa_k  g]\mathbb E_{\mu_\Delta}[h]+i\mathbb E_{\mu_\Delta}[g]  \mathbb E_{\mu_\Delta}[\hat\kappa_kP^kh]\right)
%\left(I_{1}\left(\frac N{a_n}\right)-\frac {I_3\left(\frac N{a_n}\right)}2 \frac{\log\log n}{\log n}\right)
%\\   &+O\left(a_n^{-d-1}/\log n\right)\, .
%\end{align}
\begin{align}
\label{eq:Ck}
\mathfrak C_k(\phi,\psi)&=\mathbb E_{\mu_{\Delta}}[\tilde \psi^{(k)}\Pi'_0P^{2k}\tilde\phi^{(k)}]+ i\mathbb E_{\bar\mu}[\kappa_{-k}   \psi^{(k)}]\mathbb E_{\mu_\Delta}[\phi^{(k)}]+i\mathbb E_{\bar\mu}[ \psi^{(k)}]  \mathbb E_{\bar\mu}[\kappa_k\phi^{(k)}]\\
&=
\mathbb E_{\bar\mu}[\phi]
\left(\mathbb E_{\mu_\Delta}\left[\tilde\psi^{(k)} H\right]+i\mathbb E_{\bar\mu}[\kappa_{-k}   \psi^{(k)}]\right)+i
\mathbb E_{\bar\mu}[\psi]
\sum_{j\ge 0}
\mathbb E_{\bar\mu}[\kappa\circ \bar T^{j}\phi^{(k)}]\, .\nonumber
\end{align}
Thus
\begin{align}\label{intermediate2}
&\int_{\bar M} \phi^{(k)}1_{\{\kappa_n=N\}}\psi^{(k)} \circ \bar T^n \, d\bar\mu\\
\nonumber&=\frac 1{a_n^d}
   \left(I_0\left(\frac{N}{a_n}\right)-\frac 12 I_2\left(\frac{N}{a_n}\right) \frac{\log\log n}{\log n} + O\left(\frac 1{\log n}\right)
     \right)\mathbb E_{\bar\mu}[\phi]\mathbb E_{\bar\mu}[\psi]\\
\nonumber&+ \frac {\mathfrak C_{k(n)}(\phi,\psi)}{(a_n)^{d+1}}
\left(I_{1}\left(\frac N{a_n}\right)-\frac {I_3\left(\frac N{a_n}\right)}2 \frac{\log\log n}{\log n}\right)  +O(a_n^{-d-1}/\log n)\, .
\end{align}
In particular, since $|\psi^{(k)}|\le\psi^{(k,+)}\le |\psi|+\omega_{-k}^k(\psi,\cdot)$ with $\psi^{(k,+)}(x):=\sum_{y\in\xi_{-k}^k(x)}|\psi(y)|$,
\begin{align*}\label{intermediate3}
&\left|\int_{\bar M} 1_{\{\kappa_n=N\}}(\phi.\psi \circ \bar T^n-\phi^{(k)}.\psi^{(k)}\circ\bar T^n) \, d\bar\mu\right|\\
&\le
\int_{\bar M} 1_{\{\kappa_n=N\}}\left(\left|\phi.\omega_{-k}^k(\psi,\cdot) \circ \bar T^n\right|+\left|\omega_{-k}^k(\phi,\cdot).\psi^{(k,+)}\circ\bar T^n\right|\right) \, d\bar\mu
\nonumber
\end{align*}
which is dominated, up to a multiplicative constant by
\begin{align*}
\nonumber&\ll\frac {\Vert\phi\Vert_\infty\Vert\omega_{-k}^k(\psi,\cdot)\Vert_{L^1(\bar\mu)}+\Vert\psi^{(k,+)}\Vert_\infty\Vert\omega_{-k}^k(\phi,\cdot)\Vert_{L^1(\bar\mu)}}{a_n^d} \\
\nonumber&+ \frac {\mathfrak C_{k(n)}(|\phi|,\omega_{-k}^k(\psi,\cdot))
+\mathfrak C_{k(n)}(\omega_{-k}^k(\phi,\cdot),|\psi^{(k,+)}|))
%k(n)(\Vert\phi\Vert_\infty\Vert\omega_{-k}^k(\psi,\cdot)\Vert_{L^1(\bar\mu)+\Vert\psi\Vert_\infty\Vert\omega_{-k}^k(\phi,\cdot)\Vert_{L^1(\bar\mu)}\mathfrak C_{k(n)}(\phi,\psi)
}{(a_n)^{d+1}}
%\left(I_{1}\left(\frac N{a_n}\right)-\frac {I_3\left(\frac N{a_n}\right)}2 \frac{\log\log n}{\log n}\right) 
 +O(a_n^{-d-1}/\log n)\\
&= O(a_n^{-d-1}/\log n)\, ,
\end{align*}
since for $(h,g)=(|\phi|,\omega_{-k}^k(\psi,\cdot))$ or $(h,g)=(\omega_{-k}^k(\phi,\cdot),\psi^{(k,+)})$,
\[
\mathfrak C_{k(n)}(h,g)\ll \| g\|_\infty \|h\|_\infty + k(n)(\|g\|_{L^{p_2}(\bar\mu)}
\|h\|_{L^1(\bar\mu)}+\|h\|_{L^{p_2}(\bar\mu)}
\|g\|_{L^1(\bar\mu)})\, ,
\]
(since $\kappa\in L^{p_2/(p_2-1)(\bar\mu)}$)
and using our assumptions.
% \eqref{AAA002} and the first part of \eqref{COND2}.
Therefore
\begin{align}\label{intermediate3}
&\int_{\bar M} \phi 1_{\{\kappa_n=N\}}\psi \circ \bar T^n \, d\bar\mu\\
\nonumber&=\frac 1{a_n^d}
   \left(I_0\left(\frac{N}{a_n}\right)-\frac 12 I_2\left(\frac{N}{a_n}\right) \frac{\log\log n}{\log n} + O\left(\frac 1{\log n}\right)
     \right)\int_{\bar M} \phi\, d\bar\mu\, \int_{\bar M} \psi\, d\bar\mu\\
\nonumber&+ \frac {\mathfrak C_{k(n)}(\phi,\psi)}{(a_n)^{d+1}}
\left(I_{1}\left(\frac N{a_n}\right)-\frac {I_3\left(\frac N{a_n}\right)}2 \frac{\log\log n}{\log n}\right)  +O(a_n^{-d-1}/\log n)\, .
\end{align}
Note that $\phi=\phi_0+\bar\mu(\phi)$ and $\psi=\psi_0+\bar\mu(\psi)$, with $\phi_0=\phi-\bar\mu(\phi)$ and $\psi_0=\psi-\bar\mu(\psi)$. So
%With these notations
\begin{align}\label{decompdecor}
&\int_{\bar M} \phi 1_{\{\kappa_n=N\}}\psi \circ \bar T^n \, d\bar\mu\\
\nonumber&=\int_{\bar M} \phi_0 1_{\{\kappa_n=N\}}\psi\circ \bar T^n \, d\bar\mu
+\bar\mu(\phi)\int_{\bar M}  1_{\{\kappa_n=N\}}\psi_0 \circ \bar T^n \, d\bar\mu
+\bar\mu(\phi)\bar\mu(\psi)\bar\mu(\kappa_n=N)\, .
\end{align}
Applying \eqref{intermediate3} with $(1_{\bar M},1_{\bar M})$ instead of $(\phi,\psi)$ leads to
\begin{equation}\label{decor11}
\bar\mu(\kappa_n=N)=\frac 1{a_n^d}
   \left(I_0\left(\frac{N}{a_n}\right)-\frac 12 I_2\left(\frac{N}{a_n}\right) \frac{\log\log n}{\log n} + O\left(\frac 1{\log n}\right)
     \right)  +O(a_n^{-d-1}/\log n)\, ,
\end{equation}
since $\mathbb E_{\mu_\Delta}[\Pi'_01]=0$ (see Remark~\ref{rem:derivopi}).
Now let us study $\int_{\bar M} \phi_0 1_{\{\kappa_n=N\}}\psi\circ \bar T^n \, d\bar\mu$.
In what follows we show that $\mathfrak C_{k(n)}(\phi_0,\psi)$
converges as $n\rightarrow +\infty$ with rate in $O((\log n)^{-1})$. Note that
%We start by noticing that
%\begin{align*}
%&\left|\mathbb E_{\bar\mu}[\phi^{(k(n)}-\phi]\left(\mathbb E_{\mu_\Delta}\left[\tilde\psi^{(k(n))} H\right]+i\mathbb E_{\bar\mu}[\kappa_{-k(n)}   \psi^{(k(n))}]\right)\right|\\
%&=\Vert\omega_{-k(n)}^{k(n)}(\phi,\cdot)\Vert_{L^1(\bar\mu)}\left(\Vert\psi\Vert_\infty\Vert H\Vert_{L^1(\mu_\Delta)}+k(n)\Vert \kappa\Vert_{L^1(\bar\mu)}\Vert\psi\Vert_\infty\right)\\
%&\ll a_n^{-\frac 12-u}(\log n)^{-1}\Vert\psi\Vert_\infty
%\end{align*}
%due to \eqref{AAA02} and since $k(n)=O(a_n^{\frac 12-u})$.
%\textbf{We first show $J_2=O\left(a_n^{-d-1}/\log n\right)$ 
%{\color{blue}
%We start by considering the case whenconsidering $\phi$ is replaced by $\phi_0:=\phi-\mathbb E_{\bar\mu}[\phi]$,
%}
%for which we note that
\begin{equation}\label{int0}
\mathfrak C_{k(n) }(
\phi_0
,\psi)=i\mathbb E_{\bar\mu}[\psi]\sum_{j\ge 0}
\mathbb E_{\bar\mu}[\kappa\circ \bar T^{j}\phi^{(k(n))}]\, .
\end{equation}
Let us prove the two following estimates 
\begin{equation}\label{limit}
\left\vert 
%\mathbb E_{\bar\mu}[\psi^{(k(n))}]
\sum_{j\ge 0}
\mathbb E_{\bar\mu}[\kappa\circ \bar T^{j}\phi^{(k(n))}]-
%\mathbb E_{\bar\mu}[\psi]
\sum_{j\ge 0}
\mathbb E_{\bar\mu}[\kappa\circ \bar T^{j}\phi]\right|=O((\log n)^{-1})\, ,
\end{equation}
%and that
\begin{equation}\label{finitesum}
\sum_{j\ge 0}\left|
\mathbb E_{\bar\mu}[\kappa\circ \bar T^{j}\phi]\right|<\infty\, .
\end{equation}
%We wish to prove that this quantity converges to $\mathbb E_{\bar\mu}[\psi]\sum_{j\ge 0} \mathbb E_{\bar\mu}[\kappa\circ \bar T^{j}\phi]$ as $k\rightarrow +\infty$.
\begin{itemize}
\item \textbf{Proof of~\eqref{finitesum}.} First, for 
every integers $j,k_j\ge 0$, we have
%We first prove that 
%as soon as
%{\color{red}
%\begin{equation}
%\label{COND1}
%\exists p_2>2,\quad  
%\sum_{k}\Vert\omega_{-k}^{k}(\phi,\cdot)\Vert_{L^{p_2}(\bar\mu)}<\infty\, .
%\end{equation}
%}
%Indeed, for all $p_2>2$ and every integers $j,k_j\ge 0$
\begin{align*}
\left|\mathbb E_{\bar\mu}[\kappa\circ \bar T^{j}\phi]\right|
&\le
|\mathbb E_{\bar\mu}[\kappa\circ \bar T^{j}\phi^{(k_j)}]|+
|\mathbb E_{\bar\mu}[\kappa\circ \bar T^{j}(\phi-\phi^{(k_j)})]|\\
%&\le|\mathbb E_{\mu_\Delta}[\hat\kappa\circ \bar T^{j+k_j} (\tilde\phi^{(k_j)}-\mu_\Delta(\tilde\phi^{(k_j)})]|+\Vert \kappa\Vert_{L^{\frac {p_2}{p_2-1}}(\bar\mu)}\Vert\omega_{-k_j}^{k_j}(\phi,\cdot)\Vert_{L^{p_2}(\bar\mu)}\\
&\le|\mathbb E_{\mu_\Delta}[\hat\kappa P^{j+k_j}(\tilde\phi^{(k_j)}-\mu_\Delta(\tilde\phi^{(k_j)})]|+\Vert \kappa\Vert_{L^{\frac {p_2}{p_2-1}}(\bar\mu)}
\Vert\omega_{-k_j}^{k_j}(\phi,\cdot)\Vert_{L^{p_2}(\bar\mu)}\, .
\end{align*}
Now using the fact that $\Vert P^{2k_j}\tilde\phi^{(k_j)}-\mu_\Delta(\tilde\phi^{(k_j)})\Vert_{\mathcal B_0}\ll\Vert\phi\Vert_\infty$ (which comes from \eqref{PukPuphi}) and using \eqref{spgap-Sz-bis} for $t=0$, we conclude that
\begin{equation}\label{majo}
 \left|\mathbb E_{\bar\mu}[\kappa\circ \bar T^{j}\phi]\right|\le C \Vert \hat\kappa\Vert_{L^{q_1
%\frac p{p-1}
}(\mu_\Delta)}\theta^{j-k_j}\Vert\phi\Vert_\infty+\Vert \kappa\Vert_{L^{\frac {p_2}{p_2-1}}(\bar\mu)}\Vert\omega_{-k_j}^{k_j}(\phi,\cdot)\Vert_{L^{p_2}(\bar\mu)}\, .
\end{equation}
%Due to our asumptions,
%% \eqref{COND2}, 
%$\sum_{k}\Vert\omega_{-k}^{k}(\phi,\cdot)\Vert_{L^{p_2}(\bar\mu)}<\infty$ 
% and thus, t
Taking $k_j=j/2$,
%\lfloor 2\log j/\log \eta\rfloor$so that $\Vert\omega_{-k_j}^{k_j}(\phi,\cdot)\Vert_\infty\ll\eta^{k_j}\ll j^2$, 
we obtain
\begin{equation}
\sum_{j\ge 0}  \left|\mathbb E_{\bar\mu}[\kappa\circ \bar T^{j}\phi]\right|\ll \sum_{j\ge 0} \left(\Vert \hat\kappa\Vert_{L^{\frac p{p-1}}}\theta^{ j/2}\Vert\phi\Vert_\infty+\Vert\omega_{-\lfloor j/2\rfloor}^{\lfloor j/2\rfloor}(\phi,\cdot)\Vert_{L^{p_2}(\bar\mu)}\right)<\infty\, .
\end{equation}
This ends the proof of \eqref{finitesum}.
\item 
 \textbf{Proof of~\eqref{limit}}. Recall that $\|\mathbb E_{\bar\mu}[\psi^{(k)}-\psi]\|
\le\Vert\omega_{-k}^{k}(\psi,\cdot)\Vert_{L^1(\bar\mu)}=O((\log n)^{-1})$.
% Note that to prove ~\eqref{limit} it suffices to show that
%\[\left|\mathbb E_{\bar\mu}[\kappa\circ \bar T^{j}\phi^{(k)}]-\sum_{j\ge 0}\mathbb E_{\bar\mu}[\kappa\circ \bar T^{j}\phi]\right|=O((\log n)^{-1})\, .\]
Observe that, due to \eqref{majo} with  $k_j:=\lfloor j/2\rfloor$ and $k=k(n)$, 
\begin{align*}
&\sum_{j\ge 0}
\left|\mathbb E_{\bar\mu}[\kappa\circ \bar T^{j}(\phi^{(k)}-\phi)]\right|
=\sum_{j=0}^ {2k-1}\left|\mathbb E_{\bar\mu}[\kappa\circ \bar T^{j}(\phi^{(k)}-\phi)]\right|+\sum_{j\ge 2k} \left|\mathbb E_{\bar\mu}[\kappa\circ \bar T^{j}(\phi^{(k)}-\phi)]\right|\\
&\le
2k \Vert \kappa\Vert_{L^{\frac{p_2}{p_2-1}}(\bar\mu)}
\Vert  \omega_{-k}^k(\phi,\cdot)\Vert_{L^{p_2}(\bar\mu) }\\
&\ \ \ \ \ \ \ \ 
+\sum_{j\ge 2k}\left(C
 \Vert \hat\kappa\Vert_{L^{q_1
%\frac p{p-1}
}(\mu_\Delta)}\theta^{j/2}2\Vert\phi\Vert_\infty+\Vert \kappa\Vert_{L^{\frac {p_2}{p_2-1}}(\bar\mu)}\Vert\omega_{-\lfloor j/2\rfloor}^{\lfloor j/2\rfloor}(\phi,\cdot)\Vert_{L^{p_2}(\bar\mu)}\right)\, ,
\end{align*}
which is $O((\log n)^{-1})$
due to our assumptions.
%~\eqref{COND2}. 
This ends the proof of~\eqref{limit}.
%{\color{red}
%\begin{equation}\label{COND2}
%\Vert  \omega_{-k(n)}^{k(n)}(\phi,\cdot)\Vert_{L^{p_2}(\bar\mu) }=o(k(n)^{-1}/\log n)\quad\mbox{and}\quad\sum_{j\ge k(n)}\Vert\omega_{-j}^{j}(\phi,\cdot)\Vert
%_{L^{p_2}(\bar\mu)}=O((\log n)^{-1})\, .
%\end{equation}
%}
\end{itemize}
Applying \eqref{int0} combined with \eqref{limit} and \eqref{finitesum}, we obtain
\begin{align}\label{DDD1}
\mathfrak C_k(\phi_0,\psi)=i\mathbb E_{\bar\mu}[\psi]\sum_{j\ge 0}\mathbb E_{\bar\mu}[\kappa\circ \bar T^{j}\phi]
+O((\log n)^{-1})\, ,
%{(a_n)^{d+1}}\left(I_{1}\left(\frac N{a_n}\right)-\frac {I_3\left(\frac N{a_n}\right)}2 \frac{\log\log n}{\log n}\right)
%\left( \mathbb E_{\mu_{\Delta}}[g\Pi'_0h]+ i\mathbb E_{\mu_{\Delta}}[\hat\kappa_k  g]\mathbb E_{\mu_\Delta}[h]+i\mathbb E_{\mu_\Delta}[g]  \mathbb E_{\mu_\Delta}[\hat\kappa_kP^kh]\right)
%+O\left(a_n^{-d-1}/\log n\right)\, .
\end{align}
and thus, due to \eqref{intermediate3},
\begin{equation}\label{phi0psi}
\int_{\bar M} \phi_0 1_{\{\kappa_n=N\}}\psi \circ \bar T^n \, d\bar\mu
= \frac {i\mathbb E_{\bar\mu}[\psi]\sum_{j\ge 0}\mathbb E_{\bar\mu}[\kappa\circ \bar T^{j}\phi]}{(a_n)^{d+1}}\tilde I_1\left(n,\frac N{a_n}\right)
  +O\left(\frac{a_n^{-d-1}}{\log n}\right)\, ,
\end{equation}
with
$\tilde I_1(n,x):=
I_{1}\left(x\right)-\frac {I_3\left(x\right)}2 \frac{\log\log n}{\log n}$.
%Recall the formula of $\mathfrak C_k(\phi,\psi)$ {\color{blue} introduced in~\eqref{eq:Ck}.}
%Applying Corollary~\ref{constant} (with $\phi=\psi=1$), suing that $-I_1(-x)=I_1(x)$ and  $-I_3(-x)=I_3(x)$ and recalling that $\mathbb E_{\mu_\Delta}[\Pi'_01]=0$ (see~\ref{rem:derivopi}), we obtain
%{\color{blue}
%$\mathfrak C_k(1,1)=0$ and thus
%\begin{align*}
%\mathfrak C_k(\phi,\psi)&=\mathfrak C_k(\phi_0,\psi)+\bar\mu(\phi)\mathfrak C_k(1,\psi_0)+\bar\mu(\phi)\bar\mu(\psi)\mathfrak C_k(1,1)\\
%&=\mathfrak C_k(\phi_0,\psi)+\bar\mu(\phi)\mathfrak C_k(1,\psi_0)\, .
%\end{align*}
%}
%\begin{align*}
%J_2 &= \frac{\mathfrak C_k(\phi_0,\psi)}{(a_n)^{d+1}}\left(I_{1}\left(\frac N{a_n}\right)-\frac {I_3\left(\frac N{a_n}\right)}2 \frac{\log\log n}{\log n}\right)\\
%&+\frac{\mathfrak C_k(\phi,\psi_0)}{(a_n)^{d+1}}\left(I_{1}\left(\frac{-N}{a_n}\right)-\frac{I_3\left(\frac{-N}{a_n}\right)}2 \frac{\log\log n}{\log n}\right)
%+O\left(a_n^{-d-1}/\log n\right)\, .
%\end{align*}
% decompose $\phi$ and $\psi$ as follows~: 
%\begin{align}\label{DDD1}
%&\int_{\bar M} \phi_0 1_{\{\kappa_n=N\}}\psi\circ \bar T^n\, d\bar\mu\\%
%&=i\frac {\mathbb E_{\bar\mu}[\psi]\sum_{j\ge 0}
%\mathbb E_{\bar\mu}[\kappa\circ \bar T^{j}\phi]}{(a_n)^{d+1}}
%\left( \mathbb E_{\mu_{\Delta}}[g\Pi'_0h]+ i\mathbb E_{\mu_{\Delta}}[\hat\kappa_k  g]\mathbb E_{\mu_\Delta}[h]+i\mathbb E_{\mu_\Delta}[g]  \mathbb E_{\mu_\Delta}[\hat\kappa_kP^kh]\right)
%\left(I_{1}\left(\frac N{a_n}\right)-\frac {I_3\left(\frac N{a_n}\right)}2 \frac{\log\log n}{\log n}\right)
%\\   &
%+O\left(a_n^{-d-1}/\log n\right)\, .
%\end{align}
Now let us prove that a similar formula holds for $(1,\psi_0)$, where $\psi_0=\psi-\bar\mu(\psi)$ thanks to time reversibility.
%an analogous formula is also valid for $(\phi,\psi_0)$
%In this sense, l
Let $\tau:\bar M\rightarrow\bar M$ be
given by $\tau(q,\vec v)=(q,2(\vec n_q,\vec v)\vec n_q-\vec v)$ where $\vec n_q$
is the unit normal vector to $\partial Q$ directed in $Q$. Observe that $\tau$ preserves $\bar\mu$ and satisfies the following relations:
$\tau\circ\tau=id$, $\tau\circ\bar T^n\circ \tau=\bar T^{-n}$ and $\kappa\circ\bar T^k\circ\tau:=-\kappa\circ \bar T^{-k-1}$.
In particular, $\kappa_n\circ \tau\circ \bar T^n=-\kappa_n$. 
Therefore
\begin{align*}
&\int_{\bar M}
% \phi
 1_{\{\kappa_n=N\}}\psi_0\circ \bar T^n\, d\bar\mu
=\int_{\bar M}
% \phi\circ\tau 
1_{\{\kappa_n\circ \tau=N\}}\psi_0\circ \bar T^n\circ\tau\, d\bar\mu\\
&=\int_{\bar M} 
%\phi\circ
\tau 1_{\{\kappa_n\circ \tau=N\}}\psi_0\circ\tau\circ \bar T^{-n}\, d\bar\mu
=\int_{\bar M}
% \phi\circ\tau\circ \bar T^n 
1_{\{\kappa_n\circ \tau\circ \bar T^n=N\}}\psi_0\circ\tau \, d\bar\mu
\\
&=\int_{\bar M}\psi_0\circ\tau1_{\{\kappa_n=-N\}}
 %\phi\circ\tau\circ \bar T^n 
 \, d\bar\mu\, .
\end{align*}
%\begin{align*}
%&\int_{\bar M}
% \phi
% 1_{\{\kappa_n=N\}}\psi_0\circ \bar T^n\, d\bar\mu
%=\int_{\bar M}
% \phi\circ\tau 
%1_{\{\kappa_n\circ \tau=N\}}\psi_0\circ \bar T^n\circ\tau\, d\bar\mu\\
%&=\int_{\bar M} 
%\phi\circ
%\tau 1_{\{\kappa_n\circ \tau=N\}}\psi_0\circ\tau\circ \bar T^{-n}\, d\bar\mu
%=\int_{\bar M}
% \phi\circ\tau\circ \bar T^n 
%1_{\{\kappa_n\circ \tau\circ \bar T^n=N\}}\psi_0\circ\tau \, d\bar\mu
%\\
%&=\int_{\bar M}\psi_0\circ\tau1_{\{\kappa_n=-N\}}
 %\phi\circ\tau\circ \bar T^n 
% \, d\bar\mu\, .
%\end{align*}
Observe that
$\omega_{-k}^k(\psi_0\circ\tau,x)=\omega_{-k}^k(\psi_0,\tau(x))$ and so,
the composition by $\tau$ preserves the $L^r$ norms of $\omega_{-k}^k$.
%Recalling the formula of $C_k(\phi,\psi)$ introduced in~\eqref{intermediate}, we obtain
Thus, applying \eqref{phi0psi} to the couple $(\psi_0\circ \tau,1)$ instead of $(\phi_0,\psi)$, we obtain
\begin{align*}
&\int_{\bar M} 1_{\{\kappa_n=N\}}\psi_0 \circ \bar T^n \, d\bar\mu
= \frac {i\sum_{j\ge 0}\mathbb E_{\bar\mu}[\kappa\circ \bar T^{j}\psi\circ\tau]}{(a_n)^{d+1}}\tilde I_1\left(n,-\frac N{a_n}\right)
  +O\left(\frac{a_n^{-d-1}}{\log n}\right)\\
&= \frac {i\sum_{j\le -1}\mathbb E_{\bar\mu}[\kappa\circ \bar T^{j}\psi]}{(a_n)^{d+1}}
\tilde I_1\left(n,\frac N{a_n}\right)+O\left(\frac{a_n^{-d-1}}{\log n}\right)\, .
\end{align*}
since $\mathbb E_{\bar\mu}[\kappa\circ \bar T^{j}\psi\circ\tau]=\mathbb E_{\bar\mu}[-\kappa\circ \bar T^{-j-1}\psi]$, 
$-\tilde I_1(n,-x)=\tilde I_1(n,x)$.
%and  $-I_3(-x)=I_3(x)$.
%Observe now that
%\[
%\phi.\psi\circ\bar T^n=\phi_0.\psi\circ\bar T^n+\mathbb E_{\bar\mu}[\phi]\psi_0\circ\bar T^n+\mathbb E_{\bar\mu}[\phi]E_{\bar\mu}[\psi]\, .
%\] 
The claimed estimate follows from
this last estimate combined with~\eqref{decompdecor}, \eqref{decor11} and
\eqref{phi0psi}.
%~\eqref{DDD1} and~\eqref{DDD2}. 
%and Corollary~\ref{constant} with $\phi=\psi=1$ 
% and using that $\mathbb E_{\mu}[\Pi'_01]=0$ (see~\ref{rem:derivopi})}, we obtain \eqref{LLTfinal}.
\end{proof}

For all $N\in\mathbb Z^d$, we set $\phi_N:=\phi 1_{M_N}$.
Observe that $\omega_{-k}^\infty (h,\cdot)\le\omega_{-k}^k(h,\cdot)\le L_{h,\eta}\eta^k$. Therefore, Theorem~\ref{THM0} (resp. Theorem~\ref{THM1}) is a direct consequence of the previous (resp. following) theorem (for Theorem~\ref{THM1}
we also use the remark just after its statement).
\begin{theo}\label{LLT2}
Let $(k(n))_n$ be a sequence of integers diverging to $+\infty$ such that $k(n)=O( a_n/\log n)$.
Let $\phi,\psi: M\rightarrow \mathbb C$ be  two measurable functions.
% such that 
\begin{itemize}
\item[(I)] If
%\begin{equation}\label{HYP1}
$\sum_{N\in\mathbb Z^d}(\Vert \phi_N\Vert_\infty+\Vert \psi_N\Vert_\infty)<\infty$ and
%\quad\mbox{and}\quad
$\lim_{k\rightarrow +\infty}
  \int_M(\omega_{-k}^k(\phi,\cdot)
  + \omega_{-k}^{\infty}(\psi,\cdot))\, d\mu=0$.
%\end{equation}
Then
\begin{equation}\label{MIXBILL}
\int_M \phi.\psi\circ T^n\, d\mu= \frac{I_0(0)}{a_n^d}\int_M\phi\, d\mu
      \int_M\psi\, d\mu+o(a_n^{-d})\, .
\end{equation}
\item[(II)] If moreover there exists
%we assume the two following conditions
%\begin{equation}\label{Hyp01}
%\exists 
$\gamma\in(0,1)$ such that
%,\quad 
$\sum_N |N|^\gamma(\Vert\phi_N\Vert_\infty+\Vert\psi_N\Vert_\infty)<\infty$,
%\end{equation}
%and
and if
%\begin{equation}\label{Hyp02}
$\int_M(\omega_{-k(n)}^{k(n)}(\phi,\cdot)
  + \omega_{-k(n)}^{\infty}(\psi,\cdot))\, d\mu=O((\log n)^{-1})$,
%\, ,\end{equation}
then
\begin{equation}\label{MIXBILL2}
\int_M \phi.\psi\circ T^n\, d\mu= \frac {I_0\left(0\right)}{a_n^d}
   \left(1-
%\frac d2 
%\, I_2\left(0\right)
\frac{d\log\log n}{2\log n}\right)
\int_M\phi\, d\mu  \int_M\psi\, d\mu
+O\left(\frac 1{a_n^d\log n}\right)\, .
\end{equation}
\item[(III)]
If moreover $k(n)\ll a_n^{\frac 12-u}$ for some $u\in(0,1)$ and if 
%the following conditions hold true
%\begin{equation}\label{Hyp11}
there exists $\gamma\in(0,1)$ such that
$\sum_N|N|^{1+\gamma}(\Vert\phi_N\Vert_\infty+\Vert\psi_N\Vert_\infty)<\infty$,
%\end{equation}
%and
and
%\begin{equation}\label{Hyp12}
$\int_M(\omega_{-k(n)}^{k(n)}(\phi,\cdot)
  + \omega_{-k(n)}^{k(n)}(\psi,\cdot))\, d\mu=O((a_n\log n)^{-1})$,
%\end{equation}
then
\begin{align*}\int_M \phi.\psi \circ T^n \, d\mu&=\frac {I_0\left(0\right)}{a_n^d}
   \left(1- \frac{d\log\log n +O(1)}{2\log n}
%+ O\left(\frac 1{\log n}\right)
     \right)\int_{ M} \phi\, d\mu\, \int_{ M} \psi\, d\mu
%\\&\ \ \ \  \ \ \ \ 
+O\left(\frac 1{a_n^{d+1}\log n}\right)\, .
\end{align*}
\end{itemize}
\end{theo}
\begin{proof}
For every $x\in M$ we write $\bar x\in\bar M$ for the class of $x$ modulo $\mathbb Z^d$ for the position.
Let us write $\bar\phi_N(\bar x)=\phi_N(x)$, for every $x\in M_N$.
Observe that
\[
  \int_M \phi.\psi \circ T^n \, d\mu= 
      \sum_{N_1,N_2\in\mathbb Z^2}
      \int_{\bar M} \bar\phi_{N_1}\mathbf 1_{\{\kappa_n=N_2-N_1\}}
         \bar\psi_{N_2} \circ \bar T^n \, d\bar\mu\, .
\]
Now, 
setting again $\tilde I_0(n,x):=I_0(x)-\frac {I_{2}(x)}2\frac{\log\log n}{\log n}$
and applying \eqref{error}, we obtain
\begin{align}\label{EEE1}%\label{TLLBILL0}
&\int_M \phi.\psi \circ T^n \, d\mu=\sum_{N_1,N_2\in\mathbb Z^d} \frac 1{a_n^d}
\tilde I_0\left(n ,\frac{N_2-N_1}{a_n}\right)
\int_{M}\phi_{N_1}\, d\mu  \int_{M}\psi_{N_2}\, d\mu\\
\nonumber&+  
O\left(a_n^{-d}\sum_{N_1,N_2\in \mathbb Z^d}I_0\left(\frac{N_2-N_1}{2a_n}\right)
\Vert\phi_{N_1}\Vert_{\infty}\Vert\omega_{-k(n)}^\infty(\psi_{N_2},\cdot)\Vert_{L^1(\bar\mu)})\right)\\
\nonumber&+O\left(a_n^{-d}\sum_{N_1,N_2\in \mathbb Z^d}I_0\left(\frac{N_2-N_1}{2a_n}\right)\Vert\psi_{N_2}\Vert_{\infty}\Vert\omega_{-k(n)}^{k(n)}(\phi_{N_1},\cdot)\Vert_{L^1(\bar\mu)})\right)
+O\left(\frac {\Vert\phi_{N_1}\Vert_\infty\Vert\psi_{N_2}\Vert_\infty}{a_n^d\log n}\right)\, .
\end{align}
\eqref{MIXBILL} follows from \eqref{EEE1} and the Lebesgue dominated convergence theorem, since 
$I_0$ and $I_2$ are bounded and since $I_0$ is continuous at 0.

Assume now the assumptions of (II).
%\eqref{Hyp01} and \eqref{Hyp02}.
Since $|I_0(X)-I_0(0)|+|I_2(X)-I_2(0)|=O(X^\gamma)$,
replacing $I_0\left(\frac{N_2-N_1}{a_n}\right)-\frac 12 \, I_2\left(\frac{N_2-N_1}{a_n}\right)...$ by $I_0\left(0\right)-\frac 12 \, I_2\left(0\right)...$ in \eqref{EEE1} leads to an error term in $O(a_n^{-d-\gamma})$. Since $I_0$ is bounded and
to our assumptions,
%due to the first item of \eqref{HYP1} and to \eqref{Hyp02},
 the first error term in 
\eqref{EEE1} is in $O(a_n^{-d}/\log n)$. This completes the proof of 
\eqref{MIXBILL2}.

Finally, we assume the assumptions of (III).
% \eqref{Hyp11} and \eqref{Hyp12} and prove the last point of the theorem.
We start from \eqref{BBB1} for $(\bar\phi_{N_1},\bar\psi_{N_2},N_2-N_1)$ instead of $(\phi,\psi,N)$.
Our assumptions combined with $|I_{2j}(x)-I_{2j}(0)|\ll x^{1+\gamma}$ 
%\eqref{Hyp11} 
ensure that we can replace $I_0(N/a_n)$ and $I_2(N/a_n)$ in $J_1$
up to an error (after summation over $N_1,N_2\in\mathbb Z^2$) in
$O(a_n^{-d-1-\gamma})$. 
%Moreover t
%The first condition in \eqref{HYP1} and \eqref{Hyp12} 
%also
Our assumptions 
ensure that we can replace
$\phi^{(k)}$ and $\psi^{(k)}$ by respectively $\phi$ and $\psi$ in $J_1$ up to a total error (after summation)
in $O(a_n^{-d-1}/\log n)$.
The fact that 
$|I_{2j+1}(x)|\ll x^{\gamma}$, combined with
our conditions,
%the first condition in \eqref{HYP1},  
%also 
implies that the contribution (after summation) of
$J_2$ is in $O(a_n^{-d-1}/\log n)$.
We prove as in the previous result that
 the contribution (after summation) of
$J_3$ is in $O(a_n^{-d-1}/\log n)$. We conclude using $I_2(0)=dI_0(0)$.
\end{proof}
\begin{proof}[Proof of Theorem~\ref{THEOCob}]
We proceed as in the proof of Proposition~\ref{LLT00} with same notation
$I_{\ell_1,\ell_2,N}$ introduced just before Proposition~\ref{LLT00} with $A=\Sigma^2$ and
\[
J_{\ell_1,\ell_2,n,N,N_0}:=\int_{-[\beta a_n,\beta a_n]^d}e^{-iu\cdot N/a_n}(-iu\cdot N_0)^{\ell_1}(-Au\cdot u)^{\ell_2}\lambda_{u/a_n}^{n-2k}\, du
\]
Using again  \eqref{lambdan-k} and  \eqref{intlambda}, we observe that
\begin{equation}
\label{diffJ}
\forall\delta_0\in(0,1],\quad 
:=J_{\ell_1,\ell_2,n,0,N_0} + O\left(\frac{N_0^{\ell_1}N^{\delta_0}}{a_n^{\delta_0}}\right)\, ,
%\frac 1{(2\pi)^d}\int_{-[\beta a_n,\beta a_n]^d}e^{-iu\cdot N/a_n}(-iu\cdot N)^{\ell_1}(-Au\cdot u)^{\ell_2}\lambda_{u/a_n}^{n-m}\, du\, ,
\end{equation}
\begin{align}\label{relationIJ-bill}
J_{\ell_1,\ell_2,n,0,N_0}&=I_{\ell_1,\ell_2,N_0}(0)
+
I_{\ell_1,\ell_2+1,N_0}(0) \frac{\log\log n}{2\log n}+O\left(\frac {N_0^{\ell_1}}{\log n}\right)\, .
%\label{relationIJ}&=I'_{\ell_1,\ell_2}(n)+O\left(\frac 1{\log n}+(N/a_n)^\delta\right)\, ,
\end{align}
%and notice that it satisfies \eqref{relationIJ-bill} thanks to \ref{intlambda}.
%Moreover, 
%\begin{equation}\label{DLI}
%I_{\ell_1,\ell_2,N}(N/a_n)=I_{\ell_1,\ell_2,N}(0)+O(|N|^{\delta}/a_n^\delta)=I_{\ell_1,\ell_2,N}(0)-i\frac{I_{\ell_1+1,\ell_2,N}}{a_n}+O(|N|^{1+\delta}/a_n^{1+\delta})\, ,
%\end{equation}

%We assume without loss of generality that $\phi,\psi$ take positive values.
For every $N\in\mathbb Z^d$, we
consider the functions $h_N,g_N:\bar M\rightarrow \mathbb C$
such that, for any $x\in M_N$, $h_N(\bar x):=\phi_N(x)$ and $g_N(\bar x):=\psi_N(x)$ 
where $\bar x\in\bar M$ is the class of $x$ modulo $\mathbb Z^2$ for the position.
We take $k=k(n):=\lfloor (d+m) (\log n)/|\log \eta|\rfloor$ with $m:=3/2$
in the setting of the second item (which implies that $\eta^{k(n)}\ll a_n^{-d-2m}$) and define the approximating functions
$\phi^{(k)},\psi^{(k)}:M\rightarrow\mathbb R$ ang $h_N^{(k)},g_N^{(k)}:\bar M\rightarrow\mathbb C$ given by
$\phi^{(k)}(x):=\inf_{\xi_{-k}^k( x)} \phi_+-\sup_{\xi_{-k}^k( x)} \phi_-$,
$\psi^{(k)}( x):=\inf_{\xi_{-k}^\infty(x)} \psi_+ - \sup_{\xi_{-k}^\infty(x)} \psi_-
$,
$
h_N^{(k)}(\bar x):=\inf_{\xi_{-k}^k( \bar x)} (h_N)_+-\sup_{\xi_{-k}^k( \bar x)} (h_N)_-$
%\quad\mbox{and}\quad
$g_N^{(k)}(\bar x):=\inf_{\xi_{-k}^k( \bar x)} (g_N)_+-\sup_{\xi_{-k}^k( \bar x)} (g_N)_-$.
%\]
Due to our choice of $k$ and to our assumptions on $\phi,\psi$,
\begin{equation}\label{ESTI1}
\int_M \phi.\psi\circ(id-T)^m\circ T^n\, d\mu=\int_M \phi^{(k)}.\psi^{(k)}\circ(id-T)^m\circ T^n\, d\mu+O\left(\eta^{k(n)}\right)\, .
\end{equation}
The error term in the previous formula 
is in $O(a_n^{-d-2m})$.
%goes to 0 faster than any power $n$ as $n\rightarrow +\infty$. 
So we focus on the following integral
\[
\int_M \phi^{(k)}.\psi^{(k)}\circ T^n\, d\mu =\sum_{N_1,N_2\in\mathbb Z^d}\mathbb E_{\bar\mu}[ h_{N_1}^{(k)}1_{\{\kappa=N_2-N_1\}}g_{N_2}^{(k)}\circ \bar T^n]\, .
\]
Recall, from the proof of Theorem~\ref{LLT20} that, due the definition of
 $h_N^{(k)},g_N^{(k)}$ there exist
$\tilde h_N^{(k)},\tilde g_N^{(k)}:\Delta\rightarrow\mathbb R$
such that $
h_N^{(k)}\circ T^k\circ\mathfrak p_1=\tilde h_N^{(k)}\circ\mathfrak p_2$,
$g_N^{(k)}\circ T^k\circ\mathfrak p_1=\tilde g_N^{(k)}\circ\mathfrak p_2$,
\begin{equation}
\label{controltilde}
\Vert \tilde g_N^{(k)}\Vert_\infty\le\Vert g_N\Vert_\infty\quad\mbox{and}\quad
\sup_{t}\Vert P_t^kP_t\tilde h_N^{(k)}\Vert_ {\mathcal B_0}\le\Vert h_N\Vert_\infty\, ,
\end{equation}
and thus, as seen at the begining of the proof of Proposition~\ref{LLT1},
\begin{align*}\nonumber
\int_M \phi^{(k)}.\psi^{(k)}\circ T^n\, d\mu&=\sum_{N_1,N_2\in\mathbb Z^d}
\mathbb E_{\mu_\Delta}\left[ \tilde h_{N_1}^{(k)}1_{\{\hat\kappa_n\circ f^k=N_2-N_1\}}\tilde g_{N_2}^{(k)}\circ \bar T^n\right]\\
&=\sum_{N_1,N_2\in\mathbb Z^d}\frac 1{(2\pi)^d}\int_{[-\pi,\pi]^d}
e^{-it\cdot(N_2-N_1)}
\mathbb E_{\mu_\Delta}\left[P_t^k( \tilde g_{N_2}^{(k)}P_t^{n-2k}P_t^kP^k\tilde h_{N_1}^{(k)})\right]\, .
\end{align*}
\begin{itemize}
\item Let us assume the assumptions of item (a) of the theorem. Then
\begin{align} 
&\int_M \phi^{(k)}.\psi^{(k)}\circ(I-T)^m\circ T^n\, d\mu=\sum_{j=0}^m\frac{m!}{j!(m-j)!}(-1)^j \int_M \phi^{(k)}.\psi^{(k)}\circ T^{n+j}\, d\mu\nonumber\\
&=  \sum_{N_1,N_2\in\mathbb Z^d}\frac 1{(2\pi)^d}\int_{[-\pi,\pi]^d}
e^{-it\cdot(N_2-N_1)}\sum_{j=0}^m\frac{m!}{j!(m-j)!}(-1)^j 
\mathbb E_{\mu_\Delta}\left[P_t^k (\tilde g_{N_2}^{(k)}P_t^{n+j-2k}P_t^kP^k\tilde h_{N_1}^{(k)})\right]
% +O\left(\eta^{(\log n)^2}\right)
\nonumber\\
&=  \sum_{N_1,N_2\in\mathbb Z^d}\frac 1{(2\pi)^d}\int_{[-\pi,\pi]^d}
e^{-it\cdot(N_2-N_1)}
\mathbb E_{\mu_\Delta}\left[ P_t^k(\tilde g_{N_2}^{(k)}(I-P_t)^mP_t^{n-2k}P_t^kP^k\tilde h_{N_1}^{(k)})\right]\, . \nonumber
%+O\left(\eta^{(\log n)^2}\right)\\
\end{align}
Now, as seen in the proof of the first item of Proposition~\ref{LLT00} combined with
\eqref{controltilde}
\begin{align}
\nonumber
&\int_M \phi^{(k)}.\psi^{(k)}\circ(I-T)^m\circ T^n\, d\mu =\frac 1{(2\pi)^da_n^d}\sum_{N_1,N_2\in\mathbb Z^d}\int_{[-\beta a_n,\beta a_n]^d} e^{-iu\cdot(N_2-N_1)/a_n}\lambda_{u/a_n}^{n-2k}(1-\lambda_{u/a_n})^m\\
&\ \ \ \times\mathbb E_{\mu_{\Delta}}[P_{u/a_n}^k(\tilde g^{(k)}_{N_2}\Pi_{u/a_n}P_{u/a_n}^kP_{u/a_n}\tilde h^{(k)}_{N_1})]\, du +O\left(\theta^n\Vert g_{N_2}\Vert_{\infty}\Vert h_{N_1}\Vert_{\infty}\right)\nonumber\\
&=-\frac {(\log n+\log\log n)^m}{2^ma_n^{d+2m}}\left(O((\log n)^{m-1})+\sum_{N_1,N_2\in\mathbb Z^d}
J_{0,m,n,0,0}
%\left(I_{0,m,N_2-N_1}(0)-I_{0,m+1,N_2-N_1}(0) \frac{\log\log n}{2\log n}\right)
\mathbb E_{\mu_\Delta}[P_{u/a_n}^k\tilde g^{(k)}_{N_2}]\mathbb E_{\mu_\Delta}[P_{u/a_n}^kP_{u/a_n}\tilde h^{(k)}_{N_1}]\right)\, ,\label{ESTI3}
%\nonumber\\
%&+O\left(a_n^{-d-2m}(\log n)^{m-1}\Vert g_{N_2}\Vert_{\infty}\Vert h_{N_1}\Vert_{\infty})\right)\, ,\label{ESTI3}
\end{align}
where we used
\eqref{lambda}, \eqref{bornelambda}, \eqref{diffJ}, \eqref{Pi} and
$\log(|a_n/u|)=\log a_n+O(\log|u|)$ exactly as we obtained \eqref{cobfinal} in the proof of Proposition \ref{LLT00}.
Moreover
\begin{align}
\mathbb E_{\mu_\Delta}[P_{u/a_n}^kP_{u/a_n}\tilde h^{(k)}_{N_1}]
&=\mathbb E_{\mu_\Delta}[\tilde h^{(k)}_{N_1}]
+O\left(    \frac{k}{a_n}\Vert\kappa\Vert_{L^1(\bar\mu)} \Vert h_{N_1}\Vert_\infty\right)\nonumber\\
&=\mathbb E_{\bar\mu}[ h_{N_1}]
+O\left(    \frac{k}{a_n}\Vert\kappa\Vert_{L^1(\bar\mu)} \Vert h_{N_1}\Vert_{(\eta)}\right)\, ,\label{errorEsph}
\end{align}
%and analogously
\begin{equation}\label{errorEspg}
\mbox{and}\quad
\mathbb E_{\mu_\Delta}[P_{u/a_n}^k\tilde g^{(k)}_{N_2}]
=\mathbb E_{\bar\mu}[ g_{N_2}]
+O\left(    \frac{k}{a_n}\Vert\kappa\Vert_{L^1(\bar\mu)} \Vert g_{N_2}\Vert_{(\eta)}\right)\, ,
\end{equation}
Combining this with \eqref{ESTI1}, \eqref{ESTI3}, with our summability assumptions and with \eqref{relationIJ-bill}, up to an error in $O(a_n^{-d-2m}(\log n)^{m-1})$,
%we obtain the first assertion of the Proposition, due to \eqref{relationIJ-bill}.
we obtain a dominating term in
\[
-\frac{(\log(n\log n))^m}{2^m(n\log n)^{\frac d2+m}}\left(I_{0,m,0}(0)+I_{0,m+1,0}(0)\frac{\log\log n}{2\log n}\right)\int_{M}\phi\, d\mu\int_M\psi\, d\mu\, .
\]
But, due to~\eqref{formulaI}, $I_{0,k,0}=
 \frac{\Delta^k\Phi(0)}{\sqrt{\det \Sigma^2}}$ and so
\begin{align*}
I_{0,k,0}
%&=\frac{\Delta^k\Phi(0)}{\sqrt{\det \Sigma^2}}\\
&=\sum_{i_1,...,i_k=1}^d\frac{\partial^{2k}}{\partial^2 x_{i_1}\cdots\partial^2 x_{i_k}}\frac{\Phi}{\sqrt{\det \Sigma^2}}(0)
=\sum_{k_1+\cdots+k_d=k}^d\frac{k!}{k_1!...k_d!}\frac{\partial^{2k}}{\partial^{2k_1} x_{1}\cdots\partial^{2k_d} x_{d}}\frac{\Phi}{\sqrt{\det \Sigma^2}}(0)\\
&=\sum_{k_1+\cdots+k_d=k}^d\frac{k!\Phi(0)}{k_1!...k_d!\sqrt{\det \Sigma^2}}\prod_{j=1}^d(-1)^{k_j}\mathbb E[Z_j^{2k_j}]
=\frac{(-1)^k\Phi(0)\mathbb E[(Z_1^2+...+Z_d^2)^{k}]}{\sqrt{\det \Sigma^2}}\\
&=(-1)^k\Phi_{\Sigma^2}(0)d(d+2)...(d+2k-2)\, ,
%&=\sum_{k_1+\cdots+k_d=k}^d\frac{k!}{k_1!...k_d!}\prod_{j=1}^d\frac{(-1)^{k_j}(2k_j)!}{2^{k_j}\, k_j!}\frac{\Phi(0)}{\sqrt{\det A}}=(-1/2)^k\sum_{k_1+\cdots+k_d=k}^d\frac{k!(2k_1)!...(2k_d)!}{(k_1!...k_d!)^2}\frac{\Phi(0)}{\sqrt{\det A}}
\end{align*} 
where $Z_j$ are independent standard gaussian random variables 
since $\frac{\partial^{2k}}{\partial^{2k} x_j}\Phi(x)=P_{2k}(x_j)\Phi(x)$,
with $P_{2k}$ is a polynomial such that $P_{2k}(0)=\frac{\partial^{2k}}{\partial^{2k} x_j}\Phi(0)=\frac{(-1)^k(2k)!}{2^k\, k!}=(-1)^k\mathbb E[Z_j^{2k}]$(we use also the well-known moments of the chi-squared distribution of $Z_1^2+...+Z_d^2$).
So $I_{0,m,0}(0)+I_{0,m+1,0}(0)\frac{\log\log n}{2\log n}=
%(-1)^m\Phi_{\Sigma^2}(0)d(d+2)...(d+2m-2)
I_{0,m,0}(0)\left(1-\frac{(d+2m)\log\log n}{2\log n}\right)$ and we conclude by the comments after the statement of Theorem~\ref{THEOCob}.
\item Let us prove item (b):
\begin{align*}
&\int_M \phi.\psi\circ T^n\, d\mu
=\int_M \phi^{(k)}.\psi^{(k)}\circ T^n\, d\mu+O(\eta^k)\\
&=\sum_{N_2\in\mathbb Z^d}\mathbb E_{\mu_{\Delta}}[\tilde h_{0}^k( 1_{\{\hat \kappa_n\circ f^k=N_2-N\}}+ 1_{\{\hat \kappa_n\circ f^k=N_2+N\}}-2\times 1_{\{\hat \kappa_n\circ f^k=N\}})\tilde g_{N_2}^{(k)}\circ f^n]+O(\eta^k)\\
&=\sum_{N_2\in\mathbb Z^d}\frac 1{(2\pi)^d}\int_{[-\pi,\pi]^d}e^{it\cdot N_2}(e^{-it\cdot N}+e^{it\cdot N}-2)\mathbb E_{\mu_{\Delta}}[P_t^k(\tilde g_{N_2}^{(k)}P_t^{n-2k}P_t^kP^k\tilde h_0^{(k)})]\, dt\, ,
\end{align*}
%Therefore $\int_M \phi.\psi\circ T^n\, d\mu$ 
\begin{align*}
&\int_M \phi.\psi\circ T^n\, d\mu=O(\theta^n\sum_{N_2\in\mathbb Z^d}
\Vert g_{N_2}\Vert_{L^{q_0}(\mu_\Delta)}\Vert h_0\Vert_{\infty})\\
&+\sum_{N_2\in\mathbb Z^d}
\frac 1{(2\pi)^d}\int_{[-\beta ,\beta]^d}e^{it\cdot N_2}(e^{-it\cdot N}+e^{it\cdot N}-2)\lambda_{t}^{n-2k}\mathbb E_{\mu_{\Delta}}[P_{t}^k(\tilde g_{N_2}^{(k)}\Pi_{t}P_{t}^kP^k\tilde h_0^{(k)})]\, dt\, .
%&=\sum_{N_2\in\mathbb Z^d}\frac 1{(2\pi)^da_n^d}\int_{[-\beta a_n,\beta a_n]^d}e^{iu\cdot N_2/a_n}(e^{-iu\cdot N/a_n}+e^{iu\cdot N/a_n}-2)\lambda_{u/a_n}^{n-2k}\mathbb E_{\mu_{\Delta}}[P_{u/a_n}^k(\tilde g_{N_2}^{(k)}\Pi_{u/a_n}P_{u/a_n}^kP^k\tilde h_0^{(k)})]\, du\\
\end{align*}
Now using the change of variable $t=u/a_n$, the expansion of the exponential and
expansion \eqref{Pi} of $\Pi$ with $\gamma=1$, we obtain that $\int_M \phi.\psi\circ T^n\, d\mu$ is equal to
\begin{align*}
&\sum_{N_2\in\mathbb Z^d}
\frac 1{(2\pi)^da_n^{d+2}}\int_{[-\beta a_n,\beta a_n]^d}e^{iu\cdot N_2/a_n}(iu\cdot N)^2\lambda_{u/a_n}^{n-2k}\mathbb E_{\mu_{\Delta}}[P_{u/a_n}^k\tilde g_{N_2}^{(k)}]\mathbb E_{\mu_\Delta}[P_{u/a_n}^kP^k\tilde h_0^{(k)}]\, du\\
&+O\left(a_n^{-d-3}\sum_{N_2\in\mathbb Z^d}
\Vert g_{N_2}\Vert_{L^{q_0}(\mu_\Delta)}\Vert h_0\Vert_{\infty}\right)\, ,
\end{align*}
%and so
\begin{align*}
&\int_M \phi.\psi\circ T^n\, d\mu=O\left((\log n)a_n^{-d-3}\sum_{N_2\in\mathbb Z^d}
\Vert g_{N_2}\Vert_{L^{q_0}(\mu_\Delta)}\Vert h_0\Vert_{\infty}\right)\\
&+\sum_{N_2\in\mathbb Z^d}
\frac 1{(2\pi)^da_n^{d+2}}\int_{[-\beta a_n,\beta a_n]^d}e^{iu\cdot N_2/a_n}(iu\cdot N)^2\lambda_{u/a_n}^{n-2k}\mathbb E_{\mu_{\Delta}}[\tilde g_{N_2}^{(k)}]\mathbb E_{\mu_\Delta}[\tilde h_0^{(k)}]\, du\, ,
\end{align*}
due to \eqref{errorEsph} and \eqref{errorEspg}.
Thus
\begin{align*}
\int_M \phi.\psi\circ T^n\, d\mu&=-
\frac {1}{a_n^{d+2}}\sum_{N_2\in\mathbb Z^d}\mathbb E_{\mu_{\Delta}}[\tilde g_{N_2}^{(k)}]\mathbb E_{\mu_\Delta}[\tilde h_0^{(k)}] J_{2,0,n,N_2,N}
+O\left((\log n)a_n^{-d-3}\right)\\
&=-
\frac {1}{a_n^{d+2}}J_{2,0,n,0,N}\sum_{N_2\in\mathbb Z^d}\mathbb E_{\mu_{\Delta}}[\tilde g_{N_2}^{(k)}]\mathbb E_{\mu_\Delta}[\tilde h_0^{(k)}] 
+O\left((\log n)a_n^{-d-3}+a_n^{-d-2-\delta}\right)\, ,
\end{align*}
%and finally $\int_M \phi.\psi\circ T^n\, d\mu$ 
%which is equal to
\begin{align*}
\int_M \phi.\psi\circ T^n\, d\mu=&-
\frac {1}{a_n^{d+2}}J_{2,0,n,0,N}\sum_{N_2\in\mathbb Z^d}\mathbb E_{\bar\mu}[g_{N_2}]\mathbb E_{\bar\mu}[h_0] 
+O\left((\log n)a_n^{-d-3}+a_n^{-d-2-\delta}\right)\\
&=-
\frac {1}{a_n^{d+2}}J_{2,0,n,0,N}\int_M \psi\, d\mu \mathbb E_{\bar\mu}[h_0] 
+O\left((\log n)a_n^{-d-3}+a_n^{-d-2-\delta}\right)\, ,
\end{align*}
%where we used 
%successively 
using \eqref{diffJ} and then our definition of $k(n)$.
We conclude 
%thanks to 
by \eqref{relationIJ-bill} and \eqref{I20N0}.
\end{itemize}
\end{proof}

\appendix

\section{Tail probability of the return time to the initial cell}
This appendix is devoted to the proof of Theorem~\ref{THMreturntime} (without assuming \eqref{H0}).
Since $\tau_0$ is constant along stable curves, there exists
$\hat\tau_0:\Delta\rightarrow \mathbb N$ such that $\tau_0\circ \mathfrak p_1=\hat\tau_0\circ\mathfrak p_2$.
We use the classical Dvoretzky and Erd\"os argument \cite{DE} combined with the estimates provided by our Proposition~\ref{PROP1} and~\eqref{LLT0} (we do not need Proposition~\ref{prop-lambda} here).
Considering the last visit time $n$ to the 0-cell before time $N$, we observe that
\[
1=\sum_{n=0}^N\bar\mu\left(\kappa_n=0,\ \tau_0\circ \bar T^n>N-n\right)
=\sum_{n=0}^N\mathbb E_{\mu_{\Delta}}\left[\mathbf 1_{\{\hat\tau_0>N-n\}}P^n(\mathbf 1_{\{\hat \kappa_n=0\}})\right]\, .
\]
Moreover, it follows from \eqref{spgap-Sz}, \eqref{spgap-Sz-bis}, \eqref{spgap-Sz-bisbis} and  Proposition~\ref{PROP1} that 
\begin{align*}
&P^n(\mathbf 1_{\{\hat \kappa_n=0\}})-\bar\mu(\kappa_n=0)=\frac 1{(2\pi)^d}\int_{[-\pi,\pi]^d}(Id-\Pi_0)P_t^n(\mathbf 1)\, dt\\
&=\frac 1{(2\pi)^d}\int_{[-\beta,\beta]^d}\lambda_t^n(Id-\Pi_0)\Pi_t(\mathbf 1)\, dt+O(\theta^n)\\
%&=\frac 1{(2\pi)^d}\int_{[-\beta,\beta]^d}\lambda_t^n\, dt+O\left(\theta^n+\int_{[-\beta,\beta]^d}e^{-na|t|^d\log(|t|^{-1})}|t|^\gamma\, dt\right)\\
&=
%\frac 1{(2\pi)^d}\int_{[-\beta,\beta]^d}e^{-n\Sigma^2t\cdot t\log(|t|^{-1})+O(n|t|^2)}\, dt+
O\left(\theta^n+\int_{[-\beta,\beta]^d}e^{-na|t|^2\log(|t|^{-1})}|t|^\gamma\, dt\right)=
%\frac 1{(2\pi)^2(n\log n)^{\frac d2}}\int_{[-\beta\sqrt{n\log n},\beta\sqrt{n\log n}]^2}e^{-\frac 12\Sigma^2u\cdot u\left(1+\frac{\log\log n}{\log n}-\frac{2\log|u|}{\log n}\right)}\, dt+
O\left(\theta^n+(n\log n)^{-\frac {d+\gamma}2}\right)
\end{align*}
in $L^{p'}(\mu_{\bar\Delta})$ with $p'>1$ and $\gamma>1$ as in Proposition~\ref{PROP1}, using $\Pi_0(\Pi'_0(\mathbf 1))=0$.
Therefore
\[1=
\sum_{n=0}^N\left[\bar\mu(\tau_0>N-n)\bar\mu(\kappa_n=0)
+\eps_{n,N}
\right]\, ,
\]
with $\left|\eps_{n,N}\right|=O\left(\bar\mu(\tau_0>N-n)^{\frac 1{q'}}(\theta^n+(n\log n)^{-\frac {d+\gamma}2})\right)$ uniformly in $(n,N)$
% with $n\le N$ 
 (with $\frac 1{q'}+\frac 1{p'}=1$). 
Since $(M,\mu,T)$ is recurrent, $\lim_{N\rightarrow +\infty}\bar\mu(\tau_0>N)=0$, so
$\lim_{N\rightarrow +\infty}\sum_{n=0}^N\varepsilon_{n,N}=0$ and
\begin{equation}\label{centralEQ}
1=\lim_{N\rightarrow +\infty}
\sum_{n=0}^N\bar\mu(\tau_0>N-n)\bar\mu(\kappa_n=0)\, .
\end{equation}
In particular $1\ge \limsup_{N\rightarrow +\infty}\bar\mu(\tau_0>N)
\sum_{n=0}^N\bar\mu(\kappa_n=0)$.
It follows from $\bar\mu(\kappa_n=0)\sim\frac{\Phi_{\Sigma^2}(0)}{(n\log n)^{\frac d2}}$
(\cite{SV07}) that
$\sum_{n=0}^N\bar\mu(\hat\kappa_n=0)\sim\Phi_{\Sigma^2}(0)\mathfrak e_N$
with $\mathfrak e_N=\log\log N$ if $d=2$ and with
$\mathfrak e_N=\sqrt{\frac N{\log N}}$ if $d=1$.
Therefore we obtain the following upper bound:
\begin{equation}\label{limsupreturntime}
\limsup_{N\rightarrow +\infty}\Phi_{\Sigma^2}(0)\mathfrak e_N \bar\mu(\tau_0>N) \le 1\, .
\end{equation}
\begin{itemize}
\item \underline{Assume $d=2$}.  It remains to prove that the lower bound coincide with the above upperbound. To this end, starting from \eqref{centralEQ} observe that, for $0\le m_N\le N$,
\begin{align*}
1&\le \liminf_{N\rightarrow +\infty} \left(\sum_{n=0}^{m_N-1}\bar\mu(\tau_0>N-n)\bar\mu(\kappa_n=0)+
\sum_{n=m_N}^N
%\bar\mu(\tau_0>N-n)
\bar\mu(\kappa_n=0)\right)\\
&\le \liminf_{N\rightarrow +\infty}\left(\bar\mu(\tau_0>N-m_N)\sum_{n=0}^{m_N-1}\bar\mu(\kappa_n=0)+\sum_{n=m_N}^N\bar\mu(\kappa_n=0)\right)\, .
%\left[\mathbf 1_{\{\hat\tau_0>N-n\}}P^n(\mathbf 1_{\{\hat \kappa_n=0\}})\right]
\end{align*}
applying this with $N=N'\lfloor\log N'\rfloor$ and $m_N=N'\left(\lfloor\log N'\rfloor-1\right)$, we obtain
\begin{align*}
1
&\le \liminf_{N\rightarrow +\infty}\bar\mu(\tau_0>N')\Phi_{\Sigma^2}(0)\log\log (N')+O\left((\log N')^{-2}
\right)\, ,
%&\le \bar\mu(\tau_0>N')\left(N_\eps+\sum_{n=N_\eps}^{N'(\lfloor\log N'\rfloor-1)-1}\bar\mu(\kappa_n=0)\right)+\sum_{n=N'(\lfloor\log N'\rfloor-1)}^{m=N'\lfloor\log N'\rfloor}\bar\mu(\kappa_n=0)\, .
\end{align*}
since $\sum_{n=N'(\lfloor\log N'\rfloor-1)}^{N'\lfloor\log N'\rfloor}\bar\mu(\kappa_n=0)=O\left(\log\frac{\log(N'\lfloor \log N'\rfloor)+\log(1-\frac 1{\lfloor \log N'\rfloor}) }{\log(N'\lfloor \log N'\rfloor)}\right)=O\left((\log N')^{-2}
\right)$ and  $\log\log m_N\sim\log \log N'$.
% and so $\liminf_{N\rightarrow +\infty}\Phi_{\Sigma^2}(0)(\log\log N)\bar\mu(\tau_0>N)\ge 1$. 
Combining this with \eqref{limsupreturntime}, we finally obtain \eqref{returntime}.
%\[\bar\mu(\tau_0>N)\sim\frac 1{\Phi_{\Sigma^2}(0)\log\log N}\, \quad\mbox{as}\quad N\rightarrow +\infty\, .\]
\item \underline{Assume $d=1$}.
Upper bound \eqref{limsupreturntime}
% combined with the diagonal argument 
ensures that the sequence of non increasing functions $((\mathfrak e_N\bar\mu(\tau_0 >xN))_{x\in(0,\infty)})_{N\ge 1}$ admits limit points for the pointwise convergence except at discontinuity points of the limit. Consider a subsequence indexed by $(N_k)_k$ converging to a function $\psi$.
%with discontinuity set $\mathcal D_\psi$, i.e. for all $x\in(0,\infty)\setminus\mathcal D_\psi$, $\psi(x):=\lim_{k\rightarrow+\infty}\mathfrak e_{N_k}\bar\mu(\tau_0 >xN_k)$.
It follows from \eqref{centralEQ} that, for all $y\in(0,\infty)$,
\begin{equation}\label{lim1}
 \lim_{N\rightarrow +\infty}\int_{-1}^{\lfloor Ny\rfloor}\bar\mu(\tau_0>\lfloor Ny\rfloor-x)\bar\mu(\kappa_{\lceil x\rceil}=0)\, dx =1\, .
\end{equation}
Note that $\int_{[-1,\eps N]\cup [\lfloor Ny\rfloor-N\eps,\lfloor Ny\rfloor]}
\bar\mu(\tau_0>\lfloor Ny\rfloor-x)\bar\mu(\kappa_{\lceil x\rceil}=0)\, dx$ is less than
\[
\sqrt{\frac{\log(N/2)} {N/2}} c'\mathfrak e_{\eps N}+ \sum_{k=1}^{N\eps}\sqrt{\frac{\log(k)}k}\frac{c'}{\sqrt{(N/2)\log(N/2)}}\le c '' \sqrt{\frac{\eps N\log( N)}{N\log (\eps N)}}
\]
which, combined with \eqref{lim1} leads to
\[
\limsup_{N\rightarrow +\infty}\left|\int_{\eps N}^{\lfloor Ny\rfloor-N\eps}\bar\mu(\tau_0>\lfloor Ny\rfloor-x)\bar\mu(\kappa_{\lceil x\rceil}=0)\, dx-1\right|=O(\sqrt{\eps})\, .
\]
But, for every $\varepsilon>0$, $\int_{\eps N}^{\lfloor Ny\rfloor-N\eps}\bar\mu(\tau_0>\lfloor Ny\rfloor-x)\bar\mu(\kappa_{\lceil x\rceil}=0)\, dx$ is equal to
\begin{align*}
%&N\int_{\eps}^{(\lfloor Ny\rfloor-N\eps)/N}\bar\mu(\tau_0>\lfloor Ny\rfloor-Nu)\bar\mu(\kappa_{\lceil Nu\rceil}=0)\, du\\&=
&\int_{
%(\lfloor \eps N\rfloor -1)
\eps}^{(\lfloor Ny\rfloor-N\eps)/N}\mathfrak e_N\bar\mu(\tau_0>\lfloor Ny\rfloor-Nu)(\sqrt{N\log N}\bar\mu(\kappa_{\lceil Nu\rceil}=0))\, du\\
&\rightarrow \int_{\eps}^{y-\eps}\psi(y-u)\frac{\Phi_{\Sigma^2}(0)}{\sqrt{u}}\, du\, ,\quad
\mbox{as }N=N_k\rightarrow +\infty\, ,
\end{align*}
due to the dominated convergence theorem since \eqref{limsupreturntime}
and \eqref{LLT0} ensure that
\begin{align*}
N\bar\mu(\tau_0> &\lfloor Ny\rfloor-Nu)\bar\mu(\kappa_{\lceil Nu\rceil}=0)\le cN\frac {1}{\mathfrak e_{N_\varepsilon}\sqrt{ N\eps\log( N\eps)}}=O(\varepsilon^{-1})\, .
\end{align*}
From which we conclude that $\quad \int_{0}^y\psi(y-u)\frac{\Phi_{\Sigma^2}(0)}{\sqrt{u}}\, du=1$ for any $y\in(0,\infty)$.
Observe that $\psi_0(z):=\frac 1{a\sqrt{z}}$ with $a=\Phi_{\Sigma^2}(0)\int_0^1\frac {dt}{\sqrt{t(1-t)}}$
is a solution of $\int_{0}^y\psi(y-u)\frac{\Phi_{\Sigma^2}(0)}{\sqrt{u}}\, du=1$.
Recall that $\int_0^1\frac {dt}{\sqrt{t(1-t)}}
%=\int_0^{\frac\pi 2}2\, dx
=\pi$ (using 
%for example the change of variable $t=(\sin x)^2$
for example the Euler Beta function), so $a=\pi\Phi_{\Sigma^2}(0)=\sqrt{\frac \pi{2\Sigma^2}}$.
Thus, for all $y\in(0,\infty)$, $\int_{0}^y(\psi-\psi_0)(y-u)\frac{\Phi_{\Sigma^2}(0)}{\sqrt{u}}\, du=0$.
Moreover, due to \eqref{limsupreturntime}, $\psi(x)\le C/\sqrt{x}$ and thus $\psi-\psi_0$ is integrable. We conclude by the Titchmarsh's Convolution Theorem (see e.g. \cite{Doos}) that $\psi-\psi_0\equiv 0$. Thus $\psi_0$ is the unique limit point and so we have proved \eqref{returntimedim1}.
\end{itemize}

\section{Justifying equation~\eqref{RnDA}}
\label{sec-RnDA}
%Throughout this section we l
Let $v:\Delta\rightarrow\mathbb R$
% so that $v$ satisfies
satisfying the assumptions 
%in the statement 
of
% Proposition~\ref{prop-expproj} (hence, in the statement of
Lemma~\ref {lem-Rn} and $w\in\mathcal B_1$.
Using the pointwise formula $R (\hat\kappa_\sigma v w)(y)=\sum_{a\in\mathcal{A}} e^{\chi (y_a)} w(y_a) (v\hat\kappa_\sigma)(y_a)$
(and in particular~\eqref{eq:GM}), the arguments used in the proof of ~\cite[Proposition 12.1]{MelTer17} show that
the $\|\cdot \|_{\cB_0}$ norm of $R (v\hat\kappa_\sigma\cdot)$ is bounded by
$\|v\hat\kappa_\sigma\|_{L^1(\mu_Y)}$.
We recall the main inequalities, which in turn will help us justify equations~\eqref{RnDA}.
An important assumption  used throughout ~\cite[Proof of Proposition 12.1]{MelTer17}
is ~\cite[Assumption (A1)]{MelTer17}. In our case,  $v\hat\kappa_\sigma$ is constant on the partition elements $a\in\mathcal{Y}$ and thus, it automatically satisfies 
~\cite[Assumption (A1)]{MelTer17}. In our case, ~\cite[Assumption (A1)]{MelTer17} translates into
$|\sup_a (v\hat\kappa_\sigma)-\inf_a (v\hat\kappa_\sigma)|\le C \inf_a (v\hat\kappa)_\sigma$, for all $a\in\mathcal{Y}$ and some $C>0$.
This allows for a direct application of the arguments in ~\cite[Proof of Proposition 12.1]{MelTer17} .
By the argument~\cite[Proof of Proposition 12.1 a)]{MelTer17} with $\varphi\in\{v\hat\kappa_\sigma,v1_Y\}$, we have that for some $C>0$,
\[
\|R (\varphi w)\|_\infty\le C \|w\|_\infty \|v\hat\kappa_\sigma\|_{L^1(\mu_Y)}.
\]
since $\varphi$ is constant on the partition elements $a\in\mathcal{Y}$.
%we obtain that $\|R (v w)\|_\infty\le C \|w\|_\infty \|1_Yv\|_{L^1(\mu_Y)}$.
Recall that $|\cdot|_{\cB_0}$ is the seminorm used in defining the norm $\|\cdot\|_{\cB_0}$. A simplified version of~\cite[Proof of Proposition 12.1 b)]{MelTer17}
(which can deal the $p$- derivative, for $p\in (1,2)$,  in $t$ of $R(e^{it\hat\kappa_\sigma}v)$) with $\varphi\in\{v1_Y,v\hat\kappa_\sigma, 1_{\{\sigma=n\}} v\hat\kappa_\sigma, 1_{\{\sigma=n\}}(e^{it \hat\kappa_\sigma}-1-it\hat\kappa_\sigma)v\}$
ensures that
\[
\|R (\varphi w)\|_{\cB_0}
\le C \|w\|_{\cB_0}  \|\varphi\|_{L^1(\mu_Y)}\, ,
\]
since the function $\varphi$ is  constant on the partition elements $a\in\mathcal{A}$.
%Similarly, $ \Vert R (v w)\Vert_{\cB_0}\le C \|w\|_{\cB_0} \|1_Y v\|_{L^1(\mu_Y)}$. Also, since the function  is constant on the partition elements $a\in\mathcal{A}$,the argument of ~\cite[Proof of Proposition 12.1]{MelTer17} applies to this function as well yielding
%$\hat\kappa_\sigma$,
%\[\|R (1_{\{\sigma=n\}} \hat\kappa_\sigma vw)\|_{\cB_0}\le C\|w\|_{\cB_0} \|1_{\{\sigma=n\}} v\hat\kappa_\sigma\|_{L^1(\mu_Y)}.\]
In particular, $\|R (1_{\{\sigma=n\}}  v)w\|_{\cB_0}\le C\|w\|_{\cB_0} \|1_{\{\sigma=n\}} 1_Y v\|_{L^1(\mu_Y)}$.
%Further, working with $$ (instead of $1_{\{\sigma=n\}} v\hat\kappa_\sigma$),which is again constant on  the partition elements $a\in\mathcal{A}$ since $v$ and $\hat\kappa_\sigma$ are, we have
%\[\|R (1_{\{\sigma=n\}} (e^{it\hat\kappa_\sigma}-1-it\hat\kappa_\sigma) vw)\|_{\cB_0}\le C\|w\|_{\cB_0} \|1_{\{\sigma=n\}} (e^{it\hat\kappa_\sigma}-1-it\hat\kappa_\sigma) v\|_{L^1(\mu_Y)}.\]

\section{On  the constant $c_0(v,w)$ }
\label{sec-convsum}
%In this section w
We show that the sums in the expression of  $c_0(v,w)$ 
%(as defined 
in 
%statement of 
Lemma~\ref{lem-eigf2} are absolutely convergent.
% (with exponential rate).
%In particular, in Lemma~\ref{lemma-lemexpk1y} below we show that $\int_{\Delta} \hat\kappa\, 1_{Y}\circ f^{j}\, d\mu_\Delta$  decays exponentially.
%As explained in Remark~\ref{nodirectmix}, t
Since we do not know if $\hat\kappa$ is in $\mathcal B$, our strategy is to use $1_{(a,l)}$ and exploit that $\hat\kappa$ is constant on $(a,l)
$. The argument below is delicate since the $\mathcal B$-norm of $1_{(a,l)}$ increases exponentially fast in $\sigma(a)$.
% for  $v\in\cB$  and $w \in L^\infty (\mu_\Delta)$. 
% Instead,  we shall exploit the decay in the first result below.

\begin{lem}
\label{lemma-Ppower}
%There exists $
%Let $w\in\mathcal B_0$. 
%Then 
$\sup_{a\in Y,\ l\in\{0,...,\sigma(a)-1\}}\|P_0^{\sigma(a,l)}(\cdot 1_{(a,l)})\|_{\mathcal L(\mathcal B_0\rightarrow \cB)}<\infty
%\le C \mu_\Delta((a,l)) \|w\|_ {\cB_0}
$.
%,  for all $(a,\ell)\in\mathcal{Y}\times \Z^d$ and for some $C>0$.
\end{lem}
\begin{proof}
Let $y\in Y\times\{ \ell_0\}$ with $\ell_0\ne 0$ and $w\in\cB_0$. We have $P^{\sigma(a,l)}(w1_{(a,l)})(y)=0$ and so $P^{\sigma(a,l)}(w1_{(a,l)}-\mu_\Delta(w1_{(a,l)}))(y)=-\mu_\Delta(w1_{(a,l)})$. 
Let $y\in Y\times\{0\}$, we have $P^{\sigma(a,l)}(w1_{(a,l)})(y)=\chi(y_{a})w(f^ly_a)$,
where $y_{a}$ is the unique element of $a\cap f^{-\sigma(a)}(\{y\})=a\cap F^{-1}(\{y\})$.
Thus,
\begin{align*}
&\sup_{x,y\in Y\times\{0\}}\frac{P^{\sigma(a,l)}(w1_{(a,l)})(x)-P^{\sigma(a,l)}(w1_{(a,l)})(y)}{\beta^{s_0(x,y)}}\\
&\le \sup_{x,y\in Y} \frac{e^{-\alpha_{\sigma(a,l)}(f^l(x_a))}-e^{-\alpha_{\sigma(a,l)}(f^l(y_a))}
%\chi(x_a)-\chi(y_a)
}{\beta^{s_0(x,y)}}\|w1_{(a,l)}\|_\infty
+ \sup_{x,y\in Y} \frac{w(f^l(x_a))-w(f^l(y_a))}{\beta^{s_0(x,y)}}\sup_{a} \chi\, \\
%&\le \sup_{(x,y)\in Y\times\{\mathcal B_0\}} \frac{e^{-\alpha_a(x_a)}-e^{-\alpha_a(y_a)}}{\beta^{s_0(x,y)}}\|w1_{(a,l)}\|_\infty+ \sup_{(x,y)\in Y\times\{0\}} \frac{w(f^l(x_a))-w(f^l(y_a))}{\beta^{s_0(x_a,y_a)+\sigma(a,l)}}\sup_{a} \chi\, \\
&\le \|w\|_{\cB_0}2\sup_{x\in Y} e^{-\alpha_{\sigma_a}(f^l(x_a))}
%+e^{-\alpha_{\sigma_a}(f^l(x_a))})
\left(1+\sum_{k=0}^{\sigma(a,l)-1}\frac{|\alpha(f^{l+k}(x_a))-\alpha(f^{l+k}(y_a))|}{\beta^{s_0(x,y)}}\right)
%\\&\ll \|w\|_{\cB_0}( e^{-\alpha_{\sigma_a}(f^l(x_a))}+e^{-\alpha_{\sigma_a}(f^l(x_a))})
\le C \|w\|_{\cB_0}\mu_\Delta(a)\, ,
\end{align*}
where we have used that $s_0(x,y)=s_0(x_a,y_a)-\sigma(a,l)\le s_0(x_a,y_a)$, that $\left|\alpha(x)-\alpha(y)\right|\le C'_1\beta^{s_0(x,y)}$ as soon as
$s_0(x,y)\ge 1$ and finally the  distorsion bounds. Moreover, as required,
$$\sup_{(x,y)\in Y\times\{0\}}P^{\sigma(a,l)}(w 1_{(a,l)})(x) \le \|w\|_\infty\sup_{(a,l)}e^{\alpha_{\sigma(a,l)}}\le C\|w\|_\infty\mu_\Delta(a)\, .$$~\end{proof}

\begin{lem}
\label{lemma-lemexpk1y}
There exist $\theta_0,\theta_2\in (0,1)$ so that $|\int_{\Delta} \hat\kappa\, 1_{Y}\circ f^{j}\, d\mu_\Delta|=O(\theta_0^j)$ and
so that $w\in\cB$,
%satisfying the assumptions of Proposition~\ref{prop-expproj}, 
 $|\int_{\Delta} \hat\kappa\circ f^{j}\, w\, d\mu_\Delta|=O(\theta_2^j
%\| v\|_{L^b(\mu_\Delta)}
\|w\|_{\cB})$.
\end{lem}
\begin{proof} Let $p_1<2$ and $\epsilon\in(0,1)$.
%, let $\theta$ so that~\eqref{spgap-Sz-bis} holds. Let $\theta_1\in (0,1)$ so that 
Note that $\mu_\Delta(\sigma\ge m)\le \sum_{k\ge m}k\mu_Y(\sigma\ge k)\ll\theta_1^{m(1-\epsilon)}$.
Using  H\"older inequality and Lemma~\ref{lemma-Ppower}, we compute that
\begin{align*}
&\int_{\Delta} \hat\kappa\, 1_{Y}\circ f^{j}\, d\mu_\Delta=\int_{\Delta}(\hat\kappa 1_{\{\sigma<j/2\}}-\mu_\Delta(\hat\kappa 1_{\{\sigma<j/2\}}) \, 1_Y\circ f^{j}   \, d\mu_\Delta+2\|  \hat\kappa\|_{L^{p_1}(\mu_\Delta)}  (\mu_\Delta(\sigma\ge j/2))^{\frac {p_1-1}{p_1}}\\
 &\ \ \ =\sum_{m\in\mathbb Z^d}m\sum_{(a,l)\, :\, \hat\kappa(a,l)=m,\, \sigma(a,l)<j/2} \, \int_\Delta   P^j(1_{(a,l)}-\mu_{\Delta}((a,l))) 1_Y  \, d\mu_\Delta+O\left(\theta_1^{\frac{(1-\epsilon)j(p_1-1)}{2p_1}}\right)\, .
% &\ \ \ \ \ \ \ \ \ \ \ \ +     2\|  \hat\kappa\|_{L^{p_1}(\mu_\Delta)}  \left(\sum_{m>j/2}m\mu_\Delta(Y)\mu_Y(\sigma\ge m)\right)^{\frac {p_1-1}{p_1}}\, .
\end{align*}
Thus $\int_{\Delta} \hat\kappa\, 1_{Y}\circ f^{j}\, d\mu_\Delta$ is dominated, up to a multiplicative constant, by
\begin{align*}
& \sum_{m\in\mathbb Z^d}|m|\sum_{(a,l)\, :\, \hat\kappa(a,l)=m,\, \sigma(a,l)<j/2}  \theta^{j-\sigma(a,l)} \mu_\Delta((a,l))+O\left(\theta_1^{\frac{(1-\epsilon)j(p_1-1)}{2p_1}}\right)
%    \|  \hat\kappa\|_{L^{p_1}(\mu_\Delta)}  \left(\sum_{m>j/2}m\theta_1^m\right)^{\frac {p_1-1}{p_1}}
%\\
%&\ \ \ 
\ll \theta^{j/2}
%\|\hat\kappa\|_{L^1(\mu_\Delta)}
+
\theta_1^{\frac{(1-\epsilon)j(p_1-1)}{2p_1}}
%    \|  \hat\kappa\|_{L^{p_1}(\mu_\Delta)} \theta_1^{\frac {(1-\epsilon)j(p_1-1)}{2p_1}}\ll \theta_0^{j}
\, .
\end{align*}
%with $\epsilon\in(0,1)$ small enough so that such that $\theta_1^{\frac{(1-\epsilon)(p_1-1)}{2p_1}}\le \theta_0$.
If $w\in\cB$ and $p>2$, \eqref{spgap-Sz-bis} ensures that
$
\left|\int_{\Delta}\hat\kappa\circ f^j.w\, d\mu_\Delta\right|\le C\Vert \hat \kappa\Vert_{L^{\frac p{p-1}}(\mu_\Delta)}\theta^j\Vert w\Vert_{\cB}$.
%The proof of the second estimate is similar and omitted. The assumption that  $v$ is constant on $(a, l)$ is not required.~
\end{proof}
%\begin{rem}\label{nodirectmix}
%We note that in the proof of Lemma~\ref{lemma-lemexpk1y} we cannot simply exploit mixing for $v = 1_{(a,l)}$ and $w = 1_Y$ because the norm of $1_{(a,l)}$ in $\cB$ increases exponentially fast in $\sigma(a)$.\end{rem}

We end with a technical lemma (of flavour somewhat similar to that of Lemma~\ref{lemma-lemexpk1y}) used in the proof of  in Lemma~\ref{cor-betcont}. 
%For its statement we let $\theta$ so that~\eqref{spgap-Sz-bis} holds, $\theta_0=\sqrt{\theta}$ and set $r=\theta_0^{1/\varepsilon}$, with $\varepsilon>2$ satisfying ~\eqref{condr}. We also let $\delta$ satisfying~\eqref{eq-delta}.

\begin{lem}
\label{lem-justif...}
Let $\delta$ as in the proof of Lemma~\ref{lem-eigf2}.
There exists $C>0$ so that, for any 
$v\in L^b(\mu_\Delta)$ and $w\in\cB_0$ satisfy the assumptions on $v,w$ in the statement  of Proposition~\ref{prop-expproj} and any  $\xi\in\C$ so that
$|\xi|<(1-\delta)^{-1}$,
\[
\sum_{(a,\ell)}\sum_{r\ge 0}|\xi|^{-r}\int_{\Delta} \left| P_0^{r}\left(vw1_{(a,\ell)}-\int_{(a,l)} vw\, d\mu_\Delta\right)\, (\hat\kappa1_Y\circ f^j)\right| d\mu_\Delta\le 
C\|w\|_{\mathcal B_0}
\Vert v\Vert_{L^b(\mu_\Delta)}\, ,
\]
\end{lem}

\begin{proof}
Note that $
\Vert P_0^{\sigma(a)-\ell}\left(wv1_{(a,\ell)}-\mu_{\Delta}(wv1_{(a,\ell)})\right)\Vert_{\mathcal B}\ll \|w\|_{\mathcal B_0}\mu_\Delta(|v|\, 1_{(a,\ell)})$
since $v$ is is constant on the $a\in\mathcal{Y}$ and due to Lemma~\ref{lemma-Ppower}. Thus,
\begin{align*}
&J_1:=\sum_{(a,\ell)}\sum_{r\ge\sigma(a)-\ell}|\xi|^{-r}\int_{\Delta} \left| P_0^{r}\left(vw1_{(a,\ell)}-\int_{(a,l)} vw\, d\mu_\Delta\right)\, (\hat\kappa1_Y\circ f^j)\right| d\mu_\Delta\\
&\le\sum_{(a,\ell)}\sum_{r\ge\sigma(a)-\ell} (1-\delta)^{-r} C_0\theta^{r-\sigma(a)+\ell}\|w\|_{\mathcal B_0}\mu_\Delta(|v|\, 1_{(a,\ell)})\Vert \hat\kappa\Vert_{L^{p/(p-1)}(\mu_\Delta)}\, ,
\end{align*}
\begin{align*}
J_1&\ll \sum_{(a,\ell)}(1-\delta)^{\ell-\sigma(a)}\|w\|_{\mathcal B_0}\Vert v1_{(a,\ell)}\Vert_{L^1(\mu_\Delta)}\le     \left(\sum_{(a,\ell)}(1-\delta)^{\frac p{p-1}(\ell-\sigma(a))} \mu_\Delta(a,\ell) \right)^{\frac {p-1}p}\!\!\!\!\!\!\!\!  \|w\|_{\mathcal B_0}\Vert v\Vert_{L^p(\mu_\Delta)}\\
&\ll    \left(\sum_{a}(1-\delta)^{-\frac p{p-1}\sigma(a)} \mu_Y(a) \right)^{\frac {p-1}p} \!\!\!\!\!\!\|w\|_{\mathcal B_0}\Vert v\Vert_{L^p(\mu_\Delta)}\ll
 \left(\sum_{m\ge 1}(1-\delta)^{-\frac p{p-1}m}
\theta_1^m
% \mu_Y(\sigma\ge m) 
\right)^{\frac {p-1}p}\!\!\!\!\!\! \|w\|_{\mathcal B_0}\Vert v\Vert_{L^p(\mu_\Delta)}.
\end{align*}
In the previous displayed chain of equations we have used that $\mu_Y(\sigma\ge m)\ll \theta_1^m$, $(1-\delta)^{-1}\theta<1$
and that $(1-\delta)^{-\frac p{p-1}}\theta_1<1$.  This takes care of the tails of the above sums.
Next, we deal with the partials sums (up to $\sigma$) in the chain of equations  below~\eqref{U2N't} leading to $U_{1,N'}(t)=0$.
\begin{align*}
&J_2:=\sum_{(a,\ell)}\sum_{r=0}^{\sigma(a)-\ell}|\xi|^{-r}\int_{\Delta} \left| P_0^{r}\left(vw1_{(a,\ell)}-\mu_\Delta(vw1_{(a,l)})\right)\, (\hat\kappa1_Y\circ f^j)\right| d\mu_\Delta\\
%&\le \sum_{(a,\ell)}\sum_{r=0}^{\sigma(a)-\ell}(1-\delta)^{-r}\left\Vert vw1_{(a,\ell)}-\int_{(a,l)} vw\, d\mu_\Delta\right\Vert_{L^p(\mu_\Delta)}\, \Vert \hat\kappa\Vert_{L^{\frac p{p-1}}(\mu_\Delta)}\\
&\ll \sum_{(a,\ell)}(1-\delta)^{\ell-\sigma(a)}\left\Vert vw1_{(a,\ell)}-\int_{(a,l)} vw\, d\mu_\Delta\right\Vert_{L^p(\mu_\Delta)}\Vert \hat\kappa\Vert_{L^{\frac p{p-1}}(\mu_\Delta)}\, ,
\end{align*}
%Thus $J_2\ll \sum_{(a,\ell)}(1-\delta)^{\ell-\sigma(a)}\|w\|_\infty|C_{(a,\ell)}|(\mu_{\Delta}((a,\ell)))^{\frac 1p}$ and so
\begin{align*}
J_2&\ll \sum_{(a,\ell)}(1-\delta)^{\ell-\sigma(a)}\|w\|_{\mathcal B_0}|C_{(a,\ell)}|(\mu_{\Delta}((a,\ell)))^{\frac 1b+(\frac 1p-\frac 1{b})}\\
&\ll \left(\sum_{(a,\ell)}(1-\delta)^{\frac b{b-1}(\ell-\sigma(a))}(\mu_{\Delta}((a,\ell)))^{\frac b{b-1}(\frac 1p-\frac 1{b})}\right)^{\frac{b-1}b}\|w\|_{\mathcal B_0}
\left(\sum_{(a,\ell)}|C_{(a,\ell)}|^b(\mu_{\Delta}((a,\ell)))\right)^{\frac 1b}\, .
\end{align*}
It follows that $J_2\ll 
\left(\sum_{a}(1-\delta)^{-\frac b{b-1}\sigma(a)}(\mu_{Y}(a))^{\frac b{b-1}(\frac 1p-\frac 1{b})}\right)^{\frac{b-1}b}\|w\|_{\mathcal B_0}\Vert v\Vert_{L^b(\mu_\Delta)}$ and finally
\[
J_2\ll   \left(\sum_{m\ge 1}(1-\delta)^{-\frac b{b-1}m}\theta_1^{m\frac b{b-1}(\frac 1p-\frac 1{b})}\right)^{\frac{b-1}b}\|w\|_{\mathcal B_0}
\Vert v\Vert_{L^b(\mu_\Delta)}\, .
\]
%Note that, above, we have the bound in $\|w\|_{\mathcal B_0}$ first and only after we can bound by $\Vert v\Vert_{L^b(\mu_\Delta)}$.
\end{proof}

 \paragraph{\bf Acknowledgements.}
The research of FP was partially supported by the IUF, Institut Universitaire de France.
The research of DT was partially supported by EPSRC grant EP/S019286/1.
We wish also to thank the mathematical departments of the Universities of Brest and Exeter for their hospitality.

We thank the referees for their careful reading of the manuscript which helped us to substantially improve the presentation.


\begin{thebibliography}{10}

\bibitem{ADb} J.~Aaronson, M.~Denker.
\emph{A local limit theorem for stationary processes in the domain of attraction of a normal distribution.}  
In N. Balakrishnan, I.A. Ibragimov, V.B. Nevzorov, eds.,
Asymptotic methods in probability and statistics with applications. International conference, St. Petersburg, Russia, 1998, Basel: Birkh{\"a}user,
 (2001) 215--224.


\bibitem{BalintGouezel06}
P.~B{\'a}lint, S.~Gou{\"e}zel. Limit theorems in the stadium billiard.
  \emph{Comm. Math. Phys.} \textbf{263} (2006) 461--512.

\bibitem{Bleher} P.~M.~Bleher. Statistical properties of two-dimensional periodic Lorentz gas with infinite horizon. \emph{J. Stat.
Phys.} \textbf{66}, 1 (1992), 315--373.

\bibitem{BMT18}
H.~Bruin, I.~Melbourne, D.~Terhesiu. {Rates of mixing for nonMarkov
  infinite measure semiflows}. 
{\em Trans.\ Amer.\ Math.\ Soc.\ } {\bf 371} (2019), 7343--7386.

\bibitem{BunimovichSinai:1981}
L.A.~Bunimovich, Ya.G.~Sinai.
Statistical properties of Lorentz gas with periodic configuration of scatterers,
 \emph{Comm. Math. Phys.}, \textbf{78} (1981), 479--497.


\bibitem{BCS91}
L.~A.~Bunimovich, Ya.~G.~Sina\u{\i}, N.~I.~Chernov.
Statistical properties of two-dimensional hyperbolic billiards. (Russian) 
\emph{Uspekhi Mat. Nauk} \textbf{46} (1991), no. 4(280), 43--92, 192; translation in 
\emph{Russian Math. Surveys} \textbf{46} (1991), no. 4, 47--106. 

\bibitem{Chernov99}
N.~Chernov. Decay of correlations and dispersing billiards. \emph{J. Statist.
  Phys.} \textbf{94} (1999) 513--556.

\bibitem{Conze99} J.-P.~Conze. Sur un crit\`ere de r\'ecurrence en dimension 2 pour les marches stationnaires, applications, \emph{Erg. Th., Dyn. Syst.} \textbf{19} (1999) 1233--1245.

\bibitem{DN1} D.~Dolgopyat, P.~N\'andori.
{On mixing and the local central limit theorem for hyperbolic flows}.
  \emph{Erg. Th., Dyn. Sys.} (to appear) arXiv:1710.08568.

\bibitem{DN2} D.~Dolgopyat, P.~N\'andori.
{Infinite measure renewal theorem and related results}.  \emph{Bulletin of London Math. Soc.} (to appear) https://arxiv.org/abs/1709.04074


\bibitem{DNP} D.~Dolgopyat, P.~N\'andori, F.~P\`ene.
{Asymptotic expansion of correlation functions for $\Z^d$ covers of hyperbolic flows},
Preprint://arxiv.org/abs/1908.11504
	
\bibitem{DSV}
D.~Dolgopyat, D.~Sz\'asz, T.~Varj\'u. Recurrence properties of planar Lorentz process. \emph{Duke Math. J.}, \textbf{142} (2008), no. 2, 241--281.

\bibitem{Doos} R. Doss. an elementary proof of Titchmarsh's convolution theorem.
\emph{Proc. AMS}, \textbf{104} (1988), no. 1, 181--184.


\bibitem{DE}
A.~Dvoretzky, P.~Erd\"os. Some problems on random walk in space. 
\emph{Proc. Second Berkeley Sympos. Math}.
Statist. Probab., Berkeley, CA: Univ. California Press (1951) 353--367

\bibitem{GouezelPhD}
S.~Gou{\"e}zel. {Vitesse de d\'ecorr\'elation et th\'eor\`emes limites pour les
  applications non uniform\'ement dilatantes}. {\mbox{Ph.\ D.} Thesis}, Ecole
  Normale Sup\'erieure, 2004.

\bibitem{Gouezel11} S.~Gou{\"e}zel.
{Correlation asymptotics from large deviations in dynamical systems with infinite measure.}
\emph{Colloq. Math.} \textbf{125} (2011) 193--212

\bibitem{Gouezel04b} S.~Gou{\"e}zel. 
{Central limit theorems and stable laws for intermittent maps}.
\emph{Prob.\ Th.\ and Rel.\ Fields} \textbf{1} (2004) 82--122.

\bibitem{Gouezel05}
S.~Gou{\"e}zel. Berry-Esseen theorem and local limit theorem for non uniformly expanding maps. 
\emph{Annales de l'Institut Henri Poincar{\'e}, Prob. and Stat.} \textbf{41} (2005) 997--1024.

\bibitem{Gouezel07}
S.~Gou{\"e}zel. {Statistical properties of a skew product with a curve of
  neutral points}. \emph{Ergodic Theory Dynam. Systems} \textbf{27} (2007)
  123--151.

\bibitem{Gouezel10b}
S.~Gou{\"e}zel. Characterization of weak convergence of {B}irkhoff sums for
  {G}ibbs-{M}arkov maps. \emph{Israel J. Math.} \textbf{180} (2010) 1--41.

\bibitem{KellerLiverani99} G.~Keller, C.~Liverani. 
Stability of the spectrum for transfer operators.
\emph{Annali della Scuola Normale Superiore di Pisa, Classe di Scienze} \textbf{19} (1999) 141--152.

\bibitem{LSV}C. Liverani, B. Saussol, S. Vaienti.
A probabilistic approach to intermittency.
\emph{Ergodic Theory and Dynamical Systems} \textbf{19} (1999) 671--683.

\bibitem{Lorentz} H.~A.~Lorentz, The motion of electrons in metallic bodies, \emph{Koninklijke Nederlandse Akademie van Wetenschappen (KNAW), proceeding of the section of sciences}, \textbf{7}, 2 (1905), 438--593. 

\bibitem{MT12} I.\ Melbourne, D.\ Terhesiu, 
{Operator renewal theory and mixing rates for dynamical
systems with infinite measure,}  
\emph{Invent.\ Math.\ } {\bf 1} (2012) 61--110.


\bibitem{MelTer17} I.\ Melbourne, D.\ Terhesiu. 
{Operator renewal theory for continuous time dynamical systems with finite and infinite measure}. 
{\em Monatsh. Math.} \textbf{182} (2017) 377--431.


\bibitem{pene09IHP} F.~P\`ene. Planar Lorentz process in a random scenery.
\emph {Annales de l'Institut Henri Poincar\'e, Probabilit\'es et Statistiques} \textbf{45} no 3(2009) 818--839 . 

\bibitem{FPBS10} F.~P\`ene,  B.\ Saussol. Back to balls in billiards.
\emph{Comm. Math. Phys.} \textbf{293} (2010) 837--866.

\bibitem{Pene18} F.~P{\`e}ne.
Mixing and decorrelation in infinite measure: the case of the periodic Sinai Billiard.
\emph{Annales de l'Institut Henri Poincar{\'e}, Prob. and Stat.} \textbf{55},  1,  (2019) 378--411

\bibitem{PTho18} F.~P\`ene, D. Thomine. Potential kernel, hitting probabilities and distributional asymptotics.
 \emph {Erg. Theory and Dy. Syst.} (In Press). https://arxiv.org/abs/1702.06625

\bibitem{PTho19a} F.~P\`ene, D. Thomine. Central limit theorems for the $\Z^2$ periodic Lorentz gas.
Israel Journal of Mathematics (to appear). https://arxiv.org/abs/1909.05514



\bibitem{Sarig02}
O.~M. Sarig. {Subexponential decay of correlations.} 
\emph{Invent.\ Math.} \textbf{150} (2002) 629--653.

\bibitem{Schmidt98} K.~Schmidt, On joint recurrence, \emph{C. R. Acad. Sci.} Paris, S\'erie I, \textbf{327} (1998) 837--842.


\bibitem{Sinai70} Ya.~G.~Sinai, Dynamical systems with elastic reflections, 
\emph{Russ. Math. Survey} 25, 1 (1970) 137--189.

\bibitem{SV04} D.~Sz\'{a}sz, T.~Varj\'{u}, 
 Local limit theorem for the Lorentz process and its recurrence in the plane,
\emph{Ergodic Theory Dynam. Systems} \textbf{24} (2004) 254--278.

\bibitem{SV07}
D. Sz{\'a}sz, T. Varj{\'u}. Limit Laws and Recurrence for the Planar Lorents Process with Infinite Horizon. 
\emph{J. Statist. Phys. } \textbf{129} (2007) 59--80.  

\bibitem{Terhesiu16}
D. Terhesiu. Mixing rates for intermittent maps of high exponent.
  \emph{Probab. Theory Related Fields} \textbf{166} (2016) 1025--1060.

\bibitem{Young98} L.-S.\ Young.
Statistical properties of dynamical systems with some hyperbolicity.
\emph{Ann.\ of Math.} {\bf 147} (1998) 585--650.


\end{thebibliography}
\end{document}